%% file: randomLSsdecompose.tex
\documentclass{article}
\usepackage[left=4cm, right=4cm, top=4cm, bottom=2.25cm]{geometry}
\usepackage{amsmath,color}
\usepackage{amsfonts,amsthm}
\usepackage{amssymb,enumerate,enumitem,verbatim}
\usepackage[pdftex]{graphicx}

\usepackage{setspace,scalefnt}
\usepackage[usenames,dvipsnames,svgnames,table]{xcolor}
\usepackage{tikz,caption,subcaption}
\usetikzlibrary{calc}
\usetikzlibrary{snakes}
\input{guff}

\usepackage{hyperref}

\definecolor{darkgreen}{rgb}{0.1,0.7,0.1}

\newcommand\subsetsim{\mathrel{%
  \ooalign{\raise0.2ex\hbox{$\subset$}\cr\hidewidth\raise-0.8ex\hbox{\scalebox{0.9}{$\sim$}}\hidewidth\cr}}}

  \makeatletter
  \newcommand{\labelinthm}[1]{%
     \label{temp#1}
     \protected@write \@auxout {}{\string \newlabel{#1}{{\emph{\ref{temp#1}}}{\thepage}{\emph{\ref{temp#1}}}{temp#1}{}} }%
  }
  \makeatother
  \newcounter{propcounter}

  \newcommand{\llpoly}{\stackrel{\scriptscriptstyle{\text{\sc poly}}}{\ll}}

\newcommand{\eref}[1]{\emph{\ref{#1}}}

\newcommand\claimproofend{\renewcommand{\qedsymbol}{$\boxdot$}
\end{proof}
\renewcommand{\qedsymbol}{$\square$}}

\newcommand\bal{\mathrm{pt}}
\newcommand\balvx{{\mathrm{bal},2}}
\newcommand\balcol{{\mathrm{bal},1}}

\newcommand\epsforabsthatwasepszero{\eps}
\newcommand\mathcalJJJ{\mathcal{J}_0}

\newcommand\col{\mathrm{col}}
\newcommand\partit{\mathrm{pt}}
\newcommand\cov{\mathrm{cov}}
\newcommand\abs{\mathrm{abs}}

\newcommand\forb{\mathrm{forb}}

\newcommand\mult{\mathrm{mult}}

\newcommand\fa{\mathrm{fa}}
\newcommand\tr{\mathrm{tr}}
\newcommand\cM{\mathcal{M}}
\newcommand\nsimAB{\not\sim_{A/B}}
\newcommand\simAB{\sim_{A/B}}
\newcommand\odd{\mathrm{odd}}
\newcommand\even{\mathrm{even}}
\newcommand\image{\mathrm{im}}
\newcommand\wvhp{$1-n^{-\omega(1)}$}

\title{Almost every Latin square has a decomposition into transversals}
\author{Candida Bowtell\thanks{School
of Mathematics, University of Birmingham, Edgbaston, Birmingham, B15 2TT, UK. Research supported by Leverhulme Trust Early Career Fellowship ECF--2023--393. Email: {\tt c.bowtell@bham.ac.uk}}\and Richard Montgomery\thanks{Mathematics Institute, University of Warwick, Coventry, CV4 7AL, UK. Research supported by the European Research Council (ERC) under the European Union Horizon 2020 research and innovation programme (grant agreement No.\ 947978). Email: {\tt richard.montgomery@warwick.ac.uk} }}

\newif\ifabbrev
\abbrevfalse
\newif\ifabstractout
\newif\ifseconeout
\newif\ifbitsoftwoout
\newif\ifbitsofthreeout
\newif\ifsecfourout
\newif\ifsecfiveout
\newif\ifsecsixout
\newif\ifsecsevenout

\abstractouttrue
\seconeouttrue
\bitsoftwoouttrue
\bitsofthreeouttrue
\secfourouttrue
\secfiveouttrue
\secsixouttrue
\secsevenouttrue

\abstractoutfalse
\seconeoutfalse
\bitsoftwooutfalse
\bitsofthreeoutfalse
\secfouroutfalse
\secfiveoutfalse
\secsixoutfalse
\secsevenoutfalse

\begin{document}

\maketitle

\abstract{\ifabstractout\else
In 1782, Euler conjectured that no Latin square of order $n\equiv 2\; \textrm{mod}\; 4$ has a decomposition into transversals. While confirmed for $n=6$ by Tarry in 1900, Bose, Parker, and Shrikhande constructed counterexamples in 1960 for each $n\equiv 2\; \textrm{mod}\; 4$ with $n\geq 10$. We show that, in fact, counterexamples are extremely common, by showing that if a Latin square of order $n$ is chosen uniformly at random then with high probability it has a decomposition into transversals.\fi}

\newpage

\setcounter{tocdepth}{2}
\tableofcontents

\newpage

\section{Introduction}\label{sec:intro}
\input{1intro}


\section{Proof overview and preliminaries}\label{sec:proofsketch}
\input{2proofsketch}


\section{Set-up and division into key lemmas}\label{sec:setup}
\input{3setup}


\section{Part~\ref{partA}: Absorption schematic}\label{sec:absorb}
\input{4absorption}


\section{Random Latin squares and links}\label{sec:rand}
\input{5links}


\section{Part~\ref{partB}: Realisation of the absorption structure}\label{sec:real}
\input{6realisation}


\section{Part~\ref{partC}: Covering, balancing, and the partition of the final edges}\label{sec:balanceandcover}
\input{7coverandbalance}


\bibliographystyle{abbrv}
\bibliography{LSdecomp}

\end{document}

%% file: guff.tex
\newtheorem{lemma}{Lemma}[section]
\newtheorem{corollary}[lemma]{Corollary}
\newtheorem{theorem}[lemma]{Theorem}
\newtheorem{prop}[lemma]{Proposition}

\theoremstyle{definition}
\newtheorem{defn}[lemma]{Definition}

\newtheorem{claim}{Claim}

\theoremstyle{remark}
\newtheorem*{rmk}{Remark}

\newtheorem*{example}{Example}
\newtheorem*{nte}{Note}

\global\long\def\eps{\varepsilon}

\global\long\def\N{\mathbb{N}}

\global\long\def\P{\mathbb{P}}

\global\long\def\cH{\mathcal{H}}

\global\long\def\cI{\mathcal{I}}
\global\long\def\E{\mathbb{E}}

\global\long\def\re{\begin{rmk}}
\global\long\def\mark{\end{rmk}}
\global\long\def\ex{\begin{example}}
\global\long\def\ple{\end{example}}
\global\long\def\no{\begin{nte}}
\global\long\def\ted{\end{nte}}
\global\long\def\en{\begin{compactenum}}
\global\long\def\um{\end{compactenum}}
\global\long\def\li{\begin{compactitem}}
\global\long\def\st{\end{compactitem}}
\global\long\def\de{\begin{defn}}
\global\long\def\fn{\end{defn}}
\global\long\def\cor{\begin{corollary}}
\global\long\def\ary{\end{corollary}}
\global\long\def\lem{\begin{lemma}}
\global\long\def\ma{\end{lemma}}
\global\long\def\arr{\begin{array}}
\global\long\def\ay{\end{array}}
\global\long\def\pr{\begin{proof}}
\global\long\def\oof{\end{proof}}

\global\long\def\gapp{\vspace{1cm}\indent}

%% file: 1intro.tex
\ifseconeout\else
A \emph{Latin square of order $n$} is an $n$ by $n$ grid filled with $n$ symbols so that each row and column contains each symbol exactly once. A transversal in a Latin square of order $n$ is a collection of $n$ cells which share no row, column, or symbol. Latin squares have a long history preceding their modern study; for more on this, we recommend the historical survey by Andersen~\cite{andersen2007history}, while the broader study of transversals in Latin squares is covered in surveys by Wanless~\cite{wanless2011transversals} and Montgomery~\cite{montgomerysurvey}.

In 1782, Euler~\cite{OGeuler} considered: for which $n$ is there a Latin square of order $n$ which can be decomposed into $n$ disjoint transversals? The case $n=4$ was the topic of an old recreational mathematics problem~\cite{ozanam1723recreations}, while Euler was initially particularly interested in the case $n=6$, considering his famous `36 officers problem'. In this problem, there are 36 officers of 6 different ranks from 6 different regiments, with an officer of each rank in each regiment. Can they stand in a 6 by 6 grid so that each row and each column contains officers of different ranks and different regiments? If there were a solution, then, neglecting the ranks, giving each officer the symbol of their regiment will form a Latin square of order $6$. For each rank, the set of officers of that rank marks out a transversal, so that this arrangement would give a decomposition of the Latin square of order 6 into 6 disjoint transversals\footnote{Neglecting the regiments and affixing each officer with the symbol of their rank also gives a Latin square (see Figure~\ref{fig:LSexample}), which is \emph{orthogonal} to the Latin square given by the regiments. That is, all possible $n^2$ pairs of symbols appear in matching row/column pairs of the two Latin squares.  Finding two orthogonal Latin squares of order $n$ is equivalent to the formulation of finding a Latin square of order $n$ which decomposes into transversals. In this paper, we will use the latter form.}.

Euler believed there was no solution to his 36 officer's problem, though this was not confirmed until work by Tarry~\cite{tarry1900probleme} in 1900. More generally, after demonstrating that there are Latin squares of order $n$ which can be decomposed into $n$ disjoint transversals when $n\not\equiv 2\, \textrm{mod}\, 4$, Euler conjectured that there are no examples when $n\equiv 2\,\textrm{mod}\,4$. This is true for $n=2$ and $n=6$, but, in 1959, Bose and Shrikhande~\cite{BS} showed that Euler's conjecture is false by constructing counterexamples for $n=22$ and $n=50$, before, shortly after, showing with Parker~\cite{BPS} that the conjecture is false for every $n\equiv\, 2\,\textrm{mod}\, 4$ with $n\geq 10$.

The development of the probabilistic method has shown the power of considering random objects as potential counterexamples. It is interesting then, to ask how common counterexamples to Euler's conjecture are, and, in particular, whether a random Latin square of order $n\equiv 2\, \textrm{mod}\, 4$ is typically a counterexample? For each $n\in \N$, let $\mathcal{L}(n)$ be the set of Latin squares of order $n$ which use symbols in $[n]=\{1,\ldots,n\}$, and let $L_n$ be drawn uniformly at random from $\mathcal{L}(n)$.
In 1990, van Rees~\cite{vanrees1990subsquares} conjectured that a random Latin square $L_n$ should not have a decomposition into transversals with high probability (whp), however, Wanless and Webb~\cite{wanless2006existence} observed in 2006 that numerical calculations suggest that $L_n$ should have such a decomposition with high probability.

It has long been known that, when $n$ is even, a Latin square of order $n$ may not have even a single transversal (as, for example, seen by the canonical example of the addition table for $\mathbb{Z}_{2m}$, for any $m\in \mathbb{Z}$). However, any Latin square of order $n$ does contain a large partial transversal, that is, a large collection of cells which share no row, column, or symbol. The natural extremal problem on the size of the largest partial transversal that always exists is the topic of the well-known Ryser-Brualdi-Stein conjecture~\cite{Brualdi,Ryser,Stein}, with origins from 1967, which suggests that every Latin square of order $n$ should have a transversal when $n$ is odd, and a partial transversal with $n-1$ cells when $n$ is even. Following a long-standing bound of Shor~\cite{shor} (whose proof was later corrected by Hatami and Shor~\cite{hatamishor}), significant progress towards the Ryser-Brualdi-Stein conjecture has been made in recent years by studying Latin squares from the perspective of edge-coloured graphs (as we do here, and as is described in Section~\ref{subsec:proofsketch}). In particular, after significant progress by Keevash, Pokrovskiy, Sudakov and Yepremyan~\cite{KPSY}, Montgomery~\cite{montgomery2023proof} showed that, for sufficiently large $n$, every Latin square of order $n$ has a partial transversal with $n-1$ cells. This comes very close to a single transversal, while in this paper we wish to determine whether, with high probability, we can find $n$ disjoint transversals in a random Latin square. In every Latin square of order $n$ this is not possible. Indeed, clearly for every even order $n$ we have examples of Latin squares with not even a single transversal, and Wanless and Webb~\cite{wanless2006existence} confirmed the existence of Latin squares which do not have a decomposition into transversals for every order $n>3$. However, some approximate version of this is true. In particular, Montgomery, Pokrovskiy and Sudakov~\cite{montgomery2018decompositions} showed that every Latin square of order $n$ contains $(1-o(1))n$ disjoint partial transversals with $(1-o(1))n$ cells.

This is all to say that every large Latin square has some approximation of the properties we want to find in a random Latin square whp. However, finding these exact properties in a random Latin square whp is surprisingly difficult. For example, it is very challenging to show even that a typical random Latin square contains at least one transversal, and this was proved only in 2020, by Kwan~\cite{kwan2020almost}.
A significant part of the challenge is finding a way to study a random Latin square. Roughly, this can reasonably be pinned to the rigidity of Latin squares; that is, that it is hard to make small modifications to a Latin square to reach another Latin square.

 In \cite{kwan2020almost}, Kwan's main focus was the closely related problem of finding a perfect matching in a uniformly random Steiner triple system of order $n\equiv 3 \,\textrm{mod}\, 6$, using methods that could be adapted for transversals in Latin squares  (see below as well as~\cite{kwan2020almost} for more details on this, and its connection to transversals in Latin squares). Ferber and Kwan~\cite{ferbkwan} subsequently showed that a random Steiner triple system of order $n\equiv 3 \,\textrm{mod}\, 6$ contains disjoint perfect matchings covering all but $o(n^2)$ of its edges. Though they did not do so, similar adaptations to their methods appear capable of showing that a random Latin square of order $n$ has, with high probability, $(1-o(1))n$ disjoint transversals. In this paper, we will show that, in fact, with high probability a random Latin square contains $n$ disjoint transversals. In particular, then, the proportion of Latin squares of order $n\equiv 2\,\mod\, 4$ which provide a counterexample to Euler's conjecture tends to 1 as $n$ tends to infinity.

\fi
 \begin{theorem}\label{thm:mainLSversion}
 A random Latin square of order $n$ has a decomposition into transversals with probability $1-o(1)$.
 \end{theorem}
\ifseconeout\else

Since the result of Kwan~\cite{kwan2020almost}, two alternative methods have been developed to show that a random Latin square of order $n$ has a transversal with high probability, each moreover strengthening this result in different ways. Firstly, Eberhard, Manners, and Mrazovi\'c~\cite{eberhard2023transversals} gave a remarkably tight estimate on the number of transversals in a typical random Latin square of order $n$, using tools from analytic number theory. Then, Gould and Kelly~\cite{gould2023hamilton} developed techniques from their previous work with K\"uhn and Osthus~\cite{gould2022almost} to show that a random Latin square  is likely to contain a particular type of transversal known as a `Hamilton transversal', using more combinatorial methods than~\cite{eberhard2023transversals}, but which are distinctly different to those in the original approach of Kwan~\cite{kwan2020almost}. To prove Theorem~\ref{thm:mainLSversion}, we also take a combinatorial approach. Before discussing this further, we will discuss the connection of our work to resolvable designs.

\begin{figure}[t]
\begin{center}
\includegraphics{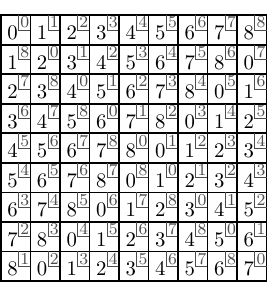}\hspace{2cm}\includegraphics{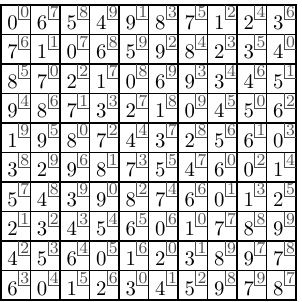}
\end{center}
\caption{Two Latin squares decomposed into transversals. On the left, the addition group of integers $\mod 9$ is given, which is then decomposed into transversals indicated by integers in the top right, starting with the transversal along the leading diagonal (marked by 0) which is then moved to the right by $1 \mod 9$ 8 times to create 8 new transversals. On the right,  Bose, Parker, and Shrikhande's example of a Latin square of order $10$ with a transversal decomposition~\cite{BS}.}\label{fig:LSexample}
\end{figure}

\medskip

\noindent\textbf{Resolvable designs.} Our main result also forms part of the area of hypergraph decompositions, with particularly strong links to resolvable designs. An \emph{$(n,q,r,\lambda)$-design} is a collection $S$ of $q$-element subsets of an $n$-element set $X$ such that every $r$-element subset of $X$ is contained in exactly $\lambda$ sets in $S$. The design, moreover, is called \emph{resolvable}, if $S$ can be partitioned into perfect matchings, that is, collections of vertex-disjoint sets in $S$ which cover every element of $X$. Resolvable designs have a long history, dating back to Kirkman's schoolgirl problem~\cite{kirkman1850note} from 1850 (see~\cite{wilson2003early} for a detailed history). For any given parameters, the existence of a resolvable $(n,q,r,\lambda)$-design requires some simple necessary divisibility conditions (see~\cite{keevash2018existence}). Subject to these, when $r=2$ and $n$ is large, resolvable $(n,q,r,\lambda)$-designs were shown to exist by Ray-Chaudhuri and Wilson~\cite{ray1971solution,ray1973existence} in the 1970's,
while for $r>2$ and $n$ sufficiently large, resolvable $(n,q,r,\lambda)$-designs were shown to exist by Keevash~\cite{keevash2018existence} in 2018, following his revolutionary proof of the existence of designs in~\cite{keevash2014existence}.

A Steiner triple system of order $n$ is an $(n,3,2,1)$-design; they were observed to exist if and only if $n\equiv 1,3 \,\mod \, 6$ by Steiner in 1853 (see~\cite{wilson2003early}). Equivalently, a 3-uniform hypergraph $\mathcal{H}$ is a Steiner triple system if it has $n$ vertices, each pair of which is contained in exactly one edge. A Latin square $L$ of order $n$ is equivalent to a 3-partite 3-uniform hypergraph $\mathcal{H}_L$ with $n$ vertices in each class $A$, $B$ and $C$, representing the rows, columns and symbols respectively, where we add an edge $abc$ exactly when the symbol $c$ appears in $L$ in the cell $(a,b)$.
Then, the conditions for a Latin square imply that every pair of vertices of $\mathcal{H}_L$ from different classes appear in exactly one edge of $\mathcal{H}_L$. Furthermore, a transversal in $L$ corresponds exactly to a perfect matching in $\mathcal{H}_L$, so that $L$ has a decomposition into transversals if and only if $\mathcal{H}_L$ is resolvable. Thus, we have shown that the hypergraph $\mathcal{H}_L$ corresponding to a random Latin square $L$ is with high probability resolvable.

As mentioned above, in 2020 Kwan~\cite{kwan2020almost} showed that if a Steiner triple system of order $n\equiv 3\,\mod \;6$ is chosen uniformly at random then it has a perfect matching with high probability. Subsequently, Ferber and Kwan~\cite{ferbkwan} showed that almost every such Steiner triple system has $(1/2-o(1))n$ disjoint perfect matchings, where these matchings must then use all but $o(n^2)$ of the triples, so that the triple system is thus almost resolvable.
Ferber and Kwan~\cite{ferbkwan} conjectured that, if $n\equiv\, 3\mod 6$, then almost every Steiner triple system of order $n$ is resolvable. That is, that the equivalent result to Theorem~\ref{thm:mainLSversion} should hold for Steiner triple systems.
It would seem that new ideas are needed, however, to show this. In particular, we will use results (discussed in Section~\ref{sec:latinlit}) which do not have a known analogue in the non-partite setting as they follow by counting perfect matchings in bipartite graphs. The techniques used in~\cite{ferbkwan,kwan2020almost} are quite different from those we use here. In \cite{kwan2020almost} a random Latin square is studied by approximating it using a modified random triangle removal process, while \cite{ferbkwan} additionally uses a generalisation of the sparse regularity
lemma to hypergraphs in conjunction with a generalisation to linear hypergraphs of the resolved K\L R conjecture.

\medskip

\noindent\textbf{Our methods.} To prove Theorem~\ref{thm:mainLSversion}, we will construct an intricate absorption structure. We first construct a template independent of the Latin square, and then adapt the template to a randomly chosen Latin square using the semi-random method applied in auxiliary hypergraphs, finding the properties we require to hold whp in a random Latin square using the deletion method and (implicitly) the switching method.
Often, we will require strong recent developments of these techniques, along with further novelties. These techniques are described in detail where appropriate throughout the paper, beginning with an overview of the proof, and the rest of the paper, in Section~\ref{sec:proofsketch}. Let us highlight here, though, some particularly key points about our proof.

The first is that we develop an `absorption schematic' (in Part~\ref{partA} of our proof) which gives a sparse set of possible local corrections that together can make any (reasonable) globally-balanced set of corrections (see Section~\ref{sec:keylemmaone} and Section~\ref{sec:absorb} for full details). This is a template for building an absorber which is independent of our work in random Latin squares, and thus may be useful elsewhere.

Secondly, the switching method can and has been used directly to find small substructures in random Latin squares (see, for example,~\cite{gould2023hamilton}), but instead we will use the deletion method. This was used by Kwan, Sah, and Sawhney~\cite{kwan2022large} to bound above the likely number of certain substructures, and, as well as developing this, we will show how to use the deletion method to bound below the likely number of some particular substructures we will use. This, and its advantages over using the switching method directly, is discussed in Section~\ref{sec:discussiondeletionmethod}.

Finally, let us note here that a major source of the complexity in finding our required absorption structure in a random Latin square via the semi-random method is that it is found in three applications of the semi-random method to an auxiliary hypergraph, the last of which depends on a previous application. That is, we will find part of the absorption structure and require it to satisfy some carefully chosen properties so that we can then apply the semi-random method again to find certain paths connecting up this initial structure. While this requires us to use a forbidding list of properties, and to apply a recent implementation of the semi-random method using weight functions to record desirable properties (see Section~\ref{sec:nibble}), all of these properties will confirm some simple heuristic.
\fi

%% file: 2proofsketch.tex
\ifbitsoftwoout\else
In this section, we will sketch the overall form of our proof of Theorem~\ref{thm:mainLSversion}, before briefly outlining the rest of the paper and then covering some preliminaries. Most notably, these preliminaries include recalling an implementation of the semi-random method (in Section~\ref{sec:nibble}) and the results we will use to prove many of the properties in random Latin squares (in Section~\ref{sec:latinlit}).
\fi

\subsection{Proof sketch}\label{subsec:proofsketch}\ifbitsoftwoout\else
As is now common (and following, for example,~\cite{KPSY}), we will approach Theorem~\ref{thm:mainLSversion} by studying an equivalent formulation in properly coloured graphs. Let $K_{n,n}$ be the complete bipartite graph with vertex classes $A$ and $B$, where $|A|=|B|=n$. A proper colouring of $K_{n,n}$ is a colouring of the edges so that no two edges which share a vertex have the same colour. An optimal colouring is a proper colouring which uses the minimum number of colours among all proper colourings, which, for $K_{n,n}$, is $n$. We will always assume $K_{n,n}$ is properly coloured using colours from $C:=[n]=\{1,\ldots,n\}$.

Given a Latin square $L$ of order $n$ whose rows are indexed by $A$ and columns by $B$, which furthermore uses the symbol set $[n]$, we can define an equivalent optimal colouring of $K_{n,n}$ as follows. For each $a\in A$ and $b\in B$, let the colour of $ab$, denoted by $c(ab)$, be the symbol in the cell of $L$ whose row corresponds to $a$ and whose column corresponds to $b$. That a Latin square has $n$ symbols with no symbol appearing twice in any row or any column immediately implies that this colouring uses $n$ colours and is proper, and thus we have an optimal colouring of $K_{n,n}$. Similarly, a Latin square of order $n$ can be constructed from any optimal colouring of $K_{n,n}$, and thus the optimal colourings of $K_{n,n}$ correspond exactly to the Latin squares of order $n$. Furthermore, it is easy to see that a transversal in a Latin square corresponds exactly under this equivalence to a perfect matching in the corresponding optimally coloured $K_{n,n}$ which has a different colour on each of its edges. We refer to such a matching as a {\it rainbow perfect matching}. Further connections and related problems on rainbow subgraphs can be found in the recent survey by Pokrovskiy~\cite{Alexeysurvey}.

We will show the following equivalent version of Theorem~\ref{thm:mainLSversion}.

\fi
\begin{theorem}\label{thm:main}
Let $G$ be an optimally coloured copy of $K_{n,n}$ chosen uniformly at random from all such colourings. Then, with probability $1-o(1)$, $G$ has a decomposition into rainbow perfect matchings.
\end{theorem}
\ifbitsoftwoout\else

We write $\mathcal{G}^{\col}_n$ for the collection of optimally properly coloured copies of $K_{n,n}$ coloured with colour set $C=[n]$ and write $G \sim G^{\col}_n$ when $G$ is selected uniformly at random from $\mathcal{G}^{\col}_n$. Our aim, then, is to show that $G\sim G^{\col}_n$ with high probability has a decomposition into $n$-edge (perfect) rainbow matchings, $M_1, \ldots, M_n$. We refer to these rainbow matchings as our {\it target} matchings. Using the semi-random method (as, for example, implemented in Latin squares by Montgomery, Pokrovskiy and Sudakov~\cite{montgomery2018decompositions}) it can be shown that any $G\in \mathcal{G}^{\col}_n$ contains $n$ disjoint rainbow matchings of size $(1-o(1))n$. With care, this could be used along with a random partitioning of the remaining edges (somewhat like we do in Section~\ref{sec:C4}) to show that, with high probability, $G\sim G^{\col}_n$ can be decomposed into $n$ rainbow subgraphs ${M}_1,\ldots,{M}_n$, one for each of our $n$ target matchings, which each have $n$ edges and are close to perfect matchings, in that they have maximum degree at most 2 and $(1-o(1))\cdot 2n$ of the vertices have degree~$1$.
Our aim is to take such a relaxed decomposition, and correct it into $n$ perfect rainbow matchings. To do so, we will use methods falling under the overall general technique of `absorption', as codified by R\"odl, Ruci\'nski, and Szemer\'edi~\cite{rodl2006dirac} in 2006. The fundamental idea here is that we should prepare for the corrections we will need to make at the end by initially choosing parts of our random subgraphs to allow later corrections to be made. In particular, this preparation and care at the start ensures that we are able to make a large number of different possible corrections, which subsequently leads to more flexibility in completing to a suitable relaxed decomposition, which we then know can be corrected into perfect rainbow matchings due to the care taken at the start.

To illustrate this, let us give a simple example (see also Figure~\ref{fig:simpleswitcher}).
Suppose a vertex $x$ has degree 2 in $M_1$ and degree 0 in $M_2$ while a vertex $y$ has degree 2 in $M_2$ and degree 0 in $M_1$. If there are matchings $F_1\subset M_1$ and $F_2\subset M_2$ such that $F_1$ and $F_2$ have the same colours and have the same vertex set except that $V(F_1)$ contains $x$ but not $y$ and $V(F_2)$ contains $y$ but not $x$, then we can correct the degrees of $x$ and $y$ in $M_1$ and $M_2$
by switching $F_1$ and $F_2$ between these near-matchings. That is, letting $M_1'=(M_1\setminus F_1)\cup F_2$ and $M_2'=(M_2\setminus F_2)\cup F_1$, we have two subgraphs which still have $n$ edges and are still rainbow, use all the edges in $M_1\cup M_2$ (so are still edge-disjoint from $M_3,\ldots,M_n$), and in which the degrees of all the vertices in $M_1'$ and $M_2'$ are the same as in $M_1$ and $M_2$
except now $x$ and $y$ both have degree 1 in these subgraphs. We call $(F_1,F_2)$ an $\{(1,x),(2,y)\}$-switcher, as it can alter the degree of $x$ and $y$ in $M_1$ and $M_2$ while keeping the edge colours and other vertex degrees the same.

We could ensure that $M_1$ and $M_2$ contain together such a switcher by finding in $G$ an $x,y$-path whose odd and even edges form subgraphs which are rainbow and use the same colour set, and assigning the even edges to $M_1$ and the odd edges to $M_2$ (see Figure~\ref{fig:simpleswitcher} for a shorter example of the switcher we eventually use). As discussed later, some such paths will likely exist when we choose $G\sim G^\col_n$, though they will be rare enough that we will have to specifically construct our matchings to contain such a switcher. Given edge-disjoint rainbow near-matchings $M_1, \ldots, M_n$, $i, j \in [n]$ and either $x,y \in A$ or $x,y \in B$, we define an {\it $\{(i,x),(j,y)\}$-switcher} to be an even length $x,y$-path $P$ consisting of edges $e_1e_2\ldots e_{2s}$ for some $s \in \mathbb{N}$ such that $M_{\odd}:=\{e_{2k-1}: k \in [s]\} \subseteq M_i$ and $M_{\even}:=\{e_{2k}: k \in [s]\} \subseteq M_j$, with $M_{\odd}$ and $M_{\even}$ being rainbow matchings with the same colour set. As described above, such an $x,y$-path $P$ would enable us to switch edges in $M_i$ and $M_j$ such that the updated subgraphs are still rainbow in the same colour set and now in $M_i$ vertex $x$ has degree one less and $y$ has degree one more, and in $M_j$ vertex $y$ has degree one less and vertex $x$ has degree one more. 

\fi


\begin{center}
\begin{figure}
\begin{center}
\begin{tikzpicture}
\def\vheight{0.6}
\def\vwidth{10}
\def\spacer{0.35}
\def\upp{0.2}
\def\thmid{-3.15}
\def\thmst{1.3}
\def\lineinw{-1.8}
\def\dropp{0.8}

\foreach \x in {0,1}
\foreach \y in {0,1}
{
\coordinate (A\x\y) at ($(\x*\vwidth,\y*\vheight)$);
\coordinate (B\x\y) at ($(\x*\vwidth,\y*\vheight)-(0,\dropp)$);
\coordinate (C\x\y) at ($(\x*\vwidth,\y*\vheight)-(0,0.5*\dropp)+(0,1*\y-1*0.5)$);
}

\coordinate (Asmm) at ($0.5*0.71*(A01)+0.5*0.29*(A10)+0.5*0.71*(A00)+0.5*0.29*(A11)$);
\coordinate (Bsmm) at ($(Asmm)-(0,\dropp)$);
\foreach \labb/\coll/\filll/\www/\hhh in {Asmm/blue!30/white/0.35*0.58/0.5,Bsmm/blue!30/white/0.35*0.58/0.5}
{
\def\newwidth{\www*\vwidth}
\def\newheight{\hhh*\vheight}
\begin{scope}[shift={(\labb)}]
\draw [\coll,fill=\filll,thick,rounded corners] ({-\newwidth},0) -- ({-\newwidth},{\newheight}) -- ({\newwidth},{\newheight}) -- ({\newwidth},{-\newheight}) -- ({-\newwidth},{-\newheight}) -- ({-\newwidth},0);
\end{scope}
}

\def\vxsp{0.7}


\coordinate (A1) at ($(Asmm)-(\vxsp,0)$);
\coordinate (A2) at ($(Asmm)$);
\coordinate (A3) at ($(Asmm)+(\vxsp,0)$);

\coordinate (B1) at ($(Bsmm)-0.5*(\vxsp,0)$);
\coordinate (B2) at ($(Bsmm)+0.5*(\vxsp,0)$);
\coordinate (Y) at ($(Bsmm)+1.5*(\vxsp,0)$);
\coordinate (X) at ($(Bsmm)-1.5*(\vxsp,0)$);

\draw [thick,blue] (X) -- (A1);
\draw [thick,blue] (B2) -- (A2);
\draw [thick,red] (B1) -- (A1);
\draw [thick,red] (A3) -- (B2);
\draw [thick,darkgreen] (A2) -- (B1);
\draw [thick,darkgreen] (A3) -- (Y);

\draw ($(X)-(0.2,0)$) node {$x$};
\draw ($(Y)+(0.2,0)$) node {$y$};

\foreach \cccooo in {A1,A2,A3,B1,B2,X,Y}
{
\draw [fill] (\cccooo) circle[radius=0.05];
}
\end{tikzpicture}\;\;\;
\begin{tikzpicture}
\def\vheight{0.6}
\def\vwidth{10}
\def\dropp{0.8}
\coordinate (Asmm) at ($0.5*0.71*(A01)+0.5*0.29*(A10)+0.5*0.71*(A00)+0.5*0.29*(A11)$);
\coordinate (Bsmm) at ($(Asmm)-(0,\dropp)$);
\foreach \labb/\coll/\filll/\www/\hhh in {Asmm/blue!00/white/0.15*0.58/0.5,Bsmm/blue!00/white/0.15*0.58/0.5}
{
\def\newwidth{\www*\vwidth}
\def\newheight{\hhh*\vheight}
\begin{scope}[shift={(\labb)}]
\draw [\coll,fill=\filll,thick,rounded corners] ({-\newwidth},0) -- ({-\newwidth},{\newheight}) -- ({\newwidth},{\newheight}) -- ({\newwidth},{-\newheight}) -- ({-\newwidth},{-\newheight}) -- ({-\newwidth},0);
\end{scope}
}

\def\vxsp{0.7}


\coordinate (A1) at ($(Asmm)-(\vxsp,0)$);
\coordinate (A2) at ($(Asmm)$);
\coordinate (A3) at ($(Asmm)+(\vxsp,0)$);

\coordinate (B1) at ($(Bsmm)-0.5*(\vxsp,0)$);
\coordinate (B2) at ($(Bsmm)+0.5*(\vxsp,0)$);
\coordinate (Y) at ($(Bsmm)+1.5*(\vxsp,0)$);
\coordinate (X) at ($(Bsmm)-1.5*(\vxsp,0)$);

{\draw [thick,blue] (X) -- (A1);}
{\draw [thick,red] (A3) -- (B2);}
{\draw [thick,darkgreen] (A2) -- (B1);}

\draw ($(X)-(0.2,0)$) node {$x$};
\draw ($(Y)+(0.2,0)$) node {$y$};

\draw ($0.0*(Y)+1*(A3)+(0.75,0)$) node {$\in M_i$};

\draw ($0.5*(X)+0.5*(A1)-(0.75,0)$) node {$=$};


\foreach \cccooo in {A1,A2,A3,B1,B2,X,Y}
{
\draw [fill] (\cccooo) circle[radius=0.05];
}
\end{tikzpicture}
\begin{tikzpicture}
\def\vheight{0.6}
\def\vwidth{10}
\def\dropp{0.8}
\coordinate (Asmm) at ($0.5*0.71*(A01)+0.5*0.29*(A10)+0.5*0.71*(A00)+0.5*0.29*(A11)$);
\coordinate (Bsmm) at ($(Asmm)-(0,\dropp)$);
\foreach \labb/\coll/\filll/\www/\hhh in {Asmm/blue!00/white/0.15*0.58/0.5,Bsmm/blue!00/white/0.15*0.58/0.5}
{
\def\newwidth{\www*\vwidth}
\def\newheight{\hhh*\vheight}
\begin{scope}[shift={(\labb)}]
\draw [\coll,fill=\filll,thick,rounded corners] ({-\newwidth},0) -- ({-\newwidth},{\newheight}) -- ({\newwidth},{\newheight}) -- ({\newwidth},{-\newheight}) -- ({-\newwidth},{-\newheight}) -- ({-\newwidth},0);
\end{scope}
}

\def\vxsp{0.7}


\coordinate (A1) at ($(Asmm)-(\vxsp,0)$);
\coordinate (A2) at ($(Asmm)$);
\coordinate (A3) at ($(Asmm)+(\vxsp,0)$);

\coordinate (B1) at ($(Bsmm)-0.5*(\vxsp,0)$);
\coordinate (B2) at ($(Bsmm)+0.5*(\vxsp,0)$);
\coordinate (Y) at ($(Bsmm)+1.5*(\vxsp,0)$);
\coordinate (X) at ($(Bsmm)-1.5*(\vxsp,0)$);

{\draw [thick,blue] (B2) -- (A2);}
{\draw [thick,red] (B1) -- (A1);}
{\draw [thick,darkgreen] (A3) -- (Y);}

\draw ($(X)-(0.2,0)$) node {$x$};
\draw ($(Y)+(0.2,0)$) node {$y$};

\draw ($0.0*(Y)+1*(A3)+(0.75,0)$) node {$\in M_j$};


\draw ($0.5*(X)+0.5*(A1)-(0.75,0)$) node {$+$};

\foreach \cccooo in {A1,A2,A3,B1,B2,X,Y}
{
\draw [fill] (\cccooo) circle[radius=0.05];
}
\end{tikzpicture}
\end{center}
\caption{A simple $\{(i,x),(j,y)\}$-switcher.}\label{fig:simpleswitcher}
\end{figure}
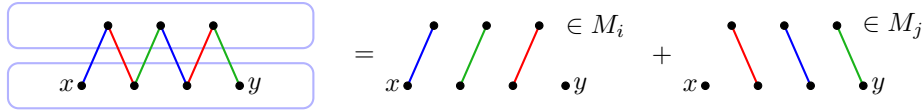
\end{center}


\ifbitsoftwoout\else
These switchers are a simple mechanism to make a small correction to two near-matchings. We wish to keep these switchers as simple as possible as it will be difficult to find many of them in a random optimal colouring of $K_{n,n}$. We will, then, need to build and use these switchers very carefully, due to the following two main considerations.
\begin{itemize}
\item \textbf{We can only create few switchers.} There are around $n^4$ choices for $i,j,u,v$ for which we could find an $\{(i,u),(j,v)\}$-switcher, but the switchers we use need to be chosen edge-disjointly, so we can use at most $\Theta(n^2)$ of them.
\item \textbf{We have to make local alterations in pairs.} Each near-matching $\hat{M}_i$, $i\in [n]$, will have a set of vertices that are not covered, $R_i'$, and a set of vertices that are covered twice, $T_i$. Using switchers, we cannot swap out $v\in T_i$ from $\hat{M}_i$ (i.e., altering its degree from 2 to 1) and swap in $u\in R_i'$ without swapping $u$ out of another near-matching $\hat{M}_j$ and swapping $v$ into the same near-matching, $\hat{M}_j$, so we must find some $j \in [n]$ such that $u \notin R_j'$ and $v \notin T_j$, and that this swap is desirable.
\end{itemize}

To reduce the number of switchers we will need, we will divide our target matchings into `families', and ensure that the corrections we need to make at the end can be done only by switching vertices between target matchings in the same families. Furthermore, we will ensure that the vertices uncovered or covered twice by each near-matching $\hat{M}_i$ belong to sets $R_i$ and $T_i$, respectively (where $R_i$ will function as a `vertex reservoir' in absorption terminology). We will arrange the families into `tribes', where families in the same tribe will help each other to make their corrections by participating in swaps that do not change the degrees or edge colours for that matching. The tribes will make their corrections essentially independently of each other, reducing the number of potential switchers we need to create.

Having created our absorption schematic, we then find the switching paths we need with high probability in the randomly chosen graph $G\sim G^\col_n$.
This will result in rainbow near-matchings $\hat{M}_1',\ldots,\hat{M}_n'$ that we then wish to extend to get $n$-edge perfect rainbow near-matchings $\hat{M}_1, \ldots, \hat{M}_n$. For each $i\in [n]$, in extending to get $\hat{M}_i$ we wish to use exactly all the colours not used on $\hat{M}_i'$, while also covering a set $L_i$ of all the vertices without at least one edge in $\hat{M}_i'$ which are not in $R_i$. The flexibility to do this will come from being able to find edges to cover those vertices in $L_i$ that have their other end point in the reservoir set $R_i$, where $|R_i|$ is much larger than $|L_i|$. Furthermore, we wish to do this so that we can make the final corrections only within each family, so we will need to do these extensions so that the corrections needed for the matchings within each family are suitably balanced.

To summarise, then, we will divide the proof into three mains parts, as follows.

\fi
\begin{enumerate}[label = \textbf{\Alph{enumi}}]
\item {Creating an absorption schematic.} \label{partA}
\item {Realising the absorption schematic by finding disjoint rainbow near-matchings.}\label{partB}
\item {Extending the near-matchings to perfect rainbow near-matchings, so that furthermore the required corrections are balanced within families and any unused vertices lie in the reservoir sets.}  \label{partC}
\end{enumerate}
\ifbitsoftwoout\else

Each of these parts is relatively complex and thus we will sketch our methods in detail before beginning each part. However, for now, we make some brief remarks on some of the ideas we use. For Part~\ref{partA}, we will use ideas from template-based absorption approaches (whose roots lie in the introduction of distributive absorption in~\cite{montgomery2018spanning}) to reduce the number of different switchers we need to create. To decompose the final changes that need to be made to obtain appropriate pairs of switches, we will use methods inspired by work of Barber, K\"uhn, Lo and Osthus~\cite{barber2016edge}.

For Part~\ref{partB}, to find the switching paths we will use the semi-random method as created by R\"odl~\cite{rodlandhisnibble}, using a strong recent implementation by Ehard, Glock and Joos~\cite{EGJnibble}. For this, we will need to find tight bounds on, for example, for each pair of distinct vertices $x,y$ in the same class (i.e., in $A$ or $B$) the likely number of $x,y$-paths of length 62 in $G$ whose odd and even edges use the same 31 (distinct) colours. For the upper-bound on their number that holds with high probability we will use the deletion method of R\"odl and Ruci\'nski~\cite{rodl1995threshold}, following and developing recent work of Kwan, Sah and Sawhney~\cite{kwan2022large}. In contrast to the direct use of the switching method used by, for example, related work of Gould, Kelly, K\"uhn and Osthus~\cite{gould2022almost}, we show how to use the deletion method for the lower bound as well, where more discussion on this approach, and why we take it, can be found in Section~\ref{sec:discussiondeletionmethod}.

For Part~\ref{partC}, we again use the semi-random method along with greedy methods to cover vertices not in the vertex reservoir for each target matching. To balance the corrections needed within each family, we use structures similar to those in Figure~\ref{fig:simpleswitcher}, but we will not need to set them aside beforehand for this purpose. Finally, then, we carefully randomly partition the remaining edges between the matchings.

\medskip \textbf{Paper outline.} In the rest of this section we will describe some of our key notation before covering some other preliminaries. In Section~\ref{sec:setup}, we will give a detailed set-up and divide the proof of Theorem~\ref{thm:main} formally into three key lemmas, representing Parts \ref{partA}, \ref{partB} and \ref{partC} respectively. In Sections~\ref{sec:absorb},~\ref{sec:real} and \ref{sec:balanceandcover} we will carry out Parts~\ref{partA},~\ref{partB} and~\ref{partC} of our proof respectively, whilst in Section~\ref{sec:rand} we obtain bounds on the numbers of switchers with certain fixed colours and vertices, which are crucial for aspects of the proof in both Parts~\ref{partB} and~\ref{partC}.

\fi


\subsection{Notation}\ifbitsoftwoout\else
Much of the notation used throughout this paper is introduced when it first appears, but here we will comment on some notation used throughout the paper. For each $n\in \N$, $K_{n,n}$ is the complete bipartite graph with $n$ vertices in each class, where we use $A$ and $B$ throughout as its two vertex classes. We use $C=[n]=\{1,\ldots,n\}$ throughout as our set of colours, and also use $[n]_0=\{0,1,\ldots,n\}$.
For vertices $x,y\in V(K_{n,n})$, we use $x\simAB y$ to indicate that $x$ and $y$ are in the same vertex class $A$ or $B$, and $x\not\simAB y$ to indicate that $x$ and $y$ are in different vertex classes.
For a set of colours $D \subseteq [n]$, $\mathcal{G}^\col_D$ is the collection of all properly coloured bipartite (simple) graphs with vertex classes $A$ and $B$ which have exactly $n$ edges of each colour in $D$.  Then, $\mathcal{G}^\col_{[n]}$ corresponds to the collection of all Latin squares of order $n$ using the symbols in $[n]$ under the standard equivalence recalled at the start of this section. We write $G \sim G^{\col}_D$ when $G$ is a graph chosen uniformly at random from $\mathcal{G}^\col_D$. When $G\in \mathcal{G}^\col_{[n]}$ and $D\subset [n]$, the graph $G|_D$ is the subgraph of $G$ of edges with colour in $D$.

We will occasionally use multisets, where the elements may occur with repetition. Here, we write $X =_\mult Y$ to require not only that the elements of $X$ and $Y$ are the same, but that the multiplicity of each element is the same. For each edge $e$ in a coloured graph $G$, we write $c_G(e)$ to denote the colour of the edge $e$, often dropping the subscript when it is clear from context. Furthermore, unless stated otherwise, we write $H \subset G$, to mean that $H$ is a subgraph of $G$ which inherits the colouring from $G$. That is, for each $e \in E(H)$ we have that $c_H(e)=c_G(e)$.
An $x,y$-path $P$ of length $\ell$ is a path with $\ell$ edges which has $x$ and $y$ as its endvertices, and we set $\ell(P)=\ell$. We often have an implicit direction on such a path $P$, and when referring to its $k$th edge we count from $x$.

The notation $\llpoly$ is used throughout the paper to compare variables. Where $\alpha \llpoly \beta$, this means that there is some constant $C>0$ which can be chosen so that any required inequalities in the rest of the proof hold if $\alpha\leq \beta^C/C$. For longer hierarchies (as found primarily at \eqref{eq:hierarchy}), these implicit constants are to be chosen from right to left. A detailed overview of this notation can be found in \cite{montgomery2018decompositions}. Less commonly, we also use the more standard notation $\alpha\ll \beta$ where this means that there is some non-negative decreasing function $f$ such that what follows will hold for all $\alpha\leq f(\beta)$. Thus, $\alpha\llpoly \beta$ means that such a function $f$ can be taken to be polynomial in $\beta$. We also use `big-O' notation, as standard. For any $a,b,c\in\mathbb{R}$, we say $a=b\pm c$ if $b-c\leq a\leq b+c$. For any hypergraph $\mathcal{H}$, we use $\Delta^c(\mathcal{H})$ to denote the maximum codegree of $\mathcal{H}$.
\fi


\subsection{Matchings in hypergraphs via the semi-random method}\label{sec:nibble}\ifbitsoftwoout\else
Almost-regular hypergraphs with small codegrees have an almost-perfect matching. This statement summarises a chain of results using the `R\"odl nibble' (also known as the `semi-random method'), that was initiated by R\"odl~\cite{rodlandhisnibble} in 1985, and has since seen a range of qualitative improvements in the variables implicit in `almost'-regular, `small' codegrees and `almost'-perfect.
The `polynomial' bounds we will use result from a sequence of improvements due to Frankl and R\"odl~\cite{frankl1985near}, Pippenger (see~\cite{EGJnibble}), Alon, Kim, and Spencer~\cite{alon1997nearly}, and Kostochka and R\"odl~\cite{kostochka1998partial}.

Beyond good bounds, we will use that the almost-perfect matching found can in fact have a variety of pseudorandom conditions. For example, in a hypergraph $\mathcal{H}$, given a medium-sized vertex set $V$, we may wish the almost-perfect matching to cover almost all of $V$. Alon and Yuster~\cite{alon2005hypergraph} built on work of Pippenger and Spencer~\cite{pippenger1989asymptotic} to give the first result of this kind, showing that the almost-perfect matching could be found to be pseudorandom with respect to many pre-specified vertex sets. We will use the following result of Ehard, Glock, and Joos~\cite{EGJnibble}, which gives good bounds on the various parameters involved, while producing an almost-perfect matching that is pseudorandom with respect to many pre-specified weight functions.
\fi
\begin{theorem}\cite[Theorem~1.2]{EGJnibble}\label{thm:nibble}\label{thm:nibbleorig}
Suppose $\delta\in (0,1)$ and $r\in \N$ with $r\geq 2$, and let $\eps=\delta/50r^2$. Then, there exists $\Delta_0$ such that, for all $\Delta\geq \Delta_0$, the following holds.

Let $\cH$ be an $r$-uniform hypergraph with $\Delta(\cH)\leq \Delta$ and $\Delta^c(\mathcal{H})\leq \Delta^{1-\delta}$ as well as $e(\mathcal{H})\leq \exp(\Delta^{\eps^2})$. Suppose that $\mathcal{W}$
 is a set of at most $\exp(\Delta^{\eps^2})$ weight functions on $E(\mathcal{H})$. Then, there exists a matching $\mathcal{M}$ in $\mathcal{H}$ such that $\omega(\mathcal{M})=(1\pm \Delta^{-\eps})\omega(E(\mathcal{H}))/\Delta$ for all $\omega\in \mathcal{W}$ with $\omega(E(\mathcal{H}))\geq \max_{e\in E(\mathcal{H})}\omega(e)\Delta^{1+\delta}$.
\end{theorem}


\subsection{Results on random Latin squares from switching methods}\label{sec:latinlit}
\ifbitsoftwoout\else
A Latin rectangle of order $n$ with $k$ rows is a $k\times n$ array filled with $n$ symbols so that each symbol appears in each row or column exactly once. Thus, a Latin rectangle of order $n$ with $n$ rows is a Latin square of order $n$, and picking $k$ rows of any Latin square forms a Latin rectangle.
Working in edge-coloured graphs, under the correspondence given at the start of Section~\ref{subsec:proofsketch} to move between Theorem~\ref{thm:mainLSversion} and Theorem~\ref{thm:main}, such a Latin rectangle corresponds to a complete bipartite graph which is properly coloured with $n$ colours before $n-k$ vertices are deleted from one vertex class.
More naturally, we will consider instead here the following equivalence to $k$-regular bipartite graphs with $n$ vertices in each vertex class which are properly coloured with $k$ colours.

Let $D\subset [n]$ be a set of size $k$. Let $A$ and $B$ be our two vertex classes of $n$ vertices. Let $G$ be a $k$-regular bipartite graph with vertex classes $A$ and $B$ which is properly coloured with the colours in $D$.
Let $L(G)$ be the $k\times n$ grid with rows indexed by $D$ and columns indexed by $A$, where we put the symbol $b\in B$ in the cell indexed by $(c,a)$, with $c\in D$ and $a\in A$, exactly if $ab$ is an edge in $G$ with colour $c$. Note that $L$ is a bijection from $\mathcal{G}^\col_D$ to the set of Latin rectangles with rows indexed by $D$, columns indexed by $A$, and the set of symbols given by $B$.

We will use the following result, which is \cite[Theorem~3.3]{kwan2022large} (itself a direct result of McKay and Wanless~\cite[Proposition~4]{mckay1999most}) rephrased equivalently in random optimal colourings (see also~\cite[Proposition~4.4]{gould2023hamilton}).

\begin{theorem}\label{thm:modelswapKSS} Let $D\subset [n]$ and $H,H'\in \mathcal{G}^\col_{D}$. Let $G\sim G_{[n]}^\col$. Then,
\[
\frac{\P(G|_D=H)}{\P(G|_D=H')}=e^{O(n\log^2n)}.
\]
\end{theorem}

Given $a\in A$, $b\in B$, $c\in [n]$ and $G\sim G^\col_{[n]}$, the probability that $ab$ has colour $c$ in $G$ is $1/n$ by symmetry. In other words, if $H$ is the bipartite graph with vertex classes $A$ and $B$ which has only one edge, an edge between $a$ and $b$ with colour $c$, then $\P(H\subset G)=1/n$. As long as $H$ has few edges and is properly coloured with colours in $[n]$, we might hope to show that $\P(H\subset G)$ is close to $(1/n)^{e(H)}$. The rigidity of optimal colourings of $K_{n,n}$ (i.e., the corresponding rigidity of Latin squares), however, makes it difficult to determine this probability.
However, we can have a good bound on the corresponding probability for $\P(H\subset G|_D)$, where $D\subset [n]$ contains all of the colours of the edges of $H$ (and not too many edges have the same colour). The following bound is a direct consequence of Theorem 3.4 in~\cite{kwan2022large} translated into random optimal colouring. The original was, in turn, a direct consequence of a result by Godsil and McKay~\cite[Theorem~4.7]{godsil1990asymptotic}) which was proved using the switching method.

\begin{theorem}\label{thm:fixedsubgraph} Let $\delta\leq 1/10$ and $D=[\delta n]$.
Let $H$ be a properly coloured bipartite graph with vertex classes $A$ and $B$ which uses colours from $D$ in which each colour appears at most $\delta n$ times.

Let $G\sim {G}^\col_D$. Then,
\[
\P(H\subset G)=\left(\frac{1+O(\delta)}{n}\right)^{e(H)}.
\]
\end{theorem}

We will always apply Theorem~\ref{thm:modelswapKSS} and Theorem~\ref{thm:fixedsubgraph} together, and therefore it will be convenient to do this through the following corollary.
\fi
\begin{corollary}\label{cor:latinsquareprobabilities}
Let $\delta\leq 1/10$ and $G\sim {G}^\col_{[n]}$. Let $H$ be a properly coloured bipartite graph with vertex classes $A$ and $B$ which has at most $\delta n$ colours, each of which is in $[n]$ and is used at most $\delta n$ times in the colouring.
Then,
\[
\P(H\subset G)=e^{O(\delta\cdot e(H)+n\log^2n)}\cdot n^{-e(H)}.
\]
\end{corollary}\ifbitsoftwoout\else
\begin{proof} Let $D\subset [n]$ be a set of $\delta n$ colours containing $C(H)$. Let $G'\sim G^\col_D$, so that, by
 Theorem~\ref{thm:fixedsubgraph},
\begin{align}
\P(H\subset G')=\left(\frac{1+O(\delta)}{n}\right)^{e(H)}= e^{O(\delta\cdot e(H))}n^{-e(H)}.\label{eqn:HSGprimenew}
\end{align}
Then, by Theorem~\ref{thm:modelswapKSS}, we have
\begin{align*}
\P(H\subset G)&=\sum_{\hat{G}\in \mathcal{G}^{\col}_D:H\subset \hat{G}}\P(G|_D=\hat{G})
=\sum_{\hat{G}\in \mathcal{G}^{\col}_D:H\subset \hat{G}}\frac{e^{O(n\log^2n)}}{|\mathcal{G}^{\col}_D|}
=e^{O(n\log^2n)}\cdot \P(H\subset G')\\
&\overset{\eqref{eqn:HSGprimenew}}{=}e^{O(\delta\cdot e(H)+n\log^2n)}\cdot n^{-e(H)},
\end{align*}
as required.
\end{proof}
\fi


\subsection{Concentration inequalities}
\ifbitsoftwoout\else We will use the following standard version of Chernoff's bound (see, for example,~\cite{alon2016probabilistic}).\fi

\begin{lemma}\label{chernoff}
Let $n$ be an integer and $0\le \delta,p \le 1$. If $X$ is a binomially or hypergeometrically distributed random variable with mean $\mu=\mathbb{E} [X] = np,$ then
$$\P(X>(1+\delta) \mu) \le e^{-\delta^2\mu/2}\quad\quad\quad \text{ and }\quad\quad\quad \P(X<(1-\delta) \mu) \le e^{-\delta^2\mu/3}.$$
\end{lemma}

\ifbitsoftwoout\else We will also use McDiarmid's inequality (see~\cite[Lemma~1.2]{mcdiarmid1989method}), in the following form.\fi

\begin{lemma}\label{lem:mcdiarmidchangingc} Let $X_1,\ldots,X_m$ be independent random variables taking values in $\mathcal{X}_1,\ldots ,\mathcal{X}_m$ respectively, and let $c_i>0$ for each $i\in [m]$. Let $f:\mathcal{X}^m\to \mathbb{R}$ be a function of $X_1,\ldots,X_m$ such that, for all $i\in [m]$, $x_i'\in \mathcal{X}_i$, and $x_j\in \mathcal{X}_j$ for each $j\in [m]$,
we have
\[
|f(x_1,\ldots,x_{i-1},x_i,x_{i+1},\ldots,x_m)-f(x_1,\ldots,x_{i-1},x_i',x_{i+1},\ldots,x_m)|\leq c_i.
\]
Then, for all $t>0$,
\[
\P(|f(X_1,\ldots,X_m)-\mathbb{E}(f(X_1,\ldots,X_m))|\geq t)\leq 2e^{-\frac{2t^2}{\sum_{i=1}^mc_i^2}}.
\]
\end{lemma}

%% file: 3setup.tex
In this section, we will choose the variables we will use in Section~\ref{sec:variables}, using them in part to partition the vertex/colour/edge sets in Section~\ref{sec:choosevxsets} in preparation for constructing the matchings. We then state three key lemmas corresponding to the three parts of our proof in Sections~\ref{sec:keylemmaone} to \ref{sec:keylemmathree}. The first key lemma is included here to give concrete details of the absorption schematic, but we do not use it directly in this section, applying it to prove the second key lemma later. In Section~\ref{sec:proofofmainthmfromkeylemmas} we use the second and third key lemmas in combination to prove Theorem~\ref{thm:main}.


\subsection{Variables}\label{sec:variables}
Recall that, to prove Theorem~\ref{thm:main}, our aim is to show that a uniformly random choice of an optimally coloured copy of $K_{n,n}$ decomposes into $n$ disjoint rainbow perfect matchings, with high probability.
For each target matching $i\in [n]$, we will have vertex sets $R_i,S_i,T_i,U_i,V_i,W_i,X_i,Y_i,Z_i$, as depicted in Figure~\ref{fig:setpartition}. We now choose variables, where, for example, for each $i\in [n]$, $R_i$ will be chosen (in Section~\ref{sec:choosevxsets}) to be a random vertex set with around $2p_Rn$ vertices. Take the following variables:
\begin{align}\label{eq:hierarchy}
\frac{1}{n} \llpoly p_{\mathrm{tr}},p_{\mathrm{fa}}&\llpoly\eps\llpoly \gamma \llpoly \beta \llpoly p_{\mathrm{cov}}\llpoly p_{\balcol}  \ldots \nonumber\\
&\hspace{2cm}\ldots\llpoly p_{\balvx}\llpoly p_{{\bal}},\alpha \llpoly p_T\llpoly p_U\llpoly p_V \llpoly p_W,\frac{1}{\log n},
\end{align}
where, after setting up some more variables, we will ensure $p_W$ and $p_{{\bal}}$ satisfy two equations (see \eqref{eqn:forpartition}).
Take the following variables (which are stated first for future reference, and then after briefly explained).
\[
p_R=(1+\alpha)p_T\;\;\;\;\;\;\;\;\; p_S=p_U+p_V+p_W\;\;\;\;\;\;\;\;\; p_{\mathcal{I}}=24(p_S-p_R)\;\;\;\;\;\;\;\;\; p_{\mathcal{J}}=72(p_S-p_R),
\]
\[
p_X=(1+\beta)p_T\;\;\;\;\;\;\;\;\;\;\;\;\; p_Y=(1+\beta)121(p_S-p_R)\;\;\;\;\;\;\;\;\;\;\;\;\; p_Z=(1+\beta)\cdot 61\cdot 72(p_S-p_R)
\]
\[
p_1=\frac{2(1+\beta)p_R}{(1+\alpha)(1-p_\bal)}\;\;\;\;\;\;\;\;\;\;\;\;\; p_2=\frac{50(1+\beta)(p_S-p_R)}{1-p_\bal}\;\;\;\;\;\;\;\;\;\;\;\;\;
p_3=\frac{62\cdot 72\cdot (1+\beta)(p_S-p_R)}{1-p_\bal}
\]
\[
p_{\abs}=1-p_{\bal}\;\;\;\;\;\;\;\;\;\;\;\;\;\beta_0=\frac{1}{1+\beta}\;\;\;\;\;\;\;\;\;\;\;\;\;p_{S- R}=p_S-p_R
\]
Furthermore, take the values of $p_W$ and $p_\bal$ appropriately so that
\begin{equation}\label{eqn:forpartition}
p_S+p_X+p_Y+p_Z=1 \;\;\;\;\;\;\text{ and }\;\;\;\;\;\;p_1+p_2+p_3=1,
\end{equation}
where for the first equation this is possible as $p_W$ is at the top of the hierarchy at \eqref{eq:hierarchy} and for the second equation we can set $p_\bal$ so that this holds and then, as we do now, check that $p_\bal$ fits into this hierarchy.
From $p_S+p_X+p_Y+p_Z=1$, we have
\[
p_S+(1+\beta)(1+\alpha)^{-1}p_R+(121+61\cdot 72)\cdot (1+\beta)(p_S-p_R)=1,
\]
so that, from $p_1+p_2+p_3=1$, we have
\begin{align}
(1-p_\bal)&=2(1+\beta)(1+\alpha)^{-1}p_R+(1+\beta)50(p_S-p_R)+(1+\beta)\cdot 62\cdot 72 (p_S-p_R)\nonumber\\
&=1-p_S+(1+\beta)(1+\alpha)^{-1}p_R+(1+\beta)(p_S-p_R)\nonumber\\
&=1-p_S+(1+\beta)p_T+(1+\beta)(p_S-p_R)\nonumber\\
&=1+\beta(p_T+p_S-p_R)+p_T-p_R\nonumber\\
&=1+\beta(p_T+p_S-p_R)-p_T\cdot \alpha,\label{eq:newnewnew}
\end{align}
and hence
\begin{align}\label{eqn:ppt}
p_\bal&=(1\pm \sqrt{\beta})\alpha p_T.
\end{align}
Thus, as $p_R \approx p_T$ and $\beta \llpoly \alpha \llpoly p_T$, it follows that $\beta \llpoly p_\bal, \alpha \llpoly p_T$, as required.

To give some explanation behind these variables (which will perhaps only really make sense when they are used), we note the following. Later, we will construct a set $\mathcal{I}$ representing the switchers we wish to create (as discussed in Section~\ref{subsec:proofsketch}),
and we will have $|\mathcal{I}|\approx p_{\mathcal{I}}n^2$. The set $\mathcal{J}$ with size around $p_{\mathcal{J}}n^2$ will be found in Part~\ref{partB2} and represent simpler switchers, where each element of $\mathcal{I}$ will give rise to 3 elements of $\mathcal{J}$, and
therefore we have chosen $p_{\mathcal{J}}=3p_{\mathcal{I}}$.
As discussed in Section~\ref{sec:proofsketch}, for each $i\in [n]$ the set $R_i$ will be a little larger than $T_i$, and this is why we have chosen $p_R=(1+\alpha)p_T$. For each $i\in[n]$, we will match $T_i$ into $X_i$, and so $X_i$ should be a little larger than $T_i$, and thus we have chosen $p_X=(1+\beta)p_T$.

For each $i\in [n]$ and $u\in S_i\setminus R_i$ (where $|S_i\setminus R_i|\approx 2(p_S-p_R)n$), we will wish to create 24 switchers involving switching $u$ out from the $i$th near-matching, and thus as our instructions will give two such pairs $(i,u)$ to switch between, we have chosen $p_{\mathcal{I}}=24(p_S-p_R)$. For each $i\in [n]$ and $u\in S_i\setminus R_i$, for the $i$th matching we will initially  assign 1 vertex from $Y_i$ common to all the 24 pairs in $\mathcal{I}$ involving $(i,u)$ and 5 other distinct vertices from $Y_i$ to each of these pairs. (For each pair, from the 6 assigned vertices we will take 3 pairs into $\mathcal{J}$, explaining why $p_{\mathcal{J}}=3p_{\mathcal{I}}$). As, for each $i\in [n]$, we should use slightly fewer than $2p_Yn$ vertices (the rough size of $Y_i$), we have set $p_Y=(1+\beta)\cdot (1+24\cdot 5)(p_S-p_R)$. For each pair in $\mathcal{J}$, we will construct a path as in Figure~\ref{fig:simpleswitcher} but with length 62, using internal vertices in $Z_i$ when the $i$th near-matching is involved. For each $i\in [n]$, this will be around $2p_{\mathcal{J}}n$ pairs and we will have $|Z_i|\approx 2p_Zn$, and thus we have chosen $p_Z$ to be a little larger than $61p_{\mathcal{J}}$.

The variables $p_1$, $p_2$ and $p_3$ will be used to partition the colours into sets $D_1$, $D_2$ and $D_3$, where we reserve edges with probability $p_\bal$ for Part~\ref{partC}, and use the remaining edges of each colour set for Parts~\ref{partB1} to \ref{partB3}, respectively.
For Part~\ref{partB1}, we find for each $i\in [n]$ a matching from $T_i$ into $X_i$, and thus use around $2p_Tn^2$ edges in total; thus $(1-p_\bal)p_1$ is a little larger than $2p_T=2p_R\cdot (1+\alpha)^{-1}$.
For Part~\ref{partB2}, for similar reasons to those in our discussion with $Y_i$, for each $i\in [n]$ and $u\in S_i\setminus R_i$ we find one edge and then, additionally, two edges for each pair in $\mathcal{I}$ involving $(i,u)$, for around
 $2(p_{S}-p_{R})n^2+2p_{\mathcal{I}}n^2$ edges in total. Therefore, we have chosen $p_2$ so that $(1-p_\bal)$ is a little larger than  $2(p_{S}-p_R+p_{\mathcal{I}})=50(p_{S}-p_R)$. Finally, for Part~\ref{partB3}, we will find 62 edges for the switching path for each pair in $\mathcal{J}$, for around $62p_{\mathcal{J}}n^2$ edges in total, and thus we have chosen $p_3$ so that $(1-p_\bal)$ is a little larger than $p_3=62p_{\mathcal{J}}=62\cdot 72(p_S-p_R)$.


\subsection{Tribes, families, and partitions of the vertices, colours and edges}\label{sec:choosevxsets}
We now partition our target matchings into families, where the families are grouped into tribes, and partition our vertices, colours, and edges into different sets which are used for different purposes throughout the paper.

\smallskip

\textbf{Tribes and families.} Let $\mathcal{T}$ be a set with size $\lceil p_{\tr}^{-1}\rceil$, which we use to index our tribes, and, for each $\tau \in \mathcal{T}$, let $\mathcal{F}_\tau$ be a set with size $\lceil p_{\fa}^{-1}\rceil$, which we use to index the families of the tribe $\tau$. Partition $[n]$ as equally as possible into $I_{\tau}$, $\tau \in \mathcal{T}$, and then, for each $\tau \in \mathcal{T}$, partition each $I_{\tau}$ as equally as possible into
$I_\phi$, $\phi\in \mathcal{F}_\tau$. 
Let $\mathcal{F}=\bigcup_{\tau \in \mathcal{T}}\mathcal{F}_\tau$ be the set of all of the families.
Note that $|\mathcal{F}|=\lceil p_{\fa}^{-1}\rceil \lceil p_{\tr}^{-1}\rceil=(1\pm \eps^2)p_{\fa}^{-1}p_{\tr}^{-1}$, for each $\tau\in \mathcal{T}$,
$|I_{\tau}|=(1\pm \eps^2)p_{\tr}n$, and, for each $\phi\in \mathcal{F}$, $|I_\phi|=(1\pm \eps^2)p_{\tr}p_{\fa}n$. For this, we have used that $1/n\llpoly p_\tr,p_\fa\llpoly \eps$.

\smallskip

\noindent \textbf{Vertex partitions (see Figure~\ref{fig:setpartition}).} Recall that $A$ and $B$ are disjoint vertex sets with size $n$, which we always use as the vertex classes of our complete bipartite graph. Using \eqref{eqn:forpartition}, independently, for each $\tau\in \mathcal{T}$,
partition $A\cup B$ into $S_\tau$, $X_\tau$, $Y_\tau$, and $Z_\tau$ so that the location of each vertex $v$ is independent and such that
\[
\P(v\in S_\tau)=p_S, \;\;\;\;\; \P(v\in X_\tau)=p_X, \;\;\;\;\; \P(v\in Y_\tau)=p_Y\;\;\text{ and }\;\; \P(v\in Z_\tau)=p_Z.
\]
For each $\tau\in \mathcal{T}$, using that $p_S=p_U+p_V+p_W$, independently, for each $\phi\in \mathcal{F}_\tau$, partition $S_\tau$ into vertex sets $U_\phi$, $V_\phi$ and $W_\phi$ by, for each $v\in S_\tau$, choosing the location of $v$ independently at random so that
\[
\P(v\in U_\phi)=p_U/p_S,\;\;\;\;\; \P(v\in V_\phi)=p_V/p_S,\;\; \text{ and }\;\; \P(v\in W_\phi)=p_W/p_S.
\]
For each $\tau\in \mathcal{T}$, $\phi\in \mathcal{F}_\tau$ and $i\in I_\phi$, take disjoint sets $R_i,T_i\subset U_\phi$ such that, 
for each $v\in U_\phi$, the location of $v$ is chosen independently at random so that
\[
\P(v\in R_i)=p_R/p_U\;\;\;\text{ and }\;\;\;\P(v\in T_i)=p_T/p_U.
\]
For each $\tau\in \mathcal{T}$ and $\phi\in \mathcal{F}_\tau$, let $S_\phi=S_\tau$, $X_\phi=X_\tau$, $Y_\phi=Y_\tau$, and $Z_\phi=Z_\tau$. For each $\tau\in \mathcal{T}$, $\phi\in \mathcal{F}_\tau$ and $i\in I_\phi$, let $S_i=S_\tau$, $X_i=X_\tau$, $Y_i=Y_\tau$, $Z_i=Z_\tau$, $U_i=U_\phi$, $V_i=V_\phi$ and $W_i=W_\phi$.

For each $i\in [n]$, create $X_i=X_{i,0}\cup X_{i,1}$ by, for each $v\in X_i$, independently at random letting $v\in X_{i,0}$ with probability $\beta_0$, and, similarly, create $Y_i=Y_{i,0}\cup Y_{i,1}$ and $Z_i=Z_{i,0}\cup Z_{i,1}$.

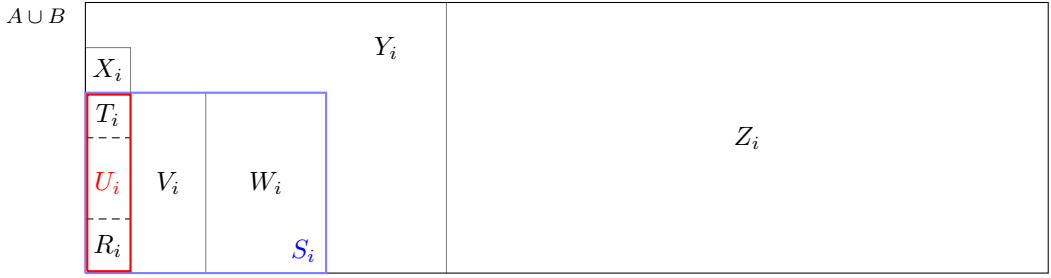
\begin{figure}
\begin{center}
\begin{tikzpicture}

\def\wi{8}
\def\hgt{6}

\draw ($(0,0.57*\hgt)-0.14285*0.4*(\wi,0)-(0.2,0)$) node {\footnotesize $A\cup B$};
\draw [black!50] ($(0.6*\wi,0)$) to ++($(0,0.6*\hgt)$);
\draw [black!50] ($(0,0.4*\hgt)$) to ++($(0.4*\wi,0)$);
\draw [black!50] ($(0.2*\wi,0)$) to ++($(0,0.4*\hgt)$);
\draw [black!50] ($(0.075*\wi,0)$) to ++($(0,0.4*\hgt)$);
\draw [black,densely dashed] ($(0,0.12*\hgt)$) to ++($(0.075*\wi,0)$);
\draw [black,densely dashed] ($(0,0.3*\hgt)$) to ++($(0.075*\wi,0)$);
\draw [black!50] ($(0,0.4*\hgt)$) to ++($(0.075*\wi,0)$) to ++($(0,0.1*\hgt)$) to ++($(-0.075*\wi,0)$) to ++($(0,-0.1*\hgt)$);

\draw [black] ($(0,0)$) to ++($(1.6*\wi,0)$) to ++($(0,0.6*\hgt)$) to ++($(-1.6*\wi,0)$) to ++($(0,-0.6*\hgt)$);
\draw [blue!50,thick] (0,0) to ++ ($(0.4*\wi,0)$) to ++($(0,0.4*\hgt)$) to ++($(-0.4*\wi,0)$) to ++($(0,-0.4*\hgt)$);
\draw [red,thick] (0.025,0.025) to ++($(0.075*\wi-0.025,0)$) to ++($(0,0.4*\hgt-0.05)$) to ++($(-0.075*\wi+0.025,0)$) to ++($(0,-0.4*\hgt+0.05)$);

\draw ($(1.1*\wi,0.3*\hgt)$) node { $Z_i$};
\draw ($(0.5*\wi,0.5*\hgt)$) node { $Y_i$};
\draw ($(0.0375*\wi,0.35*\hgt)$) node { $T_i$};
\draw ($(0.0375*\wi,0.06*\hgt)$) node { $R_i$};
\draw [red]($(0.0375*\wi,0.2*\hgt)$) node { $U_i$};
\draw ($(0.5*0.075*\wi+0.5*0.2*\wi,0.2*\hgt)$) node { $V_i$};
\draw ($(0.3*\wi,0.2*\hgt)$) node { $W_i$};
\draw ($(0,0.4*\hgt)+(0,0.05*\hgt)+(0.0375*\wi,0)$) node { $X_i$};
\draw [blue] ($(0.4*\wi,0)+(-0.3,0.3)$) node { $S_i$};
\end{tikzpicture}
\end{center}\vspace{-0.5cm}
\caption{For each $i\in [n]$, we have a partition of $A\cup B$ into $S_i\cup X_i\cup Y_i\cup Z_i$ (which is the same for individuals $i$ in the same tribe), a partition $S_i=U_i\cup V_i\cup W_i$ (which is the same for individuals $i$ in the same family) and disjoint sets $R_i,T_i\subset U_i$ (which are distinct for each individual $i \in [n]$).}\label{fig:setpartition}
\end{figure}

\smallskip

\noindent \textbf{Colour partitions.}
Let $C=[n]$, so that $C$ is the set of colours we will use. Partition $C=D_1\cup D_2\cup D_3$ by, for each $c\in C$, choosing the location of $c$ independently at random so that
\[
\P(c\in D_1)=p_1,\;\;\;\;\; \P(c\in D_2)=p_2,\;\; \text{ and }\;\; \P(c\in D_3)=p_3.
\]
For each $i\in [n]$, let $C_i\subset C$ be formed by including each colour independently at random with probability $p_\bal$.
For each $j\in [3]$ and $i\in [n]$, let $D_{j,i}\subset D_j$ be formed by including each colour independently at random with probability $1-\beta_0=\beta/(1+\beta)$.

\smallskip

\noindent \textbf{Edge partition.} Let $G\sim G^\col_{[n]}$. Partition $E(G)=E^\bal\cup E^\abs$, by choosing the location of each $e\in E(G)$ independently at random so that $\P(e\in E^\bal)=p_\bal$ and $\P(e\in E^\abs)=p_\abs$ (using that $p_\abs=1-p_\bal$).
Partition $E^\abs=E^\abs_0\cup E^\abs_1$, by choosing the location of $e\in E^\abs$ independently at random so that if $e\in E^\abs$ then $\P(e\in E^\abs_0)=\beta_0$ and $\P(e\in E^\abs_1)=1-\beta_0$. Furthermore, partition $E^\abs_1$ as $E^\abs_{1,A}\cup E^\abs_{1,B}\cup E^\abs_{1,M}$ by choosing the location of each edge independently and uniformly at random.


\subsection{Part~\ref{partA}: Absorption schematic}\label{sec:keylemmaone}
To state our main result for Part~\ref{partA}, we use the following two definitions.
\begin{defn}
Given a collection $\mathcal{C}\subset \{\{(i,u),(j,v)\}:i,j\in [n],i\neq j,u,v\in A\cup B,u\neq v\}$, $i\in [n]$ and $u\in A\cup B$, we say that $(i,u)$ is $(\leq 1)$-balanced in $\mathcal{C}$ if either

\hspace{-0.25cm}$\bullet$ there is exactly one  $(j,v)$ with $\{(i,u),(j,v)\} \in \mathcal{C}$ {and} exactly one  $(v, j)$ with $\{(i,v), (j,u)\} \in \mathcal{C}$, or

\hspace{-0.25cm}$\bullet$ there is no  $(j,v)$ such that $\{(i,u),(j,v)\} \in \mathcal{C}$ {and} no $(v, j)$ such that $\{(i,v), (j,u)\} \in \mathcal{C}$.
\end{defn}

\begin{defn} We say $u\sim_{A/B}v$ if either $u,v\in A$ or $u,v\in B$.
\end{defn}

We now state our key lemma which encapsulates Part~\ref{partA}, which provides an `absorption schematic' which, as mentioned in Section~\ref{subsec:proofsketch}, we will use to tell us which switchers we should find in Part~\ref{partB}. The key lemma is proved in Section~\ref{sec:absorb}.

\begin{lemma}\label{keylemma:absorption}
With high probability, the sets $R_i,S_i,T_i$, $i\in [n]$, set-up as detailed in Sections~\ref{sec:variables} and~\ref{sec:choosevxsets} satisfy the following.

For each $\tau\in \mathcal{T}$, there exists a collection
\begin{equation}\label{eqn:Itaugoodpairs}
\mathcal{I}_\tau\subset \{\{(i,u),(j,v)\}: i,j\in I_\tau,i\neq j, u \in  S_i \setminus (R_i \cup T_j), v \in S_j \setminus (T_i \cup R_j),u\neq v,u{\sim}_{A/B}v\}
\end{equation}
such that the following hold.
\stepcounter{propcounter}
\begin{enumerate}[label = {\emph{\textbf{\Alph{propcounter}\arabic{enumi}}}}]
\item \labelinthm{prop:abs:regularityout} For each $i\in I_\tau$ and $u\in S_i\setminus R_i$, there are exactly $24$ pairs $(j,v)$ such that $\{(i,u),(j,v)\}\in \cI_\tau$.
\item \labelinthm{prop:abs:nocodegree} For each distinct $i,j\in I_\tau$ and $u\in S_i\setminus R_i$, there is at most one $v\in S_j\setminus R_j$ with $\{(i,u),(j,v)\}\in \cI_\tau$.
\item \labelinthm{prop:abs:boundedin}\labelinthm{prop:abs:lowcodegree0} For each $i\in I_\tau$ and $u\in S_i\setminus T_i$, there are at most $n^{1/3}$ pairs $(j,v)$ such that $\{(i,v),(j,u)\}\in \cI_\tau$.
\item \labelinthm{prop:abs:lowcodegree1} For each distinct $i,j\in I_\tau$, there are at most $n^{1/3}$ pairs $(u,v)$ with $\{(i,u),(j,v)\}\in \cI_\tau$.
\item \labelinthm{prop:abs:lowcodegree2} For each distinct $j,j'\in I_\tau$, there are at most $n^{1/3}$ tuples $(i,u,v,v')$ for which we have that $\{(i,u),(j,v)\},\{(i,u),(j',v')\}\in \cI_\tau$.
\item \labelinthm{prop:abs:lowcodegree3} For each $j\in I_\tau$ and $u\in S_j\setminus T_j$ there are at most $n^{1/3}$ pairs $(i,v)$ with $\{(i,u),(j,v)\}\in \cI_\tau$.

\item \labelinthm{prop:abs:corrections} For any collection of sets $R_i'\subset R_i$, $i \in I_\tau$ such that,
for each $i\in [n]$, $|R_i'|=|T_i|$, and, for each $\phi\in \mathcal{F}_\tau$, $\bigcup_{i \in I_\phi} R_i' =_{\mult} \bigcup_{i \in I_\phi} T_i$, there exists $\mathcal{C} \subseteq \mathcal{I}_\tau$ satisfying the following.
\begin{enumerate}[label = {\emph{\textbf{\Alph{propcounter}\arabic{enumi}.\arabic{enumii}}}}]
\item \labelinthm{prop:A:correct1} For every $i \in I_\tau$ and $u \in T_i$, there is exactly one  $(j,v)$ such that $\{(i,u),(j,v)\} \in \mathcal{C}$.
\item \labelinthm{prop:A:correct2} For every $i \in I_\tau$ and $u \in R_i'$, there is exactly one  $(v,j)$ such that $\{(i,v),(j,u)\} \in \mathcal{C}$.
\item \labelinthm{prop:A:correct3} For every $i\in I_\tau$ and $u\in R_i\setminus R_i'$, there is no $(v,j)$ such that $\{(i,v),(j,u)\}\in \mathcal{C}$.
\item \labelinthm{prop:A:correct4} For every $i \in I_\tau$ and $u \in S_i\setminus (R_i \cup T_i)$, $(i,u)$ is $(\leq 1)$-balanced in $\mathcal{C}$.
\end{enumerate}
\end{enumerate}
\end{lemma}


\subsection{Part~\ref{partB}: Realisation of the absorption structure}\label{sec:keylemmatwo}
We now state our key lemma which gives the result of Part~\ref{partB} (using the schematic found in Part~\ref{partA}). It is proved in Section~\ref{sec:real}, using the work in Section~\ref{sec:rand}.

\begin{lemma}\label{keylemma:realisation}
Take the set-up detailed in Sections~\ref{sec:variables} and~\ref{sec:choosevxsets}, where, in particular, we have $G\sim G^\col_{[n]}$ and that the edges of $G$ appear in $E^{\mathrm{abs}}\subset E(G)$ independently at random with probability $p_{\mathrm{abs}}$, while, for each $i\in [n]$, $C_i\subset C$ is a random set of colours where each colour is included independently at random with probability $p_{{\bal}}=1-p_\abs$.

Then, with high probability, there are edge-disjoint subgraphs $\hat{M}_1,\ldots,\hat{M}_n$ in $G[E^{\mathrm{abs}}]$ such that the following hold.
\stepcounter{propcounter}
\begin{enumerate}[label = {\emph{\textbf{\Alph{propcounter}\arabic{enumi}}}}]
\item \labelinthm{prop:real:regularity}
\begin{enumerate}[label = {\emph{\textbf{\alph{enumii})}}}]
\item \labelinthm{prop:B1b}For each $v \in V(G)$ and $\phi\in \mathcal{F}$, there are at most $4\beta p_\tr p_\fa n$ $i \in I_\phi$ such that $v \notin V(\hat{M}_i) \cup R_i$.
\item \labelinthm{prop:B1c}For each $c \in C(G)$ and $\phi\in \mathcal{F}$, there are at most $2\beta p_\tr p_\fa n$ $i \in I_\phi$ such that $c \notin C(\hat{M}_i) \cup C_i$.
\item For each $v \in V(G)$, the degree of $v$ in $G[E^\abs] \setminus (\bigcup_{i \in [n]} \hat{M}_i)$ is at most $2\beta n$.\labelinthm{prop:B1d}
\item For each $i\in [n]$, there are at most $4\beta n$ vertices in $V(G)\setminus R_i$ that have degree $0$ in $\hat{M}_i$.\labelinthm{prop:B1e}
\end{enumerate}
\item \labelinthm{prop:real:vertices} For each $i\in [n]$, every vertex in $R_i$ has degree 0 in $\hat{M}_i$, every vertex in $T_i$ has degree 2 in $\hat{M}_i$, and every other vertex has degree 0 or 1 in $\hat{M}_i$.
\item \labelinthm{prop:real:colours} For each $i\in [n]$, $\hat{M}_i$ is a rainbow subgraph with colours in $C\setminus C_i$.
\item \labelinthm{prop:real:absorption} If there exist edge-disjoint matchings $\tilde{M}_1,\ldots,\tilde{M}_n$ in $G-\hat{M}_1-\ldots-\hat{M}_n$ with the following properties, then $G$ has a decomposition into perfect rainbow matchings.
\begin{enumerate}[label = {\emph{\textbf{\roman{enumii})}}}]
\item For each $i\in [n]$, $\tilde{M}_i$ is vertex disjoint from $\hat{M}_i$ and contains every vertex outside of $R_i$ that has degree 0 in $\hat{M}_i$.\labelinthm{prop:Bi}
\item For each $i\in [n]$, $\hat{M}_i\cup \tilde{M}_i$ is an $n$-edge rainbow subgraph.\labelinthm{prop:Bii}
\item Letting $R_i'=V(G)\setminus V(\hat{M}_i\cup \tilde{M_i})$ for each $i\in [n]$, we have, for each $\tau\in \mathcal{T}$ and $\phi\in \mathcal{F}_\tau$, that $|R_i'|=|T_i|$ and $\bigcup_{i\in I_\phi}R_i'=_{\mult} \bigcup_{i\in I_\phi}T_i$.\labelinthm{prop:Biii}
\end{enumerate}
\end{enumerate}
\end{lemma}


\subsection{Part~\ref{partC}: Covering, balancing, and the partition of the final edges}\label{sec:keylemmathree}
We now state the key lemma for Part~\ref{partC}, which we will use to partition the final edges, and which is proved in Section~\ref{sec:balanceandcover}.

\begin{lemma}\label{keylemma:completion}
Take the set-up detailed in Sections~\ref{sec:variables} and~\ref{sec:choosevxsets}, where, in particular, we have $G\sim G^\col_{[n]}$ and that the edges of $G$ appear in $E^{{\bal}}\subset E(G)$ independently at random with probability $p_{{\bal}}$, while, for each $i\in [n]$, $C_i\subset C$ is a random set of colours where each colour is included independently at random with probability $p_{{\bal}}$.
Then, with high probability, we have the following.

Suppose we have an edge set $\hat{E}\subset E(G)$, and sets $\hat{V}_i\subset V(G)$ and $\hat{C}_i\subset C(G)$, $i\in [n]$, which satisfy the following properties.
\stepcounter{propcounter}
\begin{enumerate}[label = {\emph{\textbf{\Alph{propcounter}\arabic{enumi}}}}]
\item \labelinthm{key3:need0} $E^{{\bal}}\subset \hat{E}$.
\item \labelinthm{key3:extra} For each $i\in [n]$, $|\hat{V}_i|=2|\hat{C}_i|+|T_i|$.
\item \labelinthm{key3:need1} For each $i\in [n]$, $R_i\subset \hat{V}_i$ and $|\hat{V}_i\setminus {R}_i|\leq 4\beta n$.
\item \labelinthm{key3:need2} For each $i\in [n]$, $C_i\subset \hat{C}_i$ and $|\hat{C}_i\setminus C_i|\leq 2\beta n$.
\item For each $v\in V(G)$ and $\phi\in \mathcal{F}$, $|\{i\in I_\phi:v\in \hat{V}_i\setminus R_i\}|\leq 4\beta p_\tr p_\fa n$.\labelinthm{prop:novertexmissingtoomuch}
\item For each $c\in C$ and $\phi\in \mathcal{F}$, $|\{i\in I_\phi:c\in \hat{C}_i\setminus C_i\}|\leq 2\beta p_\tr p_\fa n$.\labelinthm{prop:nocolourmissingtoomuch}
\item For each $v\in V(G)$, $|\{e\in \hat{E}\setminus E^\bal:v\in V(e)\}|\leq 2\beta n$.\labelinthm{prop:novertexintoomanyedges}
\item \labelinthm{key3:need3} For each $v\in V(G)$, $|\{e\in \hat{E}:v\in V(e)\}|=|\{i\in [n]:v\in \hat{V}_i\}|-|\{i\in [n]:v\in T_i\}|$.
\item \labelinthm{key3:need4} For each $c\in C$, $|\{e\in \hat{E}:c(e)=c\}|=|\{i\in [n]:c\in \hat{C}_i\}|$.
\end{enumerate}

Then, $\hat{E}$ can be partitioned into matchings $\tilde{M}_1,\ldots,\tilde{M}_n$ such that the following hold.
\begin{enumerate}[label = {\emph{\textbf{\Alph{propcounter}\arabic{enumi}}}}]\addtocounter{enumi}{9}
\item For each $i\in [n]$, $\tilde{M}_i$ is a rainbow matching with colour set $\hat{C}_i$. \labelinthm{key3:outcome1}
\item For each $i\in [n]$, $\hat{V}_i\setminus R_i\subset V(\tilde{M}_i)\subset \hat{V}_i$.\labelinthm{key3:outcome2}
\item For each $\phi\in \mathcal{F}$, $\bigcup_{i \in I_\phi} R_i\setminus V(\tilde{M}_i) =_{\mult} \bigcup_{i \in I_\phi} T_i$.\labelinthm{key3:outcome3}
\end{enumerate}
\end{lemma}


\subsection{Proof of Theorem~\ref{thm:main} subject to Lemmas~\ref{keylemma:realisation} and~\ref{keylemma:completion}}\label{sec:proofofmainthmfromkeylemmas}
To finish this section, we deduce Theorem~\ref{thm:main} from Lemmas~\ref{keylemma:realisation} and~\ref{keylemma:completion}.
\begin{proof}[Proof of Theorem~\ref{thm:main}]
Take the set-up detailed in Sections~\ref{sec:variables} and~\ref{sec:choosevxsets} with $G\sim G^\col_{[n]}$.
By Lemma~\ref{keylemma:completion}, we have, with high probability, that if $\hat{E}\subset E(G)$, and, for each $i\in [n]$, $\hat{V}_i\subset V(G)$ and $\hat{C}_i\subset C(G)$, are such that \ref{key3:need0}--\ref{key3:need4} hold, then $\hat{E}$ can be partitioned into matchings $\tilde{M}_1,\ldots,\tilde{M}_n$ such that \ref{key3:outcome1}--\ref{key3:outcome3} hold.
Furthermore, by Lemma~\ref{keylemma:realisation}, with high probability there are disjoint subgraphs
$\hat{M}_1,\ldots,\hat{M}_n$ in $G[E^{\mathrm{abs}}]$ such that \ref{prop:real:regularity}--\ref{prop:real:absorption} hold.

For each $i\in [n]$, let $\hat{C}_i=C\setminus C(\hat{M}_i)$ and $\hat{V}_i=V(G)\setminus V(\hat{M}_i)$. Let $\hat{E}$ be the set of edges of $E(G)$ which are not in $\hat{M}_1,\ldots,\hat{M}_n$, so that, as these subgraphs are all in $E^{\abs}$, we have that $E^\bal\subset \hat{E}$ and hence \ref{key3:need0} holds. For each $i\in [n]$, by \ref{prop:real:vertices}, we have
\[
|\hat{V}_i|=2n-|V(\hat{M}_i)|=2n-2|E(\hat{M}_i)|+|T_i|=2n-2|C(\hat{M}_i)|+|T_i|=2|\hat{C}_i|+|T_i|,
\]
and therefore \ref{key3:extra} holds. For each $i\in [n]$, note that, by \ref{prop:real:colours}, we have that $C_i\subset \hat{C}_i$, and, by \ref{prop:real:vertices}, $R_i\subset \hat{V}_i$. Then, combining this with \ref{prop:real:regularity} \ref{prop:B1e}, we have that both \ref{key3:need1} and \ref{key3:need2} hold.

For each $v\in V(G)$ and $\phi\in \mathcal{F}$, by \ref{prop:real:regularity} \ref{prop:B1b},
\begin{align*}
|\{i\in I_\phi:v\in \hat{V}_i\setminus R_i\}|&=|I_\phi|-|\{i\in I_\phi:v\in R_i\}|-|\{i\in I_\phi:v\in V(\hat{M}_i)|\leq 4\beta p_\tr p_\fa n,
\end{align*}
and therefore \ref{prop:novertexmissingtoomuch} holds.
For each $c\in C$, by \ref{prop:real:regularity} \ref{prop:B1c},
\[
|\{i\in I_\phi:c\in \hat{C}_i\setminus C_i\}|=|I_\phi|-|\{i\in I_\phi:c\in C(\hat{M}_i)\}|-|\{i\in I_\phi:c\in C_i\}|\leq 2\beta p_\tr p_\fa n,
\]
and therefore \ref{prop:nocolourmissingtoomuch} holds. Note that \ref{prop:novertexintoomanyedges} follows from \ref{prop:real:regularity} \ref{prop:B1d}.

For each $v\in V(G)$,
\begin{align*}
|\{e\in \hat{E}:v\in V(e)\}|&=n-|\{i\in [n]:v\in V(\hat{M}_i)\}|-|\{i\in [n]:v\in T_i\}|\\
&=|\{i\in [n]:v\in \hat{V}_i\}|-|\{i\in [n]:v\in T_i\}|,
\end{align*}
and thus \ref{key3:need3} holds. Furthermore, for each $c\in C$, we have
\[
|\{e\in \hat{E}:c(e)=c\}|=n-|\{i\in [n]:c\in C(\hat{M_i})\}|=|\{i\in [n]:c\in \hat{C}_i\}|,
\]
and thus \ref{key3:need4} holds.

Therefore, \ref{key3:need0}--\ref{key3:need4} all hold. Then, by Lemma~\ref{keylemma:completion}, $\hat{E}$ can be partitioned into matchings $\tilde{M}_1,\ldots,\tilde{M}_n$ such that \ref{key3:outcome1} -- \ref{key3:outcome3} hold.
From \ref{key3:outcome2}, for each $i\in [n]$, as $\hat{V}_i=V(G)\setminus V(\hat{M}_i)$, we have that $\hat{M}_i$ and $\tilde{M}_i$ are vertex-disjoint, and $\tilde{M}_i$ contains every vertex outside of $R_i$ with degree 0 in $\hat{M}_i$, and therefore \ref{prop:Bi} in \ref{prop:real:absorption} holds. For each $i\in [n]$, by
 \ref{key3:outcome1}, we have $C(\tilde{M}_i)=\hat{C}_i=C\setminus C(\hat{M}_i)$, and thus, as both $\tilde{M}_i$ and $\hat{M}_i$ are rainbow, \ref{prop:Bii} in \ref{prop:real:absorption} holds.
Now, for each $i\in [n]$, let
$R_i'=V(G)\setminus V(\hat{M}_i\cup \tilde{M_i})=R_i\setminus V(\tilde{M}_i)$, so that, by \ref{key3:outcome3},  \ref{prop:Biii} in \ref{prop:real:absorption} holds. Therefore, as
 $\tilde{M}_1,\ldots,\tilde{M}_n$ satisfy \ref{prop:Bi}--\ref{prop:Biii} of \ref{prop:real:absorption}, by \ref{prop:real:absorption}, $G$ has a decomposition into perfect rainbow matchings, as required.
\end{proof}

%% file: 4absorption.tex
In this section, we will prove Lemma~\ref{keylemma:absorption}. In Section~\ref{subsec:sketchabs}, we sketch how we construct our absorption structure, and divide its construction into three parts, which we call Parts~\ref{partA1}--\ref{partA3}. In Section~\ref{subsec:basic}, we give the basic properties we will need for the vertex sets involved in Lemma~\ref{keylemma:absorption}. In Section~\ref{sec:partA1}, we carry out Part~\ref{partA1}. In Section~\ref{sec:propauxgraph}, we construct an auxiliary graph that we will use as a template in our construction. We then carry out Parts~\ref{partA2} and~\ref{partA3} in Sections~\ref{sec:partA2} and~\ref{sec:partA3} respectively, before using this to complete the proof of Lemma~\ref{keylemma:absorption} in Section~\ref{sec:finalpartA}.

\ifsecfourout
\else


\subsection{Sketch of the absorption schematic}\label{subsec:sketchabs}
In this section, and only in this section, we will use colours to index the target matchings to make for easier visualisation. Note that the lemma we wish to prove, Lemma~\ref{keylemma:absorption}, does not involve the colours of the edges of $G\sim G^\col_{[n]}$, or indeed any graph $G$, but only the vertex subsets we have chosen randomly in $A\cup B$. In this section, then, we will consider the target matching $M_i$ to have colour $i$.

We will make the corrections within each tribe independently of the other tribes, so for this sketch let us fix $\tau\in \mathcal{T}$.
The key absorption property we want to develop is \ref{prop:abs:corrections}. The property considers any collection of sets $R_i'\subset R_i$, $i \in I_\tau$, such that, for each $i\in [n]$, $|R_i'|=|T_i|$, and, for each $\phi\in \mathcal{F}_\tau$, $\bigcup_{i \in I_\phi} R_i' =_{\mult} \bigcup_{i \in I_\phi} T_i$.
At the very end of our constructions, this represents that, for each $i\in [n]$, in the $i$th near-matching, every vertex will have degree 1 except for the vertices in $R_i'$ which will have degree 0 and the vertices in $T_i$ which will have degree 2. As discussed in Section~\ref{subsec:proofsketch}, we wish to find a set $\mathcal{C}$ of pairs $\{(i,u),(j,v)\}$ with $i,j\in I_\tau$ and $u,v\in S_\tau$  such that, if for each pair $\{(i,u),(j,v)\}$ we decrease the degree of $u$ and increase the degree of $v$ by $1$ in the $i$th near-matching while making the reverse change in the $j$th near-matching, then we will correct all the near-matchings indexed by $I_\tau$ into actual matchings. These corrections need to be made in pairs so that they can be done without affecting any of the other restrictions.

We will use auxiliary coloured multi-digraphs to represent the changes that this makes (see Figures~\ref{fig:reducecoloursurplus}--\ref{fig:sparsecoloursacrossvertexpair}). For example, in a digraph with vertex set $S_\tau$, we use an edge from $u$ to $v$ with colour $i$ to represent in the $i$th near-matching the decrease of the degree of $u$  by 1 and the increase of the degree of $v$ by 1. Thus, the change wrought by an $\{(i,u),(j,v)\}$-switcher can be represented by a pair of directed edges: an edge $\vec{uv}$ with colour $i$ and an edge $\vec{vu}$ with colour $j$.

In Part~\ref{partA}, our first task is to take any arbitrary collection of sets $R_i'\subset R_i$, $i \in I_\tau$, as described, and find, for each $\phi\in \mathcal{F}_\tau$, a collection $\mathcal{C}_\phi$ of pairs $\{(i,u),(j,v)\}$ with $i,j\in I_\phi$
and $u,v\in U_\phi$ which will, overall, for each $i\in I_\phi$, in the $i$th near-matching increase the degree of each vertex in $R_i'$ by 1 and decrease the degree of each vertex in $T_i$ by 1 without changing any of the other vertex degrees.
This we will do as Part~\ref{partA1}, in Section~\ref{sec:partA1}.
Though used in a very different way, to do this we are inspired by elements of a strategy of Barber, K\"uhn, Lo, and Osthus~\cite{barber2016edge} in their work decomposing complete graphs into copies of a small fixed graph. We can represent the change we wish to make by adding to the vertex set $S_\tau$ an arbitrary directed perfect matching with colour $i$ from $T_i$ into $R_i'$ for each $i\in I_\tau$.
This will have a decomposition of its edges into directed cycles (from the condition $\bigcup_{i \in I_\phi} R_i' =_{\mult} \bigcup_{i \in I_\phi} T_i$ for each $\phi\in \mathcal{F}_\tau$). Inspired by some ideas from \cite{barber2016edge}, we will make changes to the directed matchings chosen, remove some edges, and add additional directed edges to create a decomposition instead into directed 2-cycles, so that the collection of directed edges still makes the same corrections overall. For example, if we have an edge $\vec{uv}$ coloured $i$, and replace it with edges $\vec{uw}$ and $\vec{wv}$ both coloured $i$ (for some other vertex $w$), then the change described at $w$ by these two directed arrows would not result in the change of the degree of $w$ in the $i$th near-matching. We can also add monochromatic directed cycles, which we do for certain 2- and 3-cycles, where the cycles will not change the degree of any of its vertices in the $i$th near-matching.
The operations we used are depicted in Figures~\ref{fig:reducecoloursurplus} and~\ref{fig:replacetriangles}.
After these operations, we need that at every vertex there is at most 1 in-edge and at most 1 out-edge of each colour, as we will only swap a vertex out of or into a near-matching at most once. To aid with this, when replacing a directed colour-$i$ edge from an original matching $T_i$ to $R'_i$ by a directed path with colour $i$, the interior vertices of the path (which will have length 3) will be chosen in $U_i\setminus (T_i\cup R'_i)$.

Ideally, for each pair $\{(i,u),(j,v)\}$ in $\mathcal{C}_\phi$, we would have an $\{(i,u),(j,v)\}$-switcher (see Section~\ref{subsec:proofsketch}) within the $i$th and $j$th matching. However, we have not imposed any condition on which pairs can appear in $\mathcal{C}_\phi$, so there are $\Theta(|I_\phi|^2|U_\phi^2|)=\Theta(p_\tr^2 p_\fa^2p_U^2n^4)$ possible such pairs.
For one family alone, this is many more than the number of switchers we could find edge-disjointly in the $n^2$ edges of $G\sim G^\col_{[n]}$.
Therefore, in Parts~\ref{partA2} and~\ref{partA3}, we develop a much sparser set of pairs $\mathcal{I}_\tau$, such that for any such $\mathcal{C}_\phi$, $\phi\in \mathcal{F}$, we can find a set $\mathcal{C}'\subset \mathcal{I}_\tau$ which makes the same changes overall as $\bigcup_{\phi\in \mathcal{F}_\tau}\mathcal{C}_\phi$.

We do this using ideas from template-based absorption, originating from distributive absorption (as introduced in~\cite{montgomery2018spanning}). A much simplified idea (expressed slightly vaguely) here is the following. Suppose we have a set of $n$ vertices $V$ and can construct some `switcher' between any pair of vertices in $V$.
We could create $\binom{n}{2}$ switchers, and this would allow us to swap any pair of vertices from $V$. However, if we have any connected graph $H$ with vertex set $V$ and create a `switcher' for each $e\in E(H)$, then, for any pair of vertices $x$ and $y$, we could take a path from $x$ to $y$ in $H$ and swap vertices along this path in order to swap $x$ and $y$. Thus, we can swap any pair of vertices using at most $n-1$ switchers. In practice, we often take $H$ to be a sparse, well-expanding graph (for example a sparse random regular graph) so that these paths of swaps are not too long, but so that the graph $H$ still has $O(n)$ edges, and thus require only $O(n)$ switchers. Very roughly, in Parts~\ref{partA2} and~\ref{partA3} we will use two rounds of auxiliary graphs to reduce the number of switchers required for each tribe from $\Theta(p_\tr p_\fa^2p_U^2n^4)$ to $O(p_\tr p_Un^2\log^2n)$. Thus,
in total over all the tribes, the number of switchers we will require is $O(p_Sn^2)$. The choice of the variable $p_U$ will then allow us to fit all the required switchers into the $n^2$ edges of $G$. We will find the auxiliary graphs we use (roughly speaking) in the template role of $H$ in Section~\ref{sec:propauxgraph}.

In Section~\ref{sec:partA2} we develop the auxiliary graph found in Section~\ref{sec:propauxgraph} so that it can be used to take the set $\mathcal{C}_\phi$ and use instead a similar set $\mathcal{C}'_\phi$ which makes the same overall changes but each $\{(i,u),(j,v)\}\in \mathcal{C}$ is only between certain pairs of vertices $(u,v)$. The goal here is to replace each $\{(i,u),(j,v)\}$ by, for some $r\in \N$, a sequence of pairs
\begin{equation}\label{eqn:addverticestopath}
\{(i,u),(j,v_1)\},\{(i,v_1),(j,v_2)\},\{(i,v_2),(j,v_3)\}\ldots, \{(i,v_{r-1}),(j,v_r)\}, \{(i,v_{r}),(j,v)\},
\end{equation}
so that, over all the pairs in $\mathcal{C}'_\phi$, we only use few pairs of vertices that can appear as $(v_i,v_{i+1})$, $(u,v_1)$ or $(v_r,v)$.
This corresponds to replacing the edges $\vec{uv}$ with colour $i$ and $\vec{vu}$ with colour $j$ by a directed $u,v$-path with colour $i$ and a directed $v,u$-path with colour $j$ with the same vertex set (see Figure~\ref{fig:sparsevxpairs}). For each $\phi\in \mathcal{F}_\tau$, the vertices $v_1,\ldots,v_r$ at \eqref{eqn:addverticestopath} will come from $V_\phi$.

In Section~\ref{sec:partA3} we then show that, essentially, we can take the sets $\mathcal{C}'_\phi$, $\phi\in\mathcal{F}_\tau$, and use instead similar sets $\mathcal{C}''_\phi$, $\phi\in \mathcal{F}_\tau$, which make the same overall changes but each $\{(i,u),(j,v)\}\in \bigcup_{\phi\in\mathcal{F}_\tau}\mathcal{C}''_\phi$
is now only between certain pairs of vertices $(u,v)$ and only uses certain pairs of colours $(i,j)$.
Similarly to before, the goal here is to replace each $\{(i,u),(j,v)\}$ by, for some $r\in \N$, a sequence of pairs
\begin{equation}\label{eqn:reducecolpairs}
\{(i,u),(i_1,v)\},\{(i_1,u),(i_2,v)\},\{(i_2,u),(i_3,u)\}\ldots, \{(i_{r-1},u),(i_r,v)\}, \{(i_r,u),(j,v)\},
\end{equation}
so that, over all the pairs in $\mathcal{C}''_\phi$, we only use few pairs of colours that can appear as $(i_j,i_{j+1})$, $(i,i_1)$ or $(i_r,j)$.
This corresponds to considering the edges $\vec{uv}$ with colour $i$ and $\vec{vu}$ with colour $j$, adding a directed edge in both directions between $u$ and $v$ with colour $i_j$ for each $j\in [r]$, and then pairing up these edges as indicated by \eqref{eqn:reducecolpairs} (see also Figure~\ref{fig:sparsecoloursacrossvertexpair}). Where $\phi\in \mathcal{F}_\tau$ is such that $i,j\in I_\phi$, the colours $i_1,\ldots,i_r$ at \eqref{eqn:reducecolpairs} will each come from some $I_{\phi'}$ with $\phi'\in \mathcal{F}_\tau\setminus \{\phi\}$ such that $u,v\in W_{\phi'}$. This is the part of the proof where the families in the same tribe assist each other in making the corrections. We do this as we still want the property in $\mathcal{C}''_\phi$, $\phi\in \mathcal{F}_\tau$, that, working in coloured arrows, we never want to have more than 1 in-edge or more than 1 out-edge of any colour at any vertex.

This will allow us to find the sparse collection of pairs $\mathcal{I}_\tau$ from which we can find pairs to make any of our required corrections. Finally, then, for Part~\ref{partA}, in Part~\ref{partA4}, we add some more pairs to $\mathcal{I}_\tau$ so that when in Part~\ref{partB} we find these switchers this can be done using the semi-random method. For this we need that, for each $i\in I_\tau$ and each $u\in S_i\setminus R_i$, the same number of paths will be found starting at $u$ for the $i$th near-matching, which corresponds to the same number of pairs of $\mathcal{I}_\tau$ containing $(i,u)$ (i.e., that \ref{prop:abs:regularityout} holds). While adding these pairs we ensure that the conditions \ref{prop:abs:boundedin}--\ref{prop:abs:lowcodegree3} continue to hold, where they will hold for the initial sparse collection of pairs by our careful constructions. These conditions are used to ensure low codegrees in certain auxiliary hypergraphs in which we use the semi-random method. 
Part~\ref{partA4} is carried out in Section~\ref{sec:finalpartA}, which completes the proof of Lemma~\ref{keylemma:absorption}.

Where relevant, we include further sketches in this section, but for now we finish with the following summary of the subparts of Part~\ref{partA}.

\begin{enumerate}[label = \textbf{\Alph{enumi}}]
\item Create the absorption schematic, which can be represented by a sparse collection $\mathcal{I}_\tau$ of edge-coloured directed 2-cycles. Then, for any sets $R_i'\subset R_i$, $i\in I_\tau$, with $|R_i'|=|T_i|$ for each $i\in [n]$ and $\bigcup_{i \in I_\phi} R_i' =_{\mult} \bigcup_{i \in I_\phi} T_i$,
\begin{enumerate}[label = \textbf{\Alph{enumi}.\arabic{enumii}}]
\item Find a collection of edge-coloured directed 2-cycles which can make the required changes.\label{partA1}
\item Find such a collection where the 2-cycles only use certain pairs of vertices.\label{partA2}
\item Find such a collection where the 2-cycles only use certain pairs of colours/vertices (i.e., only 2-cycles corresponding to pairs in $\mathcal{I}_\tau$).\label{partA3}
\item Add further pairs to $\mathcal{I}_\tau$ to regularise the schematic.\label{partA4}
\end{enumerate}
\end{enumerate}


\subsection{Basic properties of the vertex partition}\label{subsec:basic}

We will use the following properties of the vertex partitions.

\begin{lemma}\label{lem:setsizesetc}
With high probability, the following all hold.
\stepcounter{propcounter}
\begin{enumerate}[label = {\emph{\textbf{\Alph{propcounter}\arabic{enumi}}}}]
\item \labelinthm{prop:vxpartitionfirst} For each $i\in [n]$ and $X\in \{A,B\}$, $|R_i\cap X|=(1\pm \epsforabsthatwasepszero)p_Rn$, $|S_i\cap X|=(1\pm \epsforabsthatwasepszero)p_Sn$, $|T_i\cap X|=(1\pm \epsforabsthatwasepszero)p_Tn$, $|U_i\cap X|=(1\pm \epsforabsthatwasepszero)p_Un$, $|V_i\cap X|=(1\pm \epsforabsthatwasepszero)p_Vn$, $|W_i\cap X|=(1\pm \epsforabsthatwasepszero)p_Wn$, $|X_i\cap X|=(1\pm \epsforabsthatwasepszero)p_Xn$, $|Y_i\cap X|=(1\pm \epsforabsthatwasepszero)p_Yn$, and $|Z_i\cap X|=(1\pm \epsforabsthatwasepszero)p_Zn$.
\item \labelinthm{prop:verticesinTi} For each $\tau\in \mathcal{T}$, $\phi\in \mathcal{F}_\tau$ and $v\in U_\phi$, $|\{i\in I_\phi:v\in T_i\}|=(1\pm \epsforabsthatwasepszero)p_Tp_U^{-1}p_\tr p_\fa n$.
\item \labelinthm{prop:verticesinUiminusRiTi} For each $\tau\in \mathcal{T}$, $\phi\in \mathcal{F}_\tau$ and $v\in U_\phi$,
\[
|\{i\in I_\phi:v\in U_i\setminus (R_i\cup T_i)\}|\geq (1-\epsforabsthatwasepszero)(1-(p_R+p_T)/p_U)p_\tr p_\fa n\geq (1-\sqrt{p_T})|I_\phi|.
\]
\item \labelinthm{prop:pairsinFtauUVphi} For each $\tau\in \mathcal{T}$ and distinct $u,v\in S_{\tau}$, $|\{\phi \in \mathcal{F}_{\tau}:u,v\in U_\phi \cup V_\phi\}|=(1\pm \epsforabsthatwasepszero)(p_U+p_V)^2p_S^{-2}p_{\fa}^{-1}$.
\item \labelinthm{prop:pairsinStauWphi} For each $\tau\in \mathcal{T}$ and distinct $u,v\in S_{\tau}$, $|\{i\in I_{\tau}:u,v\in W_i\}|=(1\pm \epsforabsthatwasepszero)p_W^2p_S^{-2}p_{\tr}n$.
\item \labelinthm{prop:sizeofsubsetVphiWphi'} For each $\tau\in \mathcal{T}$ and distinct $\phi, \phi' \in \mathcal{F}_{\tau}$, $|(U_\phi \cup V_\phi) \cap W_{\phi'}|=2(1 \pm \epsforabsthatwasepszero)(p_U+p_V)p_Wp_S^{-1}n$.
\item \labelinthm{prop:pairsuvinVphiWphi'} For each $\tau\in \mathcal{T}$ and distinct $\phi, \phi' \in \mathcal{F}_{\tau}$, $|\{\{u,v\} \subset S_{\tau}: u \sim_{A/B} v,~ u,v \in (U_{\phi} \cup V_{\phi}) \cap W_{\phi'}\}|=(1 \pm \epsforabsthatwasepszero)(p_U+p_V)^2p_W^2p_S^2n^2/2$.
\item \labelinthm{prop:numberofUVphicontainingu} For each $u \in S_{\tau}$, $|\{\phi \in \mathcal{F}_{\tau}: u \in (U_{\phi} \cup V_{\phi})\}|=(1 \pm \epsforabsthatwasepszero)(p_U+p_V)p_{\fa}^{-1}$.
\item \labelinthm{prop:pairsphiphi'withuinVphiWphi'} For each $u \in S_{\tau}$, $|\{(\phi, \phi'):\phi,\phi' \in \mathcal{F}_{\tau}, u \in (U_{\phi} \cup V_{\phi}) \cap W_{\phi'}\}|=(1 \pm \epsforabsthatwasepszero)(p_U+p_V)p_Wp_{\fa}^{-2}/2$.\labelinthm{prop:vxpartitionlast}
\end{enumerate}
\end{lemma}
\begin{proof} Each of these properties holds with high probability by an application of Lemma~\ref{chernoff} and a union bound. To avoid undue repetition, we will only prove a sample of these explicitly.

\smallskip

\noindent \ref{prop:vxpartitionfirst} for $R_i$, $i\in [n]$: Let $\tau\in \mathcal{T}$, $\phi\in \mathcal{F}_\tau$ and $i\in I_\phi$. For each $v\in A\cup B$, by the partitioning in Section~\ref{sec:choosevxsets},
\[
\P(v\in R_i)=\P(v\in R_i|v\in U_i)\cdot \P(v\in U_\phi|v\in S_\tau)\cdot \P(v\in S_\tau)=(p_R/p_U)\cdot(p_U/p_S)\cdot p_S=p_R.
\]
For each $X\in \{A,B\}$, as $|X|=n$ and $1/n\llpoly \epsforabsthatwasepszero,p_R$, by Lemma~\ref{chernoff}, with probability $1-\exp(-\omega(\log n))$ we have $|R_i\cap X|=(1\pm \epsforabsthatwasepszero)p_Rn$. Thus, by a union bound, with high probability, $|R_i\cap X|=(1\pm \epsforabsthatwasepszero)2p_Rn$ for each $i\in [n]$.

\smallskip

\noindent \ref{prop:pairsinFtauUVphi}: Let $\tau\in \mathcal{T}$ and let $u,v\in S_{\tau}$ be distinct. For each $\phi\in \mathcal{F}_\tau$, $\P(u,v\in U_\phi \cup V_\phi)=(p_U+p_V)^2/p_S^2$. As $|\mathcal{F}_\tau|=(1\pm \epsforabsthatwasepszero^2)p_\fa^{-1}$, and $p_\tr\llpoly \epsforabsthatwasepszero,p_U,p_V\llpoly p_W\llpoly 1/\log n$, and $p_S\geq p_W$, we have that,
by Lemma~\ref{chernoff}, with probability $1-\exp(-\omega(\log n))$ we have $|\{\phi \in \mathcal{F}_{\tau}:u,v\in U_\phi \cup V_\phi\}|=(1\pm \epsforabsthatwasepszero)(p_U+p_V)^2p_{\fa}^{-1}$. Thus,  by a union bound, with high probability, \ref{prop:pairsinFtauUVphi} holds.
\end{proof}


\subsection{Part~\ref{partA1}: Initial 2-cycle decomposition}\label{sec:partA1}

In Part~\ref{partA1}, we prove Lemma~\ref{lem:partA1}. After stating the lemma, we discuss it from the perspective of the auxiliary coloured directed graph discussed in Section~\ref{subsec:sketchabs}.

\begin{lemma}\label{lem:partA1} Let $R_i,T_i,U_i$, $i\in [n]$, satisfy \eref{prop:vxpartitionfirst}--\eref{prop:vxpartitionlast}.
Let $\tau\in \mathcal{T}$ and $\phi\in \mathcal{F}_\tau$. For each $i\in I_\phi$, let $R_i'\subset R_i$ satisfy $|R_i'|=|T_i|$, and suppose that $\bigcup_{i \in I_\phi} R_i' =_{\mult} \bigcup_{i \in I_\phi} T_i$.

Then, there exists a set
\begin{equation}\label{eq:Csitsokay}
\mathcal{C}_\phi\subset \{\{(i,u),(j,v)\}:i,j\in I_\phi,i\neq j,u\in U_\phi\setminus (R_i\cup T_j)\text{ and }v\in U_\phi\setminus (T_i\cup R_j),u\neq v\}.
\end{equation}
such that the following hold.

\stepcounter{propcounter}
\begin{enumerate}[label = {\emph{\textbf{\Alph{propcounter}\arabic{enumi}}}}]
\item \labelinthm{conc:A1-1} For every $i \in I_\phi$ and $u \in T_i$, there is exactly one  $(v,j)$ such that $\{(i,u),(j,v)\} \in \mathcal{C}_\phi$.
\item \labelinthm{conc:A1-2} For every $i \in I_\phi$ and $u \in R_i'$ there is exactly one  $(v,j)$ such that $\{(i,v),(j,u)\} \in \mathcal{C}_\phi$.
\item \labelinthm{conc:A1-2b} For every $i\in I_\phi$ and $u\in R_i\setminus R_i'$ there is no $(v,j)$ such that $\{(i,v),(j,u)\}\in \mathcal{C}_\phi$.
\item \labelinthm{conc:A1-3} For every $i \in I_\phi$ and $u \in U_i\setminus (R_i \cup T_i)$,  $(i,u)$ is $(\leq 1)$-balanced in $\mathcal{C}_\phi$.
\end{enumerate}
\end{lemma}

Consider the set $\mathcal{C}_\phi$ in Lemma~\ref{lem:partA1} and, for each $\{(i,u),(j,v)\}\in \mathcal{C}_\phi$, add $\vec{uv}$ with colour $i$ and $\vec{vu}$ with colour $j$ to create the auxiliary coloured multi-digraph $D''$. Note that this has a natural decomposition into directed 2-cycles. Then, for each $i\in I_\phi$ and $u\in U_i$, using $d_{D''}^{+,i}(u)$ and $d_{D''}^{-,i}(u)$ as the out- and in-degree of $u$ in the colour-$i$ edges in $D''$ respectively, the following table shows the degrees around $u$ according to \ref{conc:A1-1}--\ref{conc:A1-3}.
\begin{equation}\label{eq:cyclebalance}
\begin{array}{l|c}
&\big(d_{D''}^{+,i}(u),d_{D''}^{-,i}(u)\big)\\
\textcolor{white}{.}\vspace{-0.3cm}&\\%
\hline
u\in T_i & (1,0)\\
u\in R_i' & (0,1)\\
u\in R_i\setminus R_i' & (0,0)\\
u\in U_i\setminus (R_i \cup T_i) & (0,0)\text{ or }(1,1)
\end{array}
\end{equation}
Thus, the directed arrows of colour $i$ correspond to the required corrections of the degrees of $u$ in the $i$th near-matching if $u\in R_i'\cup T_i$, while not affecting the degree of $u$ if $u\in U_i\setminus (R_i' \cup T_i)$.

To find $D''$ in the proof of Lemma~\ref{lem:partA1}, we start by finding a similar coloured multi-digraph which has a directed cycle decomposition, for which, for each $i\in I_\phi$, the vertices in $U_i\setminus (R_i \cup T_i)$ are in no edges with colour $i$ (and the corresponding version of \eqref{eq:cyclebalance} holds). Via some maximalisation in its construction, we show that in fact this will be a rainbow cycle decomposition. We then add directed edges to $D$ to get $D'$ which satisfies the corresponding version of \eqref{eq:cyclebalance}, but has a decomposition into rainbow triangles/2-cycles. Then, we take each rainbow triangle in $D'$ and replace it with some 2-cycles, creating $D''$ while ensuring \eqref{eq:cyclebalance} is still satisfied. These operations are depicted in Figures~\ref{fig:reducecoloursurplus} and~\ref{fig:replacetriangles}.

\begin{proof}[Proof of Lemma~\ref{lem:partA1}]
For each $i\in I_\phi$, using that $|T_i|=|R_i'|$, let $F_i$ be an arbitrary perfect matching between $T_i$ and $R_i'$ with edges directed from $T_i$ to $R_i'$ which each have colour $i$.
Let $D$ be the directed edge-coloured multigraph with vertex set $U_\phi$ and edge set $\bigcup_{i\in I_\phi}F_i$.

\smallskip

\noindent\textbf{Initial cycle decomposition.} Note that, for each $v\in U_\phi$, as $\bigcup_{i \in I_\phi} R_i' =_{\mult} \bigcup_{i \in I_\phi} T_i$, we have
\[
d^+_D(v)=|\{i\in I_\phi:v\in T_i\}|=|\{i\in I_\phi:v\in R_i'\}|=d^-_D(v).
\]
Therefore, as is well-known, $D$ has a decomposition of its edges into edge-disjoint directed cycles. Take such a decomposition, $C_1,\ldots,C_r$ say.

\smallskip

\noindent\textbf{Rainbow cycle decomposition.} We claim that there is a choice of the matchings $F_i$, $i\in I_\phi$, for which there is a cycle decomposition of $D$ in which every cycle is rainbow, that is, no colour appears more than once on any one cycle in the decomposition. Subject to the constraints so far, then, choose $F_i$, $i\in I_\phi$, $r\in \N$, and $C_1,\ldots,C_r$, to maximise $r$.
Suppose, for contradiction, that there is some cycle $C_j$, for some $j\in [r]$, which is not rainbow. Then, let $\vec{u_1u_2}$, $\vec{u_3u_4}$ be two edges of $C_j$ which have the same colour, $i$ say. We thus have that $u_1,u_3\in T_i$ and $u_2,u_4\in R_i'$. Note that replacing   $\vec{u_1u_2}$, $\vec{u_3u_4}$ in $F_i$ by  $\vec{u_1u_4}$, $\vec{u_3u_2}$ and replacing $C_j$ by the
two cycles in $C_j-\vec{u_1u_2}-\vec{u_3u_4}+\vec{u_1u_4}+\vec{u_2u_3}$ will give a choice of the directed matchings with $r+1$ directed cycles (see Figure~\ref{fig:reducecoloursurplus}a)), a contradiction.

\begin{center}
\begin{figure}[b]
\begin{center}
\begin{tikzpicture}
\def\vrad{0.6}
\def\vwidth{10}
\def\spacer{0.35}
\def\upp{0.2}
\def\thmid{-3.15}
\def\thmst{1.3}
\def\lineinw{-1.8}
\def\dropp{0.8}

\def\vxsp{0.7}

\def\midvec(#1,#2){
($0.9*(#1)+0.1*(#2)$)--($0.1*(#1)+0.9*(#2)$)
}


\foreach \n in {1,2,3}
{
\coordinate (A\n) at ($(120*\n+180:\vrad)$);
\draw [fill] (A\n) circle[radius=0.05];
}
\foreach \n in {1,2,3}
{
\coordinate (B\n) at ($2.5*(\vrad,0)+(120*\n:\vrad)$);
\draw [fill] (B\n) circle[radius=0.05];
}
\draw ($(A2)+(-1.2,0.2)$) node {a)};

\draw ($(A2)+(0,0.2)$) node {$u_1$};
\draw ($(B1)+(0,0.2)$) node {$u_2$};
\draw ($(B2)+(0,-0.25)$) node {$u_3$};
\draw ($(A1)+(0,-0.25)$) node {$u_4$};
\draw ($0.25*(A1)+0.25*(A2)+0.25*(B1)+0.25*(B2)$) node {$C_j$};

\draw [thick,blue,->] \midvec(A1,A3);
\draw [thick,red,->] \midvec(A3,A2);
\draw [thick,blue,->] \midvec(B1,B3);
\draw [thick,darkgreen,->] \midvec(B3,B2);
\draw [thick,orange,->] \midvec(A2,B1);
\draw [thick,orange,->] \midvec(B2,A1);
\end{tikzpicture}
\begin{tikzpicture}
\draw [white](0,-0.955) -- (0,0);
\draw (0,0) node {$\implies$};
\end{tikzpicture}
\begin{tikzpicture}
\def\vrad{0.6}
\def\vwidth{10}
\def\spacer{0.35}
\def\upp{0.2}
\def\thmid{-3.15}
\def\thmst{1.3}
\def\lineinw{-1.8}
\def\dropp{0.8}

\def\vxsp{0.7}

\def\midvec(#1,#2){
($0.9*(#1)+0.1*(#2)$)--($0.1*(#1)+0.9*(#2)$)
}


\foreach \n in {1,2,3}
{
\coordinate (A\n) at ($(120*\n+180:\vrad)$);
\draw [fill] (A\n) circle[radius=0.05];
}
\foreach \n in {1,2,3}
{
\coordinate (B\n) at ($2.5*(\vrad,0)+(120*\n:\vrad)$);
\draw [fill] (B\n) circle[radius=0.05];
}

\draw ($(A2)+(0,0.2)$) node {$u_1$};
\draw ($(B1)+(0,0.2)$) node {$u_2$};
\draw ($(B2)+(0,-0.25)$) node {$u_3$};
\draw ($(A1)+(0,-0.25)$) node {$u_4$};

\draw [thick,blue,->] \midvec(A1,A3);
\draw [thick,red,->] \midvec(A3,A2);
\draw [thick,blue,->] \midvec(B1,B3);
\draw [thick,darkgreen,->] \midvec(B3,B2);
\draw [thick,orange,->] \midvec(A2,A1);
\draw [thick,orange,->] \midvec(B2,B1);
\end{tikzpicture}\;\;\;\;\;\;\begin{tikzpicture}
\draw [dashed](0,-0.55) -- (0,1.55);
\end{tikzpicture}\;\;\;\;\;
\begin{tikzpicture}
\def\vrad{0.6}
\def\vwidth{10}
\def\spacer{0.35}
\def\upp{0.2}
\def\thmid{-3.15}
\def\thmst{1.3}
\def\lineinw{-1.8}
\def\dropp{0.8}

\def\vxsp{0.7}

\def\midvec(#1,#2){
($0.9*(#1)+0.1*(#2)$)--($0.1*(#1)+0.9*(#2)$)
}


\foreach \n in {1,2,3}
{
\coordinate (A\n) at ($(120*\n+180:\vrad)$);
\draw [fill] (A\n) circle[radius=0.05];
}
\foreach \n in {1,2,3}
{
\coordinate (B\n) at ($2.5*(\vrad,0)+(120*\n:\vrad)$);
\draw [fill] (B\n) circle[radius=0.05];
}

\draw ($(A2)+(-1.2,0.2)$) node {b)};

\draw ($(A2)+(0,0.2)$) node {$v_1$};
\draw ($(B1)+(0,0.2)$) node {$v_2$};
\draw ($(B2)+(0,-0.25)$) node {$v_3$};
\draw ($(A1)+(0,-0.25)$) node {$v_4$};

\draw [thick,blue,->] \midvec(A1,A3);
\draw [thick,red,->] \midvec(A3,A2);
\draw [thick,violet,->] \midvec(B1,B3);
\draw [thick,darkgreen,->] \midvec(B3,B2);
\draw [thick,brown,->] \midvec(A2,B1);
\draw [thick,orange,->] \midvec(B2,A1);

\draw [dotted] (A1) -- (A2) -- (B2) -- (B1);

\end{tikzpicture}
\begin{tikzpicture}
\draw [white](0,-0.955) -- (0,0);
\draw (0,0) node {$\implies$};
\end{tikzpicture}
\begin{tikzpicture}
\def\vrad{0.6}
\def\vwidth{10}
\def\spacer{0.35}
\def\upp{0.2}
\def\thmid{-3.15}
\def\thmst{1.3}
\def\lineinw{-1.8}
\def\dropp{0.8}

\def\vxsp{0.7}

\def\midvec(#1,#2){
($0.9*(#1)+0.1*(#2)$)--($0.1*(#1)+0.9*(#2)$)
}


\foreach \n in {1,2,3}
{
\coordinate (A\n) at ($(120*\n+180:\vrad)$);
\draw [fill] (A\n) circle[radius=0.05];
}
\foreach \n in {1,2,3}
{
\coordinate (B\n) at ($2.5*(\vrad,0)+(120*\n:\vrad)$);
\draw [fill] (B\n) circle[radius=0.05];
}

\draw ($(A2)+(0,0.2)$) node {$v_1$};
\draw ($(B1)+(0,0.2)$) node {$v_2$};
\draw ($(B2)+(0,-0.25)$) node {$v_3$};
\draw ($(A1)+(0,-0.25)$) node {$v_4$};

\draw [thick,blue,->] \midvec(A1,A3);
\draw [thick,red,->] \midvec(A3,A2);
\draw [thick,violet,->] \midvec(B1,B3);
\draw [thick,darkgreen,->] \midvec(B3,B2);
\draw [thick,brown,->] \midvec(A2,B1);
\draw [thick,orange,->] \midvec(B2,A1);

\draw [thick,violet,->] ($0.9*(A2)+0.1*(A1)+(-0.05,0)$)--($0.1*(A2)+0.9*(A1)+(-0.05,0)$);
\draw [thick,violet,->] ($0.9*(A1)+0.1*(A2)+(0.05,0)$)--($0.1*(A1)+0.9*(A2)+(0.05,0)$);

\draw [thick,red,->] ($0.9*(B2)+0.1*(B1)+(0.05,0)$)--($0.1*(B2)+0.9*(B1)+(0.05,0)$);
\draw [thick,red,->] ($0.9*(B1)+0.1*(B2)+(-0.05,0)$)--($0.1*(B1)+0.9*(B2)+(-0.05,0)$);

\draw [thick,blue,->] ($0.8*(B2)+0.2*(A2)+(0.05,0.05)$)--($0.1*(B2)+0.9*(A2)+(0.05,0.05)$);
\draw [thick,blue,->] ($0.8*(A2)+0.2*(B2)+(-0.05,-0.05)$)--($0.1*(A2)+0.9*(B2)+(-0.05,-0.05)$);
\end{tikzpicture}
\end{center}
\caption{a) Each cycle $C_j$ we consider must be rainbow, for otherwise we would replace, for example, the orange edges $\vec{u_1u_2}$ and $\vec{u_3u_4}$ with orange edges $\vec{u_1u_4}$ and $\vec{u_3u_2}$.\vspace{0.1cm}
\newline
\textcolor{white}.\;\;\;b) For each rainbow cycle $C_j$, we take a set $E_j=\{v_4v_1,v_1v_3,v_3v_2\}$ of edges whose addition allows an (undirected) decomposition into triangles, and put a directed 2-cycle with colour $i_e$ on each $e\in E_j$.}\label{fig:reducecoloursurplus}
\end{figure}
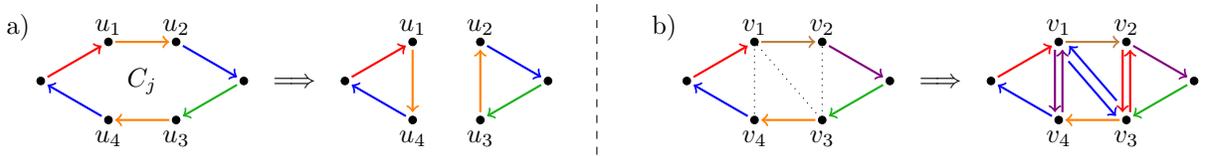
\end{center}

\smallskip

\noindent\textbf{Triangle/2-cycle decomposition.} Now, let $I\subset [r]$ index the cycles $C_j$, $j\in [r]$, which have length at least 4. For each $j\in I$, where $\ell_j$ is the length of the cycle $C_j$, take a set $E_j$ of $\ell_j-3$ (undirected) pairs of vertices from $V(C_j)$ such that the undirected graph underlying $C_j+E_j$ is a union of triangles in which each vertex has degree at most 4 (see Figure~\ref{fig:reducecoloursurplus}). For each $j\in I$ and $e=xy\in E_j$, as we will show is possible, greedily pick $i_e\in I_{\phi}$ under the following rules.
\begin{enumerate}[label = \arabic{enumi})]
\item For each $j\in I$ and $e=xy\in E_j$, $x,y\in U_{\phi}\setminus (R_{i_e}\cup T_{i_e})$.\label{cond:short1}
\item For each $i\in I_\phi$, and $x\in U_\phi$, there is at most one pair $(j,e)$ with $j\in I$, $e\in E_j$, $x\in V(e)$ and $i_e=i$.\label{cond:short2}
\item For each $i\in I_\phi$, there are at most $32p_Tn$ vertices $x\in U_\phi$ for which there is some pair $(j,e)$ with $j\in I$, $e\in E_{j}$, $x\in V(e)$ and $i_e=i$.\label{cond:short3}
\end{enumerate}

Note that the number of times any $x\in U_\phi$ appears in an edge $e=xy$ for some $e \in E_j$ and $j \in I$ is
\begin{equation}\label{eqn:xappearance}
\leq 2|\{j\in [r]:x\in V(C_j)\}|\leq 2|\{i\in I_\phi:x\in T_i\}|\overset{\ref{prop:verticesinTi}}{\leq} 3p_Tp_U^{-1}p_\tr p_\fa n.
\end{equation}
Thus, when we greedily choose some $i_e \in I_{\phi}$, we have that by \ref{prop:verticesinUiminusRiTi} there are at least $p_\tr p_\fa n/2$ options for $i_e\in I_\phi$ so that \ref{cond:short1} is satisfied. Of these, the number for which there are more than $32p_Tn$ vertices $x\in U_\phi$ for which there is some pair $(j,e)$ with $j\in I$, $e\in I_{j'}$, $x\in V(e)$ and $i_e=i$ is, using \eqref{eqn:xappearance}, at most
\[
\frac{|U_\phi|\cdot 3p_Tp_U^{-1}p_\tr p_\fa n}{32p_Tn}\overset{\ref{prop:vxpartitionfirst}}\leq \frac{8p_Tp_\tr p_\fa n^2}{32p_Tn}= \frac{p_\tr p_\fa n}{4}.
\]
 Hence, since $p_T\llpoly p_U$, we can choose $i_e$ so that \ref{cond:short2} and \ref{cond:short3} hold.

Take $D$ and, for each $j\in I$ and $xy\in E_{j}$, add both the edge $\vec{xy}$ and $\vec{yx}$ to $D$ with colour $i_{xy}$, and call the resulting directed multigraph $D'$ (see Figure~\ref{fig:reducecoloursurplus}b)). Note that, from this construction, $D'$ has an edge decomposition into rainbow directed cycles of length 2 and 3. Take such a decomposition, and let $r'$ be the number of cycles of length 3, labelling them as $C'_1,\ldots,C_{r'}'$.

\smallskip

\noindent\textbf{2-cycle decomposition.}
For each $j\in [r']$, label the vertices of $C'_j$ as $x_j,y_j,z_j$ and its colours as $a_j,b_j,c_j$, so that $\vec{x_jy_j}$, $\vec{y_jz_j}$, $\vec{z_jx_j}$ have colour $a_j$, $b_j$ and $c_j$, respectively. For each $i\in I_\phi$, let $U_i^-\subset U_\phi$ be the subset of vertices $x\in U_\phi$ for which there is some pair $(j,e)$ with $j\in I$, $e\in I_{j'}$, $x\in V(e)$ and $i_e=i$, and note that, by \ref{cond:short3}, $|U_i^-|\leq 32p_Tn$.
For each $j\in [r']$, choose distinct vertices $x_j',y_j',z_j'$ and a colour $i_j$ under the following rules.
\begin{enumerate}[label = \roman{enumi})]
\item For each $j\in [r']$ and each $i\in \{a_j,b_j,c_j\}$, $x_j',y_j',z_j'\in U_\phi\setminus (R_i\cup T_i\cup U_i^-)$.\label{cond:shorti}
\item For each $j\in [r']$, $i_j$ is such that $x_j',y_j',z_j'\in U_\phi\setminus (R_{i_j}\cup T_{i_j}\cup U_{i_j}^-)$.\label{cond:shortii}
\item For each $i\in I_\phi$, and $v\in U_\phi\setminus (R_i\cup T'_i)$, there is at most one $j\in [r']$ with  $i\in \{a_j,b_j,c_j,i_j\}$ and $v\in \{x_j',y_j',z_j'\}$.\label{cond:shortiii}
\end{enumerate}

Similarly to the previous step, this can be done greedily, where this uses $|U_i^-|\leq 32p_Tn$ for each $i\in I_\phi$ to select $x_j',y_j',z_j'$. Then, \eqref{eqn:xappearance} and \ref{prop:verticesinUiminusRiTi} can be used to select $i_j$.

As depicted in Figure~\ref{fig:replacetriangles}, take $D'$ and, for each $j\in [r']$, remove the edges in $C_j'$ and add the edges $\vec{x_jx_j'}$, $\vec{x_j'y_j'}$, $\vec{y_j'y_j}$ with colour $a_j$, the edges $\vec{y_jy_j'}$, $\vec{y_j'z_j'}$, $\vec{z_j'z_j}$ with colour $b_j$, the edges $\vec{z_jz_j'}$, $\vec{z_j'x_j'}$, $\vec{x_j'x_j}$ with colour $c_j$, and the edges $\vec{y'_jx_j'}$, $\vec{x_j'z_j'}$, $\vec{z_j'y_j'}$ with colour $i_j$, and call the resulting directed multigraph $D''$. By construction, $D''$ has a decomposition into directed rainbow cycles with length 2. Take such a decomposition, and, letting $r''$ be the number of cycles, let these cycles be $C''_1,\ldots,C''_{r''}$.
Let $\mathcal{C}_\phi$ be the set of pairs $\{(i,u),(j,v)\}$ for each cycle $C''_{i'}$, $i'\in [r'']$, with vertex set $\{u,v\}$ and an edge from $u$ to $v$ with colour $i$ and an edge from $v$ to $u$ with colour $j$. Note that, in our construction, for each $i\in I_\phi$ and $u\in T_i$, we never added an edge with colour $i$ directed into $u$, and, for each $i\in I_\phi$ and $v\in R_i$, we never added an edge with colour $i$ directed out of $v$, so that \eqref{eq:Csitsokay} holds.
We now show that $\mathcal{C}_\phi$ satisfies \ref{conc:A1-1}--\ref{conc:A1-3}.

\begin{center}
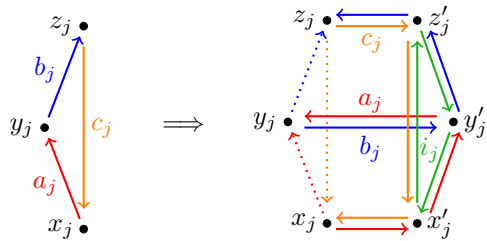
\begin{figure}[b]
\begin{center}
\begin{tikzpicture}
\def\vrad{0.6}

\def\vxsp{0.7}

\def\midvec(#1,#2){
($0.9*(#1)+0.1*(#2)$)--($0.1*(#1)+0.9*(#2)$)
}


\foreach \n in {1,2,3}
{
\coordinate (A\n) at ($(120*\n+180+90:\vrad)$);
}
\foreach \n in {1,2,3}
{
\coordinate (B\n) at ($2.5*(0,\vrad)+(120*\n+90:\vrad)$);
}

\coordinate (C3) at ($0.5*(A2)+0.5*(B1)$);
\coordinate (C4) at ($(C3)+(2.2,0)$);
\coordinate (A4) at ($(A3)+(1.2,0)$);

\draw [fill] (A3) circle[radius=0.05];
\draw [fill] (B3) circle[radius=0.05];
\draw [fill] (C3) circle[radius=0.05];

\draw ($(C3)-(0.3,0)$) node {$y_j$};
\draw ($(B3)-(0.3,0)$) node {$z_j$};
\draw ($(A3)-(0.3,0)$) node {$x_j$};

\draw ($0.5*(A3)+0.5*(C3)-(0.25,0.1)$) node {\textcolor{red}{$a_j$}};
\draw ($0.5*(B3)+0.5*(C3)-(0.25,-0.1)$) node {\textcolor{blue}{$b_j$}};
\draw ($0.5*(A3)+0.5*(B3)+(0.25,0)$) node {\textcolor{orange}{$c_j$}};


\draw [thick,orange,->] \midvec(B3,A3);
\draw [thick,red,->] \midvec(A3,C3);
\draw [thick,blue,->] \midvec(C3,B3);
\end{tikzpicture}\;\;
\begin{tikzpicture}
\draw [white](0,-1.6005) -- (0,0);
\draw (0,0) node {$\implies$};
\end{tikzpicture}\;\;
\begin{tikzpicture}
\def\vrad{0.6}
\def\vwidth{10}
\def\spacer{0.35}
\def\upp{0.2}
\def\thmid{-3.15}
\def\thmst{1.3}
\def\lineinw{-1.8}
\def\dropp{0.8}

\def\vxsp{0.7}

\def\midvec(#1,#2){
($0.9*(#1)+0.1*(#2)$)--($0.1*(#1)+0.9*(#2)$)
}


\foreach \n in {1,2,3}
{
\coordinate (A\n) at ($(120*\n+180+90:\vrad)$);
}
\foreach \n in {1,2,3}
{
\coordinate (B\n) at ($2.5*(0,\vrad)+(120*\n+90:\vrad)$);
}

\coordinate (C3) at ($0.5*(A2)+0.5*(B1)$);
\coordinate (C4) at ($(C3)+(2.2,0)$);
\coordinate (A4) at ($(A3)+(1.2,0)$);
\coordinate (B4) at ($(B3)+(1.2,0)$);

\draw ($(C3)-(0.3,0)$) node {$y_j$};
\draw ($(B3)-(0.3,0)$) node {$z_j$};
\draw ($(A3)-(0.3,0)$) node {$x_j$};

\draw ($(C4)+(0.3,0)$) node {$y_j'$};
\draw ($(B4)+(0.3,0)$) node {$z_j'$};
\draw ($(A4)+(0.3,0)$) node {$x_j'$};

\draw ($0.5*(C4)+0.5*(C3)+(0,0.25)$) node {\textcolor{red}{$a_j$}};
\draw ($0.5*(C4)+0.5*(C3)+(0,-0.35)$) node {\textcolor{blue}{$b_j$}};
\draw ($0.5*(B3)+0.5*(B4)+(0,-0.3)$) node {\textcolor{orange}{$c_j$}};

\draw [fill] (A3) circle[radius=0.05];
\draw [fill] (B3) circle[radius=0.05];
\draw [fill] (C3) circle[radius=0.05];
\draw [fill] (A4) circle[radius=0.05];
\draw [fill] (B4) circle[radius=0.05];
\draw [fill] (C4) circle[radius=0.05];

\draw [thick,dotted,orange,->] \midvec(B3,A3);
\draw [thick,dotted,red,->] \midvec(A3,C3);
\draw [thick,dotted,blue,->] \midvec(C3,B3);

\draw [thick,red,->] ($0.9*(A3)+0.1*(A4)+(0,-0.075)$)--($0.1*(A3)+0.9*(A4)+(0,-0.075)$);
\draw [thick,orange,->] ($0.9*(A4)+0.1*(A3)+(0,0.075)$)--($0.1*(A4)+0.9*(A3)+(0,0.075)$);
\draw [thick,orange,->] ($0.9*(B3)+0.1*(B4)+(0,-0.075)$)--($0.1*(B3)+0.9*(B4)+(0,-0.075)$);
\draw [thick,blue,->] ($0.9*(B4)+0.1*(B3)+(0,0.075)$)--($0.1*(B4)+0.9*(B3)+(0,0.075)$);
\draw [thick,blue,->] ($0.9*(C3)+0.1*(C4)+(0,-0.075)$)--($0.1*(C3)+0.9*(C4)+(0,-0.075)$);
\draw [thick,red,->] ($0.9*(C4)+0.1*(C3)+(0,0.075)$)--($0.1*(C4)+0.9*(C3)+(0,0.075)$);

\draw [thick,orange,->] ($0.9*(B4)+0.1*(A4)+(-0.125,0)$)--($0.1*(B4)+0.9*(A4)+(-0.125,0)$);
\draw [thick,darkgreen,<-] ($0.9*(B4)+0.1*(A4)+(0,-0.0)$)--($0.1*(B4)+0.9*(A4)+(0,-0.0)$);

\draw [thick,darkgreen,<-] ($0.9*(C4)+0.1*(B4)+(0,0.0)$)--($0.1*(C4)+0.9*(B4)+(0,0.0)$);
\draw [thick,blue,->] ($0.9*(C4)+0.1*(B4)+(0.125,0)$)--($0.1*(C4)+0.9*(B4)+(0.125,0)$);

\draw [thick,darkgreen,<-] ($0.9*(A4)+0.1*(C4)+(0,0.0)$)--($0.1*(A4)+0.9*(C4)+(0,0.0)$);
\draw [thick,red,->] ($0.9*(A4)+0.1*(C4)+(0.125,0)$)--($0.1*(A4)+0.9*(C4)+(0.125,0)$);

\draw [darkgreen] ($0.5*(A4)+0.5*(C4)+(-0.075,0.3)$) node {$i_j$};

\end{tikzpicture}
\end{center}\caption{Each directed triangle $x_jy_jz_j$ is replaced by a collection of 2-cycles, where each vertex has balanced in- and out-degree in each colour except for $x_j,y_j,z_j$ which maintain the same in- and out-degree in each colour.}\label{fig:replacetriangles}
\end{figure}
\end{center}

\smallskip

\noindent\ref{conc:A1-1}: Note that, for any pair $(i,u)$ such that $i \in I_{\phi}$ and $u \in T_i$, we have that $u$ was contained in an edge directed out of $u$ in colour $i$ in $D$. In order to construct $D''$ from $D$, either we did not add any other edges directed out of $u$ in colour $i$ and kept the original edge from $D$ in $D''$, or, in building $D''$ we replaced the out-edge from $u$ in colour $i$ by a path of length three containing exactly one out-edge from $u$ in colour $i$. Since $D''$ defines $\mathcal{C}_{\phi}$ and there is exactly one pair $\{(u,i),(v,j)\} \in \mathcal{C}_{\phi}$ for each out-edge in $D''$ from $u$ in colour $i$, we have that \ref{conc:A1-1} holds.

 \smallskip

 \noindent\ref{conc:A1-2}: Similarly, for $i \in I_{\phi}$ and $u \in R_i'$, we have that the number of $(v,j)$ such that $\{(v,i),(u,j)\} \in \mathcal{C}_{\phi}$ is the number of in-edges of colour $i$ at $u$ in $D''$. As our construction has exactly one in-edge to $u$ with colour $i$ in each of the digraphs $D$, $D'$, and then $D''$, we have that, similarly \ref{conc:A1-2} holds.

\smallskip

\noindent\ref{conc:A1-2b}: To see that \ref{conc:A1-2b} holds, note that we need to ensure that $D''$ contains no in-edge to $u \in R_i \setminus R_i'$ in colour $i$. Note that by construction we have that $D$ contains no such edge. Furthermore, any edges that were added in $D'$ and in $D''$ were chosen precisely so that a new edge to a vertex $u$ would be in a colour $i'$ such that $u \notin R_{i'}$. That is, at no point did we add an in-edge to $u$ in colour $i$, since $u \in R_i$, so $D''$ is as required.

\smallskip

\noindent\ref{conc:A1-3}: For this, it suffices to show that, for each $i \in I_{\phi}$ and $u \in U_i\setminus (R_i \cup T_i)$, either $u$ appears in no edge of colour $i$, or $u$ appears in exactly one in-edge and exactly one out-edge of colour $i$. We have that, in $D$, $u$ appears in no edge of colour $i$. When we build $D'$, either $u$ remains in no edge of colour $i$, or the colour $i$ is assigned to some edge $uv \in \bigcup_{j \in I} E_j$ and by rule \ref{cond:short2} this happens no more than once.
In the first case, it then follows similarly by rule \ref{cond:shortiii} that either in building $D''$ no edge of colour $i$ containing $u$ is added, or there is exactly one $j \in [r']$ such that $u \in \{x'_j, y'_j, z'_j\}$ and $i \in \{a_j, b_j, c_j, i_j\}$ --- either way, to build $D''$ from $D'$ we add exactly one in-edge of colour $i$ to $u$ and exactly one out-edge of colour $i$ to $u$.
In the second case, we have that $i$ is assigned to exactly one edge $uv \in \bigcup_{j \in I} E_j$.
By construction, then, we have that in $D'$ there is both an in-edge in colour $i$ and an out-edge in colour $i$ containing $u$, and $u \in U_i^-$. Thus when building $D''$ from $D'$, by rule \ref{cond:shortii}, $u$ is not chosen as a vertex $x'_j, y'_j, z'_j$ paired with colour $i$. In particular, this means that when shifting from $D'$ to $D''$,  either the previous in- and out-edges in colour $i$ which contain $u$ are left as before, or if one is removed, it is replaced by a path which contains exactly one edge in the same direction to or from $u$ in colour $i$ as the one that was removed. Thus, $u$ remains in exactly one in-edge in colour $i$ and exactly one out-edge in colour $i$, completing the proof of \ref{conc:A1-3} and hence the lemma.
\end{proof}


\subsection{An auxiliary sparse well-connected graph}\label{sec:propauxgraph}

The next lemma shows the existence of an auxiliary graph which is used in Part~\ref{partA3}, specifically, in the proof of Lemma~\ref{lem:partA3} in Section~\ref{sec:partA3}. We prove it, however, before embarking on Part~\ref{partA2} as it is a useful preliminary to a similar construction which we use in the proof of Lemma~\ref{lem:partA2}, our lemma for carrying out Part~\ref{partA2}.
 The main idea of the following lemma is to build a sparse graph $K$ which is the union of trees with roots in a set $U$, such that we may pair the vertices of $U$ up in any way and always find a collection of vertex-disjoint paths in $K$ which connect these pairs. As we will use binary trees, here we use $\log = \log_2$. The properties of the graph we construct could, as in other template-based approaches, be found using appropriate random graphs (if not with quite such a low maximum degree). We use this explicit construction in preparation for the similar version in Section~\ref{sec:partA2}, where we want a more delicate property of the auxiliary graph, as explained there.

\begin{lemma}\label{lem:auxiliarygraph}
Let $1/n\llpoly p\llpoly \log^{-1}n$. Then, there is a graph $K$ with vertex set $[n]$ and $\Delta(K)\leq 4$ containing an independent set $U\subset V(K)$ with $|U|= pn$ and the following property.

Given any $r\in \N$ and any set of vertex-disjoint pairs $x_1y_1,\ldots,x_ry_r\in U^{(2)}$, there is a set of vertex-disjoint paths $P_i$, $i\in [r]$, in $K$ with internal vertices in $V(K)\setminus U$ such that, for each $i\in [r]$, $P_i$ is an $x_i,y_i$-path.
\end{lemma}
\begin{proof} Let $U\subset [n]$ have size $pn$. Let $\ell$ be such that $2^{\ell}\leq n/10\log n < 2^{\ell+1}$, and note that, as $\ell\leq \log n$, we have $(\ell+1)\cdot 2^{\ell}\leq n$. Using this, take disjoint sets $V^{(0)},V^{(1)},\ldots,V^{(\ell)}$
in $[n]$ with size $2^{\ell}$ such that $U\subset V_0$. Let $m=|U|=pn$. For each $i\in [\ell]_0$,
label the vertices in $V^{(i)}$ as $v_{i,1},\ldots,v_{i,2^{\ell}}$, so that, in particular, $U=\{v_{0,1},\ldots,v_{0,m}\}$.

Let $K$ be the empty graph with vertex set $[n]$ and, for each $i\in [\ell-1]_0$, $j\in [2^\ell]$ and $r\in [2]$, add an edge to $K$ from $v_{i,j}$ to $v_{i+1,s}$ where $s$ is such that $s=2(j-1)+r\;\mathrm{mod}\;2^\ell$, noting that we have added 2 edges to $v_{i,j}$ into $V^{(i+1)}$,
so that $d_{K}(v_{i,j},V^{(i+1)})=2$. We will show that $K$ has the properties we require.
For this, for each $j\in [2^\ell]$, $v=v_{0,j}$ and $i\in [\ell]_0$, let
\begin{equation}\label{eqn:Lvi_simple}
L(v,i)=\{v_{i,j'}:\exists j''\in [2^\ell]\text{ s.t.\ }2^i(j-1)+1\leq j''\leq 2^i(j-1)+2^{i}\text{ and }j'=j''\;\textrm{mod}\;2^\ell\},
\end{equation}
noting in particular that $L(v,0)=\{v_{0,j}\}=\{v\}$, $L(v,\ell)=V^{(\ell)}$ and, for each $i\in [\ell]_0$, $|L(v,i)|=2^i$. For each $j\in [2^\ell]$ and $v=v_{0,j}$,
let $F_v$ be the graph with vertex set $V(F_v)=\bigcup_{i\in [\ell]_0}L(v,i)$ and edge set $E(K[V(F_v)])$.

We now show that $\Delta(K)\leq 4$ and, for each $v\in V^{(0)}$, that $F_v$ is a binary tree rooted at $v$, in the following two claims.

\begin{claim}\label{clm:maxdegKphi-easier}
$\Delta(K)\leq 4$.
\end{claim}
\begin{proof}[Proof of Claim~\ref{clm:maxdegKphi-easier}]
Recall that, for each $i\in [\ell-1]_0$ and $j\in [2^\ell]$, $d_{K}(v_{i,j},V^{(i+1)})=2$.
Now, let $i\in \{2,\ldots,\ell\}$ and $j\in [2^\ell]$. Note that, for each $j'\in [2^\ell]$ we added an edge from $v_{i-1,j'}$ to $v_{i,j}$ only when there was some $r\in [2]$ such that $2(j'-1)+r= j\;\mathrm{mod}\; 2^\ell$, so that there are exactly 2 such $j'\in [2^\ell]$, where in both cases we have $r\in [2]$ such that $r=j\;\mathrm{mod}\; 2$. Therefore, $d_{K}(v_{i,j},V^{(i-1)})=2$.
Thus, as we only added edges to $K$ between $V^{(i)}$ and $V^{(i+1)}$ for each $i\in [\ell-1]_0$, we have that $\Delta(K)\leq 4$, as required.
\claimproofend

\begin{claim}\label{clm:graphsaretrees-easier}
For each $j\in [2^\ell]$ and $v=v_{0,j}$, $F_v$ is a binary tree rooted at $v$ such that, for each $i\in [\ell]_0$, the vertices in the $i$th level of $F_v$, $V(F_v)\cap V^{(i)}$, are those in $L(v,i)$ (as defined at \eqref{eqn:Lvi_simple}).
\end{claim}
\begin{proof}[Proof of Claim~\ref{clm:graphsaretrees-easier}]
Let $j\in [2^\ell]$ and set $v=v_{0,j}$. Recall that, for each $i\in [\ell]_0$, $|L(v,i)|=2^i$, and, for each $i\in [\ell-1]_0$ and $j'\in [2^\ell]$, $d_{K}(v_{i,j'},V^{(i+1)})=2$.
Therefore, to show the claim, it is sufficient to show that, for each $i\in [\ell]$ and $w\in L(v,i)$, there is some $w'\in L(v,i-1)$ with $w'w\in E(K)$.

Let, then, $i\in [\ell]$ and let $j'\in [2^\ell]$ be such that $v_{i,j'}\in L(v,i)$. Using the definition of $L(v,i)$, let $r\in [2^{i-1}]$ and $s\in [2]$ be such that $2^{i}(j-1)+2(r-1)+s=j'\;\mathrm{mod}\;2^\ell$.
Then, let $j''\in [2^{\ell}]$ be such that $2^{i-1}(j-1)+r=j''\;\textrm{mod}\;2^\ell$, so that $v_{i-1,j''}\in L(v,i-1)$, and note that we added an edge from $v_{i-1,j''}$ to $v_{i,j'}$ in $K$ as $j'=2(j''-1)+s\;\mathrm{mod} \;2^\ell$ and $s\in [2]$. Thus, for $w=v_{i,j'}\in L(v,i)$, there is some $w'=v_{i-1,j''}\in L(v,i-1)$ with $w'w\in E(K)$, as required.
\claimproofend

We now show that the trees $F_v$, $v\in U=\{v_{0,1},\ldots,v_{0,m}\}$, are well spread out, particularly at their lower levels, as follows.

\begin{claim}\label{clm:treesarewellspread-easier}
For each $u\in U$ and $i\in [\ell]_0$, there are at most $\lfloor 2^{i}/(100\log n)\rfloor$ vertices $u'\in U$ with $V^{(i)}\cap (V(F_{u})\cap V(F_{u'}))\neq \emptyset$.
\end{claim}
\begin{proof}[Proof of Claim~\ref{clm:treesarewellspread-easier}]
Let $u\in U$ and $i\in [\ell]_0$. First, note that if $i\geq \ell-10$, then $2^i/(100\log n)\geq 2^{\ell-10}/100(\log n)\geq m$, as $p \ll \log^{-1}n$, so that there are at most $\lfloor 2^{i}/100\log n\rfloor$ vertices $u'\in U$. Assume, then, that $i\leq \ell-10$.

Let $u'\in U$ with $V^{(i)}\cap (V(F_{u})\cap V(F_{u'})\neq \emptyset$.
Let $r,r'\in [m]$ be such that $u=v_{0,r}$ and $u'=v_{0,r'}$. Then, from Claim~\ref{clm:graphsaretrees-easier}, and the definition of $L(u,i)$ and $L(u',i)$ at \eqref{eqn:Lvi_simple} we have that there is some $j$ such that
\[
2^i(r-1)+1\leq j\leq 2^i(r'-1)+2^{i}
\]
and some $j'$ such that $j=j'\;\mathrm{mod}\;2^\ell$ and
\[
2^i(r'-1)+1\leq j'\leq 2^i(r'-1)+2^{i},
\]
so that, setting $x=j-j'$, we have $x=0\;\mathrm{mod}\;2^{\ell}$ and
\begin{equation}\label{eqn:sxineq_easier}
2^i(r-r')-2^i\leq x \leq 2^i(r-r')+2^i.
\end{equation}
As $r,r'\in [m]$, we have $-(m+1)\cdot 2^i\leq x\leq (m+1)\cdot 2^i$.

As $i\leq \ell-10$, for each $r$ and $x$, there is at most one value of $r'$ for which \eqref{eqn:sxineq_easier} holds. Therefore, the number of choices of $u'\in U$ so that $V^{(i)}\cap (V(F_{u})\cap V(F_{u'}))\neq \emptyset$ is at most the number of choices of $x$ for which $x=0\;\mathrm{mod}\;2^{\ell}$ and $-(m+1)\cdot 2^i\leq x\leq (m+1)\cdot 2^i$. There are at most $\lceil(2m+3)2^i/2^{\ell}\rceil$ such values of $x$. As $m=pn$ and $2^{\ell+1}\geq n/(10\log n)$ and $p\llpoly \log^{-1}n$, there are thus at most $\lceil 2^i/(100\log n)\rceil$ such values of $x$. Noting that $x=0$ is always a solution, which gives that $r=r'$ and so $u=u'$, we have that the claim holds.
\claimproofend

Let then $r\in \N$ and let $x_1y_1,\ldots,x_ry_r$ be vertex-disjoint pairs in $U^{(2)}$. Let $I\subset [r]$ be a maximal subset for which there are vertex-disjoint paths $P_i$, $i\in I$, with internal vertices in $V(K)\setminus U$ such that, for each $i\in I$, $P_i$ is an $x_i,y_i$-path in $F_{x_i}\cap F_{y_i}$ with at most 2 vertices in each set $V^{(i')}$, for each $i'\in [\ell]_0$.

Suppose, for contradiction that $I\neq [r]$, and let $i\in [r]\setminus I$.
Let $V^\forb$ be the set of internal vertices in the paths $P_{i'}$, $i'\in I$.

\begin{claim}\label{clm:resilientlyexpand-easier}
For each $i'\in [\ell]_0$ and $u\in \{x_i,y_i\}$, $|V(F_{u})\cap V^{(i')}\cap V^{\forb}|\leq 2^{i'}/10\log n$.
\end{claim}
\begin{proof}[Proof of Claim~\ref{clm:resilientlyexpand-easier}] Let $i'\in [\ell]_0$.
By Claim~\ref{clm:treesarewellspread-easier}, there are at most $\lfloor 2^{i'}/100\log n\rfloor$ vertices $u'\in U\setminus \{x_i\}$ with $V^{(i')}\cap (V(F_{u})\cap V(F_{x_i}))\neq \emptyset$.
As the pairs $x_{i'}y_{i'}$, $i'\in I$, are disjoint, and $i\notin I$, there are thus at most $\lfloor 2^{i'}/(100\log n)\rfloor$ values of $i'\in I$ for which $V(P_{i'})$ intersects with $V(F_{x_i})\cap V^{(i')}$.
As each of these paths intersect with $V^{(i')}$ in at most 2 vertices, we therefore have that $|V(F_{x_i})\cap V^{(i')}\cap V^{\forb}|\leq 2^{i'}/(10 \log n)$. As, similarly, $|V(F_{y_i})\cap V^{(i')}\cap V^{\forb}|\leq 2^{i'}/(10 \log n)$, the claim follows.
\claimproofend

Then, let $F'_{x_i}$ be the connected component of $F_{x_i}-V^\forb$ which contains $x_i$. Note that, for each $i'\in [\ell]_0$, removing a vertex from $L(x_i,i')$ removes at most $2^{\ell-i'+1}$ vertices from the connected component of $F_{x_i}$ which contains $x_i$. Therefore, by Claim~\ref{clm:resilientlyexpand-easier}, the number of vertices in $F_{x_i}$ which are not in $F'_{x_i}$ is at most
\[
\sum_{i'=0}^\ell |V(F_{x_i})\cap V^{(i')}\cap V^{\forb}|\cdot 2^{\ell-i'+1} \leq \sum_{i'=0}^\ell 2^{\ell+1}/10\log n\leq 2^\ell/4,
\]
so that, in particular, $|V(F'_{x_i})\cap V^{\ell}|\geq 2^{\ell}-2^{\ell}/4>|V^{\ell}|/2$.

Similarly, letting $F'_{y_i}$ be the connected component of $F_{y_i}-V^\forb$, we have $|V(F'_{y_i})\cap V^{\ell}|>|V^{\ell}|/2$. Therefore, $F'_{x_i}$ and $F'_{y_i}$ intersect on $V^{\ell}$,
 Let $P_{i}$ be a shortest $x_i,y_i$-path in $F_{x_i}'\cup F_{y_i}'$. Note that by this minimality $P_i$ contains at most 2 vertices from each set $V^{(i')}$, $i'\in [\ell+1]$, and, by construction $P_{i}$ has no vertices in $V^\forb$ and therefore no vertices in $V(P_{i'})$ for each $i'\in I$. Thus, the paths $P_{i'}$, $i'\in I\cup \{i\}\subset [r]$, contradict the choice of $I$. Therefore, we must have $I=[r]$, and thus have the required paths $P_{i'}$, $i'\in [r]$.
\end{proof}


\subsection{Part~\ref{partA2}: 2-cycles using few vertex pairs}\label{sec:partA2}

For each vertex $v\in U_\phi$ we will attach an edge from $v$ to the roots of 5 binary
trees consisting of vertices from $V_\phi$, which have been chosen so that, across all $v\in U_\phi$ their vertices at each level are very well spread (see Claim~\ref{clm:treesarewellspread}), each pair of trees intersect completely in the last layer, and the union of all these trees has low maximum degree (as follows from Claim~\ref{clm:maxdegKphi}). This construction is similar to that given in Lemma~\ref{lem:auxiliarygraph}, but we use some more delicate properties of it. This is because, instead of finding a collection of vertex-disjoint paths $P_i$, $i\in [r]$, as before, we find a larger collection of paths, $\mathcal{P}$ say, many more than could be vertex-disjoint, but so that many specified pairs of paths are vertex-disjoint except for possibly on their endvertices (as in \ref{conc:A2:2} below). This requires us to take more care in the choice of paths.

\begin{lemma}\label{lem:partA2} Let $U_i,V_i$, $i\in [n]$, satisfy \eref{prop:vxpartitionfirst}--\eref{prop:vxpartitionlast}.
Let $\tau\in \mathcal{T}$ and $\phi\in \mathcal{F}_\tau$. Then, there is a graph $K_\phi$ with vertex set $U_\phi\cup V_\phi$ and $\Delta(K_\phi)\leq 5$ and the following property.

Suppose $\mathcal{C}\subset \{(i,u),(j,v):i,j\in I_\phi,u,v\in U_\phi,i\neq j,u\neq v,u\sim_{A,B}v\}$ satisfies the following properties.

\stepcounter{propcounter}
\begin{enumerate}[label = {\emph{\textbf{\Alph{propcounter}\arabic{enumi}}}}]
\item For each $i\in I_\phi$ and $u\in U_\phi$, there is at most one pair $(j,v)$ with $\{(i,u),(j,v)\}\in \mathcal{C}$.\labelinthm{prop:A2:1}
\item For each $i\in I_\phi$ and $v\in U_\phi$, there is at most one pair $(u,j)$ with $\{(i,u),(j,v)\}\in \mathcal{C}$.\labelinthm{prop:A2:2}
\end{enumerate}

Then, there are paths $P_e$, $e\in \mathcal{C}$, in $K_\phi$ with the following properties.

\begin{enumerate}[label = {\emph{\textbf{\Alph{propcounter}\arabic{enumi}}}}]\addtocounter{enumi}{2}
\item For each $e=\{(i,u),(j,v)\}\in \mathcal{C}$, $P_e$ is a $u,v$-path with internal vertices in $V_\phi\cap A$ if $u,v\in A$ and internal vertices in $V_\phi\cap B$ if $u,v\in B$.\labelinthm{conc:A2:1}
\item For each $e=\{(i,u),(j,v)\},e'=\{(i',u'),(j',v')\}\in \mathcal{C}$ with $e\neq e'$, if $\{i,j\}\cap \{i',j'\}\neq \emptyset$, then $V(P_e)$ and $V(P_{e'})$ intersect only on $\{u,v\}\cap \{u',v'\}$.\labelinthm{conc:A2:2}
\end{enumerate}
\end{lemma}
\begin{proof}
For each $e=\{(i,u),(j,v)\}\in \mathcal{C}$, we have $u,v\in A$ or $u,v\in B$, and wish to find $P_e$ such that, as in \ref{conc:A2:1}, all of the vertices of $P_e$ are in $A$ in the first case, and in $B$ in the second case. Therefore, $K_\phi$ will be the disjoint union of two graphs $K_\phi^A$ and $K_\phi^B$ with vertex set $(U_\phi\cup V_\phi)\cap A$ and $(U_\phi\cup V_\phi)\cap B$, respectively.
To reduce notation, we will give the construction for $K_\phi^A$, where the construction for $K_\phi^B$ follows identically but with $B$ in place of $A$.

Let $\ell$ be such that $2^{\ell}\leq p_Vn/\log n < 2^{\ell+1}$, and note that, as $\ell\leq (\log n)-1$, $(\ell+1)\cdot 2^{\ell}\leq p_Vn/2$. Using this, and \ref{prop:vxpartitionfirst}, take disjoint sets $V_\phi^{(0)},V_\phi^{(1)},\ldots,V_\phi^{(\ell)}$ in $V_\phi\cap A$ with size $2^{\ell}$ and, for each $i\in [\ell]_0$, enumerate $V_\phi^{(i)}$ as $\{v_{i,1},\ldots,v_{i,2^{\ell}}\}$. Let $K_{\phi}^-$ be the empty graph with vertex set $V_\phi^{(0)}\cup V_\phi^{(1)}\cup\ldots\cup V_\phi^{(\ell)}$ and, for each $i\in [\ell-1]_0$, $j\in [2^\ell]$ and $r\in [2]$, add an edge from $v_{i,j}$ to $v_{i+1,s}$ where $s$ is such that $s=2(j-1)+r\;\mathrm{mod}\;2^\ell$, so that $d_{K_\phi^-}(v_{i,j},V_\phi^{(i+1)})=2$.
For each $j\in [2^\ell]$, $v=v_{0,j}$ and $i\in [\ell]_0$, let
\begin{equation}\label{eqn:Lvi}
L(v,i)=\{v_{i,j'}:\exists j''\in [2^\ell]\text{ s.t.\ }2^i(j-1)+1\leq j''\leq 2^i(j-1)+2^{i}\text{ and }j'=j''\;\textrm{mod}\;2^\ell\},
\end{equation}
noting in particular that $L(v,0)=\{v_{0,j}\}=\{v\}$, $L(v,\ell)=V_\phi^{(\ell)}$ and, for each $i\in [\ell]_0$, $|L(v,i)|=2^i$. For each $j\in [2^\ell]$ and $v=v_{0,j}$,
let $F_v$ be the graph with vertex set $V(F_v)=\bigcup_{i\in [\ell]_0}L(v,i)$ and edge set $E(K_{\phi}^-[V(F_v)])$.

We will use that $\Delta(K_\phi^-)\leq 4$ and, for each $v\in V_\phi^{(0)}$, $F_v$ is a binary tree rooted at $v$, in the following two claims. As these two claims are proved virtually identically to Claims~\ref{clm:maxdegKphi-easier}~and~\ref{clm:graphsaretrees-easier} respectively, we omit their proof.

\begin{claim}\label{clm:maxdegKphi}
$\Delta(K_\phi^-)\leq 4$.\hfill$\boxdot$
\end{claim}

\begin{claim}\label{clm:graphsaretrees}
For each $j\in [2^\ell]$ and $v=v_{0,j}$, $F_v$ is a binary tree rooted at $v$ such that, for each $i\in [\ell]_0$, the vertices in the $i$th level of $F_v$, $V(F_v)\cap V_\phi^{(i)}$, are those in $L(v,i)$ (as defined at \eqref{eqn:Lvi}).\hfill $\boxdot$
\end{claim}

Now, let $m_0=|U_\phi\cap A|$, so that, by \ref{prop:vxpartitionfirst}, we have $m_0=(1\pm \epsforabsthatwasepszero)p_Un$. Enumerate $U_\phi$ as $\{u_{1},\ldots,u_{m_0}\}$. For each $i\in [m_0]$ and $j\in [5]$, add an edge to $K^-_\phi$ from $u_{i}$ to $w(u_{i},j):=v_{0,5(i-1)+j}$, and call the resulting graph $K^A_\phi$, where we have used that $5m_0\leq (1+\epsforabsthatwasepszero)p_Un$ is much smaller than $|V_\phi^{(1)}|\geq p_Vn/2$ as $p_U\llpoly p_V$.
To each vertex in $U_\phi$ we have attached 5 of the binary trees $F_v$, $v\in V_\phi^{(0)}$. As all the neighbours of $U_\phi\cap A$ are distinct and within the first $5m_0\leq p\cdot p_Vn$ for some $p\ll \log^{-1}n$, the trees we have attached are well spread out at each level, in a similar way to Claim~\ref{clm:treesarewellspread-easier}. As the proof is virtually identical, we omit it.

\begin{claim}\label{clm:treesarewellspread}
For each $u\in U_\phi$, $j\in [5]$ and $i\in [\ell+1]$, there are at most $\lfloor 2^{i}/(100\log n)\rfloor$ pairs $(u',j')\neq (u,j)$ with $u'\in U_\phi$, $j'\in [5]$ and $V_\phi^{(i)}\cap (V(F_{w(u,j)})\cap V(F_{w(u',j')}))\neq \emptyset$.\hfill$\boxdot$
\end{claim}

Similarly, form the graph $K^B_\phi$, and let $K_\phi$ be the graph with the vertex set $U_\phi\cup V_\phi$ and edge set $E(K^A_\phi)\cup E(K^B_\phi)$. We will show that $K_\phi$ has the desired property.

First, note that, from Claim~\ref{clm:maxdegKphi}, and noting that we added vertex-disjoint 5-edge stars from each vertex in $U_\phi$ to vertices in $V_{\phi}$, we have that $d_{K_\phi}(v)\leq 5$ for each $v\in (U_\phi\cup V_\phi)\cap A$. By a similarly proved version of Claim~\ref{clm:maxdegKphi}, we have that this also holds for every $v\in (U_\phi\cup V_\phi)\cap B$, so that $\Delta(K_\phi)\leq 5$, as required.

We now show that the main property of $K_\phi$ holds. Let $\mathcal{C}\subset \{(i,u),(j,v):i,j\in I_\phi,i\neq j,u,v\in U_\phi,u\neq v,u\simAB v\}$ satisfying \ref{prop:A2:1} and \ref{prop:A2:2}. We will show that there are paths $P_e$, $e\in \mathcal{C}$, in $K_\phi$ such that \ref{conc:A2:1} and \ref{conc:A2:2} hold.
First, choose $r_{e}\in [5]$ for each $e=\{(i,u),(j,v)\}\in \mathcal{C}$ such that
\begin{enumerate}[label = {{\textbf{\Alph{propcounter}\arabic{enumi}}}}]\addtocounter{enumi}{4}
\item For each $e=\{(i,u),(j,v)\},e'=\{(i',u'),(j',v')\}\in \mathcal{C}$ with $e\neq e'$, if $\{i,j\}\cap \{i',j'\}\neq \emptyset$ and $\{u,v\}\cap \{u',v'\}\neq \emptyset$, then $r_{e}\neq r_{e'}$.\label{prop:jegood}
\end{enumerate}

To see that this is possible, create an auxiliary graph $L$ with vertex set $\mathcal{C}$ and for each $e=\{(i,u),(j,v)\},e'=\{(i',u'),(j',v')\}\in \mathcal{C}$ with $e\neq e'$ put an edge between $e$ and $e'$ in $L$ if $\{i,j\}\cap \{i',j'\}\neq \emptyset$ and $\{u,v\}\cap \{u',v'\}\neq \emptyset$.
Then, for each $e=\{(i,u),(j,v)\}\in \mathcal{C}$, by \ref{prop:A2:1} and \ref{prop:A2:2} there are at most 4 choices for $e'=\{(i',u'),(j',v')\}\in \mathcal{C}$ with $e\neq e'$, $\{i,j\}\cap \{i',j'\}\neq \emptyset$ and $\{u,v\}\cap \{u',v'\}\neq \emptyset$. Indeed, for such an $e'$, firstly by relabelling if necessary we can assume that $j'\neq i$ and $i'\neq j$. Then, if $i=i'$ note that we have $u'\neq u$ by \ref{prop:A2:1} as $e\neq e'$, and therefore $v'\in \{u,v\}$, which, by \ref{prop:A2:2} gives us two options for $v',j',u'$. Similarly, if $j=j'$, then there are two options for $u',i',v'$, for at most 4 options in total.
Thus, $L$ has maximum degree 4, and so can be properly coloured with 5 colours, using the colour set $[5]$. Take such a colouring, and, for each $e\in \mathcal{C}$, let $r_e$ be the colour of $e$, noting that, then, \ref{prop:jegood} holds.

Now, let $\mathcal{C}'\subset \mathcal{C}$ be a maximal set for which there are paths $P_e$, $e\in \mathcal{C}'$, in $K_\phi$ such that  \ref{conc:A2:1} and \ref{conc:A2:2} hold with $\mathcal{C}$ replaced by $\mathcal{C}'$ and, for each $e=\{(i,u),(j,v)\}\in\mathcal{C}'$, $P_e\subset F_{w(u,r_e)}\cup F_{w(v,r_e)}+uw(u,r_e)+vw(v,r_e)$ and $P_e$ contains at most 2 vertices from each set $V_\phi^{(i')}$, $i'\in [\ell]_0$. Pick such a set of paths $P_e$, $e\in \mathcal{C}'$.

Noting that we will be done if $\mathcal{C}'=\mathcal{C}$, assume, for a contradiction, that $\mathcal{C}\neq\mathcal{C}'$ and pick some $e=\{(i,u),(j,v)\}\in \mathcal{C}\setminus \mathcal{C}'$. Assume that $u,v\in A$, so that we may use the notation above, where the case for $u,v\in B$ follows similarly. We will find a path $P_e$, as depicted in Figure~\ref{fig:sparsevxpairs}, which will have the properties so that we could add $e$ to $\mathcal{C}'$ to contradict its maximality.

We first define a similar set $V^{\forb}$ of vertices as we avoided in the proof of Lemma~\ref{lem:auxiliarygraph}, but only collect together vertices from paths that we need to avoid if we are to find a path $P_e$ which we can add to the collection $P_{e'}$, $e'\in \mathcal{C}'$, with \ref{conc:A2:2} still holding for $\mathcal{C}'$ in place of $\mathcal{C}$. Therefore, let $V^{\forb}$ be the set of vertices which appear in some path $P_{e'}$ such that $e'=\{(i',u'),(j',v')\}\in \mathcal{C}'$ with $i'=i$ or $j'=j$. 
We will now show a similar claim to Claim~\ref{clm:resilientlyexpand-easier}. Its proof is similar to the proof of Claim~\ref{clm:resilientlyexpand-easier}.

\begin{claim}\label{clm:resilientlyexpand}
For each $i'\in [\ell]_0$, $|V(F_{w(u,r_e)})\cap (V_\phi^{(i')}\cap V^{\forb})|\leq 2^{i'}/(10 \log n)$ and $|V(F_{w(v,r_e)})\cap (V_\phi^{(i')}\cap V^{\forb})|\leq 2^{i'}/(10\log n)$.
\end{claim}
\begin{proof}[Proof of Claim~\ref{clm:resilientlyexpand}]
By Claim~\ref{clm:treesarewellspread}, there are at most $\lfloor 2^{i'}/100\log n\rfloor$ pairs $(v',r)\neq (u,r_e)$ with $v'\in U_\phi$, $r\in [5]$ and $V_\phi^{(i')}\cap (V(F_{w(u,r_e)})\cap V(F_{w(v',r)}))\neq \emptyset$.
Therefore, by \ref{prop:A2:1}, there are at most  $\lfloor 2^{i'}/100\log n\rfloor$ triples $(u',v',j')$ for which $e'=\{(i,u'),(j',v')\}\in \mathcal{C}'$ and $V_\phi^{(i')}\cap (V(F_{w(u,r_e)})\cap V(F_{w(v',r_{e'})}))\neq \emptyset$.
Similarly, by \ref{prop:A2:2}, there are at most  $\lfloor 2^{i'}/100\log n\rfloor$ triples $(u',v',j')$ for which $e'=\{(i,v'),(j',u')\}\in \mathcal{C}$ and $V_\phi^{(i')}\cap (V(F_{w(u,r_e)})\cap V(F_{w(v',r_{e'})}))\neq \emptyset$.
Thus, as for each $e'=\{(i',u'),(j',v')\}\in\mathcal{C}'$, $P_{e'}\subset F_{w(u',r_{e'})}\cup F_{w(v',r_{e'})}$ and $P_{e'}$ contains at most 2 vertices from $V_\phi^{(i')}$, $|V(F_{w(u,r_e)})\cap (V_\phi^{(i')}\cap V^{\forb})|\leq 2^{i'}/(10\log n)$.
Similarly, we have that $|V(F_{w(v,r_e)})\cap (V_\phi^{(i')}\cap V^{\forb})|\leq 2^{i'}/(10\log n)$.
\claimproofend

Then, let $F'_{w(u,r_e)}$ be the connected component of $F_{w(u,r_e)}-V^\forb$ which contains $w(u,r_e)$. Note that, for each $i'\in [\ell]_0$, removing a vertex from $L(w(u,r_e),i')$ removes at most $2^{\ell-i'+1}$ vertices from the connected component of $F_{w(u,r_e)}$ which contains $w (u,r_e)$. Therefore, by Claim~\ref{clm:resilientlyexpand}, the number of vertices in $F_{w(u,r_e)}$ which are not in $F'_{w(u,r_e)}$ is at most
\[
\sum_{i'=0}^\ell |V(F_{w(u,r_e)})\cap (V_\phi^{(i')}\cap V^{\forb})|\cdot 2^{\ell-i'+1} \leq \sum_{i'=0}^\ell 2^{\ell+1}/10\log n\leq 2^\ell/4,
\]
so that, in particular, $|V(F'_{w(u,r_e)})\cap V_\phi^{\ell}|\geq 2^{\ell}-2^{\ell}/4>|V_\phi^{\ell}|/2$.

Similarly, letting $F'_{w(v,r_e)}$ be the connected component of $F_{w(v,r_e)}-V^\forb$  which contains $w(v,r_e)$, we have $|V(F'_{w(v,r_e)})\cap V_\phi^{\ell}|>|V_\phi^{\ell}|/2$. Therefore, $F'_{w(u,r_e)}$ and $F'_{w(v,r_e)}$ intersect on $V_\phi^{\ell}$.
 Let $P_{e}$ be a shortest $u,v$-path in $F_{w(u,r_e)}'\cup F_{w(v,r_e)}'+uw(u,r_e)+vw(v,r_e)$. Note that by this minimality $P_e$ contains at most 2 vertices from each set $V_\phi^{(i')}$, $i'\in [\ell]_0$ and, by construction, for each $e'=\{(i',u'),(j',v')\}\in \mathcal{C}'$, $P_e$ has no vertices in $V^\forb$ and therefore no vertices in $V(P_{e'})\setminus \{u',v'\}$ if $\{i',j'\}\cap \{i,j\}\neq \emptyset$. Thus, the paths $P_{e'}$, $e'\in \mathcal{C}'\cup \{e\}\subset \mathcal{C}$ contradicts the choice of $\mathcal{C}'$. Therefore, we must have $\mathcal{C}'=\mathcal{C}$. Thus, we can choose the required paths $P_{e'}$, $e'\in \mathcal{C}$.
\end{proof}


\begin{center}
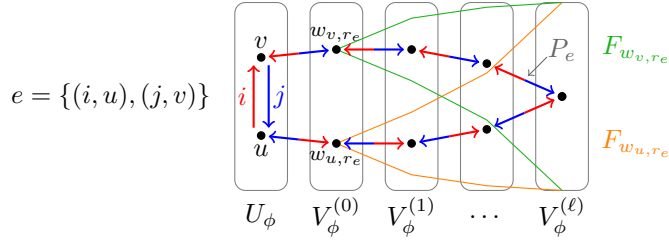
\begin{figure}
\begin{center}
\begin{tikzpicture}
\def\vrad{0.6}
\def\vwidth{10}
\def\setsp{1}
\def\setw{0.35}
\def\seth{1.25}
\def\midvec(#1,#2){
($0.9*(#1)+0.1*(#2)$)--($0.1*(#1)+0.9*(#2)$)
}
\def\midveccol(#1,#2,#3,#4){
\draw [#3,thick,<-]($0.9*(#1)+0.1*(#2)$)-- ($0.5*(#1)+0.5*(#2)$);
\draw [#4,thick,->]($0.5*(#1)+0.5*(#2)$)-- ($0.1*(#1)+0.9*(#2)$);
}


\foreach \n in {2}
{
\coordinate (v) at ($(120*\n+180:\vrad)$);
\draw [fill] (v) circle[radius=0.05];
}
\foreach \n in {1}
{
\coordinate (u) at ($(120*\n+180:\vrad)$);
\draw [fill] (u) circle[radius=0.05];
}

\foreach \n in {0,1,2,3,4}
{
\coordinate (set\n) at ($0.5*(v)+0.5*(u)+\n*(1,0)$);
\draw [rounded corners,black!50] ($(set\n)-(\setw,0)$) to ++(0,\seth) to ++(2*\setw,0) to ++(0,-2*\seth) to ++(-2*\setw,0) to ++(0,\seth);
}

\coordinate (u1) at ($(set1)-(0,0.5*\seth)$);
\coordinate (u2) at ($(set2)-(0,0.5*\seth)$);
\coordinate (u3) at ($(set3)-(0,0.35*\seth)$);
\coordinate (u4) at ($(set4)-(0,0*\seth)$);
\coordinate (v1) at ($(set1)+(0,0.5*\seth)$);
\coordinate (v2) at ($(set2)+(0,0.5*\seth)$);
\coordinate (v3) at ($(set3)+(0,0.35*\seth)$);
\coordinate (v4) at ($(set4)+(0,0*\seth)$);

\draw ($(u)+(0,-0.2)$) node {$u$};
\draw ($(v)+(0,0.2)$) node {$v$};

\draw ($(u1)+(0,-0.2)$) node {\footnotesize $w_{u,r_e}$};
\draw ($(v1)+(0,0.2)$) node {\footnotesize $w_{v,r_e}$};

\draw ($(set1)+(0,-\seth)+(0,-0.35)$) node {$V_\phi^{(0)}$};
\draw ($(set2)+(0,-\seth)+(0,-0.35)$) node {$V_\phi^{(1)}$};
\draw ($(set3)+(0,-\seth)+(0,-0.35)$) node {$\dots$};
\draw ($(set4)+(0,-\seth)+(0,-0.35)$) node {$V_\phi^{(\ell)}$};
\draw ($(set0)+(0,-\seth)+(0,-0.35)$) node {$U_\phi$};

\draw [orange] (u1) -- ($(u2)+(0,0.3333*\seth)$) -- ($(u3)+(0,0.6*\seth)$) -- ($(u4)+(0,\seth)$);
\draw [orange] (u1) -- ($(u2)-(0,0.3333*\seth)$) -- ($(u3)-(0,0.6*\seth)$) -- ($(u4)-(0,\seth)$);
\draw [darkgreen] (v1) -- ($(v2)+(0,0.3333*\seth)$) -- ($(v3)+(0,0.6*\seth)$) -- ($(v4)+(0,\seth)$);
\draw [darkgreen] (v1) -- ($(v2)-(0,0.3333*\seth)$) -- ($(v3)-(0,0.6*\seth)$) -- ($(v4)-(0,\seth)$);

\draw [thick,blue,->] ($0.9*(v)+0.1*(u)+(0.1,0)$)--($0.1*(v)+0.9*(u)+(0.1,0)$);
\draw [thick,red,->] ($0.9*(u)+0.1*(v)+(-0.1,0)$)--($0.1*(u)+0.9*(v)+(-0.1,0)$);

\draw [red] ($0.5*(u)+0.5*(v)-(0.25,0)$) node {$i$};
\draw [blue] ($0.5*(u)+0.5*(v)+(0.25,0)$) node {$j$};

\draw ($0.5*(u)+0.5*(v)-(2,0)$) node {$e=\{(i,u),(j,v)\}$};

\draw [black!60] ($(u4)+(0,0.5*\seth)$) node {$P_e$};

\draw [darkgreen] ($(u4)+(1,0.5*\seth)$) node {$F_{w_{v,r_e}}$};
\draw [orange] ($(u4)+(1,-0.5*\seth)$) node {$F_{w_{u,r_e}}$};

\draw [black!60,->] ($(u4)+(0,0.5*\seth)-(0.2,0.1)$) to ++(-0.25,-0.25);

\midveccol(u,u1,blue,red)
\midveccol(u1,u2,blue,red)
\midveccol(u2,u3,blue,red)
\midveccol(u3,u4,blue,red)
\midveccol(v4,v3,blue,red)
\midveccol(v3,v2,blue,red)
\midveccol(v2,v1,blue,red)
\midveccol(v1,v,blue,red)
\foreach \n in {1,2,3,4}
{
\draw [fill] (u\n) circle[radius=0.05];
\draw [fill] (v\n) circle[radius=0.05];
}
\end{tikzpicture}
\end{center}
\caption{A $u,v$-path $P_e$ as found in the sparse auxiliary graph $K_\phi$ using the two binary trees $F_{w_{u,r_e}}$ and $F_{w_{v,r_e}}$. In the proof of Lemma~\ref{lem:partA2}, this path will additionally avoid some set of vertices $V^{\forb}$.}\label{fig:sparsevxpairs}
\end{figure}
\end{center}


\subsection{Part~\ref{partA3}: 2-cycles using few vertex pairs and few colour pairs}\label{sec:partA3}

Our next lemma, Lemma~\ref{lem:partA3}, is the most difficult part of this section, but on proving it we will be very close to proving the main result for Part~\ref{partA}, Lemma~\ref{keylemma:absorption}. Indeed, Lemma~\ref{lem:partA3} is very similar to Lemma~\ref{keylemma:absorption}, producing for each $\tau \in \mathcal{T}$, a set $\mathcal{I}_\tau$ satisfying similar conditions as those in Lemma~\ref{keylemma:absorption}, essentially only lacking a regularity condition (i.e., we will have \ref{prop:A3:regularityout} instead of \ref{prop:abs:regularityout}). From our previous work in this section, we are well prepared to take sets $R_i'\subset R_i$, $i\in [n]$, with $|R_i'|=|T_i|$, for which, for each $\phi\in \mathcal{F}$, $\bigcup_{i \in I_\phi} R_i' =_{\mult} \bigcup_{i \in I_\phi} T_i$, and decompose the corresponding corrections we will require into a collection $\mathcal{C}$ of pairs of the form $\{(i,u),(j,v)\}$ such that $i$ and $j$ here always belong to the same family, and for such a family we only use a sparse collection of vertex pairs $u,v$ (defined by some auxiliary graph $K_\phi$).
Here, we will now replace such a pair with a collection of pairs $\{(i',u),(j',v)\}$ as at \eqref{eqn:reducecolpairs} which make the same effective change but for which $(i',j')$ come from a sparse set of pairs (using an appropriate auxiliary graph to restrict which pairs we allow). This is depicted in Figure~\ref{fig:sparsecoloursacrossvertexpair}. When replacing the pair $\{(i,u),(j,v)\}$ with  a collection of pairs $\{(i',u),(j',v)\}$ as at \eqref{eqn:reducecolpairs}, for any $i',j'\notin\{i,j\}$, we will take $i',j'$ to be not in the same family as $i,j$ but instead only in the same tribe. This is the part of the proof where individuals in different families in the same tribe help each other to develop the absorption properties.

To prove Lemma~\ref{lem:partA3}, we will use Lemma~\ref{lem:partA2} and Lemma~\ref{lem:auxiliarygraph} to build graphs $K_{\phi}$ for each $\phi \in \tau$, and $L_{\phi, uv}$ for each $\phi \in \tau$ and $uv \in E(K_{\phi})$, respectively. Lemma~\ref{lem:auxiliarygraph} allows us to conclude the existence of these useful graphs $L_{\phi, uv}$, each with its own connection property. In order to have our `codegree conditions' in Lemma~\ref{keylemma:absorption} (i.e., \ref{prop:abs:lowcodegree1}--\ref{prop:abs:lowcodegree3}), we want these graphs $L_{\phi, uv}$  not to share any edge too often. To get this property, we will take the graph as given by Lemma~\ref{lem:auxiliarygraph} and place it on the desired vertex set for $L_{\phi,uv}$ in some random manner.

\begin{center}
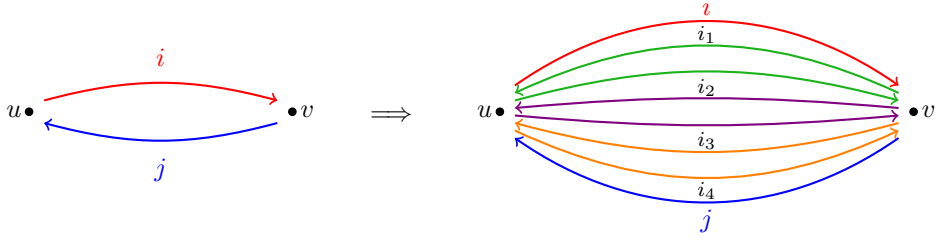
\begin{figure}
\begin{center}
\begin{tikzpicture}
\def\setsp{1}
\coordinate (u) at (-1.75*\setsp,0);
\coordinate (v) at (1.75*\setsp,0);

\draw [fill] (u) circle[radius=0.05];
\draw [fill] (v) circle[radius=0.05];

\def\spacein{0.2}
\draw [thick,red,->] ($(u)+(\spacein,0)+(0,0.15)$) to [out=15,in=165] ($(v)+(-\spacein,0)+(0,0.15)$);
\draw [thick,blue,<-] ($(u)+(\spacein,0)+(0,-0.15)$) to [out=-15,in=195] ($(v)+(-\spacein,0)+(0,-0.15)$);

\draw ($(u)-(0.2,0)$) node {$u$};
\draw ($(v)+(0.2,0)$) node {$v$};

\draw [white] (0,1.4) node {$i$};
\draw [white] (0,-1.45) node {$j$};

\draw [red] (0,0.72) node {$i$};
\draw [blue] (0,-0.77) node {$j$};
\end{tikzpicture}\;\;\;
\begin{tikzpicture}
\draw [white](0,-1.655) -- (0,0);
\draw (0,0) node {$\implies$};
\end{tikzpicture}
\;\;\;
\begin{tikzpicture}
\def\setsp{1}
\coordinate (u) at (-2.75*\setsp,0);
\coordinate (v) at (2.75*\setsp,0);

\draw [fill] (u) circle[radius=0.05];
\draw [fill] (v) circle[radius=0.05];

\def\spacein{0.2}
\draw [thick,red,->] ($(u)+(\spacein,0)+(0,0.35)$) to [out=35,in=145] ($(v)+(-\spacein,0)+(0,0.35)$);
\draw [thick,darkgreen,<-] ($(u)+(\spacein,0)+(0,0.25)$) to [out=25,in=155] ($(v)+(-\spacein,0)+(0,0.25)$);
\draw [thick,darkgreen,->] ($(u)+(\spacein,0)+(0,0.15)$) to [out=15,in=165] ($(v)+(-\spacein,0)+(0,0.15)$);
\draw [thick,violet,<-] ($(u)+(\spacein,0)+(0,0.05)$) to [out=5,in=175] ($(v)+(-\spacein,0)+(0,0.05)$);
\draw [thick,violet,->] ($(u)+(\spacein,0)+(0,-0.05)$) to [out=-5,in=185] ($(v)+(-\spacein,0)+(0,-0.05)$);
\draw [thick,orange,<-] ($(u)+(\spacein,0)+(0,-0.15)$) to [out=-15,in=195] ($(v)+(-\spacein,0)+(0,-0.15)$);
\draw [thick,orange,->] ($(u)+(\spacein,0)+(0,-0.25)$) to [out=-25,in=205] ($(v)+(-\spacein,0)+(0,-0.25)$);
\draw [thick,blue,<-] ($(u)+(\spacein,0)+(0,-0.35)$) to [out=-35,in=215] ($(v)+(-\spacein,0)+(0,-0.35)$);

\draw ($(u)-(0.2,0)$) node {$u$};
\draw ($(v)+(0.2,0)$) node {$v$};

\draw [red] (0,1.4) node {$i$};
\draw [blue] (0,-1.45) node {$j$};

\draw (0,1.025) node {\footnotesize $i_1$};
\draw (0,0.325) node {\footnotesize $i_2$};
\draw (0,-0.385) node {\footnotesize $i_3$};
\draw (0,-1.05) node {\footnotesize $i_4$};
\end{tikzpicture}
\end{center}
\caption{Replacing arrows representing a pair $\{(i,u),(j,v)\}$ with a sequence of pairs which have the same effect, but take the form $\{(i',u),(j',v)\}$ for only certain pairs $(i',j')$, as at \eqref{eqn:reducecolpairs}.}\label{fig:sparsecoloursacrossvertexpair}
\end{figure}
\end{center}


\begin{lemma}\label{lem:partA3} Let $R_i,S_i,T_i,U_i,V_i,W_i$, $i\in [n]$, satisfy \eref{prop:vxpartitionfirst}--\eref{prop:vxpartitionlast}.
Then, for each $\tau\in \mathcal{T}$, there exists a collection
\begin{equation}\label{eqn:Itau}
\mathcal{I}_\tau\subset \{\{(i,u),(j,v)\}: i,j\in I_\tau,i\neq j, u \in  S_i \setminus (R_i \cup T_j), v \in S_j \setminus (T_i \cup R_j),u\neq v,u{\sim}_{A/B}v\}
\end{equation}
such that the following hold.
\stepcounter{propcounter}
\begin{enumerate}[label = {\emph{\textbf{\Alph{propcounter}\arabic{enumi}}}}]
\item \labelinthm{prop:A3:regularityout} For each $i\in I_\tau$ and $u\in S_i\setminus R_i$, there are at most $20$ pairs $(j,v)$ such that $\{(i,u),(j,v)\}\in \cI_\tau$.
\item \labelinthm{prop:A3:nocodegree} For each $i,j\in I_\tau$ and $u\in S_i\setminus R_i$, there is at most one $v\in S_j\setminus R_j$ with $\{(i,u),(j,v)\}\in \cI_\tau$.
\item \labelinthm{prop:A3:boundedin}\labelinthm{prop:A3:lowcodegree0} For each $i\in I_\tau$ and $u\in S_i\setminus T_i$, there are at most $24$ pairs $(j,v)$ such that $\{(i,v),(j,u)\}\in \cI_\tau$.
\item \labelinthm{prop:A3:lowcodegree1} For each distinct $i,j\in I_\tau$ there are at most $n^{1/3}/2$ pairs $(u,v)$ with $\{(i,u),(j,v)\}\in \cI_\tau$.
\item \labelinthm{prop:A3:lowcodegree2} For each distinct $j,j'\in I_\tau$, there are at most $n^{1/3}/2$ tuples $(i,u,v,v')$ for which we have that $\{(i,u),(j,v)\},\{(i,u),(j',v')\}\in \cI_\tau$.
\item \labelinthm{prop:A3:lowcodegree3} For each $j\in I_\tau$ and $u\in S_j\setminus T_j$ there are at most $n^{1/3}/2$ pairs $(i,v)$ with $\{(i,u),(j,v)\}\in \cI_\tau$.
\item \labelinthm{prop:A3:corrections} For any collection of sets $R_i'\subset R_i$, $i \in I_\tau$, such that,
for each $i\in I_\tau$, $|R_i'|=|T_i|$ and, for each $\phi\in \mathcal{F}_\tau$, $\bigcup_{i \in I_\phi} R_i' =_{\mult} \bigcup_{i \in I_\phi} T_i$, there exists $\mathcal{C} \subset \mathcal{I}_\tau$ satisfying the following.
\begin{enumerate}[label = {\emph{\textbf{\Alph{propcounter}\arabic{enumi}.\arabic{enumii}}}}]
\item For every $i \in I_\tau$ and $u \in T_i$, there is exactly one  $(v,j)$ such that $\{(i,u),(j,v)\} \in \mathcal{C}$.\labelinthm{prop:A3:cor1}
\item For every $i \in I_\tau$ and $u \in R_i'$ there is exactly one  $(v,j)$ such that $\{(i,v),(j,u)\} \in \mathcal{C}$.\labelinthm{prop:A3:cor2}
\item For every $i\in I_\tau$ and $u\in R_i\setminus R_i'$, there is no $(v,j)$ such that $\{(i,v),(j,u)\}\in \mathcal{C}$.\labelinthm{prop:A3:cor3}
\item For every $i \in I_\tau$ and $u \in S_i\setminus (R_i \cup T_i)$, $(i,u)$ is $(\leq 1)$-balanced in $\mathcal{C}$.\labelinthm{prop:A3:cor4}
\end{enumerate}
\end{enumerate}
\end{lemma}
\begin{proof}
Let $\tau\in \mathcal{T}$. Using Lemma~\ref{lem:partA2}, for each $\phi\in \mathcal{F}_\tau$, let $K_\phi$ be a graph on $U_\phi\cup V_\phi$ with $\Delta(K_\phi)\leq 5$ which satisfies the property in Lemma~\ref{lem:partA2}.
 That is,  $\Delta(K_\phi)\leq 5$  and if
\[
\mathcal{C}\subset \{\{(i,u),(j,v)\}:i,j\in I_\phi,u,v\in U_\phi,i\neq j,u\neq v,u\sim_{A,B}v\}
\]
satisfies \ref{prop:A2:1} and \ref{prop:A2:2}, then there are paths $P_e$, $e\in \mathcal{C}$, in $K_\phi$ for which \ref{conc:A2:1} and \ref{conc:A2:2} hold.

We now wish, for each $\phi\in \mathcal{F}_\tau$ and $uv\in E(K_\phi)$ to define an auxiliary graph $L_{\phi,uv}$ using Lemma~\ref{lem:auxiliarygraph}. In this, $I_\phi$ will function as the set $U_{\phi}$ in that lemma, while we pick some subset $J_{\phi,uv}$ of $\{i\in I_\tau\setminus I_\phi:u,v\in W_i\}$ to use as $V_{\phi}$.
Where $\phi,\phi'\in \mathcal{F}_\tau$, $uv\in E(K_\phi)$ and  $u'v'\in E(K_{\phi'})$ satisfy $\phi\neq\phi'$ and $\{u,v\}\cap \{u',v'\}\neq \emptyset$ we will want $J_{\phi,uv}\cap J_{\phi',u'v'}=\emptyset$ for proving that the property \ref{prop:A3:corrections} holds, which motivates the following selection of the sets $J_{\phi,uv}$,  $\phi\in \mathcal{F}_\tau$ and $uv\in E(K_\phi)$.

Let $q=p_V^{-2}p_S^{2}p_{\fa}/10$. For each $\phi\in \mathcal{F}_\tau$ and $uv\in E(K_\phi)$, independently at random choose a set $J^+_{\phi,uv}\subset \{i\in I_\tau\setminus I_\phi:u,v\in W_i\}$ by including each $i\in I_\tau\setminus I_\phi$ with $u,v\in W_i$ independently at random with probability $q$.
For each $\phi\in \mathcal{F}_\tau$ and $uv\in E(K_\phi)$, let $J_{\phi,uv}$ be the set of $i\in J^+_{\phi,uv}$ which do not appear in any $J^+_{\phi',u'v'}$ with $\phi'\in \mathcal{F}_\tau$, $u'v'\in E(K_{\phi'})$, $\{u,v\}\cap \{u',v'\}\neq \emptyset$ and $(\phi,uv)\neq (\phi',u'v')$.

For each $\phi\in \mathcal{F}_\tau$ and $uv\in E(K_\phi)$, by \ref{prop:pairsinStauWphi} we have $|\{i\in I_\tau\setminus I_\phi:u,v\in W_i\}| \geq p_W^2p_S^{-2}p_{\tr}n/2$. Furthermore, for each
$i\in I_\tau\setminus I_\phi$ with $u,v\in W_i$,
\begin{align*}
|\{(\phi',u'v'):\phi'\in \mathcal{F}_\tau, u'v'\in E(K_{\phi'}), \{u,v\}\cap \{u',v'\}\neq \emptyset\}|
&\leq 5|\{\phi'\in \mathcal{F}_\tau:\{u,v\}\cap (U_{\phi'}\cup V_{\phi'})\neq \emptyset\}|\\
&\overset{\ref{prop:pairsinFtauUVphi}}{\leq}10p_V^2p_S^{-2}p_{\tr}^{-1}= 1/q,
\end{align*}
where we have used that $p_U\llpoly p_V$. Therefore,
\begin{align*}
\mathbb{E}|J_{\phi, uv}|&\geq q \cdot (1-q)^{1/p}\cdot |\{i\in I_\tau\setminus I_\phi:u,v\in W_i\}|\geq \frac{q}{e^2}\cdot \frac{p_W^2p_S^{-2}p_{\tr}n}{2}=\frac{p_W^2p_{\fa}p_{\tr}n}{2e^2p_V^2}\\
&\geq \frac{4}{p_V}p_{\fa}p_{\tr}n\geq \frac{2|I_\phi|}{p_V},
\end{align*}
where we have used that $p_V\llpoly p_W$.
 By
Chernoff's bound and a union bound, we can therefore assume the following property holds with high probability.
\stepcounter{propcounter}
\begin{enumerate}[label = {{\textbf{\Alph{propcounter}\arabic{enumi}}}}]\addtocounter{enumi}{7}
\item For each $\phi\in \mathcal{F}_\tau$ and $uv\in E(K_\phi)$, $|J_{\phi,uv}|\geq |I_\phi|/p_V$.
\label{prop:Jphij:large}
\end{enumerate}

For each $\phi\in \mathcal{F}_\tau$ and $uv\in E(K_\phi)$, using that $1/n\llpoly p_V\llpoly \log^{-1}n$, and \ref{prop:Jphij:large} and Lemma~\ref{lem:auxiliarygraph}, let $L'_{\phi,uv}$ be a graph with vertex set $I_\phi\cup J_{\phi,uv}$ and the properties in the lemma with $U_{\phi}=I_\phi$ and $V_{\phi}=J_{\phi,uv}$.
Let $\sigma_{\phi,uv}$ be a uniformly random permutation of $I_\phi\cup J_{\phi,uv}$ subject to $\sigma_{\phi,uv}(I_{\phi})=I_\phi$ and $\sigma_{\phi,uv}(J_{\phi, uv})=J_{\phi,uv}$.
Let $L_{\phi,uv}$ be the graph with vertex set $I_\phi\cup J_{\phi,uv}$ and edge set $\{\sigma_{\phi,uv}(x)\sigma_{\phi,uv}(y):xy\in E(L'_{\phi,uv})\}$. Observe that the properties of $L'_{\phi,uv}$ carry through to $L_{\phi,uv}$, which is to say that $\Delta(L_{\phi,uv})\leq 4$ and the following hold.

\begin{enumerate}[label = {{\textbf{\Alph{propcounter}\arabic{enumi}}}}]\addtocounter{enumi}{8}
\item There are no edges in $L_{\phi,uv}$ with both vertices in $I_\phi$.\label{prop:Lphi:new}
\item Given any $r\in \N$ and any vertex-disjoint pairs $a_1b_1,\ldots,a_rb_r\in I_\phi^{(2)}$, there are vertex-disjoint paths $P_i$, $i\in [r]$, in $L_{\phi,uv}$ with internal vertices in $J_{\phi,uv}$ such that, for each $i\in [r]$, $P_i$ is an $x_i,y_i$-path.\label{prop:Lphi:connectivity}
\end{enumerate}

We can now choose our set of pairs $\mathcal{I}_\tau$. Let
\begin{equation}\label{eqn:Itaudefn}
\mathcal{I}_\tau=\bigcup_{\phi\in \mathcal{F}_\tau}\bigcup_{uv\in E(K_\phi)}\{\{(i,u),(j,v)\}:ij\in E(L_{\phi,uv}),u\in S_i\setminus (R_i\cup T_j),v\in S_j\setminus (T_i\cup R_j)\}.
\end{equation}
Note that, as, for each $\phi\in \mathcal{F}_\tau$ and $uv\in E(K_\phi)$, $V(L_{\phi,uv})=I_\phi\cup J_{\phi,uv}$, we have that \eqref{eqn:Itau} holds.

We will now show that \ref{prop:A3:regularityout}--\ref{prop:A3:boundedin}, \ref{prop:A3:lowcodegree3}, and \ref{prop:A3:corrections} hold, and \ref{prop:A3:lowcodegree1} and \ref{prop:A3:lowcodegree2} hold with high probability, and therefore we can take $\mathcal{I}_\tau$ with the claimed properties. 

\smallskip

\noindent\ref{prop:A3:regularityout}: Let $i\in I_\tau$ and $u\in S_i\setminus R_i$. If $u\in U_i\cup V_i$, then, for each $\phi\in \mathcal{F}_\tau$ and any $v$ such that $uv\in E(K_\phi)$, $i\notin J_{\phi,uv}$ as $u\notin W_i$. Therefore, the only graphs $L_{\phi,uv}$ with $i\in V(L_{\phi,uv})$ are those with $i\in I_\phi$ and $uv\in E(K_\phi)$. As the sets $I_\phi$, $\phi\in \mathcal{F}_\tau$, are disjoint, and, for each $K_\phi$ there are at most 5 vertices $v$ such that $uv\in E(K_\phi)$, there are at most 5 graphs $L_{\phi,uv}$, for some $\phi$ and $v$, with $i\in V(L_{\phi,uv})$. As any graph $L_{\phi,uv}$ has maximum degree at most 4, we thus have that there are at most 20 pairs $(j,v)$ with $\{(u,i),(j,v)\}\in \mathcal{I}_\tau$.

 Suppose, then, that $u\in W_i$. Then, for each $\phi\in \mathcal{F}_\tau$ and any $v$ such that $uv\in E(K_\phi)$, $i\notin I_\phi$ as $u\notin U_i\cup V_i$.
 Therefore, the only graphs $L_{\phi,uv}$ with $i\in V(L_{\phi,uv})$ are those with $i\in J_{\phi,uv}$ and $uv\in E(K_\phi)$. For each $i\in I_\tau$, there is at most one pair $(\phi,v)$ with $i\in J_{\phi,uv}$ by the choice of the $J_{\phi,uv}$. As, here, $L_{\phi,uv}$ has maximum degree at most 4, we thus have that there are at most 4 pairs $(j,v)$ with $\{(u,i),(j,v)\}\in \mathcal{I}_\tau$. Therefore,  \ref{prop:A3:regularityout} holds in both cases  $u\notin W_i$ and $u\in W_i$.

 \smallskip

  \noindent\ref{prop:A3:nocodegree}: Suppose for contradiction that there is some $i,j\in I_\tau$, $u\in S_i\setminus T_i$ and distinct $v,v'\in S_j\setminus T_j$ with $\{(i,u),(j,v)\},\{(i,u),(j,v')\}\in \mathcal{I}_\tau$. Then, from \eqref{eqn:Itaudefn}, there is some $\phi\in \mathcal{F}$ with $uv\in E(K_\phi)$ and $ij\in E(L_{\phi,uv})$ as well as some $\phi'\in \mathcal{F}$ with $uv\in E(K_{\phi'})$ and $ij\in E(L_{\phi',uv})$. Now, as $(\phi,uv)\neq (\phi',u'v')$ and $\{u,v\}\cap \{u,v'\}\neq \emptyset$, we have that each of $i$ and $j$ cannot appear in both $J_{\phi,uv}\subset J_{\phi,uv}^+$ and $J_{\phi',uv'}\subset J_{\phi',uv'}^+$.
Then, as $V(L_{\phi,uv})=I_\phi\cup J_{\phi,uv}$ and $V(L_{\phi',uv'})=I_{\phi'}\cup J_{\phi',uv'}$, by \ref{prop:Lphi:new}, we have that (swapping the labels of $\phi',v'$ with $\phi,v$ if necessary), that $i\in I_\phi$, $j\in J_{\phi,uv}$, $j\in I_{\phi'}$ and $i\in J_{\phi',uv'}$. As $i\in J_{\phi',uv'}\subset J_{\phi',uv'}^+$, we have that $u\in W_i$. On the other hand, $uv\in E(K_\phi)$ and $i\in I_\phi$ implies that $u\in V(K_\phi)=U_\phi\cup V_\phi=U_i\cup V_i$, a contradiction. Thus, \ref{prop:A3:nocodegree} holds.

 \smallskip

 \noindent\ref{prop:A3:boundedin}: Let $i\in I_\tau$ and $u\in S_i\setminus T_i$. Let $\phi(i)$ be the unique $\phi(i)\in \mathcal{F}_\tau$ with $i\in I_{\phi(i)}$.
  If $(j,v)$ is such that $\{(i,v),(j,u)\}\in \mathcal{I}_\tau$, then there is some $\phi(j,v)\in \mathcal{F}_\tau$ such that $ij\in L_{\phi(j,v),uv}$.

  First, we count the choices for such $(j,v)$ for which $\phi(v,j)\neq \phi(i)$. In this case we have that $i\in J_{\phi(j,v),uv}$. However, there is at most one pair $(\phi(j,v),v)$ with $i\in J_{\phi(j,v),uv}$, and, having chosen this, at most 4 different vertices $j\in V(L_{\phi(j,v),uv})$ with $ij\in E(L_{\phi,uv})$. Therefore, there are at most 4 choices for $(j,v)$ with $\phi(j,v)\neq \phi(i)$ and $ij\in E(L_{\phi(j,v),uv})$.

  Second, we count the choices for $(j,v)$ for which $\phi(j,v)=\phi(i)$ and $ij\in E(L_{\phi(j,v),uv})$.  As $\Delta(K_{\phi(i)})\leq 5$, there are at most 5 choices for $v$ for which $uv\in E(K_{\phi(i)})$. As $\Delta(L_{\phi,uv})\leq 4$, there are then at most 4 choices for $j$ such that $ij\in E(L_{\phi,uv})$.

  In total, then, there at most 24 choices for $(j,v)$ such that $ij\in E(L_{\phi,uv})$ for some $\phi\in \mathcal{F}_\tau$. Thus, there are at most 24 choices for $(j,v)$ such that $\{(i,v),(j,u)\}\in \mathcal{I}_\tau$, and therefore \ref{prop:A3:boundedin} holds.

  \smallskip

 \noindent\ref{prop:A3:lowcodegree1}: Let $i,j\in I_\tau$.  Note that the pair $(u,v)$ satisfies $\{(i,u), (j,v)\} \in \mathcal{I}_{\tau}$ only if there exists $\phi \in \mathcal{F}_{\tau}$ with $uv\in E(K_{\phi})$ such that $ij \in E(L_{\phi, uv})$.
For each $\phi\in \mathcal{F}_\tau$ and $uv\in E(K_\phi)$, by \ref{prop:Lphi:new}, we have that
\begin{equation}\label{eq:inLprobinL}
\P(ij\in E(L_{\phi,uv}))\leq \frac{4}{|J_{\phi,uv}|-1}\leq \frac{5p_V}{|I_\phi|}\leq \frac{6p_V}{p_\tr p_\fa n},
\end{equation}
where we are using that $\P(ij\in  E(L_{\phi,uv}))=0$ if either $i$ or $j\notin V(L_{\phi,uv})$.
The events $\{ij\in E(L_{\phi,uv})\}$ are independent over $\phi\in \mathcal{F}_\tau$ and $uv\in E(K_\phi)$, and there are at most $|\mathcal{F}_\tau|\cdot  5n\leq 10p_{\fa}^{-1}n$ such events. Thus, the expected number of triples $(\phi,u,v)$ with $uv\in E(K_{\phi})$ and $ij \in E(L_{\phi, uv})$ is at most
\[
\frac{6p_V}{p_\tr p_\fa n}\cdot 10p_{\fa}^{-1}n\leq p_\tr^{-1}p_{\fa}^{-2}\leq n^{1/3}/4,
\]
 where we have used that $1/n\llpoly p_\tr,p_\fa$. Thus, by Lemma~\ref{chernoff}, with probability $1-\exp(-\omega(\log n))$, the number of pairs $(u,v)$ with $\{(i,u), (j,v)\} \in \mathcal{I}_{\tau}$ is at most $n^{1/3}/2$. Taking a union bound then completes the proof of \ref{prop:A3:lowcodegree1}.

 \smallskip

 \noindent\ref{prop:A3:lowcodegree2}: For each $\phi,\phi'\in \mathcal{F}_\tau$, using that $\Delta(K_{\phi}\cup K_{\phi'})\leq 10$, greedily colour the edges of $K_{\phi}\cup K_{\phi'}$ as $c_{\phi,\phi'}:E(K_{\phi}\cup K_{\phi'})\to [250]$ so that any two edges of $K_{\phi}\cup K_{\phi'}$ with the same colour are a distance at least 2 apart in $K_{\phi}\cup K_{\phi'}$ (as opposed to the more normal proper colouring where this distance is at least 1).
  Let $\mathcal{R}$ be the set of $(\phi,\phi',c,d)$ with $\phi,\phi'\in \mathcal{F}_\tau$ and $c,d\in [250]$.

Now, let $j,j' \in I_{\tau}$ be distinct. For each $(i,u,v,v')$ such that $\{(i,u),(j,v)\}, \{(i,u), (j',v')\} \in \mathcal{I}_{\tau}$ there is some $(\phi,\phi',c,d)$ such that $uv\in E(K_\phi)$, $uv'\in E(K_{\phi'})$, $ij \in E(L_{\phi, uv})$ and $ij' \in E(L_{\phi', uv'})$, and the edges $uv$ and $uv'$ have colour $c$ and $d$ respectively in the colouring $c_{\phi,\phi'}$. Note that if $(\phi,\phi',c,d)$ and $u$ are known, then $v$ and $v'$ are known, let them be $v_{\phi,\phi',c,d,u}$ and $v'_{\phi,\phi',c,d,u}$ respectively.
Then, for each $(\phi,\phi',c,d)\in \mathcal{R}$, let $\mathcal{E}_{(\phi,\phi',c,d)}$ be the set of $(i,u)$ for which $u$ has a colour-$c$ and a colour-$d$ neighbour in $K_{\phi}\cup K_{\phi'}$ under the colouring $c_{\phi,\phi'}$, $i,j\in V(K_{\phi,uv})$ and $i,j\in V(K_{\phi',uv'})$. Thus, we have, easily, that $|\mathcal{E}_{(\phi,\phi',c,d)}|\leq 2n^2$.

Now, for any fixed $(\phi,\phi',c,d)\in \mathcal{R}$, the events $\{ij \in E(L_{\phi, uv_{\phi,\phi',c,d,u}})\text{ and }ij' \in E(L_{\phi', uv'_{\phi,\phi',c,d,u}})\}$ are independent across all $(i,u)\in \mathcal{E}_{(\phi,\phi',c,d)}$. Each of these events occurs with probability at most $n^2\cdot (6p_V/p_\tr p_\fa n)^2\leq n^{1/4}$. Therefore, for each $(\phi,\phi',c,d)\in \mathcal{R}$, with probability $1-\exp(-\omega(\log n))$, we have that the number of $(i,u)\in \mathcal{E}_{(\phi,\phi',c,d)}$ with $ij \in E(L_{\phi, uv_{\phi,\phi',c,d,u}})$ and $ij' \in E(L_{\phi', uv'_{\phi,\phi',c,d,u}})$ is at most $2n^{1/4}$.
As we have $|\mathcal{R}|\leq (2p_\fa^{-1})^2\cdot 250^2$, and $1/n\llpoly p_\fa^{-1}$, using a union bound, with probability $1-\exp(-\omega(\log n))$, we have that \ref{prop:A3:lowcodegree2} holds for any distinct fixed $j,j' \in I_{\tau}$. Thus, by another union bound, \ref{prop:A3:lowcodegree2} holds with high probability.

 \smallskip

 \noindent\ref{prop:A3:lowcodegree3}: Let $j\in I_\tau$ and $u\in S_j\setminus T_j$. Then, the pair $(i,v)$ satisfies $\{(i,u), (j,v)\} \in \mathcal{I}_{\tau}$ only if there exists $\phi \in \mathcal{F}_{\tau}$ with $uv\in E(K_{\phi})$ such that $ij \in E(L_{\phi, uv})$. However, there are at most $|\mathcal{F}_{\tau}|\leq 2p_\fa^{-1}$ choices for $\phi\in \mathcal{F}_\tau$, and, after this, at most 5 choices for $v$ with $uv\in E(K_{\phi})$ and then at most 4 choices for $i$ with $ij \in E(L_{\phi, uv})$.
 Thus, in total, there are at most $2p_\fa^{-1}\cdot 4\cdot 5\leq n^{1/3}/2$ choices for $(i,v)$ with $\{(i,u), (j,v)\} \in \mathcal{I}_{\tau}$.

 \smallskip

 \noindent\ref{prop:A3:corrections}: Let $R_i'\subset R_i$, $i \in I_\tau$, be any collection of sets  such that,
 for each $i\in [n]$, $|R_i'|=|T_i|$, and, for each $\phi\in \mathcal{F}_\tau$, $\bigcup_{i \in I_\phi} R_i' =_{\mult} \bigcup_{i \in I_\phi} T_i$.
As  \ref{prop:vxpartitionfirst}--\ref{prop:vxpartitionlast} hold, by Lemma~\ref{lem:partA1}, for each $\phi \in \mathcal{F}_{\tau}$, there is a set
 \begin{equation}\label{eq:whereisCphi}
 \mathcal{C}_\phi\subset \{\{(i,u),(j,v)\}:i,j\in I_\phi,i\neq j,u\in U_\phi\setminus (R_i\cup T_j)\text{ and }v\in U_\phi\setminus (T_i\cup R_j),u\neq v\}.
 \end{equation}
 such that the following hold.

 \stepcounter{propcounter}
 \begin{enumerate}[label = {{\textbf{\Alph{propcounter}\arabic{enumi}}}}]
 \item \label{propfromA1-1} For every $i \in I_\phi$ and $u \in T_i$, there is exactly one  $(j,v)$ such that $\{(i,u),(j,v)\} \in \mathcal{C}_\phi$.
 \item \label{propfromA1-2} For every $i \in I_\phi$ and $u \in R_i'$ there is exactly one  $(j,v)$ such that $\{(i,v),(j,u)\} \in \mathcal{C}_\phi$.
 \item \label{propfromA1-2b} For every $i\in I_\phi$ and $u\in R_i\setminus R_i'$ there is no $(j,v)$ such that $\{(i,v),(j,u)\}\in \mathcal{C}_\phi$.
 \item \label{propfromA1-3} For every $i \in I_\tau$ and $u \in U_i\setminus (R_i \cup T_i)$,  $(i,u)$ is $(\leq 1)$-balanced in $\mathcal{C}_\phi$.
 \end{enumerate}

Let $D$ be the coloured multi-digraph with vertex set $S_\tau$ where, for each $\{(i,u),(j,v)\}\in \bigcup_{\phi\in \mathcal{F}_\tau}\mathcal{C}_\phi$ we add $\vec{uv}$ with colour $i$ and $\vec{vu}$ with colour $j$. Note that, for each $i\in I_\tau$, \ref{propfromA1-1}--\ref{propfromA1-3} and \eqref{eq:whereisCphi} imply that the edges with colour $i$ form exactly a vertex-disjoint collection of some directed cycles, with no 2-cycles, and $|T_i|$ directed paths from $T_i$ to $R_i'$, all of which are in $D[U_i\setminus (R_i\setminus R_i')]$.

Now, for each $\phi\in \mathcal{F}_\tau$, from \ref{propfromA1-1}--\ref{propfromA1-3}, we have that \ref{prop:A2:1} and \ref{prop:A2:2} hold with $\mathcal{C}$ replaced by $\mathcal{C}_\phi$. Therefore, there are paths $P_e$, $e\in \mathcal{C}_\phi$, in $K_\phi$ for which the following hold.
 \begin{enumerate}[label = {{\textbf{\Alph{propcounter}\arabic{enumi}}}}]\addtocounter{enumi}{4}
 \item For each $e=\{(i,u),(j,v)\}\in \mathcal{C}_\phi$, $P_e$ is a $u,v$-path with internal vertices in $V_\phi\cap A$ if $u,v\in A$ and internal vertices in $V_\phi\cap B$ if $u,v\in B$.\label{propfromA2:1}
 \item For each $e=\{(i,u),(j,v)\},e'=\{(i',u'),(j',v')\}\in \mathcal{C}'$ with $e\neq e'$, if $\{i,j\}\cap \{i',j'\}\neq \emptyset$, then $V(P_e)$ and $V(P_{e'})$ intersect only on $\{u,v\}\cap \{u',v'\}$.\label{propfromA2:2}
 \end{enumerate}

For each $e=\{(i,u),(j,v)\}\in \mathcal{C}_\phi$, arbitrarily direct $e$ from $(i,u)$ to $(j,v)$, let $\ell_e$ be the length of $P_e$ and label its vertices as $u_{e,0}=u,u_{e,1},\ldots,u_{e,{\ell_e-1}},u_{e,\ell_e}=v$.
For each $\phi\in \mathcal{F}_\tau$, let
\[
\mathcal{C}'_\phi=\bigcup_{e=\{(i,u),(j,v)\}\in \mathcal{C}_{\phi}}\bigcup_{r\in [\ell_e]}\{(i,u_{e,r-1}),(j,u_{e,r})\}.
\]

Let $D'$ be the coloured multi-digraph with vertex set $S_\tau$ where, for each $\{(i,u),(j,v)\}\in \bigcup_{\phi\in \mathcal{F}_\tau}\mathcal{C}'_\phi$ we add $\vec{uv}$ with colour $i$ and $\vec{vu}$ with colour $j$. Note that to create $D'$ from $D$, we would take each $\vec{uv}\in E(D)$, with colour $i$ say, and replace it with a directed $u,v$-path of edges with colour $i$ whose underlying path is $P_{\{(i,u),(j,v)\}}$ for some $j$, where this $j$ is unique by \ref{propfromA1-1}.
For each $i\in I_\tau$, by \ref{propfromA2:2}, the interior vertices of the paths with colour $i$ are all vertex-disjoint and lie in $V_\phi$. Therefore, from the similar property for $D$, for each $i\in I_\tau$, the edges with colour $i$ in $D'$ form exactly a vertex-disjoint collection of some directed cycles, with no 2-cycles, and  $|T_i|$ vertex-disjoint directed paths from $T_i$ to $R_i'$ in $D'[(U_i\cup V_i)\setminus (R_i\setminus R_i')]$.

Then, for each $\phi\in \mathcal{F}_\tau$ and $uv\in E(K_\phi)$, let $\mathcal{I}_{\phi,uv}$ be the set of pairs $\{i,j\}$ such that $\{(i,u),(j,v)\}\in \mathcal{C}_{\phi}'$. For each $i\in I_\tau$, as $D'$ has no directed 2-cycles of colour $i$ and every vertex has out-degree in $D'$ at most 1 in the colour $i$ edges, for each $u,v\in S_\tau$ there is at most one edge with vertex set $\{u,v\}$ and colour $i$ in $D'$.
Therefore, the pairs in $\mathcal{I}_{\phi,uv}$ are disjoint for each $\phi\in \mathcal{F}_\tau$ and $uv\in E(K_\phi)$. Thus, by \ref{prop:Lphi:connectivity}, we can find paths $Q_e$, $e\in \mathcal{I}_{\phi,uv}$, which are vertex-disjoint, such that, for each $\phi\in \mathcal{F}_\tau$ and $e=(i,j)\in \mathcal{I}_{\phi,uv}$, $Q_e$ is an $i,j$-path in $L_{\phi,uv}$ with interior vertices in $J_{\phi,uv}$. For such a path $Q_e$, let $s_e$ be the length of $Q_e$ and label its vertices as $i_{e,0}=i$, $i_{e,1}$, $\ldots$, $i_{e,s_e}=j$.

For each $\phi\in \mathcal{F}_\tau$, let
\[
\mathcal{C}''_\phi=\bigcup_{e=\{(i,u),(j,v)\}\in \mathcal{C}'_{\phi}}\bigcup_{r\in [s_e]}\{(i_{e,r-1},u),(i_{e,r},v)\},
\]
Let $\mathcal{C}=\bigcup_{\phi\in \mathcal{F}_\tau}\mathcal{C}'_\phi$, noting that it follows from \eqref{eqn:Itaudefn} that $\mathcal{C}\subset \mathcal{I}_\tau$.
Let $D''$ be the coloured multi-digraph with vertex set $S_\tau$ where, for each $\{(i,u),(j,v)\}\in \mathcal{C}$ we add $\vec{uv}$ with colour $i$ and $\vec{vu}$ with colour $j$. Note that to create $D''$ from $D'$, we would take each pair $\vec{uv},\vec{vu}\in E(D')$ with colour $i$ and $j$ respectively, such that $e=\{(i,u),(j,v)\}\in \bigcup_{\phi\in \mathcal{F}_\tau}\mathcal{C}'_\phi$, and add the edges $\vec{uv},\vec{vu}$ with colour $i_{e,r}$ for each $r\in [s_e-1]$. Here, $i_{e,r}$ is an interior vertex of the path $Q_e$, and therefore lies in $J_{\phi,uv}$, so that $u,v\in W_{i_{e,r}}$. Using this, and the definition of the sets $J_{\phi,uv}$, we have, for each $i\in [n]$, that the 2-cycles with colour $i$ that we add to get from $D'$ to $D''$ are vertex-disjoint from each other and from the edges in $D'$ with colour $i$. Thus, from the similar property for $D'$ and the construction of $D''$ from $D'$, for each $i\in I_\tau$, the edges with colour $i$ in $D'$ form exactly a vertex-disjoint collection of some directed cycles, with no 2-cycles, and  $|T_i|$ vertex-disjoint directed paths from $T_i$ to $R_i'$ in $D[(U_i\cup V_i\cup W_i)\setminus (R_i\setminus R_i')]$. From the direct definition of $D''$, we therefore have that \ref{prop:A3:cor1}--\ref{prop:A3:cor4} hold for $\mathcal{C}$, as required.
\end{proof}


\subsection{Part~\ref{partA4}: regularisation of the collection of pairs}\label{sec:finalpartA}

For each $\tau\in \mathcal{T}$, we now take a collection $\mathcal{I}_\tau$ for which \ref{prop:A3:regularityout}--\ref{prop:A3:corrections} hold, as provided by Lemma~\ref{keylemma:absorption}, and add more pairs to it so that, for all $i\in I_\tau$ and $u\in S_i\setminus R_i$, $(i,u)$ is in the same number of pairs in $\mathcal{I}_\tau$. I.e., we will have that \ref{prop:abs:regularityout} holds instead of \ref{prop:A3:regularityout}. This is the key regularity property we use when later showing the near-regularity of a certain auxiliary hypergraph in Section~\ref{sec:real}.

While achieving this regularity, we need that the pairs added will not worsen overmuch the other conditions on $\mathcal{I}_\tau$. It is useful then to compare the properties  \ref{prop:abs:regularityout}--\ref{prop:abs:corrections} and \ref{prop:A3:regularityout}--\ref{prop:A3:corrections} used in Lemma~\ref{keylemma:absorption} and Lemma~\ref{lem:partA3}, respectively. As noted, our aim is to add pairs to $\mathcal{I}_\tau$ so that \ref{prop:abs:regularityout} holds instead of \ref{prop:A3:regularityout}. The conditions \ref{prop:abs:corrections} and \ref{prop:A3:corrections} are the same, and continue to hold on the addition of new pairs to $\mathcal{I}_\tau$.
 The conditions \ref{prop:abs:lowcodegree0}--\ref{prop:abs:lowcodegree3} are a relaxation of \ref{prop:A3:lowcodegree0}--\ref{prop:A3:lowcodegree3} respectively, where an extra factor of at least 2 is permitted in the bounds they each claim. This allows us the room to add more pairs to $I_\tau$ without breaking the relaxed conditions, and we will do this by only adding pairs from some random collection.

On the other hand, \ref{prop:abs:nocodegree} and \ref{prop:A3:nocodegree} are the same condition. For convenience we repeat this, as follows.

\smallskip

\begin{minipage}{0.9\textwidth}
\begin{itemize}
\item[\ref{prop:abs:nocodegree}/\ref{prop:A3:nocodegree}] For each $i,j\in I_\tau$ and $u\in S_i\setminus R_i$, there is at most one $v\in S_j\setminus R_j$ with $\{(i,u),(j,v)\}\in \mathcal{I}_\tau$.
\end{itemize}
\end{minipage}

\smallskip

\noindent Therefore, this condition gives triples $(i,j,u)$ which cannot appear together in any pair $\{(i,u),(j,v)\}$ we add to $\mathcal{I}_\tau$, and for each other such triple $(i,j,u)$ we can add at most 1 pair $\{(i,u),(j,v)\}$. This condition is not too onerous: as we certainly always need $|{\mathcal{I}}_\tau|\leq n^2$, this is a small proportion of the triples $(i,j,u)$ with $i,j\in I_\tau$ and $u\in S_i\setminus R_i$. Thus, within the random collection of pairs which we consider for addition to $\mathcal{I}_\tau$, we will only need to remove a few more pairs in order to guarantee that \ref{prop:abs:nocodegree} holds whichever set of pairs from this random collection we add.

\begin{proof}[Proof of Lemma~\ref{keylemma:absorption}]
Using Lemma~\ref{lem:partA3}, let
\[
\mathcal{I}_\tau\subset \{\{(i,u),(j,v)\}: i,j\in I_\tau,i\neq j, u \in  S_i \setminus (R_i \cup T_j), v \in S_j \setminus (T_i \cup R_j),u\neq v,u{\sim}_{A/B}v\}
\]
satisfy \ref{prop:A3:regularityout}--\ref{prop:A3:corrections}.
For each $i\in I_\tau$ and $u\in S_i\setminus R_i$, let $\lambda_{i,u}^-$ be the number of pairs $(j,v)$ such that $\{(i,u),(j,v)\}\in \mathcal{I}_\tau$, so that $\lambda_{i,u}^-\leq 20$ by \ref{prop:A3:regularityout}, and let $\lambda_{i,u}=24-\lambda_{i,u}^-$, so that $4\leq \lambda_{i,u}\leq 24$.

Let $\mathcal{E}^A=\{(i,u,r):i\in I_\tau,u\in A\cap (S_i\setminus R_i)\text{ and }r\in [\lambda_{i,u}]\}$. Note that, as $24|\{(i,u):i\in I_\tau,A\cap (S_i\setminus R_i)\}|$ is even, and any $\{(i,u),(j,v)\}\in \mathcal{I}_\tau$ with $u\in A$ has $v\in A$, we have that $|\mathcal{E}^A|$ is even.
 Form the auxiliary graph $L^A$ with vertex set $\mathcal{E}^A$ and, for each $(i,u,r),(j,v,s)\in \mathcal{E}^A$, an edge between $(i,u,r)$ and $(j,v,s)$ if $u\notin T_j$, $v\notin T_i$, $u\neq v$, $i\neq j$,
 and there is no $w$ such that $\{(i,u),(j,w)\}\in \mathcal{I}_\tau$ or $\{(i,w),(j,v)\}\in \mathcal{I}_\tau$.
 We now show that $L^A$ has high minimum degree.

 \begin{claim} $\delta(L^A)\geq (1-\sqrt{p_T})|L^A|$ and $|L^A|\geq p_Sp_{\tr}n^2$. \label{clm:Lmindeg}
 \end{claim}
\begin{proof}[Proof of Claim~\ref{clm:Lmindeg}]
Let $(i,u,r)\in \mathcal{E}^A$, so that $i\in I_\tau$, $u\in A\cap (S_i\setminus R_i)$ and $r\in [\lambda_{i,u}]$. Let $(j,v,s)\in \mathcal{E}^A$. We have that $(j,v,s)\notin N_{L^A}((i,u,r))$ only if at least one of the following hold: \textbf{i)} $u\in T_j$,
\textbf{ii)} $v\in T_i$,
\textbf{iii)} $u=v$,
\textbf{iv)} $i=j$, \textbf{v)} there is some $w$ such that $\{(i,u),(j,w)\}\in \mathcal{I}_\tau$, or \textbf{vi)} there is some $w$ such that $\{(i,w),(j,v)\}\in \mathcal{I}_\tau$.
We will count the number of $(j,v,s)\in \mathcal{E}^A\setminus \{(i,u,r)\}$ satisfying each of \textbf{i)} -- \textbf{vi)} in turn.

Using \ref{prop:vxpartitionfirst}--\ref{prop:vxpartitionlast}, and that $|I_{\tau}|=(1\pm \epsforabsthatwasepszero)p_{\tr}n$ and $|\mathcal{F}_\tau|=(1\pm \epsforabsthatwasepszero )p^{-1}_{\fa}$ for any $\tau \in \mathcal{T}$, we have the following. \textbf{i)} By \ref{prop:verticesinTi}, there are at most $|\mathcal{F}_\tau|\cdot 2(1+\epsforabsthatwasepszero)p_Tp_U^{-1}p_\fa p_{\tr}n\leq 8p^Tp_U^{-1}p_\tr n$ choices for $j\in I_\tau$ such that $u\in T_j$, and thus at most $8p_Tp_U^{-1}p_\tr n\cdot 2p_Sn\cdot 24\leq \sqrt{p_T}p_\tr n^2/10$ choices with $(j,v,s)\in \mathcal{E}^A$ and $u\in T_j$, where we have used that $p_T\llpoly p_U$.
\textbf{ii)} By \ref{prop:vxpartitionfirst}, there are at most $2p_Tn$ choices of $v\in A\cap T_i$, so that there are thus at most $2p_{\tr}n\cdot 2p_Tn\cdot 24\leq \sqrt{p_T}p_Sp_\tr n^2/10$ choices of $(j,v,s)\in \mathcal{E}^A$ with $v\in T_i$.

\textbf{iii)} There are at most $2p_{\tr}n\cdot 1\cdot 24\leq \sqrt{p_T}p_Sp_\tr n^2/10$ choices of $(j,v,s)\in \mathcal{E}^A$ with $v\in T_i$ with $v=u$.
\textbf{iv)} By \ref{prop:vxpartitionfirst}, there are at most $2p_Sn\cdot 24\leq \sqrt{p_T}p_Sp_\tr n^2/10$ choices of $(j,v,s)\in \mathcal{E}^A$ with $j=i$.
\textbf{v)} By \ref{prop:A3:regularityout}, there are at most $20$ choices of $j$ for which there is some $w$ with $\{(i,u),(j,w)\}\in \mathcal{I}_\tau$, and therefore at most $20\cdot 24p_Sn\leq \sqrt{p_T}p_Sp_\tr n^2/10$ choices of $(j,v,s)\in \mathcal{E}^A$ for which there is some $w$ with $\{(i,u),(j,w)\}\in \mathcal{I}_\tau$.
\textbf{vi)} Similarly, but using \ref{prop:A3:boundedin}, there are at most $24\cdot 24p_Sn\leq \sqrt{p_T}p_Sp_\tr n^2/10$ choices of $(j,v,s)\in \mathcal{E}^A$ for which there is some $w$ with $\{(i,w),(j,v)\}$.
Combining all of this, the number of non-neighbours of $(i,u,r)$ in $L^A$ is at most $\sqrt{p_T}p_Sp_{\tr}n^2$.

Now, as $\lambda_{i,u}\geq 4$ for each $i\in I_{\tau}$ and $u\in A\cap (S_i\setminus R_i)$, we have that
\begin{equation*}
|L^A|=|\mathcal{E}^A|\geq 4|I_\tau|\cdot |A\cap (S_i\setminus R_i)|\geq 4(1-\epsforabsthatwasepszero)p_{\tr}n\cdot (1-\epsforabsthatwasepszero)p_Sn\geq p_Sp_{\tr}n^2,
\end{equation*}
so that, in combination with $|V(L^A)\setminus N_{L^A}(x)|\leq \sqrt{p_T}p_Sp_{\tr}n^2$ for each $x\in V(L^A)$,  we have that the claim holds.
\claimproofend

Let $p=\log^2n/|L^A|$.
Form $\hat{L}_A\subset L_A$ with $V(\hat{L}_A)=V(L_A)$ by including each edge of ${L}_A$ independently at random with probability $p$. We will show the following claim.

\begin{claim} With high probability, we have the following properties. \label{clm:Lhatgood}
\stepcounter{propcounter}
\begin{enumerate}[label = {{\textbf{\Alph{propcounter}\arabic{enumi}}}}]
\item \label{prop:regularise:mindeg} For each $U\subset V(\hat{L}^A)$ with $|U|\leq |\hat{L}^A|/10^3$, $|N_{\hat{L}^A}(U)|\geq 20|U|$.
\item For each disjoint $U,V\subset V(\hat{L}^A)$ with $|U|,|V|\geq |\hat{L}^A|/10^4$, $e_{\hat{L}^A}(U,V)\geq 20|U|$. \label{prop:regularise:joined}
\item For each $x\in \mathcal{E}^A$, there are at most 12 tuples $(j,v,s,v',s')$ for which $(j,v,s),(j,v',s)\in N_{\hat{L}^A}(x)$.\label{prop:regularise:fornocod1}
\item For each $x\in \mathcal{E}^A$, there are at most 4 tuples $(j,v,s,u,r)$ with $(i,u,r)\neq x$ for which $(j,v,s)\in N_{\hat{L}^A}(x)$ and $(i,u,r)(j,v,s)\in E(\hat{L})^A$.\label{prop:regularise:fornocod2}
\item \label{prop:regularise:lowcod0} For each $i\in I_\tau$ and $u\in S_i\setminus T_i$, there are at most $n^{1/3}/4$ pairs $(j,v)$ such that $\{(i,v),(j,u)\}\in E(\hat{L}_A)$.
\item \label{prop:regularise:lowcod1} For each distinct $i,j\in I_\tau$ there are at most $n^{1/3}/4$ pairs $(u,v)$ with $\{(i,u),(j,v)\}\in E(\hat{L}_A)$.
\item \label{prop:regularise:lowcod2} For each distinct $j,j'\in I_\tau$, there are at most $n^{1/3}/4$ tuples $(i,u,v,v')$ for which we have that $\{(i,u),(j,v)\},\{(i,u),(j',v')\}\in E(\hat{L}_A)$.
\item \label{prop:regularise:lowcod3} For each $j\in I_\tau$ and $u\in S_j\setminus T_j$ there are at most $n^{1/3}/4$ pairs $(i,v)$ with $\{(i,u),(j,v)\}\in E(\hat{L}_A)$.
\end{enumerate}
\end{claim}
\begin{proof}[Proof of Claim~\ref{clm:Lhatgood}]
\ref{prop:regularise:mindeg}: Let $U\subset V(\hat{L}^A)$ with $|U|\leq |\hat{L}^A|/10^3$. By Claim~\ref{clm:Lmindeg} and a simple double-counting argument, there are at least $|{L}^A|/2$ vertices $x\in V({L})$ with at least $|U|/2$ neighbours in $U$ in $L^A$.
For each such $x$, $\P(x\in N_{\hat{L}^A}(U))\geq 1-(1-p)^{|U|}\geq p|U|/2= (2/|L^A|)\cdot |U|\log^2n/4$. Therefore, by Lemma~\ref{chernoff}, with probability $1-\exp(-\omega(|U|\log n))$ we have that $|N_{\hat{L}^A}(U)|\geq 20|U|$. Thus, \ref{prop:regularise:mindeg} holds with high probability by a union bound.

\smallskip

\noindent\ref{prop:regularise:joined}: Let $N=|{L}^A|$. Let $U,V\subset V({L}^A)$ be disjoint with $|U|,|V|\geq N/10^4$. Then, by Claim~\ref{clm:Lmindeg}, $e_{L^A}(U,V)\geq |U||V|-\max\{|U|,|V|\}\cdot \sqrt{p_T}N\geq N^2/10^{9}$.
Then, by Lemma~\ref{chernoff}, with probability $1-\exp(-\omega(N))$, we have that  $e_{\hat{L}^A}(U,V)\geq 20N$. Thus, by a union bound, \ref{prop:regularise:joined} holds with high probability.

\smallskip

\noindent\ref{prop:regularise:fornocod1}: Let $x=(i,u,r)\in \mathcal{E}^A$. For each $j\in I_{\tau}\setminus \{i\}$, there are at most $24n$ pairs $(v,s)$ with $(j,v,s)\in \mathcal{E}^A$. Thus, the probability that there are at least 5 pairs $(v,s)$ with
$(j,v,s)\in \mathcal{E}^A$ and $x(j,v,s)\in E(\hat{L}^A)$ is at most $(24n)^5p^5\leq n^{-4}$, and the probability there are at least 2 such pairs is at most $(24n)^2p^2\leq n^{-1.9}$.
Then, the probability that there are at least $3$ values of $j\in I_{\tau}\setminus \{i\}$ for which there are at least 3 pairs $(v,s)$ with $(j,v,s)\in \mathcal{E}$ and $x(j,v,s)\in E(\hat{L}^A)$ is at most $n^3\cdot (n^{-1.9})^3=n^{-2.7}$.
Furthermore, the probability that there is some $j\in I_{\tau}\setminus \{i\}$ for which there are at least 5 pairs $(v,s)$ with $(j,v,s)\in \mathcal{E}$ and $x(j,v,s)\in E(\hat{L}^A)$ is at most $n\cdot n^{-4}=n^{-3}$. Combining these, we have that with probability at least $1-2n^{-2.7}$ the number of tuples $(j,v,s,v',s')$ for which $(j,v,s),(j,v',s')\in N_{\hat{L}^A}(x)$ is at most $2\cdot \binom{4}{2}=12$. Thus, by a union bound over $x=(i,u,r)\in \mathcal{E}^A$, we have that \ref{prop:regularise:fornocod1} holds with high probability.

\smallskip

\noindent\ref{prop:regularise:fornocod2}: Let $x=(i,u,r)\in \mathcal{E}^A$.
By Lemma~\ref{chernoff}, with probability $1-\exp(-\omega(\log n))$ the set, $Y$ say, of $y\in \mathcal{E}^A$ with $xy\in E(L^A)$ satisfies $|Y|\leq 2\log^2n$. For each $y\in Y$, the probability there is some $(u',r')\neq (u,r)$ with $(i,u',r')y\in E(L^A)$ is at most $24n\cdot p\leq n^{-0.9}$, and the probability there are at least $3$ such $(u',r')$ is at most $(24n)^3\cdot p^3\leq n^{-2.7}$. Thus, with probability at least $1-(2\log^2n)\cdot n^{-2.7}-(2\log^2n)^3\cdot n^{-3\cdot 0.9}\leq n^{-2.6}$, there is no $y\in Y$ for which there are at least $3$ pairs $(u',r')\neq (u,r)$ with $(i,u',r')y\in E(L^A)$ and there are at most 2 choices for $y\in Y$ for which there is some pair $(u',r')\neq (u,r)$ with $(i,u',r')y\in E(L^A)$. Note that when this holds then \ref{prop:regularise:fornocod2} holds for $x$. Thus, \ref{prop:regularise:fornocod2} holds with high probability by a union bound.

\smallskip

\noindent \ref{prop:regularise:lowcod0}: Let $i\in I_\tau$ and $u\in S_i\setminus T_i$. There are at most $n^2$ pairs $(j,v)$ such that $\{(i,v),(j,u)\}\in E(\hat{L}_A)$, and, for each such $(j,v)$, the probability that $\{(i,v),(j,u)\}\in E(\hat{L}_A)$ is $p\leq n^{-1.9}$. Thus, by Lemma~\ref{chernoff}, with probability $1-\omega(-\log n)$, there are at most $n^{1/3}/4$ pairs $(j,v)$ such that $\{(i,v),(j,u)\}\in E(\hat{L}_A)$. Thus, \ref{prop:regularise:lowcod0} holds with high probability by a union bound.

\smallskip

\noindent \ref{prop:regularise:lowcod1},\ref{prop:regularise:lowcod3}: These hold with high probability virtually identically to \ref{prop:regularise:lowcod0}.

\smallskip

\noindent \ref{prop:regularise:lowcod2}. Let $j,j'\in I_\tau$ be distinct. For each $(i,u,r)\in \mathcal{E}^A$, with $i\notin\{j,j'\}$, the probability there are at least 6 pairs $(v,s)$ with $(j,v,s)\in \mathcal{E}$ and $x(j,v,s) \in E(\hat{L}^A)$ or $(j',v,s)\in \mathcal{E}$ and $x(j',v,s)\in E(\hat{L}^A)$ is at most $(24n)^6(2p)^6\leq n^{-5}$.
Furthermore, the probability there are at least 2 such pairs is at most $(24n)^2(2p)^2\leq n^{-1.9}$.
Then, for each $i\in I_\tau\setminus \{j,j'\}$, with probability $1-n^{-4}-(24n)^4\cdot (n^{-1.9})^4\geq 1-2n^{3.5}$ there are at most $3$ values of $(u,r)$ with $(i,u,r)\in \mathcal{E}^A$
for which there are at least 2 such pairs, and there is no value of $u$ with $(i,u)\in \mathcal{E}^A$ for which there are more than 5 such pairs. When this happens, there are at most $3\cdot \binom{5}{2}=30$ choices for $(u,r,j,j',v,v',s,s')$ for which
 $(j,v,s),(j',v',s')\in \mathcal{E}$ and $x(j,v,s),x(j',v',s') \in E(\hat{L}^A)$.
Thus, with high probability this holds for all distinct $j,j'\in I_\tau$ and $i\in I_\tau\setminus \{j,j'\}$ by a union bound. That is, for each $j,j'\in I_\tau$ and $i\in I_\tau\setminus \{j,j'\}$ there are at most 30 choices for $(u,v,v')$ for which we have that $\{(i,u),(j,v)\},\{(i,u),(j',v')\}\in E(\hat{L}_A)$.

Let again $j,j'\in I_\tau$ be distinct. For each $i\in I_\tau\setminus \{j,j'\}$, the probability that there is some $(i,u,v,v')$ for which we have that $\{(i,u),(j,v)\},\{(i,u),(j',v')\}\in E(\hat{L}_A)$ is at most $(24n)^3p^2\leq n^{-0.8}$. Therefore, by Lemma~\ref{chernoff}, with probability $1-\exp(-\omega(\log n))$ we have that the number of $i\in I_\tau\setminus \{j,j'\}$ for which there is $(u,v,v')$ for which we have that $\{(i,u),(j,v)\},\{(i,u),(j',v')\}\in E(\hat{L}_A)$ is at most $n^{1/4}$.
Thus, by a union bound, this holds with high probability for every distinct $j,j'\in I_\tau$

In combination, the two properties we have shown hold with high probability imply that, for every distinct $j,j'\in I_\tau$, the number of tuples $(i,u,v,v')$ for which we have that $\{(i,u),(j,v)\},\{(i,u),(j',v')\}\in E(\hat{L}_A)$  is at most $n^{1/4}\cdot 30\leq n^{1/3}/4$, as required.
Thus, \ref{prop:regularise:lowcod2} holds with high probability.
\claimproofend

Thus, we can assume that \ref{prop:regularise:mindeg}--\ref{prop:regularise:lowcod3} hold. Now, let $\tilde{L}^A$ be the graph $\hat{L}^A$ where we remove any edge $(i,u,r)(j,v,s)$ if there exists some $(v',s')$ with $(i,u,r)(j,v',s')\in E(\hat{L}^A)$.
Note that, by \ref{prop:regularise:fornocod1} and \ref{prop:regularise:fornocod2} this removes at most 16 edges around any one vertex. Therefore, \ref{prop:regularise:mindeg} and \ref{prop:regularise:joined} easily imply the following.

\begin{enumerate}[label = {{\textbf{\Alph{propcounter}\arabic{enumi}'}}}]
\item \label{prop:regularise:mindeg2} For each $U\subset V(\hat{L}^A)$ with $|U|\leq |\hat{L}^A|/10^3$, $|N_{\hat{L}^A}(U)|\geq 3|U|$.
\item For each disjoint $U,V\subset V(\hat{L}^A)$ with $|U|,|V|\geq |\hat{L}^A|/10^4$, $e_{\hat{L}^A}(U,V)>0$. \label{prop:regularise:joined2}
\end{enumerate}

We now show that
\ref{prop:regularise:mindeg2} and \ref{prop:regularise:joined2} imply $\tilde{L}^A$ contains a perfect matching via Tutte's theorem. Let $U\subset V(\tilde{L}^A)$. If $U\neq \emptyset$, then as $|\tilde{L}^A|$ is even and
\ref{prop:regularise:mindeg2} and \ref{prop:regularise:joined2} easily imply that $\tilde{L}^A$ is connected, $\tilde{L}^A-U$ has no components with an odd number of vertices, i.e., it has no odd components.
If $|U|>0$ and $\tilde{L}^A-U$ has at least $|U|+1$ components then $V(\tilde{L}^A-U)$ can be partitioned into $V_1$, $V_2$ with $|V_1|,|V_2|\geq |U|/2$ so that there are no edges between $V_1$ and $V_2$ in $\tilde{L}^A-U$. By \ref{prop:regularise:joined2}, we then have $|U|/2\leq |\tilde{L}^A|/10^4$. Taking an arbitrary set $V_1'\subset V_1$ with $|V_1'|=\lceil |U|/2\rceil$, we then have by \ref{prop:regularise:mindeg2} that
\[
|U|\geq |N_{\tilde{L}^A}(V_1')|\geq 3|V_1'|>|U|,
\]
a contradiction. Thus, for every $U\subset V(\tilde{L}^A)$, $\tilde{L}^A-U$ has at most $|U|$ odd components. Therefore, by Tutte's theorem, $\tilde{L}^A$ has a perfect matching, $M^A$, say.

Similarly, form the auxiliary graph $L^B$ with vertex set $\mathcal{E}^B$ and, for each $(i,u,r),(j,v,s)\in \mathcal{E}^B$, an edge between $(i,u,r)$ and $(j,v,s)$ if $u\notin T_j$, $v\notin T_i$, $u\neq v$, $i\neq j$,
and there is no $w$ such that $\{(i,u),(j,w)\}\in \mathcal{I}_\tau$ or $\{(i,w),(j,v)\}\in \mathcal{I}_\tau$ and no $(w,r,s)$ such that $(i,u,r)(j,w,s)\in M^A$. Similarly, form $\tilde{L}^B$ and $M^B$, where the extra condition on edges in $L^B$ is small enough that Claim~\ref{clm:Lmindeg} can easily be seen to still hold with $B$ in place of $A$.

Then, let
\begin{equation}\label{eqn:finalI}
\mathcal{I}'_\tau=\mathcal{I}_\tau\cup \left(\bigcup_{(i,u,r)(j,v,s)\in M^A\cup M^B}\{(i,u),(j,v)\}\right).
\end{equation}
By the construction of $M^A\cup M^B$, all the pairs added to $\mathcal{I}_\tau$ in \eqref{eqn:finalI} are distinct and not in $\mathcal{I}_\tau$.

By the choice of the $\lambda_{i,u}$, then, we have that \ref{prop:abs:regularityout} holds with $\mathcal{I}'_\tau$ in place of $\mathcal{I}_\tau$.
\ref{prop:abs:nocodegree} follows with $\mathcal{I}'_\tau$ in place of $\mathcal{I}_\tau$ from \ref{prop:A3:nocodegree}, the definition of $L^A$ and $L^B$, and the definition of $\tilde{L}^A$ and $\tilde{L}^B$.
\ref{prop:abs:boundedin}--\ref{prop:abs:lowcodegree3} hold with $\mathcal{I}'_\tau$ in place of $\mathcal{I}_\tau$ by combining \ref{prop:A3:boundedin}--\ref{prop:A3:lowcodegree3} and \ref{prop:regularise:lowcod0}--\ref{prop:regularise:lowcod3}, as well as the corresponding versions of \ref{prop:regularise:lowcod0}--\ref{prop:regularise:lowcod3} for $\tilde{L}^B$.
Furthermore, \ref{prop:abs:corrections} follows directly from \ref{prop:A3:corrections} with $\mathcal{I}'_\tau$ in place of $\mathcal{I}_\tau$.
Finally, we have that \eqref{eqn:Itaugoodpairs} holds with $\mathcal{I}'_\tau$ in place of $\mathcal{I}_\tau$. Therefore, $\mathcal{I}'_\tau$ satisfies the conditions in Lemma~\ref{keylemma:absorption} in place of $\mathcal{I}_\tau$, completing the proof of the lemma.
\end{proof}
\fi

%% file: 5links.tex
In this section, we give tight bounds on the number of paths with certain patterns of colours, which hold with very high probability in $G\sim G^{\col}_{[n]}$. To describe this, we will use the following notation.

\begin{defn}[Patterns and links]
Say $L=(H,f)$ is a \emph{pattern} if $H$ is a graph with a specified start vertex $u_L$ and a specified end vertex $v_L\neq u_L$  and $f$ is a function from $E(H)$ to $\N$.

Given distinct vertices $u,v$ in a coloured graph $G$, and a pattern $L=(H,f)$, a \emph{$(u,v,L)$-link} is a graph $H' \subseteq G$ for which there is an isomorphism $\psi:H\to H'$ such that, for each $i\in \image(f)$, $\psi(f^{-1}(i))$ are edges in $G$ of the same colour, where this colour is distinct over $i\in \image(f)$.
\end{defn}

The main result of this section, and the only one used elsewhere, is the following theorem which counts certain paths with a given structure. In particular, we are interested in paths of length 62 with fixed endvertices $u$ and $v$, which, starting from $u$, use 31 distinct colours and then repeat each of these colours in the same order exactly once before ending at $v$. Note that the number of possible ordered choices for the $31$ colours on such a path is $(1-O(n^{-1}))n^{31}$, and for each such choice, by symmetry, the probability that starting from $u$ and choosing the edges in the required colours and order produces a path that arrives at $v$ is essentially $1/n$. Thus, the expected number of links we wish to count is very close to the parameter $\Phi_0=n^{30}$ in Theorem~\ref{thm:Llinks}.
The first property in Theorem~\ref{thm:Llinks} that holds with high probability, \ref{prop:links:totalnumber}, is that the number of links we are counting is close to this expectation. The remaining properties count the number of links given additional constraints (fixing certain vertices, edges and colours, sometimes at particular points in the path) -- each case there is a natural heuristic that, for example, if the link is additionally required to have a vertex $x$ in the $k$th position then the expected number of such links is reduced by a factor of $n^{-1}$. In Section~\ref{sec:real}, we will consider an auxiliary hypergraph whose regularity depends on the number of links containing any particular vertex, edge, or colour in a particular position.
We require this auxiliary hypergraph to be approximately regular with sufficiently small codegrees. Thus the tight bounds of \ref{prop:links:throughvertex}--\ref{prop:links:throughedge} are required for this approximate regularity, whilst only upper bounds are needed for the properties covered by \ref{prop:links:cod:twovertices}--\ref{prop:links:throughedgeandvertex} as we need only show that these are not too large in order to bound the codegrees of the auxiliary hypergraph (and also to bound certain dependencies in Section~\ref{sec:balanceandcover}).

\begin{theorem}\label{thm:Llinks}
Let $1/n\llpoly \eps$ and $G\sim G^\col_{[n]}$. 
Let $L$ be the following pattern, a path of length 62 with 31 different colours in the first 31 edges which then repeat in the same order:

\begin{center}
\begin{tikzpicture}
\def\setsp{0.4}
\def\upup{0.4}
\def\extra{0.2}

\foreach \n in {1,3,5,7,9}
{
\coordinate (v\n) at ($(\n*\setsp,0)+(0,0)$);
}
\foreach \n in {11,13,15,17,19,21,23,25}
{
\coordinate (v\n) at ($(\extra,0)+(\n*\setsp,0)+(0,0)$);
}
\foreach \n in {27,29,31}
{
\coordinate (v\n) at ($(\extra,0)+(\extra,0)+(\n*\setsp,0)+(0,0)$);
}

\foreach \n in {2,4,6,8,10}
{
\coordinate (v\n) at ($(\n*\setsp,0)+(0,\upup)$);
}
\foreach \n in {12,14,16,18,20,22,24}
{
\coordinate (v\n) at ($(\extra,0)+(\n*\setsp,0)+(0,\upup)$);
}
\foreach \n in {26,28,30}
{
\coordinate (v\n) at ($(\extra,0)+(\extra,0)+(\n*\setsp,0)+(0,\upup)$);
}

\draw [red,thick] (v1) -- (v2);
\draw [darkgreen,thick] (v2) -- (v3);
\draw [blue,thick] (v3) -- (v4);
\draw [red,thick] (v16) -- (v17);
\draw [darkgreen,thick] (v17) -- (v18);
\draw [blue,thick] (v18) -- (v19);
\draw [brown,thick] (v4) -- (v5);
\draw [brown,thick] (v19) -- (v20);
\draw [red!50,thick] (v5) -- (v6);
\draw [red!50,thick] (v20) -- (v21);
\draw [violet,thick] (v6) -- (v7);
\draw [violet,thick] (v21) -- (v22);
\draw [teal,thick] (v7) -- (v8);
\draw [teal,thick] (v22) -- (v23);
\draw [blue!50,thick] (v8) -- (v9);
\draw [blue!50,thick] (v23) -- (v24);
\draw [magenta,thick] (v9) -- (v10);
\draw [magenta,thick] (v24) -- (v25);
\draw [thick,dotted] (v10) to ++(0.2,-0.2);
\draw [thick,dotted] (v25) to ++(0.2,0.2);
\draw [thick,dotted] (v11) to ++(-0.2,0.2);
\draw [thick,dotted] (v26) to ++(-0.2,-0.2);

\draw [purple,thick] (v11) -- (v12);
\draw [purple,thick] (v26) -- (v27);
\draw [green,thick] (v12) -- (v13);
\draw [green,thick] (v27) -- (v28);
\draw [orange,thick] (v13) -- (v14);
\draw [orange,thick] (v28) -- (v29);
\draw [olive,thick] (v14) -- (v15);
\draw [olive,thick] (v29) -- (v30);
\draw [cyan,thick] (v15) -- (v16);
\draw [cyan,thick] (v30) -- (v31);

\draw [red] ($0.5*(v1)+0.5*(v2)+(-0.05,0.15)$) node {\footnotesize $1$};
\draw [red] ($0.5*(v16)+0.5*(v17)+(-0.05,-0.15)$) node {\footnotesize $1$};
\draw [darkgreen] ($0.5*(v2)+0.5*(v3)+(-0.05,-0.15)$) node {\footnotesize $2$};
\draw [darkgreen] ($0.5*(v17)+0.5*(v18)+(-0.05,0.15)$) node {\footnotesize $2$};
\draw [blue] ($0.5*(v3)+0.5*(v4)+(-0.05,0.15)$) node {\footnotesize $3$};
\draw [blue] ($0.5*(v18)+0.5*(v19)+(-0.05,-0.15)$) node {\footnotesize $3$};
\draw [brown] ($0.5*(v4)+0.5*(v5)+(-0.05,-0.15)$) node {\footnotesize $4$};
\draw [brown] ($0.5*(v19)+0.5*(v20)+(-0.05,0.15)$) node {\footnotesize $4$};
\draw [red!50] ($0.5*(v5)+0.5*(v6)+(-0.05,0.15)$) node {\footnotesize $5$};
\draw [red!50] ($0.5*(v20)+0.5*(v21)+(-0.05,-0.15)$) node {\footnotesize $5$};
\draw [violet] ($0.5*(v6)+0.5*(v7)+(-0.05,-0.15)$) node {\footnotesize $6$};
\draw [violet] ($0.5*(v21)+0.5*(v22)+(-0.05,0.15)$) node {\footnotesize $6$};
\draw [teal] ($0.5*(v7)+0.5*(v8)+(-0.05,0.15)$) node {\footnotesize $7$};
\draw [teal] ($0.5*(v22)+0.5*(v23)+(-0.05,-0.15)$) node {\footnotesize $7$};
\draw [blue!50] ($0.5*(v8)+0.5*(v9)+(-0.05,-0.15)$) node {\footnotesize $8$};
\draw [blue!50] ($0.5*(v23)+0.5*(v24)+(-0.05,0.15)$) node {\footnotesize $8$};
\draw [magenta] ($0.5*(v9)+0.5*(v10)+(-0.05,0.15)$) node {\footnotesize $9$};
\draw [magenta] ($0.5*(v24)+0.5*(v25)+(-0.05,-0.15)$) node {\footnotesize $9$};

\draw [purple] ($0.5*(v11)+0.5*(v12)+(-0.05-0.075,0.15)$) node {\footnotesize $27$};
\draw [purple] ($0.5*(v26)+0.5*(v27)+(-0.05-0.075,-0.15)$) node {\footnotesize $27$};
\draw [green] ($0.5*(v12)+0.5*(v13)+(0.05-0.075-0.075,-0.15)$) node {\footnotesize $28$};
\draw [green] ($0.5*(v27)+0.5*(v28)+(-0.05-0.075,0.15)$) node {\footnotesize $28$};
\draw [orange] ($0.5*(v13)+0.5*(v14)+(-0.05-0.075,0.15)$) node {\footnotesize $29$};
\draw [orange] ($0.5*(v28)+0.5*(v29)+(-0.05-0.075,-0.15)$) node {\footnotesize $29$};
\draw [olive] ($0.5*(v14)+0.5*(v15)+(-0.05-0.075,-0.15)$) node {\footnotesize $30$};
\draw [olive] ($0.5*(v29)+0.5*(v30)+(0.05-0.075-0.075,0.15)$) node {\footnotesize $30$};
\draw [cyan] ($0.5*(v15)+0.5*(v16)+(-0.05-0.075,0.15)$) node {\footnotesize $31$};
\draw [cyan] ($0.5*(v30)+0.5*(v31)+(-0.05-0.075,-0.15)$) node {\footnotesize $31$};


\foreach \n in {1,...,31}
{
\draw [fill] (v\n) circle[radius=0.04];
}

\draw ($(v1)-(0,0.3)$) node {$u_L$};
\draw ($(v31)-(0,0.3)$) node {$v_L$};
\end{tikzpicture}
\end{center}

\vspace{-0.2cm}

Then, with probability $1-n^{-\omega(1)}$, the following hold with $\Phi_0=n^{30}$.
\stepcounter{propcounter}
\begin{enumerate}[label = {\emph{\textbf{\Alph{propcounter}\arabic{enumi}}}}]
\item \labelinthm{prop:links:totalnumber} For each distinct $u,v\in V(G)$ with $u\sim_{A/B} v$, the number of $(u,v,L)$-links in $G$ is $(1\pm \eps)\Phi_0$.
\item \labelinthm{prop:links:throughvertex} For each $k$ with $2\leq k\leq 62$ and each distinct $u,v\in V(G)$ with $u\sim_{A/B} v$, and each $x\in V(G)\setminus \{u,v\}$ with $x\not\sim_{A/B}u,v$ if $k$ is even and $x\sim_{A/B}u,v$ if $k$ is odd,
the number of $(u,v,L)$-links in $G$ in which $x$ is the $k$th vertex is $(1\pm \eps)\cdot \Phi_0\cdot  n^{-1}$.
\item \labelinthm{prop:links:throughcolour} For each $k\in [62]$, and each distinct $u,v\in V(G)$ with $u\sim_{A/B} v$ and each $c\in C$, the number of $(u,v,L)$-links in $G$ in which the $k$th edge has colour $c$ is $(1\pm \eps)\cdot \Phi_0\cdot  n^{-1}$.
\item \labelinthm{prop:links:throughedge}
For each $k$ with $2\leq k\leq 61$ and each distinct $u,v\in V(G)$ with $u\sim_{A/B} v$, and each $xy\in E(G)$ with $\{x,y\}\cap \{u,v\}=\emptyset$, the number of $(u,v,L)$-links in $G$ which have $xy$ as the $k$th edge is
$(1\pm \eps)\cdot \Phi_0 \cdot n^{-2}$.
\item \labelinthm{prop:links:cod:twovertices} For each distinct $u,v,x,y\in V(G)$, the number of $(u,v,L)$-links in $G$ containing $x$ and $y$ is at most $10^4\cdot \Phi_0\cdot n^{-2}$.
\item \labelinthm{prop:links:cod:1vertex1colour} For each distinct $u,v,x\in V(G)$ and each $c\in C$, the number of $(u,v,L)$-links in $G$ using $x$ and $c$ in which there is not a colour-$c$ edge $ux$ or $xv$ is at most $10^4\cdot \Phi_0\cdot n^{-2}$.
\item \labelinthm{prop:links:cod:twocolours} For each distinct $u,v\in V(G)$ and each distinct $c,d\in C$, the number of $(u,v,L)$-links in $G$ using $c$ and $d$ is at most $10^4\cdot \Phi_0\cdot n^{-2}$.

\item \labelinthm{prop:links:2edgesofdifferentcoloursanddisjoint} For each distinct $u,v\in V(G)$ and each distinct $e,e'\in E(G-\{u,v\})$ with different colours and which share no vertices, the number of $(u,v,L)$-links in $G$ containing $e$ and $e'$ is at most $10^4\cdot \Phi_0\cdot n^{-4}$.
\item \labelinthm{prop:links:throughedgeandvertex} For each distinct $u,v,w\in V(G)$ and each $e\in E(G-\{u,v,w\})$, the number of $(u,v,L)$-links in $G$ containing $w$ and $e$ is at most $10^8\cdot \Phi_0\cdot n^{-3}$.
\end{enumerate}
\end{theorem}

\ifsecfiveout \label{sec:discussiondeletionmethod}
\else
We discuss our proof of Theorem~\ref{thm:Llinks} in Section~\ref{sec:discussiondeletionmethod}, before outlining the rest of this section.


\subsection{Discussion of methods and section outline}\label{sec:discussiondeletionmethod}

As we have noted in Sections \ref{sec:intro} and \ref{sec:proofsketch}, to prove Theorem~\ref{thm:Llinks} we will use the deletion method of R\"odl and Ruci\'nski~\cite{rodl1995threshold}, developing its use in Latin squares by Kwan, Sah and Sawhney~\cite{kwan2022large}, where they use it to prove likely upper bounds on the counts of different substructures. We will first describe how we use it in this way, before explaining how and why we use it for our key lower bound.

Suppose we have a collection $\mathcal{F}$ of small properly-coloured graphs that might appear in $G\sim G^{\col}_{[n]}$, and we wish to give an upper bound on the number of graphs in $\mathcal{F}$ that appear as subgraphs of $G$. As each $F\subset \mathcal{F}$ is small and properly coloured, and the probability the corresponding edge in $G$ has the same colour as that edge in $F$ is, by symmetry, $1/n$, we expect that $\P(F\subset G)\approx n^{-e(F)}$. However, the challenges of working in the uniformly random Latin square model mean that this probability is severely dominated by the error terms we need to use for the known bounds on $\P(F\subset G)$ if $e(F)$ is small (i.e., using Corollary~\ref{cor:latinsquareprobabilities} is far off even the trivial bound $\P(F\subset G)\leq 1$).
However, if $H$ is a properly coloured subgraph which is the union of many edge-disjoint
subgraphs in $\mathcal{F}$, then we can give an effective bound on $\P(H\subset G)$ (via Corollary~\ref{cor:latinsquareprobabilities}). Moreover, if $G$ contained plenty of graphs in $\mathcal{F}$ that did not overly overlap, then it would contain some such subgraph $H$.

The deletion method uses this reasoning to give an upper bound on the number of graphs in $\mathcal{F}$ that appear in $G$ that holds with high probability. The form of this argument (using similar notation to later in this section) will go as follows, for a set $D\subset [n]$ of $pn$ colours and for $G\sim G^\col_{[n]}$.
\begin{itemize}
\item Suppose $\mathcal{F}$ is a collection of $N$ properly-coloured subgraphs which might appear in $G|_D$, each with $r$ edges, and let $\kappa\in \N$.
\item Consider the collection $\mathcal{S}$ of sequences $S=(F_1,\ldots,F_\kappa)$ of subgraphs drawn from $\mathcal{F}$ which are edge-disjoint and whose union $H_S:=\cup_{i\in [\kappa]}F_i$ is properly coloured.
\item Then, as $|\mathcal{S}|\leq N^\kappa$, the expected number of such sequences $S\in \mathcal{S}$ with $H_S\subset G$ will be, by Corollary~\ref{cor:latinsquareprobabilities}, at most $e^{O(pr\kappa+n\log^2n)}(Nn^{-r})^{\kappa}$.
\item Thus, by Markov's inequality, with probability $1-n^{-\omega(1)}$, the number of $S\in \mathcal{S}$ with $H_S\subset G$ will be at most $n^{\omega(1)}\cdot e^{O(pr\kappa+n\log^2n)}(Nn^{-r})^{\kappa}\leq ((1+\eps/2)Nn^{-r})^\kappa$, where the inequality will hold if, for example, $1/n \llpoly p\llpoly \eps$ and $\kappa=n^{1.01}$.
\item If $G$ contains more than $(1+\eps)N^rn^{-r}$ of the graphs in $\mathcal{F}$ which are well enough distributed that selecting these graphs greedily shows there are at least $((1+\eps/2) N^rn^{-r})^\kappa$ sequences $S\in \mathcal{S}$ with $H_S\subset G$, then from the previous step it must be the case that, with probability $1-n^{-\omega(1)}$, $G$ contains at most $(1+\eps)N^rn^{-r}$ of the graphs in $\mathcal{F}$.
\end{itemize}

The method is applied in Sections~\ref{sec:loose3pathbound}, \ref{sec:tight7pathupperbound} and \ref{sec:path15upperlower}. These applications build in complexity, and rely on the previous applications, and so we repeat each application largely in full rather than attempting to amalgamate them into a general result. Before discussing our application of the deletion method for a lower bound, we make the following further remarks.

\begin{itemize}
\item In order to give a likely upper-bound on the number of graphs in $\mathcal{F}$ appearing in $G$, we need to show that, with high probability, there are not many heavily overlapping graphs in $\mathcal{F}$ appearing in $G$. For some of our applications of the deletion method, this will only hold with very high probability in $G\sim G_{[n]}^\col$, something we show using a simpler application of the deletion method. Where we do this, $\mathcal{B}$ will be the event that there are not many heavily overlapping graphs in $\mathcal{F}$ appearing in $G$ (see, for example, Claim~\ref{clm:AH1} and the argument just after this).
\item The sketch above uses $p\llpoly \eps$, while, for example, for Theorem~\ref{thm:Llinks}, we want to count the number of $(u,v,L)$-links in $G\sim G^\col_{[n]}$ using any colour. To make the deletion method approaches work, then, we have to count the number of substructures we are interested in that use colours in some set ($D$, $D_1$, or $D_2$) that is not too large, and sometimes using only edges in some random set of edges.
Thus, we often prove results with extra restrictions like these (e.g.\ Lemma~\ref{lem:looseupperlength3}) before using simple probabilistic arguments with a Chernoff bound or McDiarmid's inequality to give a result without these restrictions (e.g.\ Corollary~\ref{cor:fewlengththreesame}).
\item In order to apply Corollary~\ref{cor:latinsquareprobabilities} to bound $\P(H_S\subset G)$, we need to have that each colour appears at most $pn$ times in $H_S$, so this is an added condition we will take on sequences $S\in \mathcal{S}$.
\item This sketch so far follows the use of the deletion method in~\cite{kwan2022large}, though we additionally use the deletion method to control the likely spread of graphs in $\mathcal{F}$. A larger difference is that the subgraphs we are using will have some vertices contained in all of them -- for example, for some $u,v$, $\mathcal{F}$ could be a collection of coloured $u,v$-paths. Then, if we can take $\kappa$ edge-disjoint graphs from $\mathcal{F}$ which can appear in $G$, we must have $\kappa\leq n$, so that the error term $e^{O(pr\kappa+n\log^2n)}$ from Corollary~\ref{cor:latinsquareprobabilities} is too large to give the tight bounds we need as $n\log^2n$ will be larger than $\kappa$, preventing the above sketch from working.

However, it is not difficult to see how to fix this. In the example above for $u,v$-paths, we would let $H$ be the graph of edges of $G$ with colour in $D$ which contains $u$ or $v$. Then, considering instead the collection $\mathcal{F}'=\{G_F:=F-u-v:F\in \mathcal{F}\}$, we can consider $\kappa=n^{1.01}$ edge disjoint graphs $G_F\in \mathcal{F}'$ for which $F\subset G$. We can then work by conditioning on the different possible outcomes of $H$, where, as we will see (for example, in Claim~\ref{clm:AH}, and its proof) that this will introduce an additional error term of $e^{e(H)}$ into the application of Corollary~\ref{cor:latinsquareprobabilities}. This is small compared to the  error term $e^{O(n\log^2n)}$ we already have, as $e(H)\leq 2n$.
\end{itemize}

We will now describe how we use the deletion method to give a likely lower bound, as we do in the proof of Lemma~\ref{lemma:mainLlinksfirstlower} in Section~\ref{sec:path15upperlower}. In this, we have $G\sim {G}^\col_{[n]}$  and distinct $x_1,x_2,y_1,y_2\in V(G)$ with $x_1\nsimAB y_1$ and $x_2\nsimAB y_2$. We wish to give a lower bound which holds with high probability on the number of pairs $(P_1,P_2)$ of vertex-disjoint paths in $G|_D$ such that, for each $i\in [2]$, $P_i$ is an $x_i,y_i$-path of length 15 and these paths have the same colours in the same order. To make this easier, we split $D$ into $D_1\cup D_2$ and count the pairs of such paths whose edges are all in $G|_{D_1}$ except the middle edges of each path, which has colour in $D_2$.
In Section~\ref{sec:path15prep}, we show that $G|_{D_1}$ will very likely be such that for all but $o(n^4)$ distinct (uncoloured) edges $e,f\notin E(G|_{D_1})$, if $e$ and $f$ have the same colour in $G$ then there will at most only a little fewer than can be expected paths $(P_1,P_2)$ of the sort we are counting which have $e$ and $f$ as their respective middle edges. In other words, we will have $\mathcal{E}$ which is a relatively small set of pairs of uncoloured edges $e,f\notin E(G|_{D_1})$ which would not fewer than average of the structures we want if they appear in $G$ together with the same colour. Bounding above the number of these which appear in $G$ with the same colour (using the deletion method) we then have a lower bound for the structures in $G$ we wish to count.

We leave further details to the proof of Lemma~\ref{lemma:mainLlinksfirstlower}, but as this implementation is relatively complex, we think it worth noting here why using switching methods instead would be even more complicated. In this setting, using the deletion method and using the switching method for a likely lower bound on the number of certain small coloured subgraphs here is actually closely related. In particular they would both use the counting of small coloured subgraphs which can be found robustly (for example, in the work of Gould, Kelly, K\"uhn, and Osthus~\cite{gould2022almost}, the subgraphs called `spin systems' and `twist systems' in its Definition~7.4). With apologies to any readers not familiar with the switching method, we will not overburden this explanation with the details of a method we are not using, but this comparison roughly goes as follows. For the switching method, the aim would be to argue that (where $\Phi$ is the expected number of subgraphs we are counting), when $X\leq (1-\eps)\Phi$ the number of $G\in \mathcal{G}_{[n]}^\col$ with $X$ such subgraphs  is outnumbered by a $(1+\Omega_\eps(1))$ factor by those $G\in \mathcal{G}_{[n]}^\col$ with $X+1$ such subgraphs. This would then be iterated to show that a proportion of at most $(1-\Omega_\eps(1))^{\eps \Phi}$ of the graphs in $G\in \mathcal{G}_{[n]}^\col$ can have $X$ such subgraphs if $X\leq (1-2\eps)\Phi$, which translates into the very small probabilities we would need.
The point here is that this would involve counting small coloured subgraphs not only in a typical $G\in \mathcal{G}_{[n]}^\col$ but doing so in a way that remains possible as $G$ is iteratively altered by switching operations.
For our new approach to such lower bounds via the deletion method (using the notation above) we consider  a sequence of $\kappa$ small coloured subgraphs as well, but these are drawn from the same coloured graph, and this makes for a less technical approach as our parent graph is not changing.

\medskip

\noindent \textbf{Section outline.}
In Section~\ref{sec:loose3pathbound}, we give a loose bound on pairs of paths of length 3 with fixed endvertices and the same colours which holds with very high probability in $G\sim G^\col_{[n]}$. The first such result is Lemma~\ref{lem:looseupperlength3} (proved with our first application of the deletion method), which bounds the number of such paths with colours in some set $D\subset [n]$, before we use this to deduce for Corollary~\ref{cor:fewlengththreesame} a bound which holds with high probability for paths of any colours.
Then, in Section~\ref{sec:tight7pathupperbound}, we use this to ensure certain longer paths are likely to be well enough distributed that we can bound above their number using the deletion method. This will give us a strong bound on the pairs of paths of length 7 with the same colours and fixed endvertices.
In Section~\ref{sec:path15prep}, we prepare for the lower bound we show in Section~\ref{sec:path15upperlower}, using relatively straight-forward combinatorial and probabilistic arguments to give good a good upper bound on the size of the collection $\mathcal{E}$ described above. This allows us to use the deletion method to give our main likely lower bound in Section~\ref{sec:path15upperlower}, before we put all of the work in this section together to prove Theorem~\ref{thm:Llinks} in Section~\ref{sec:proofmainlinkthm}.


\subsection{Loose upper bounds for length 3 paths with the same colours}\label{sec:loose3pathbound}

As described in Section~\ref{sec:discussiondeletionmethod}, we will first prove a loose bound on the number of paths of length 3 with fixed endvertices and the same colours, using only colours in some set $D\subset [n]$, which holds with very high probability in $G\sim G^\col_{[n]}$, as follows.

\begin{lemma}\label{lem:looseupperlength3} Let $\eta=0.01$ and $1/n\llpoly p\ll\eta$. Let $D\subset [n]$ have size $pn$. Let $G\sim G_{[n]}^\col$. Let $x_1,x_2\in A$ be distinct and let $y_1,y_2\in B$ be distinct. Then, with probability $1-n^{-\omega(1)}$, there are at most $n^{1+\eta}$ pairs $(P_1,P_2)$ of vertex-disjoint paths in $G|_D$ of length three such that $P_1$ is an $x_1,y_1$-path, $P_2$ is an $x_2,y_2$-path and they have the same colours in the same order.
\end{lemma}

\begin{center}
\begin{tikzpicture}
\def\setsp{0.8}

\foreach \n/\ifBup in {1/0,2/1,3/0,4/1}
{
\coordinate (v\n) at ($(\n*\setsp,0)+\ifBup*(0,0.5)$);
}
\foreach \n/\ifBup in {5/0,6/1,7/0,8/1}
{
\coordinate (v\n) at ($(\n*\setsp,0)+\ifBup*(0,0.5)$);
}

\draw [red,thick] (v1) -- (v2);
\draw [darkgreen,thick] (v2) -- (v3);
\draw [blue,thick] (v3) -- (v4);
\draw [red,thick] (v5) -- (v6);
\draw [darkgreen,thick] (v6) -- (v7);
\draw [blue,thick] (v7) -- (v8);

\draw [red] ($0.5*(v1)+0.5*(v2)+(0,0.15)$) node {\footnotesize $1$};
\draw [red] ($0.5*(v5)+0.5*(v6)+(0,0.15)$) node {\footnotesize $1$};
\draw [darkgreen] ($0.5*(v2)+0.5*(v3)+(0,0.15)$) node {\footnotesize $2$};
\draw [darkgreen] ($0.5*(v6)+0.5*(v7)+(0,0.15)$) node {\footnotesize $2$};
\draw [blue] ($0.5*(v3)+0.5*(v4)+(0,0.15)$) node {\footnotesize $3$};
\draw [blue] ($0.5*(v7)+0.5*(v8)+(0,0.15)$) node {\footnotesize $3$};

\foreach \n in {1,...,8}
{
\draw [fill] (v\n) circle[radius=0.05];
}

\draw ($(v1)+(0,-0.3)$) node {$x_1$};
\draw ($(v5)+(0,-0.3)$) node {$x_2$};
\draw ($(v4)+(0,0.3)$) node {$y_1$};
\draw ($(v8)+(0,0.3)$) node {$y_2$};
\end{tikzpicture}
\end{center}
\begin{proof} Note that we can assume that $x_1\not\simAB y_1$ and $x_2\not\simAB y_2$, for otherwise there are trivially no such paths. Let $H$ be the graph of edges in $G$ next to $\{x_1,x_2,y_1,y_2\}$ with colour in $D$, and let $\mathcal{H}$ be the set of possibilities for $H$.
Let $\mathcal{A}$ be the property that there are more than $n^{1+\eta}$ pairs $(P_1,P_2)$ of vertex-disjoint paths in $G|_D$ of length three such that $P_1$ is an $x_1,y_1$-path, $P_2$ is an $x_2,y_2$-path and they have the same colours in the same order.
We will show the following claim.

\begin{claim}\label{clm:AH1} For each $\hat{H}\in \mathcal{H}$, $\P(\mathcal{A}|H=\hat{H})=n^{-\omega(1)}$.
\end{claim}
Given this claim, we will have
\begin{align*}
\P(\mathcal{A})&=\sum_{\hat{H}\in \cH}\P(\mathcal{A}\land (H=\hat{H}))=\sum_{\hat{H}\in \cH}\P(\mathcal{A}|H=\hat{H})\cdot \P(H=\hat{H})
\leq \max_{\hat{H}\in \cH}\P(\mathcal{A}|H=\hat{H})
= n^{-\omega(1)}.
\end{align*}

Therefore, it is sufficient to prove Claim~\ref{clm:AH1}.

\smallskip

\noindent \emph{Proof of Claim~\ref{clm:AH1}.} Let $\hat{H}\in \cH$. Let $\kappa=n^{1+\eta/2}$.
Let $\mathcal{F}_{\hat{H}}$ be the set of properly coloured graphs $F$ with vertices in $(A\cup B)\setminus \{x_1,x_2,y_1,y_2\}$ which each comprise two coloured edges, $e$ and $f$ say, with colour in $D$, such that $\hat{H}+e+f$ is a properly coloured graph which contains a pair $(P_1,P_2)$ of vertex-disjoint paths in $G|_D$ of length three such that $P_1$ is an $x_1,y_1$-path, $P_2$ is an $x_2,y_2$-path and they have the same colours in the same order.
Such a subgraph $F$ is determined by the middle edge in $P_1$ (including the edge's colour), and therefore $|\mathcal{F}_{\hat{H}}|\leq n^3$.

Let $\mathcal{S}_{\hat{H}}$ be the set of sequences $(F_1,\ldots,F_\kappa)$ of length $\kappa$ of edge-disjoint subgraphs from $\mathcal{F}_{\hat{H}}$ for which each colour appears on $\bigcup_{i\in [\kappa]}F_i$ at most $pn/2$ times, and note that, then,
\begin{equation}\label{eqn:sizeS1new1}
|\mathcal{S}_{\hat{H}}|\leq |\mathcal{F}_{\hat{H}}|^\kappa\leq n^{3\kappa}.
\end{equation}
For each $S=(F_1,\ldots,F_{\kappa})\in \mathcal{S}_{\hat{H}}$, let $H_S=\cup_{i\in[\kappa]}F_i$, so that $e(H_S)=2\kappa$, each colour appears on $H_S$ at most $pn/2$ times and every edge of $H_S$ has colour in $D$.
Let $Z_{\hat{H}}$ be the number of
$S\in \mathcal{S}_{\hat{H}}$ with $H_S\subset G$.

Now, we have that each subgraph $F\in \mathcal{F}_{\hat{H}}$ is determined by either of its edges along with which of the paths $P_1$ and $P_2$ it is in. Thus, any edge can appear in at most two graphs in $\mathcal{F}_{\hat{H}}$ appearing as subgraphs of $G$, and (as any colour in $D$ appears on $n$ edges of $G$) any colour can appear on at most $2n$ graphs in $\mathcal{F}_{\hat{H}}$ appearing as subgraphs of $G$.
Assuming $\mathcal{A}$ holds, then we have $Z_{\hat{H}}\geq (n^{1+\eta/2})^{\kappa}$, as follows. Indeed, if $\mathcal{A}$ holds, then we can pick a sequence $(F_1,\ldots,F_\kappa)$ of edge-disjoint subgraphs from $\mathcal{F}_{\hat{H}}$ by picking each $F_i$, $1\leq i\leq\kappa$, in turn, where at the selection of each $F_i$, $i\in [\kappa]$, there will be $(i-1)\cdot 2$ edges we wish to avoid and at most $2(i-1)\cdot 2/pn$ colours, so the number of possibilities for $F_i$ is at least
\[
n^{1+\eta}-2\cdot 2\kappa -2n\cdot \frac{2\kappa}{pn/2}\geq n^{1+\eta/2},
\]
as $1/n\llpoly p$ and $\eta=0.01$, and therefore $Z_{\hat{H}}\geq (n^{1+\eta/2})^{\kappa}$.

For each $S\in \mathcal{S}_{\hat{H}}$, note that $H_S\cup \hat{H}$ has at most $pn/2+4\leq pn$ edges of each colour, and every edge of $H_S\cup \hat{H}$ has colour in $D$. Therefore, for each $S\in \mathcal{S}_{\hat{H}}$ for which $H_S\cup \hat{H}$ is properly coloured, using Corollary~\ref{cor:latinsquareprobabilities} (applied to both $H_S$ and $H_S\cup \hat{H}$) and that $e(H_S)=2\kappa\geq n\log^2n\geq 4n\geq e(\hat{H})$, we have
\begin{align}
\P(H_S\subset G|H=\hat{H})&=\frac{\P((H_S\cup \hat{H})\subset G)}{\P(\hat{H}\subset G)}
=e^{O(p\cdot e(H_S)+n\log^2n)}\cdot \frac{n^{-e(H_S\cup \hat{H})}}{n^{-e(\hat{H})}}=e^{O(p\kappa)}n^{-2\kappa},\label{eqn:longversion}
\end{align}
where we have used that $p\kappa=pn^{1+\eta/2}=\Omega(n\log^2n)$ as $1/n\llpoly p$.
Therefore, as this holds for each $S\in \mathcal{S}_{\hat{H}}$ such that $H_S\cup \hat{H}$ is properly coloured,
\[
\E(Z_{\hat{H}}|H=\hat{H})\leq |\mathcal{S}_{\hat{H}}|\cdot e^{O(p\kappa)}n^{-2\kappa}\overset{\eqref{eqn:sizeS1new1}}{\leq}
n^{3\kappa}\cdot e^{O(p\kappa)}n^{-2\kappa}\leq (n^{1+O(p)})^{\kappa}=(n^{(1+\eta/2)})^\kappa\cdot n^{-\omega(1)},
\]
so that, by Markov's inequality, we have
\[
\P(\mathcal{A}|H=\hat{H}) \leq \P(Z_{\hat{H}}\geq (n^{1+\eta/2})^{\kappa})\leq \frac{\E(Z_{\hat{H}}|H=\hat{H})}{(n^{1+\eta/2})^{\kappa}}
=n^{-\omega(1)}.
\]
This completes the proof of the claim, and hence the lemma.\hspace{5.2cm}$\boxdot$
\end{proof}


Using Lemma~\ref{lem:looseupperlength3}, we now deduce a similar bound which is very likely to hold in $G\sim G_{[n]}^\col$ but where the paths can use any colours in $G$.

\begin{corollary}\label{cor:fewlengththreesame}
Let $\eta=0.02$ and $n\in \N$. Let $G\sim G_{[n]}^\col$. Let $x_1,x_2\in A$ be distinct and let $y_1,y_2\in B$ be distinct. Then, with probability $1-n^{-\omega(1)}$, there are at most $n^{1+\eta}$ pairs $(P_1,P_2)$ of vertex-disjoint paths in $G$ of length three such that $P_1$ is an $x_1,y_1$-path, $P_2$ is an $x_2,y_2$-path and they have the same colours in the same order.
\end{corollary}
\begin{proof} Let $p\ll \eta$, let $x_1,x_2\in A$ be distinct, let $y_1,y_2\in B$ be distinct and let $G\sim G_{[n]}^\col$. For each $i\in [n]$, let $\hat{D}_i$ be a uniformly random subset of $C$ with size $pn$. By Chernoff's bound, with high probability we have that, for each set $\hat{C}\subset C$ of $3$ colours, there are at least $p^3n/2$ values of $i\in [n]$ for which $\hat{C}\subset \hat{D}_i$. Furthermore, by Lemma~\ref{lem:looseupperlength3}, with probability $1-n^{-\omega(1)}$, for each $i\in [n]$, we have that there are at most $n^{1.01}$ pairs $(P_1,P_2)$ of vertex-disjoint paths in $G|_{\hat{D}_i}$ of length three such that $P_1$ is an $x_1,y_1$-path, $P_2$ is an $x_2,y_2$-path and they have the same colours in the same order.
Therefore, the number of pairs of vertex-disjoint paths $(P_1,P_2)$ in $G$ of length three such that $P_1$ is an $x_1,y_1$-path, $P_2$ is an $x_2,y_2$-path and they have the same colours in the same order is at most
\[
\frac{n\cdot n^{1.01}}{p^3n/2}\leq n^{1+\eta},
\]
as required, where the last inequality holds for sufficiently large $n$, which holds with probability $1-n^{-\omega(1)}$.
\end{proof}


\subsection{Tight upper bounds for length 7 paths with the same colours}\label{sec:tight7pathupperbound}
We now prove a tighter upper bound for pairs of paths of length 7 instead of pairs of paths of length 3.
We will use Corollary~\ref{cor:fewlengththreesame} to show that the subgraphs we now seek in $G\sim G^\col_{[n]}$ are likely to be well distributed in $G$, through the following result.

\begin{prop}\label{prop:linksusingearefew} Let $G\sim G_{[n]}^\col$. Let $x_1,x_2\in A$ and $y_1,y_2\in B$ all be distinct. Then, with probability $1-n^{-\omega(1)}$, for each $e\in E(G-\{x_1,x_2,y_1,y_2\})$, there are at most $n^{3.03}$ pairs  $(P_1,P_2)$ of vertex-disjoint paths in $G$ of length seven such that $P_1$ is an $x_1,y_1$-path, $P_2$ is an $x_2,y_2$-path, they have the same colours in the same order, and $e\in E(P_1\cup P_2)$.
\end{prop}
\begin{proof} Let $\eta=0.02$. From Corollary~\ref{cor:fewlengththreesame}, with probability $1-n^{-\omega(1)}$, for any distinct $x_1',x_2',y_1',y_2'\in V(G)$ there are at most $n^{1+\eta}$ pairs $(P_1',P_2')$ of vertex-disjoint paths in $G$ of length three such that $P_1'$ is an $x'_1,y'_1$-path, $P_2$ is an $x'_2,y'_2$-path and they have the same colours in the same order.
Assuming this property, we will now show the required property holds.

For this, let $x_1,x_2\in A$ be distinct and let $y_1,y_2\in B$ be distinct and let $e\in E(G-\{x_1,x_2,y_1,y_2\})$.
Let $(P_1,P_2)$ be a pair of vertex-disjoint paths in $G$ of length seven such that $P_1$ is an $x_1,y_1$-path, $P_2$ is an $x_2,y_2$-path, they have the same colours in the same order, and $e\in E(P_1\cup P_2)$.
Suppose  that $e$ is one of the first 4 edges of $P_1$, say the $i$th edge and note that, as $x_1\notin V(e)$, $2\leq i\leq 4$. Note further that the first 4 edges of $P_1$ are determined by choosing which of them is $e$ and additionally choosing the colour of the other edges among the 2nd, 3rd and 4th edge. As the colours on $P_1$ and $P_2$ are the same and in the same order, this then determines the first 4 edges of $P_2$. By the property from Corollary~\ref{cor:fewlengththreesame}, the pair $(P_1,P_2)$ is then determined up to at most $n^{1+\eta}$ possibilities. Thus, there are overall at most $3n^{3+\eta}$ possibilities for $(P_1,P_2)$ for which $e$ is one of the first 4 edges of $P_1$.

By the same argument, there are at most $3n^{3+\eta}$ possibilities for $(P_1,P_2)$ when $e$ is one of the last 4 edges of $P_1$, or one of the first 4 edges of $P_2$, or one of the last 4 edges of $P_2$. Therefore, in total, there are at most $12n^{3+\eta}\leq n^{3.03}$  pairs  $(P_1,P_2)$ of vertex-disjoint paths in $G$ of length seven such that $P_1$ is an $x_1,y_1$-path, $P_2$ is an $x_2,y_2$-path, they have the same colours in the same order, and $e\in E(P_1\cup P_2)$.
\end{proof}


Using Proposition~\ref{prop:linksusingearefew}, we now prove our tight upper bound likely to hold for pairs of length 7 paths with the same colour pattern and colours within some subset $D\subset [n]$, as follows.

\begin{lemma}\label{lem:tightupplength7} Let $1/n\llpoly p \llpoly \eps$. Let $D\subset [n]$ have size $pn$. Let $G\sim G_{[n]}^\col$. Let $x_1,x_2\in A$ be distinct and let $y_1,y_2\in B$ be distinct.
Then, with probability $1-n^{-\omega(1)}$, there are at most $(1+\eps)p^7n^5$ pairs $(P_1,P_2)$ of vertex-disjoint paths in $G|_D$ of length seven such that $P_1$ is an $x_1,y_1$-path, $P_2$ is an $x_2,y_2$-path and they have the same colours in the same order.
\end{lemma}

\begin{center}
\begin{tikzpicture}
\def\setsp{0.8}

\foreach \n/\ifBup in {1/0,2/1,3/0,4/1,5/0,6/1,7/0,8/1}
{
\coordinate (v\n) at ($(\n*\setsp,0)+\ifBup*(0,0.5)$);
}
\foreach \n/\ifBup in {9/0,10/1,11/0,12/1,13/0,14/1,15/0,16/1}
{
\coordinate (v\n) at ($(\n*\setsp,0)+\ifBup*(0,0.5)$);
}

\draw [red,thick] (v1) -- (v2);
\draw [darkgreen,thick] (v2) -- (v3);
\draw [blue,thick] (v3) -- (v4);
\draw [purple,thick] (v4) -- (v5);
\draw [orange,thick] (v5) -- (v6);
\draw [teal,thick] (v6) -- (v7);
\draw [brown,thick] (v7) -- (v8);

\draw [red,thick] (v9) -- (v10);
\draw [darkgreen,thick] (v10) -- (v11);
\draw [blue,thick] (v11) -- (v12);
\draw [purple,thick] (v12) -- (v13);
\draw [orange,thick] (v13) -- (v14);
\draw [teal,thick] (v14) -- (v15);
\draw [brown,thick] (v15) -- (v16);

\draw [red] ($0.5*(v1)+0.5*(v2)+(0,0.2)$) node {\footnotesize $1$};
\draw [red] ($0.5*(v9)+0.5*(v10)+(0,0.2)$) node {\footnotesize $1$};
\draw [darkgreen] ($0.5*(v2)+0.5*(v3)+(0,0.2)$) node {\footnotesize $2$};
\draw [darkgreen] ($0.5*(v10)+0.5*(v11)+(0,0.2)$) node {\footnotesize $2$};
\draw [blue] ($0.5*(v3)+0.5*(v4)+(0,0.2)$) node {\footnotesize $3$};
\draw [blue] ($0.5*(v11)+0.5*(v12)+(0,0.2)$) node {\footnotesize $3$};
\draw [purple] ($0.5*(v4)+0.5*(v5)+(0,0.2)$) node {\footnotesize $4$};
\draw [purple] ($0.5*(v12)+0.5*(v13)+(0,0.2)$) node {\footnotesize $4$};
\draw [orange] ($0.5*(v5)+0.5*(v6)+(0,0.2)$) node {\footnotesize $5$};
\draw [orange] ($0.5*(v13)+0.5*(v14)+(0,0.2)$) node {\footnotesize $5$};
\draw [teal] ($0.5*(v6)+0.5*(v7)+(0,0.2)$) node {\footnotesize $6$};
\draw [teal] ($0.5*(v14)+0.5*(v15)+(0,0.2)$) node {\footnotesize $6$};
\draw [brown] ($0.5*(v7)+0.5*(v8)+(0,0.2)$) node {\footnotesize $7$};
\draw [brown] ($0.5*(v15)+0.5*(v16)+(0,0.2)$) node {\footnotesize $7$};

\foreach \n in {1,...,16}
{
\draw [fill] (v\n) circle[radius=0.05];
}

\draw ($(v1)+(0,-0.3)$) node {$x_1$};
\draw ($(v9)+(0,-0.3)$) node {$x_2$};
\draw ($(v8)+(0,0.3)$) node {$y_1$};
\draw ($(v16)+(0,0.3)$) node {$y_2$};
\end{tikzpicture}
\end{center}

\begin{proof}
Let $H$ be the graph of edges in $G$ next to $\{x_1,x_2,y_1,y_2\}$ with colour in $D$, and let $\mathcal{H}$ be the set of possibilities for $H$.
Let $\mathcal{A}$ be the event that there are more than $(1+\eps)p^7n^5$ pairs $(P_1,P_2)$ of vertex-disjoint paths in $G|_D$ of length seven such that $P_1$ is an $x_1,y_1$-path, $P_2$ is an $x_2,y_2$-path and they have the same colours in the same order.
Let $\mathcal{B}$ be the event that, for each $e\in E(G-\{x_1,x_2,y_1,y_2\})$, there are at most $n^{3.03}$ pairs  $(P_1,P_2)$ of vertex-disjoint paths in $G$ of length seven such that $P_1$ is an $x_1,y_1$-path, $P_2$ is an $x_2,y_2$-path, they have the same colours in the same order, and $e\in E(P_1\cup P_2)$.
We will show the following claim.

\begin{claim}\label{clm:AH} For each $\hat{H}\in \mathcal{H}$, $\P(\mathcal{A}\land\mathcal{B}|H=\hat{H})=n^{-\omega(1)}$.
\end{claim}
Given this claim, as, by Proposition~\ref{prop:linksusingearefew}, $\P(\mathcal{B})=1-n^{-\omega(1)}$, we will have
\begin{align}
\P(\mathcal{A})&\leq \P(\bar{\mathcal{B}})+\P(\mathcal{A}\land \mathcal{B})= n^{-\omega(1)}+ \sum_{\hat{H}\in \cH}\P((\mathcal{A}\land \mathcal{B})\land (H=\hat{H}))\nonumber\\
&= n^{-\omega(1)}+\sum_{\hat{H}\in \cH}\P(\mathcal{A}\land \mathcal{B}|H=\hat{H})\cdot \P(H=\hat{H})
\leq n^{-\omega(1)}+\max_{\hat{H}\in \cH}\P(\mathcal{A}\land \mathcal{B}|H=\hat{H})\nonumber \\
&= n^{-\omega(1)}.\label{eqn:forlaterreference}
\end{align}

Therefore, it is left only to prove Claim~\ref{clm:AH}.

\smallskip

\noindent\emph{Proof of Claim~\ref{clm:AH}.} Let $\hat{H}\in \cH$, $\eta=0.01$ and $\kappa=n^{1+\eta}$. Let $\mathcal{F}_{\hat{H}}$ be the set of subgraphs $F$ which consist of two vertex-disjoint paths of length 5 with vertices in $(A\cup B)\setminus \{x_1,x_2,y_1,y_2\}$, where $F$ is additionally labelled\footnote{This is done as otherwise each $F$ can have 2 matching paths $(P_1,P_2)$. This did not matter analogously in the proof of Claim~\ref{clm:AH1} as we were proving a looser bound.} with a colour $c$, such that there is a pair $(P_1,P_2)$ of vertex-disjoint paths in $\hat{H}\cup F$ of length seven such that $P_1$ is an $x_1,y_1$-path, $P_2$ is an $x_2,y_2$-path and they have the same colours in the same order, all of which are in $D$, and such that $c$ is the first colour on $P_1$.
Such a pair $(P_1,P_2)$ is determined by choice of the first and last colour of $P_1$ (with at most $(pn)^2$ choices, and which determines $x_1,x_2,y_1$ and $y_2$), the 4 internal vertices of $P_1$ which are not neighbours of $x_1$ or $y_1$ in $P_1$, the 4 internal vertices of $P_2$ which are not neighbours of $x_2$ or $y_2$ in $P_2$, and the 5 colours in order appearing on the edges of $P_1$ which are not in $\hat{H}$. Note (in part for the implication of $\mathcal{A}$ holding later) that each such pair $(P_1,P_2)$ gives rise to exactly one subgraph in $\mathcal{F}_{\hat{H}}$ (labelled with the first colour of $P_1$) and each such subgraph in $\mathcal{F}_{\hat{H}}$ has exactly one such pair $(P_1,P_2)$.
In particular, then,
 $|\mathcal{F}_{\hat{H}}|\leq (pn)^2\cdot n^4\cdot n^4\cdot (pn)^5=p^7n^{15}$. Note that each subgraph in $\mathcal{F}_{\hat{H}}$ has 10 edges.

Let $\mathcal{S}_{\hat{H}}$ be the set of sequences $(F_1,\ldots,F_{\kappa})$ of length $\kappa$ of edge-disjoint subgraphs from $\mathcal{F}_{\hat{H}}$ for which each colour appears on  $\bigcup_{i\in [\kappa]}F_i$ at most $pn/2$ times, and note that, then,
\begin{equation}\label{eqn:sizeS1here}
|\mathcal{S}_{\hat{H}}|\leq p^{7\kappa}n^{15\kappa}.
\end{equation}
For each $S=(F_1,\ldots,F_{\kappa})\in \mathcal{S}_{\hat{H}}$, let $H_S=\cup_{i\in[\kappa]}F_i$, so that $e(H_S)=10\kappa$ and each colour appears on $H_S$ at most $pn/2$ times.
Let $Z_{\hat{H}}$ be the number of
$S\in \mathcal{S}_{\hat{H}}$ with $H_S\subset G$. As follows, if $\mathcal{A}$ and $\mathcal{B}$ hold, then $Z_{\hat{H}}\geq ((1+\eps/2)p^7n^5)^\kappa$.  Indeed, if $\mathcal{A}$ and $\mathcal{B}$ hold, then we can pick a sequence $(F_1,\ldots,F_\kappa)$ of edge-disjoint subgraphs from $\mathcal{F}_{\hat{H}}$ by picking each $F_i$, $1\leq i\leq\kappa$, in turn, where at the selection of each $F_i$, $i\in [\kappa]$,
as $\mathcal{A}$ and $\mathcal{B}$ hold and there will be $10(i-1)$ edges we wish to avoid and at most $10(i-1)\cdot 2/pn$ colours, the number of possibilities for the choice of $F_i$ will be at least
\[
(1+\eps)p^7n^5-10\kappa\cdot n^{3.03}-\frac{10\kappa}{pn/2}\cdot n\cdot n^{3.03}\geq (1+\eps)p^7n^5-\frac{30\cdot n^{1+\eta}\cdot n^{3.03}}{p} \geq (1+\eps/2)p^7n^5,
\]
where we have used that $1/n\llpoly p,\eps$.

For each $S\in \mathcal{H}$, $H_S\cup \hat{H}$ has at most $pn/2+4\leq pn$ edges of each colour and every edge on $H_S\cup \hat{H}$ has colour in $D$.
Therefore, for each $S\in \mathcal{S}$ such that $H_S\cup \hat{H}$ is properly coloured, similarly to \eqref{eqn:longversion}, and using Corollary~\ref{cor:latinsquareprobabilities} twice and that $p\kappa\geq n\log^2n$, we have
\[
\P(H_S\subset G|H=\hat{H})= e^{O(p\kappa+n\log^2n)}n^{-10\kappa}=e^{O(p\kappa)}n^{-10\kappa}.
\]
Thus, as this holds for every $S\in \mathcal{S}_{\hat{H}}$ such that $H_S\cup \hat{H}$ is properly coloured, and  $p\llpoly \eps$,
\begin{align*}
\E(Z_{\hat{H}}|H=\hat{H})&\leq |\mathcal{S}_{\mathcal{H}}|\cdot e^{O(p\kappa)}n^{-10\kappa}
\overset{\eqref{eqn:sizeS1here}}{\leq} 
 p^{7\kappa}n^{5\kappa}\cdot e^{O(p\kappa)}=(1+\eps/2)^{\kappa/2}p^{7\kappa}n^{5\kappa}
\\
&=((1+\eps/2)p^7n^{5})^{\kappa}\cdot {n^{-\omega(1)}}.
\end{align*}
Then, by Markov's inequality, we have
\[
\P(\mathcal{A}\land\mathcal{B}|H=\hat{H})\leq \P(Z_{\hat{H}}\geq ((1+\eps/2)p^7n^{5})^{\kappa})\leq \frac{\E(Z_{\hat{H}}|H=\hat{H})}{((1+\eps/2)p^7n^{5})^{\kappa}}
=n^{-\omega(1)}.
\]
This completes the proof of the claim, and hence the lemma.\hspace{5.2cm}$\boxdot$
\end{proof}

In the same way as Corollary~\ref{cor:fewlengththreesame} follows from Lemma~\ref{lem:looseupperlength3}, the following corollary follows from Lemma~\ref{lem:tightupplength7}, where again the paths counted can now have any colours, not just those in $D$.

\begin{corollary}\label{cor:fewlengthsevensame}
Let $1/n\llpoly \eps$. Let $G\sim G_{[n]}^\col$. Let $x_1,x_2,y_1,y_2\in V(G)$ be distinct with $x_1\nsimAB y_1$ and $x_2\nsimAB y_2$.
Then, with probability $1-n^{-\omega(1)}$, there are at most $(1+\eps)n^5$ pairs $(P_1,P_2)$ of vertex-disjoint paths in $G$ of length seven such that $P_1$ is an $x_1,y_1$-path, $P_2$ is an $x_2,y_2$-path and they have the same colours in the same order.\hfill\qed
\end{corollary}


\subsection{Preparation for the lower bound for length 15 paths}\label{sec:path15prep}
We now prepare to prove, for distinct $x_1,x_2,y_1,y_2\in V(G)$ with $x_1\nsimAB y_1$ and $x_2\nsimAB y_2$, a likely lower bound
on the number of pairs of vertex-disjoint paths with length 15 between $(x_1,y_1)$ and $(x_2,y_2)$ respectively which use the same colours in the same order. To do this, we consider the number of such pairs where the respective middle edges may instead have any colours (as pictured in Lemma~\ref{lem:lotsofLzero}), and show firstly that many such pairs of paths will exist (see Lemma~\ref{lem:lotsofLzero}). Not imposing a colour condition on the middle two edges means that this can be done using relatively simple combinatorial and probabilistic arguments. Then, using our previously shown likely upper bounds, we show that it is very likely that no pair of edges $e,f$ appear as these middle edges more than we should expect (see Corollary~\ref{cor:nottoomanyLzero}), before turning this into a result on the number of pairs $e,f$ that can appear as these middle edges distinctly less than we should expect (see Corollary~\ref{cor:Lzerostuff}).
We start with the following lemma, which gives a lower bound on these paths with colours in $D$ and edges in $E$. The set $E$ is used because eventually we will use this to control the number of edges of each colour and at each vertex in a subgraph to which we apply Corollary~\ref{cor:latinsquareprobabilities} (see $H$ in the proof of Lemma~\ref{lemma:mainLlinksfirstlower}).


\begin{lemma}\label{lem:lotsofLzero} Let $1/n\llpoly p,\eps\leq 1$. Let $D\subset [n]$ have size $pn$ and let $G\sim G^\col_{[n]}$. Let $E\subset E(G)$ be formed by including each edge independently at random with probability $p$. Let $x_1,x_2,y_1,y_2\in V(G)$ be distinct with $x_1\nsimAB y_1$ and $x_2\nsimAB y_2$.

Then, with probability $1-n^{-\omega(1)}$, there are at least
$(1- \eps)p^{42}(1-p)^2n^{14}$ pairs $(P_1,P_2)$ of vertex-disjoint rainbow paths of length 15 in $G$, each of whose middle edges have colour not in $D$ and all other edges in $E$ with colour in $D$, and such that $P_1$ is an $x_1,y_1$-path, $P_2$ is an $x_2,y_2$-path, and, apart from possibly their middle edges, the paths $P_1$ and $P_2$ have the same colours in the same order, as pictured below.

\begin{center}
\begin{tikzpicture}
\def\setsp{0.4}
\def\upup{0.4}
\def\extra{0.8}

\foreach \n in {1,3,5,7,9,11,13,15}
{
\coordinate (v\n) at ($(\n*\setsp,0)+(0,0)$);
}
\foreach \n in {17,19,21,23,25,27,29,31}
{
\coordinate (v\n) at ($(\extra,0)+(\n*\setsp,0)+(0,0)$);
}
\foreach \n in {2,4,6,8,10,12,14,16}
{
\coordinate (v\n) at ($(\n*\setsp,0)+(0,\upup)$);
}
\foreach \n in {16}
{
\coordinate (v16plus) at ($(\extra,0)+(\n*\setsp,0)+(0,\upup)$);
}
\foreach \n in {18,20,22,24,26,28,30}
{
\coordinate (v\n) at ($(\extra,0)+(\n*\setsp,0)+(0,\upup)$);
}

\draw [red,thick] (v1) -- (v2);
\draw [darkgreen,thick] (v2) -- (v3);
\draw [blue,thick] (v3) -- (v4);
\draw [red,thick] (v16plus) -- (v17);
\draw [darkgreen,thick] (v17) -- (v18);
\draw [blue,thick] (v18) -- (v19);
\draw [brown,thick] (v4) -- (v5);
\draw [brown,thick] (v19) -- (v20);
\draw [red!50,thick] (v5) -- (v6);
\draw [red!50,thick] (v20) -- (v21);
\draw [violet,thick] (v6) -- (v7);
\draw [violet,thick] (v21) -- (v22);
\draw [teal,thick] (v7) -- (v8);
\draw [teal,thick] (v22) -- (v23);
\draw [dotted,thick] (v8) -- (v9);
\draw [dotted,thick] (v23) -- (v24);
\draw [magenta,thick] (v9) -- (v10);
\draw [magenta,thick] (v24) -- (v25);
\draw [thick] (v10) -- (v11);
\draw [thick] (v25) -- (v26);
\draw [purple,thick] (v11) -- (v12);
\draw [purple,thick] (v26) -- (v27);
\draw [green,thick] (v12) -- (v13);
\draw [green,thick] (v27) -- (v28);
\draw [orange,thick] (v13) -- (v14);
\draw [orange,thick] (v28) -- (v29);
\draw [olive,thick] (v14) -- (v15);
\draw [olive,thick] (v29) -- (v30);
\draw [cyan,thick] (v15) -- (v16);
\draw [cyan,thick] (v30) -- (v31);

\draw [red] ($0.5*(v1)+0.5*(v2)+(-0.05,0.15)$) node {\footnotesize $1$};
\draw [red] ($0.5*(v16plus)+0.5*(v17)+(-0.05,-0.15)$) node {\footnotesize $1$};
\draw [darkgreen] ($0.5*(v2)+0.5*(v3)+(-0.05,-0.15)$) node {\footnotesize $2$};
\draw [darkgreen] ($0.5*(v17)+0.5*(v18)+(-0.05,0.15)$) node {\footnotesize $2$};
\draw [blue] ($0.5*(v3)+0.5*(v4)+(-0.05,0.15)$) node {\footnotesize $3$};
\draw [blue] ($0.5*(v18)+0.5*(v19)+(-0.05,-0.15)$) node {\footnotesize $3$};
\draw [brown] ($0.5*(v4)+0.5*(v5)+(-0.05,-0.15)$) node {\footnotesize $4$};
\draw [brown] ($0.5*(v19)+0.5*(v20)+(-0.05,0.15)$) node {\footnotesize $4$};
\draw [red!50] ($0.5*(v5)+0.5*(v6)+(-0.05,0.15)$) node {\footnotesize $5$};
\draw [red!50] ($0.5*(v20)+0.5*(v21)+(-0.05,-0.15)$) node {\footnotesize $5$};
\draw [violet] ($0.5*(v6)+0.5*(v7)+(-0.05,-0.15)$) node {\footnotesize $6$};
\draw [violet] ($0.5*(v21)+0.5*(v22)+(-0.05,0.15)$) node {\footnotesize $6$};
\draw [teal] ($0.5*(v7)+0.5*(v8)+(-0.05,0.15)$) node {\footnotesize $7$};
\draw [teal] ($0.5*(v22)+0.5*(v23)+(-0.05,-0.15)$) node {\footnotesize $7$};
\draw [] ($0.5*(v8)+0.5*(v9)+(-0.05,-0.15)$) node {\footnotesize $e$};
\draw [] ($0.5*(v23)+0.5*(v24)+(-0.05,0.15)$) node {\footnotesize $f$};
\draw [magenta] ($0.5*(v9)+0.5*(v10)+(-0.05,0.15)$) node {\footnotesize $9$};
\draw [magenta] ($0.5*(v24)+0.5*(v25)+(-0.05,-0.15)$) node {\footnotesize $9$};

\draw [black] ($0.5*(v10)+0.5*(v11)+(-0.05-0.075,-0.15)$) node {\footnotesize $10$};
\draw [black] ($0.5*(v25)+0.5*(v26)+(0.05-0.075-0.075,0.15)$) node {\footnotesize $10$};
\draw [purple] ($0.5*(v11)+0.5*(v12)+(-0.05-0.075,0.15)$) node {\footnotesize $11$};
\draw [purple] ($0.5*(v26)+0.5*(v27)+(-0.05-0.075,-0.15)$) node {\footnotesize $11$};
\draw [green] ($0.5*(v12)+0.5*(v13)+(0.05-0.075-0.075,-0.15)$) node {\footnotesize $12$};
\draw [green] ($0.5*(v27)+0.5*(v28)+(-0.05-0.075,0.15)$) node {\footnotesize $12$};
\draw [orange] ($0.5*(v13)+0.5*(v14)+(-0.05-0.075,0.15)$) node {\footnotesize $13$};
\draw [orange] ($0.5*(v28)+0.5*(v29)+(-0.05-0.075,-0.15)$) node {\footnotesize $13$};
\draw [olive] ($0.5*(v14)+0.5*(v15)+(-0.05-0.075,-0.15)$) node {\footnotesize $14$};
\draw [olive] ($0.5*(v29)+0.5*(v30)+(0.05-0.075-0.075,0.15)$) node {\footnotesize $14$};
\draw [cyan] ($0.5*(v15)+0.5*(v16)+(-0.05-0.075,0.15)$) node {\footnotesize $15$};
\draw [cyan] ($0.5*(v30)+0.5*(v31)+(-0.05-0.075,-0.15)$) node {\footnotesize $15$};


\foreach \n in {1,...,31}
{
\draw [fill] (v\n) circle[radius=0.04];
}

\draw [fill] (v16plus) circle[radius=0.04];

\draw ($(v1)-(0,0.3)$) node {$x_1$};
\draw ($(v16)+(0,0.2)$) node {$y_1$};
\draw ($(v16plus)+(0,0.2)$) node {$x_2$};
\draw ($(v31)-(0,0.3)$) node {$y_2$};
\end{tikzpicture}
\end{center}

\end{lemma}
\begin{proof}  Note that we can assume that $\eps=o(1)$.
Let $G'\in G^\col_{[n]}$. We start by proving three claims about $G'$, before using this to derive likely properties of $G\sim G^\col_{[n]}$. Let $\mathcal{P}$ be the set of pairs $(P_1,P_2)$ of vertex-disjoint rainbow paths of length 15 in $G'$, each of whose middle edge has colour not in $D$ and all its other edges have colour in $D$, and such that $P_1$ is an $x_1,y_1$-path, $P_2$ is an $x_2,y_2$-path, and, apart from possibly their middle edges, the paths $P_1$ and $P_2$ have the same colours in the same order.

\begin{claim}\label{clm:Pbig} $|\mathcal{P}|\geq (1-\eps/2)n^{14}$.
\end{claim}
\begin{proof}[Proof of Claim~\ref{clm:Pbig}]
Let $\mathcal{D}_x$ be the set of sequences $\mathbf{c}=(c_1,c_2,\ldots,c_7)$ of distinct colours such that there are vertex-disjoint paths $P_{\mathbf{c},1}$ and $P_{\mathbf{c},2}$ of length 7 in $G'$ which, respectively, are from $x_1$ and $x_2$, and have colours in the order in $\mathbf{c}$. Similarly, define $\mathcal{D}_y$, and paths $Q_{\mathbf{c},1}$ and $Q_{\mathbf{c},2}$, for each $\mathbf{c}\in \mathcal{D}_y$, starting from $y$.

Note that, if $\mathbf{c}=(c_1,c_2,\ldots,c_7)$ is chosen in order, so that, at each stage $i\in [7]$, the two paths of length $i$ leading from $x_1$ and $x_2$ with colours $c_1,\ldots,c_i$ in order are vertex-disjoint, then when choosing $c_j$ with $j\in [7]$, there are $O(1)$ colours that need to be avoided because adding an edge of that colour to one of the paths will lead to a vertex on either path or have a colour already on the paths.
 Thus, $|\mathcal{D}_x|\geq (n-O(1))^7\geq n^7-O(n^6)$, and, similarly, $|\mathcal{D}_y|\geq n^7-O(n^6)$.

For each $\mathbf{c}\in \mathcal{D}_x$, there are 16 vertices and 7 colours together in the paths $P_{\mathbf{c},1}$ and $P_{\mathbf{c},2}$, and thus $O(n^6)$ paths of length $7$ from $y_1$ or $y_2$ which contain a vertex in $P_{\mathbf{c},1}$ and $P_{\mathbf{c},2}$ or have an edge to $V(P_{\mathbf{c},1}\cup P_{\mathbf{c},2})$ with a colour in $\mathbf{c}$.
As this holds similarly for each $\mathbf{d}\in \mathcal{D}_y$, there are $n^{14}-O(n^{13})$ choices of $(\mathbf{c},\mathbf{d})$ with
$\mathbf{c}\in \mathcal{D}_x$, $\mathbf{d}\in \mathcal{D}_y$ such that $V(P_{\mathbf{c},1}\cup P_{\mathbf{c},2})$ and $V(Q_{\mathbf{d},1}\cup Q_{\mathbf{d},2})$ are disjoint and have no edge between them with colour in $\mathbf{c}$ or $\mathbf{d}$. Thus,
$|\mathcal{P}|= n^{14}-O(n^{13})\geq (1-\eps/2)n^{14}$.
\claimproofend
\begin{claim}\label{clm:Psmallcod} For each $e\in E(G'-\{x_1,x_2,y_1,y_2\})$ there are $O(n^{12})$ pairs of paths in $\mathcal{P}$ where one of the paths contains ${e}$.
\end{claim}
\begin{proof} Let $e\in E(G'-\{x_1,x_2,y_1,y_2\})$. Any $x_1,y_1$-path of length 15 in $G$ containing $e$ then contains 12 vertices not in $V(e)\cup \{x_1,x_2\}$, so there are $O(n^{12})$ different $x_1,y_1$-paths of length 15 in $G$ containing $e$. Similarly, there are there are $O(n^{12})$ different $x_2,y_2$-paths of length 15 in $G$ containing $e$. Noting that each path $P_1$ appears in at most one pair $(P_1,P_2)\in \mathcal{P}$ or $(P_2,P_1)\in \mathcal{P}$, the claim follows.
\claimproofend
\begin{claim}\label{clm:Psmallcod2} For each edge $e\in E(G')$ containing $x_1$, $x_2$, $y_1$ or $y_2$, there are $O(n^{13})$ pairs of paths in $\mathcal{P}$ where one of the paths contains ${e}$.
\end{claim}
\begin{proof} Let $e\in E(G')$ contain $x_1$, $x_2$, $y_1$ or $y_2$, and suppose without loss of generality that it contains one of $x_1,x_2$. Any $x_1,y_1$-path of length 15 in $G$ containing $e$ then contains 13 vertices not in $V(e)\cup \{x_1,x_2\}$, so there are $O(n^{13})$ different $x_1,y_1$-paths of length 15 in $G$ containing $e$.  Noting that each path $P_1$ appears in at most one pair $(P_1,P_2)\in \mathcal{P}$, the claim follows.
\claimproofend

We will now create a set of random colours $\hat{D}$ and deduce likely properties about the pairs of paths in $\mathcal{P}$ which use colours only in $\hat{D}$, except for their middle edges which do not have colour in $\hat{D}$.
Now, let $\hat{D}_0,\hat{D}_1\subset C$ be disjoint random sets of colours such that each colour $c\in C$ is independently added to $\hat{D}_0$ with probability $(1-\eps^{2})p$ and to $\hat{D}_1$ with probability $(1-\eps^{2})(1-p)$.
Let $\hat{D}\subset C$ be a random set of colours such that if $|\hat{D}_0|\leq pn$ and $|\hat{D}_1|\leq (1-p)n$, $\hat{D}$ is chosen uniformly at random from all subsets of $C$ with size $pn$ such that $\hat{D}_0\subset \hat{D}\subset C\setminus \hat{D}_1$, and, otherwise, $\hat{D}$ is chosen uniformly at random from all subsets of $C$ with size $pn$. Note that, by Lemma~\ref{chernoff}, with probability $1-n^{-\omega(1)}$, we have $\hat{D}_0\subset \hat{D}\subset C\setminus \hat{D}_1$.

Let $X$ be the number of pairs of paths $(P_1,P_2)\in \mathcal{P}$ whose edges all have colour in $\hat{D}_0$ and which are in $E$, except for their middle edges which have colour in $\hat{D}_1$ and which may be in $E$ or not.
For each $(P_1,P_2)\in \mathcal{P}$, the probability $(P_1,P_2)$ satisfies these conditions is (as the middle edges may or may not have the same colour) at least $(1-\eps^{2})^{16}p^{28}p^{14}(1-p)^2$, and thus,
from Claim~\ref{clm:Pbig}, we have
\[
\E X\geq (1-\eps^{2})^{16}p^{42}(1-p)^2\cdot (1-\eps/2)n^{14}\geq (1-2\eps/3)p^{42}(1-p)^2n^{14}.
\]
Now, by Claim~\ref{clm:Psmallcod}, there is some $\lambda_1=O(1)$ such that, if ${e}\in E(G'-\{x_1,x_2,y_1,y_2\})$, then changing whether or not ${e}$ is in $E$ changes $X$ by at most $\lambda_1n^{12}$. Furthermore, by Claim~\ref{clm:Psmallcod2}, there is some $\lambda_2=O(1)$ such that, if ${e}\in E(G')$ does contain $x_1$, $x_2$, $y_1$ or $y_2$, then changing whether or not ${e}$ is in $E$ changes $X$ by at most $\lambda_2n^{13}$.
Finally, again by Claim~\ref{clm:Psmallcod} and Claim~\ref{clm:Psmallcod2}, and as each colour appears in $G'$ $n$ times, and has at most 4 edges touching $\{x_1,x_2,y_1,y_2\}$, there is some $\lambda_3=O(1)$ such that, for each $c\in [n]$, changing whether $c$ is in $\hat{D}_0$, or $\hat{D}_1$, or neither $\hat{D}_0$ nor $\hat{D}_1$, changes $X$ by at most $\lambda_3n^{13}$.
Thus, by Lemma~\ref{lem:mcdiarmidchangingc} with $t=\eps p^{42}(1-p)^2n^{14}/3$,
\begin{align}
\P(X\leq (1-\eps)p^{42}(1-p)^2n^{14})&\leq 2 \exp\left(-\frac{2t^2}{n^2\cdot (\lambda_1n^{12})^2+4n\cdot (\lambda_2n^{13})^2+n\cdot (\lambda_3n^{13})^2}\right)\nonumber\\
& \leq 2 \exp\left(-\Omega(t^2/n^{27}\right)=2 \exp\left(-\Omega(\eps^2p^{84}(1-p)^4n)\right)=n^{-\omega(1)},\label{eq:mcdexample}
\end{align}
where we have used that $1/n\llpoly \eps,p$. Thus, as $\hat{D}_1\subset \hat{D}\subset C\setminus \hat{D}_2$ with probability $1-n^{-\omega(1)}$, with probability $1-n^{-\omega(1)}$ there are  at least
$(1- \eps)p^{42}(1-p)^2n^{14}$ pairs $(P_1,P_2)$ of vertex-disjoint rainbow paths in $G'$, each of whose middle edge has colour not in $\hat{D}$ and all its other edges are in $E$ and have colour in $\hat{D}$, and such that $P_1$ is an $x_1,y_1$-path, $P_2$ is an $x_2,y_2$-path, and, apart from possibly their middle edges, the paths $P_1$ and $P_2$ have the same colours in the same order, as pictured below.
As the distribution $\hat{D}$ is that of a set of $pn$ colours chosen uniformly at random from $C$ for each fixed $G'\in \mathcal{G}^{\col}_{[n]}$, and $G\sim G^\col_{[n]}$, the result of the lemma follows easily.
\end{proof}


We now deduce from our previous likely upper bounds (specifically Lemma~\ref{lem:looseupperlength3} and~Lemma~\ref{lem:tightupplength7}), that no pair $e,f$ is likely to appear much more often than expected as the middle two edges in the pairs of paths counted in Lemma~\ref{lem:lotsofLzero}.

\begin{corollary}\label{cor:nottoomanyLzero}
Let $1/n\llpoly p\llpoly \eps$. Let $D\subset [n]$ have size $pn$ and let $G\sim G^\col_{[n]}$. Let $E\subset E(G)$ be formed by including each edge independently at random with probability $p$. Let $x_1,x_2,y_1,y_2\in V(G)$ be distinct with $x_1\nsimAB y_1$ and $x_2\nsimAB y_2$.

Then, with probability $1-n^{-\omega(1)}$, for any pair of edges $e,f\in E(G)$, there are at most
$(1+\eps)p^{42}n^{10}$ pairs $(P_1,P_2)$ of vertex-disjoint rainbow paths in $G$ of length 15, and all the edges of the paths apart from the two middle edges are in $E$ and have colour in $D$, and such that $P_1$ is an $x_1,y_1$-path with middle edge $e$, $P_2$ is an $x_2,y_2$-path with middle edge $f$, and, apart from possibly their middle edges, the paths $P_1$ and $P_2$ have the same colours in the same order.
\end{corollary}
\begin{proof}
By Lemma~\ref{lem:tightupplength7}, with probability $1-n^{-\omega(1)}$, we have that, for every distinct $x'_1,x'_2\in A$ and distinct $y'_1,y'_2\in B$, there are at most $(1+\eps/3)p^7n^{5}$ pairs $(P'_1,P'_2)$ of vertex-disjoint paths in $G|_D$ of length 7 such that $P'_1$ is an $x'_1,y'_1$-path, $P'_2$ is an $x'_2,y'_2$-path and they have the same colours in the same order.
By Lemma~\ref{lem:looseupperlength3}, with probability $1-n^{-\omega(1)}$, we have that, for every distinct $x'_1,x'_2\in A$ and distinct $y'_1,y'_2\in B$, there are at most $n^{1.01}$ pairs $(P'_1,P'_2)$ of vertex-disjoint paths in $G|_D$ of length 3 such that $P'_1$ is an $x'_1,y'_1$-path, $P'_2$ is an $x'_2,y'_2$-path and they have the same colours in the same order.
Assuming these properties for $G|_D$, and revealing $E$, we will show that, with probability $1-n^{-\omega(1)}$, for every distinct $a_1,a_2\in A$ and distinct $b_1,b_2\in B$, there are at most $(1+\eps/3)p^{21}n^{5}$ pairs $(P'_1,P'_2)$ of vertex-disjoint paths in $G|_D$ of length seven such that all of their edges are in $E$, $P'_1$ is an $a_1,a_2$-path, $P'_2$ is an $b_1,b_2$-path and they have the same colours in the same order.

For this, suppose that $a_1,a_2\in A$ and $b_1,b_2\in B$ are all distinct, and let $\mathcal{P}$ be the set of pairs $(P'_1,P'_2)$ of vertex-disjoint paths in $G|_D$ of length seven  such that $P'_1$ is an $a_1,a_2$-path, $P'_2$ is an $b_1,b_2$-path and they have the same colours in the same order. Then, by the property from Lemma~\ref{lem:tightupplength7}, we have $|\mathcal{P}|\leq (1+\eps/3)p^7n^5$. Furthermore, given any edge $e\in E(G-\{a_1,a_2,b_1,b_2\})$, we can count the number of pairs of paths $(P'_1,P'_2)\in \mathcal{P}$ which use $e$ by choosing which of $P_1'$ or $P_2'$ contains $e$ and which edge of the path this is (with at most 14 choices).
Then, assuming $e$ is among the first 4 edges of $P_1'$ (where the other cases follow almost identically), we can choose the other vertices for the first 4 interior vertices of $P_1'$ (with at most $n^2$ options) which also determines the first 4 interior vertices of $P_2'$, whereupon we have at most $n^{1.01}$ options to choose the remaining subpaths of $P'_1$ and $P_2'$ using the property from Lemma~\ref{lem:looseupperlength3}. Thus, each edge in $G-\{a_1,a_2,b_1,b_2\}$ is contained in at most $10n^{3.01}$ pairs of paths in $\mathcal{P}$.
Similarly, each edge in $G$ which contains a vertex in $\{a_1,a_2,b_1,b_2\}$ is contained in at most $4n^{4.01}$ pairs of paths.

Then, as the expected number of pairs of paths in $\mathcal{P}$ whose edges are all in $E$ is $p^{14}|\mathcal{P}|\leq (1+\eps/3)p^{21}n^5$, as $1/n\llpoly p,\eps$, using McDiarmid's inequality (Lemma~\ref{lem:mcdiarmidchangingc}) similarly to how we did at \eqref{eq:mcdexample} we have that, with probability $1-n^{-\omega(1)}$,  there are at most $(1+2\eps/3)p^{21}n^{5}$ pairs $(P'_1,P'_2)\in \mathcal{P}$
such that all of their edges are in $E$. Taking a union bound, with probability $1-n^{-\omega(1)}$, we will have that for any $a_1,a_2\in A$ and distinct $b_1,b_2\in B$, there are at most $(1+2\eps/3)p^{21}n^{5}$ pairs $(P'_1,P'_2)$ of vertex-disjoint paths in $G|_D$ with edges in $E$ such that $P'_1$ is an $a_1,a_2$-path, $P'_2$ is a $b_1,b_2$-path and they have the same colours in the same order.

Assuming this, the property we want for any pair of edges $e,f\in E(G)$ follows. Indeed, first note that the property we want is trivial unless these edges share no vertices and have no vertices in $\{x_1,y_1,x_2,y_2\}$. Assuming otherwise, then, we can let $a_1,b_1,a_2,b_2$ be such that $e=a_1b_1$, $b_1\nsimAB x_1$, $f=a_2b_2$ and $a_2\nsimAB x_2$. There are at most $((1+\eps/3)p^{21}n^{5})^2$ choices for paths $P_1,P_2,P_3,P_4$ such that (as depicted in Figure~\ref{fig:formakingLzero}) they all have length 7, $P_1$ is an $x_1,b_1$-path, $P_2$ is an $a_1,y_1$-path, $P_3$ is an $x_2,a_2$-path, $P_4$ is an $b_2,y_2$-path, $P_1$ and $P_3$ have the same colours in the same order and are vertex-disjoint, and $P_2$ and $P_4$ have the same colours in the same order and are vertex-disjoint. Thus, there are at most
$(1+\eps)p^{42}n^{10}$ pairs $(P'_1,P'_2)$ of vertex-disjoint rainbow paths in $G$ of length 15, and all the edges of the paths apart from the two middle edges are in $E$ and have colour in $D$, and such that $P_1'$ is an $x_1,y_1$-path with middle edge $e$, $P_2'$ is an $x_2,y_2$-path with middle edge $f$, and, apart from possibly their middle edges, the paths $P_1'$ and $P'_2$ have the same colours in the same order, as required.
\end{proof}

\begin{figure}[h]
\begin{center}
\begin{tikzpicture}
\def\setsp{0.8}
\def\setspp{1.2}
\def\setsppp{1.5}
\def\ellheight{0.6}

\coordinate (v2) at (0,0);

\foreach \n/\m/\gapp in {3/2/\setspp,4/3/\setspp,5/4/\setsp,6/5/\setspp,6b/6/\setsppp,7/6b/\setspp,8/7/\setsp,9/8/\setspp}
{
\coordinate (v\n) at ($(v\m)+(\gapp,0)$);
}

\draw [dotted,thick] (v4) -- (v5);
\draw [dotted,thick] (v7) -- (v8);

\draw [] ($0.5*(v7)+0.5*(v8)+(0,0.2)$) node {$f$};
\draw [] ($0.5*(v4)+0.5*(v5)+(0,0.2)$) node {$e$};

\foreach \n/\m/\colll/\subbb in {3/4/red/1,5/6/blue/2,6b/7/red/3,8/9/blue/4}
{
\draw ($0.5*(v\n)+0.5*(v\m)$) node {$P_{\subbb}$};
\draw [thick,\colll] ($0.5*(v\n)+0.5*(v\m)$) circle [x radius = 0.5*\setspp,y radius=0.5*\ellheight];
}

\draw ($(v3)+(-0.3,0)$) node {$x_1$};
\draw ($(v9)+(0.3,0)$) node {$y_2$};

\draw ($(v4)+(0.15,-0.3)$) node {$b_1$};
\draw ($(v5)+(0,-0.3)$) node {$a_1$};
\draw ($(v6)+(0.3,0)$) node {$y_1$};
\draw ($(v6b)+(-0.3,0)$) node {$x_2$};
\draw ($(v7)+(0.15,-0.2)$) node {$a_2$};
\draw ($(v8)+(-0.05,-0.2)$) node {$b_2$};

\foreach \n in {3,4,...,9}
{
\draw [fill] (v\n) circle[radius=0.05];
}
\draw [fill] (v6b) circle[radius=0.05];
\end{tikzpicture}
\end{center}

\vspace{-0.4cm}

\caption{The more general structure counted in the proof of Corollary~\ref{cor:nottoomanyLzero}, where $P_1$ and $P_3$ are paths with length seven with colours in the same order, as are $P_2$ and $P_4$ (using potentially some of the same colours as in $P_1$ and $P_3$).}\label{fig:formakingLzero}
\end{figure}
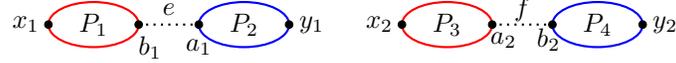


We now show that, roughly speaking, as, by Corollary~\ref{cor:nottoomanyLzero}, no pair $e,f$ can contribute as the middle two edges of the pairs of paths counted in Lemma~\ref{lem:lotsofLzero}, almost all of the possible pairs $e,f$ must contribute not much below what can be expected of them, as follows.


\begin{corollary}\label{cor:Lzerostuff}
Let $1/n\llpoly p,\eps\leq 1$. Let $D\subset [n]$ have size $pn$ and let $G\sim G^\col_{[n]}$. Let $E\subset E(G)$ be formed by including each edge independently at random with probability $p$. Let $x_1,x_2,y_1,y_2\in V(G)$ be distinct with $x_1\nsimAB y_1$ and $x_2\nsimAB y_2$.

Then, with probability $1-n^{-\omega(1)}$, for any pair of edges $e,f\in E(G)$, there are at most $\eps n^4$ pairs of distinct edges $e,f\in E(G|_{[n]\setminus D})$ for which there are at most
$(1-\eps)p^{42}n^{10}$ pairs $(P_1,P_2)$ of vertex-disjoint rainbow paths in $G$ with length 15, and all the edges of the paths apart from the two middle edges are in $E$ and have colour in $D$, and such that $P_1$ is an $x_1,y_1$-path with middle edge $e$, $P_2$ is an $x_2,y_2$-path with middle edge $f$, and, apart from possibly their middle edges, the paths $P_1$ and $P_2$ have the same colours in the same order.
\end{corollary}
\begin{proof} This follows directly from Lemma~\ref{lem:lotsofLzero} and Corollary~\ref{cor:nottoomanyLzero}, as follows.
For each $e,f\in E(G)$ let $X_{e,f}$ be the number of pairs $(P_1,P_2)$ of vertex-disjoint rainbow paths in $G$ of length 15 for which all the edges of the paths apart from the two middle edges are in $E$ and have colour in $D$, and such that $P_1$ is an $x_1,y_1$-path with middle edge $e$, $P_2$ is an $x_2,y_2$-path with middle edge $f$, and, apart from possibly their middle edges, the paths $P_1$ and $P_2$ have the same colours in the same order.
Let $\mathcal{E}$ be the set of pairs of distinct edges $e,f\in E(G|_{[n]\setminus D})$ such that $X_{e,f}\leq (1-\eps)p^{42}n^{10}$.
By Corollary~\ref{cor:nottoomanyLzero}, with probability $n^{-\omega(1)}$ we have $X_{e,f}\leq (1+\eps^2/2)p^{42}n^{10}$ for each $e,f\in E(G)$.
If $|\mathcal{E}|\geq \eps n^2$, then
\begin{align*}
\sum_{e,f\in E(G|_D)}X_{e,f}\leq &\eps n^4\cdot (1-\eps)\cdot p^{42}n^{10}+\left((e(G|_{[n]\setminus D})^2-\eps n^4\right)\cdot (1+\eps^2/2)\cdot p^{42}n^{10}
\\
&\leq \eps n^4\cdot (1-\eps)\cdot p^{42}n^{10}+\left((1-p)^2n^4-\eps n^4\right)\cdot (1+\eps^2/2)\cdot p^{42}n^{10}
\\
& \leq (1-\eps^2/2)\cdot p^{42}(1-p)^2n^{14}.
\end{align*}
However, by Lemma~\ref{lem:lotsofLzero} this does not hold with probability $1-n^{-\omega(1)}$. Therefore, with probability $1-n^{-\omega(1)}$, we must have $|\mathcal{E}|\leq \eps n^4$, as required.
\end{proof}


\subsection{Tight bounds for length 15 paths with the same colours}\label{sec:path15upperlower}

We are now ready to prove our required likely lower bound for pairs of paths of length 15. For convenience, we will record this along with our upper bound in the following result.

\begin{lemma}\label{lemma:mainLlinks} Let $1/n\llpoly \eps$. Let $G\sim {G}^\col_{[n]}$ and let $x_1,x_2,y_1,y_2\in V(G)$ be distinct with $x_1\nsimAB y_1$ and $x_2\nsimAB y_2$. Then, with probability $1-n^{-\omega(1)}$, the following holds.
There are $(1\pm \eps)n^{13}$ pairs $(P_1,P_2)$ of vertex-disjoint paths in $G$ of length 15 such that $P_1$ is an $x_1,y_1$-path, $P_2$ is an $x_2,y_2$-path, and they have the same colours in the same order.
\end{lemma}

For the lower bound in Lemma~\ref{lemma:mainLlinks}, we first prove a similar lower bound where the edges of the paths have colour within some specified subset $D_1$ or $D_2$, as follows, and some edges are in a random set $E$ of edges.

\begin{lemma}\label{lemma:mainLlinksfirstlower} Let $1/n\llpoly p_1\llpoly p_2\llpoly \eps$. Let $D_1,D_2\subset [n]$ be disjoint sets with size $p_1n$ and $p_2n$ respectively. Let $E\subset E(G)$ be formed by including each edge at random with probability $p_1$. Let $G\sim G^\col_{[n]}$  and let $x_1,x_2,y_1,y_2\in V(G)$ be distinct with $x_1\nsimAB y_1$ and $x_2\nsimAB y_2$.
Then, with probability $1-n^{-\omega(1)}$, there are at least $(1-\eps)p_1^{42}p_2n^{13}$ pairs $(P_1,P_2)$ of vertex-disjoint paths in $G$ of length 15 such that $P_1$ is an $x_1,y_1$-path, $P_2$ is an $x_2,y_2$-path, they have the same colours in the same order, all but their middle edges have colour in $D_1$ and are in $E$, and their middle edges have colour in $D_2$.
\end{lemma}
\begin{proof} Let $\mu$ satisfy $1/n\llpoly \mu \llpoly p_1$.
Let $H$ be the graph of edges in $E$ with colour in $D_1$, and let $\mathcal{H}$ be the set of possibilities for $H$. Let $\mathcal{E}$ be the set of pairs $(e,f)$ of edges of the (uncoloured) complete bipartite graph which do not appear (coloured) in $H$ such that there are at most $(1-\mu)p_1^{42}n^{10}$ pairs $(P_1,P_2)$ of vertex-disjoint paths in $H+e+f$ of length 15 such that $P_1$ is an $x_1,y_1$-path with middle edge $e$, $P_2$ is an $x_2,y_2$-path with middle edge $f$, and $P_1$ and $P_2$ have, apart from the middle edges ($e$ and $f$), the same colours in the same order.
Let $\mathcal{A}$ be the event that there are at least $\eps p_2n^3/2$ pairs $(e,f)\in \mathcal{E}$ for which $e$ and $f$ have the same colour and this colour is in $D_2$.
Let $\mathcal{B}$ be the event that there are at most $2pn$ edges of each colour in $E$ and $|\mathcal{E}|\leq \mu n^4$.

We will show the following claim.

\begin{claim}\label{clm:AH2} For each $\hat{H}\in \mathcal{H}$, $\P(\mathcal{A}\land\mathcal{B}|H=\hat{H})=n^{-\omega(1)}$.
\end{claim}
Given this claim, as $\P(\mathcal{B})=1-n^{-\omega(1)}$ by Corollary~\ref{cor:Lzerostuff} and Lemma~\ref{chernoff}, following the same reasoning as for \eqref{eqn:forlaterreference}, we have $\P(\mathcal{A})= n^{-\omega(1)}$. Furthermore, if $\mathcal{A}$ does not occur, then
the number of pairs $(P_1,P_2)$ of vertex-disjoint paths in $G$ of length 15 such that $P_1$ is an $x_1,y_1$-path, $P_2$ is an $x_2,y_2$-path, they have the same colours in the same order, all but their middle edges have colour in $D_1$ and are in $E$, and their middle edges have colour in $D_2$, is at least
\[
\left(|D_2|\cdot n(n-1)-\frac{\eps p_2n^3}{2}\right)\cdot (1-\mu)p_1^{42}n^{10}\geq (1-\eps)p_1^{42}p_2n^{13},
\]
as required. Therefore, it is left only to prove Claim~\ref{clm:AH2}.

\smallskip

\noindent\emph{Proof of Claim~\ref{clm:AH2}.} Let $\mathcal{H}'$ be the set of $\hat{H}\in\cH$ for which $\mathcal{B}$ holds whenever $H=\hat{H}$. Note that if $\hat{H}\in \cH\setminus \cH'$, then $\P(\mathcal{A}\land\mathcal{B}|H=\hat{H})\leq \P(\mathcal{B}|H=\hat{H})=0$, so Claim~\ref{clm:AH2} holds trivially in this case. Therefore, let $\hat{H}\in \cH'$.

Let $p=p_1+p_2$, $D=D_1\cup D_2$ and $\kappa=p_1 n^2$. Let $\mathcal{F}_{\hat{H}}$ be the set of subgraphs $F$ such that $F$ has exactly two edges, $e$ and $f$ say, which both have the same colour, which is in $D_2$, and are such that $(e,f)\in \mathcal{E}$ or $(f,e)\in \mathcal{E}$. As $\mathcal{B}$ holds when $H=\hat{H}$, and there are $n$ colours, we have $|\mathcal{F}_{\hat{H}}|\leq 2\mu n^5$.

Let $\mathcal{S}_{\hat{H}}$ be the set of $(F_1,\ldots,F_{\kappa})$ of sequences of length $\kappa$ of edge-disjoint subgraphs from $\mathcal{F}_{\hat{H}}$ for which each colour appears on $\bigcup_{i\in [\kappa]}F_i$ at most $p n$ times, so that, if $\mathcal{B}$ holds, then
\begin{equation}\label{eqn:sizeS1}
|\mathcal{S}_{\hat{H}}|\leq (2\mu)^{\kappa} n^{5\kappa}.
\end{equation}
For each $S=(F_1,\ldots,F_{\kappa})\in \mathcal{S}_{\hat{H}}$, let $H_S=\cup_{i\in[\kappa]}F_i$, so that $e(H_S)=2\kappa$, each colour appears on $H_S$ at most $pn$ times, and all of the colours of $H_S$ are in $D_2$.
Let $Z_{\hat{H}}$ be the number of
$S\in \mathcal{S}_{\hat{H}}$ with $H_S\subset G$. Now, observe that if $\mathcal{A}$ and $\mathcal{B}$ hold and $H=\hat{H}$, then $Z_{\hat{H}}\geq (\eps n^3/4)^\kappa$.  Indeed, first note simply that any edge $e\in E(G|_{[n]\setminus D_1})$ is in at most $n$ graphs in $\mathcal{F}_{\hat{H}}$ which are a subgraph of $G$ as such subgraphs are pairs of edges with the same colour, and thus, moreover, every colour appears on at most $n^2$ subgraphs in $\mathcal{F}_{\hat{H}}$. Therefore,
if $\mathcal{A}$ and $\mathcal{B}$ hold, then we can pick a sequence $(F_1,\ldots,F_\kappa)$ of edge-disjoint subgraphs from $\mathcal{F}_{\hat{H}}$ by picking each $F_i$, $1\leq i\leq\kappa$, in turn, where at the selection of each $F_i$, $i\in [\kappa]$,
as $\mathcal{A}$ and $\mathcal{B}$ hold and there will be $2(i-1)$ edges we wish to avoid and at most $2(i-1)\cdot (pn/2)^{-1}$ colours, the number of possibilities for the choice of $F_i$ will be at least
\[
\frac{\eps n^3}{2}-2\kappa\cdot n - \frac{2\kappa}{pn/2}\cdot n^2\geq  \frac{\eps n^3}{2}-\frac{8\kappa n}{p} \geq \frac{\eps n^3}{4},
\]
as $\kappa=p_1 n^2$ and $p_1\llpoly p,\eps$.

For each $S\in \mathcal{S}_{\hat{H}}$, we have, as $\mathcal{B}$ holds if $H=\hat{H}$, that $H_S\cup \hat{H}$ has at most $pn+pn\leq 2pn$ edges of each colour. Furthermore, the edges of $H_S\cup \hat{H}$ only have colour in $D=D_1\cup D_2$, a set of size $pn$.
Thus,  for every
$S\in \mathcal{S}_{\hat{H}}$ such that $H_S\cup \hat{H}$ is properly coloured,
similarly to \eqref{eqn:longversion}, using Corollary~\ref{cor:latinsquareprobabilities} twice and that $p\kappa\geq n\log^2n$, we have that
\begin{align}
\P(H_S\subset G|H'=\hat{H})&=\frac{\P((H_S\cup \hat{H})\subset G)}{\P(H'=\hat{H})}=\frac{\left(\frac{1+O(p)}{n}\right)^{e(\hat{H})+2\kappa}}{\left(\frac{1+O(p)}{n}\right)^{e(\hat{H})}}= (1+O(p))^{2e(\hat{H})+2\kappa}n^{-2\kappa}\nonumber\\
&= e^{O(p\kappa)}n^{-2\kappa},\label{eqn:HSGprime2}
\end{align}
where we have used that $e(\hat{H})=p_1n^2=\kappa$.
Thus, as this holds for every $S\in \mathcal{S}_{\hat{H}}$ such that $H_S\cup \hat{H}$ is properly coloured,
\begin{align*}
\E(Z_{\hat{H}}|H=\hat{H})\leq |\mathcal{S}_{\mathcal{H}}|\cdot e^{O(p\kappa)}n^{-2\kappa}
\overset{\eqref{eqn:sizeS1}}{\leq} e^{O(p\kappa)}\cdot (2\mu)^{\kappa}n^{3\kappa}
=(\eps n^3/4)^{\kappa}\cdot n^{-\omega(1)},
\end{align*}
where we have used that $1/n\llpoly \mu \llpoly \eps$.
Then, by Markov's inequality, we have
\[
\P(\mathcal{A}\land\mathcal{B}|H=\hat{H})\leq \P(Z_{\hat{H}}\geq (\eps n^3/4)^{\kappa})\leq \frac{\E(Z_{\hat{H}}|H=\hat{H})}{(\eps n^3/4)^{\kappa}}
=n^{-\omega(1)}.
\]
This completes the proof of the claim, and hence the lemma.\hspace{5.2cm}$\boxdot$
\end{proof}

Using Corollary~\ref{cor:fewlengthsevensame} and Lemma~\ref{lemma:mainLlinksfirstlower}, it is now short work to deduce Lemma~\ref{lemma:mainLlinks}.

\begin{proof}[Proof of Lemma~\ref{lemma:mainLlinks}] Let $\mathcal{P}$ be the set of pairs $(P_1,P_2)$ of vertex-disjoint rainbow paths in $G$ of length 15 with the same colours in the same order, such that $P_1$ is an $x_1,y_1$-path and $P_2$ is an $x_2,y_2$-path.
By Corollary~\ref{cor:fewlengthsevensame}, with probability $1-n^{-\omega(1)}$, for any distinct $x'_1,x'_2,y'_1,y'_2\in V(G)$ with $x'_1\nsimAB y'_1$ and $x'_2\nsimAB y'_2$, there are at most $(1+\eps)n^5$ pairs $(P_1,P_2)$ of vertex-disjoint paths in $G$ of length seven such that $P_1$ is an $x'_1,y'_1$-path, $P_2$ is an $x'_2,y'_2$-path and they have the same colours in the same order.
Note that, for any distinct $x_1,x_2,y_1,y_2\in V(G)$ with $x_1\nsimAB y_1$ and $x_2\nsimAB y_2$, there are at most $n^{8}$ pairs $(P_1,P_2)$ paths of length 8 which start at $x_1$ and $x_2$ respectively and have the same colours in the same order. Therefore, applying the property of paths of length 7 to the other ends of these paths with $y_1$ or $y_2$ as appropriate, we get that there are at most $(1+\eps)n^{13}$ pairs in $\mathcal{P}$.

Furthermore, given any edge $e\in E(G)-\{x_1,x_2,y_1,y_2\}$, and any $2\leq k\leq 8$, there are at most $n^6$ pairs $(P_1,P_2)$ paths of length 8 which start at $x_1$ and $x_2$ respectively and have the same colours in the same order and where $e$ is the $k$ edge of the path from $x_1$. Thus, arguing as above, there are $O(n^{11})$ pairs $(P_1,P_2)\in \mathcal{P}$ in which $e$ is the $k$th edge of $P_1$ from $x_1$. Working similarly, we have that there are altogether $O(n^{11})$ pairs $(P_1,P_2)\in \mathcal{P}$ in which $e\in E(P_1\cup P_2)$.
If $e\in E(G)$ contains a vertex in $\{x_1,x_2,y_1,y_2\}$, then arguing similarly to before, we have that there are at most $O(n^{12})$ pairs $(P_1,P_2)\in \mathcal{P}$ with $e\in E(P_1\cup P_2)$.

Let $p_1$ and $p_2$ satisfy $1/n \llpoly p_1\llpoly p_2\llpoly \eps$. Let $\hat{D}_1,\hat{D}_2\subset C$ be disjoint random sets of colours such that each colour $c\in C$ is independently added to $\hat{D}_0$ with probability $(1-\eps^{2})p_1$ and to $\hat{D}_2$ with probability $(1-\eps^{2})p_2$.
Let $D_1,D_2$ be disjoint random sets of colours such that if $|\hat{D}_1|\leq pn$ and $|\hat{D}_1|\leq p_2n$, $D_1$ and $D_2$ are chosen uniformly at random subject to $\hat{D}_1\subset D_1\subset C$, $\hat{D}_2\subset D_2\subset C$, $|D_1|=p_1n$, $|D_2|=p_2n$, and $D_1$ and $D_2$ are disjoint, and, otherwise  $D_1$ and $D_2$ are chosen uniformly at random subject to  $D_1,D_2\subset C$, $|D_1|=p_1n$, $|D_2|=p_2n$, and $D_1$ and $D_2$ are disjoint. Note that, by Lemma~\ref{chernoff}, with probability $1-n^{-\omega(1)}$, we have $\hat{D}_1\subset D_1$ and $\hat{D}_2\subset D_2$. Let $E\subset E(G)$ be a formed by including each edge independently at random with probability $p_1$.

Let $\mathcal{P}'\subset \mathcal{P}$ be the set of $(P_1,P_2)\in \mathcal{P}$ for which the colours of $P_1$ not on the middle edge are all in $D_1$ and in $E$, and whose middle edge has colour in $D_2$. By Lemma~\ref{lemma:mainLlinksfirstlower}, with probability $1-n^{-\omega(1)}$, we have $|\mathcal{P}'|\geq (1-\eps/3)p^{42}_1p_2n^{13}$. On the other hand, if
$|\mathcal{P}|\leq (1-\eps)n^{13}$, then, by an application of Lemma~\ref{lem:mcdiarmidchangingc} similar to \eqref{eq:mcdexample},
$|\mathcal{P}'|\leq (1-2\eps/3)p^{42}_1p_2n^{13}$ with probability $1-n^{-\omega(1)}$. Therefore, we must have that, with probability $1-n^{-\omega(1)}$, $|\mathcal{P}|\geq (1-\eps)n^{13}$, as claimed.
\end{proof}


\subsection{$L$-links: proof of Theorem~\ref{thm:Llinks}}\label{sec:proofmainlinkthm}

Finally in this section, we put our work together to prove Theorem~\ref{thm:Llinks}.

\begin{proof}[Proof of Theorem~\ref{thm:Llinks}] Note that we can assume that $\eps\ll 1$. By Corollary~\ref{cor:fewlengthsevensame} and Lemma~\ref{lemma:mainLlinks}, with probability $1-n^{-\omega(1)}$, we can assume that the following hold.
\stepcounter{propcounter}
\begin{enumerate}[label = {{\textbf{\Alph{propcounter}\arabic{enumi}}}}]
\item \label{prop:7bounds} For every distinct $x_1,x_2,y_1,y_2\in V(G)$ with $x_1\nsimAB y_1$ and $x_2\nsimAB y_2$, there are at most $(1+\eps/8)n^{5}$ pairs $(P_1,P_2)$ of vertex-disjoint paths in $G$ of
length 7 such that $P_1$ is an $x_1,y_1$-path, $P_2$ is an $x_2,y_2$-path, and they have the same colours in the same order.
\item \label{prop:15bounds} For every distinct  $x_1,x_2,y_1,y_2\in V(G)$ with $x_1\nsimAB y_1$ and $x_2\nsimAB y_2$, there are $(1\pm \eps/8)n^{13}$
pairs $(P_1,P_2)$ of vertex-disjoint paths in $G$ of length 15 such that $P_1$ is an $x_1,y_1$-path, $P_2$ is an $x_2,y_2$-path, and they have the same colours in the same order.
\end{enumerate}

Note that \ref{prop:7bounds} and \ref{prop:15bounds} easily give the following.

\begin{enumerate}[label = {{\textbf{\Alph{propcounter}\arabic{enumi}}}}]\addtocounter{enumi}{2}
\item \label{prop:15bounds:robust} For every distinct $x_1,x_2,y_1,y_2\in V(G)$ with $x_1\nsimAB y_1$ and $x_2\nsimAB y_2$, and any set $U\subset V(G)\setminus \{x_1,x_2,y_1,y_2\}$ with $|U|\leq 100$, there are $(1\pm \eps/4)n^{13}$
pairs $(P_1,P_2)$ of vertex-disjoint paths in $G-U$ of length 15 such that $P_1$ is a $x_1,y_1$-path, $P_2$ is a $x_2,y_2$-path, and they have the same colours in the same order.
\end{enumerate}

We can now show that \ref{prop:links:totalnumber}--\ref{prop:links:throughedgeandvertex} hold. We will first show that \ref{prop:links:throughedge} holds with $\eps$ replaced by $\eps/2$. Let then $k$ satisfy $2\leq k\leq 61$ and let $u,v,x,y\in V(G)$ be distinct with $u\sim_{A/B} v$ and $xy\in E(G)$. Note that, by swapping $u$ and $v$ if necessary, we can assume that $k\leq 31$. Let $H$ be a $(u,v,L)$-link in $G$ with $xy$ as its $k$th edge. Note that which of $A$ and $B$ $x$ (and thus $y$) is in, determines which vertex in the link is $x$ and which is $y$, which are the $k$th and $(k+1)th$ vertex in some order. Suppose first that $16\leq k\leq 31$. Note that we have $(1\pm \eps/6)n^{15}$ choices for picking the other edges which are the $k'$th link for $16\leq k'\leq 31$ with $k'\neq k$, which determines the 16th to 31st colour (in order) of the link, and thus, working backwards from $v$, the 47th to 62th edge of the link. Then, by \ref{prop:15bounds:robust}, there are $(1\pm \eps/4)n^{13}$ ways to complete the link, giving $(1\pm \eps/4)n^{13}\cdot (1\pm \eps/6)n^{15}=(1\pm \eps/2)\Phi_0\cdot n^{-2}$ $(u,v,L)$-links with $xy$ as the $k$th edge.

Suppose then that $2\leq k\leq 15$. Note that we have $(1\pm \eps/6)n^{14}$ choices for picking the other edges which are the $k'$th link for $2\leq k'\leq 16$ with $k'\neq k$, which also then determines the first edge of the link and its 1st to 16th colour (in order) of the link. There are then $n-O(1)$ choices for a path of length 16 with these colours in the same order and which does not use any of the known vertices so far and so that the start vertex is in the opposite vertex class $A$ or $B$ to $u$ and $v$, and thus we can use this path to give us the 32nd to 47th edge of the link. Finally, by \ref{prop:15bounds:robust}, there are at most $(1\pm \eps/4)n^{13}$ ways to complete the link, giving $(1\pm \eps/4)n^{13}\cdot (1\pm \eps/6)n^{14}\cdot (n-O(1))=(1\pm \eps/2)\Phi_0\cdot n^{-2}$ $(u,v,L)$-links with $xy$ as the $k$th edge.

This completes the proof of \ref{prop:links:throughedge} with $\eps$ replaced by $\eps/2$.
Now, note that \ref{prop:links:totalnumber}--\ref{prop:links:throughcolour} are easily implied by \ref{prop:links:throughedge} with $\eps$ replaced by $\eps/2$. Indeed, for \ref{prop:links:totalnumber}, suppose $u,v\in V(G)$ are distinct with $u\sim_{A/B}v$. There are $(n-1)^2$ choices for an edge $e$ in $G$ which does not contain $u$ or $v$, and then $(1\pm \eps/2)\Phi_0n^{-2}$ choices for an $(u,v,L)$-link in $G$ where $e$ is the 2nd edge (by \ref{prop:links:throughedge} with $\eps$ replaced by $\eps/2$).
Thus, in total there are $(n-1)^2\cdot (1\pm \eps/2)\Phi_0n^{-2}=(1\pm \eps)\Phi_0$ $(u,v,L)$-links in $G$.

Similarly, for \ref{prop:links:throughvertex}, let $k$ satisfy $2\leq k\leq 62$ and let $u,v\in V(G)$ with $u\sim_{A/B} v$ be distinct. Note that, by swapping $u$ and $v$ if necessary, we can assume that $k\leq 32$. Let $x\in V(G)\setminus \{u,v\}$ with $x\not\sim_{A/B}u,v$ if $k$ is even and $x\sim_{A/B}u,v$ if $k$ is odd.
Choose an edge $e\in E(G)$ containing $x$ but not $u$ or $v$, noting there are either $n$ or $(n-2)$ choices for $e$. Then, applying \ref{prop:links:throughedge} with $\eps$ replaced by $\eps/2$ with $e$ and $k$, we have that the number of  $(u,v,L)$-links in $G$ in which $x$ is the $k$th vertex is $(n\pm 2)\cdot (1\pm \eps/2)\cdot \Phi_0n^{-2}=(1\pm \eps)\cdot \Phi_0\cdot  n^{-1}$, as required.

For~\ref{prop:links:throughcolour}, let $u,v\in V(G)$ be distinct with $u\sim_{A/B} v$, and let $c\in C$ and $k\in [62]$. If $k=1$ or $k=62$, then note that $(u,v,L)$-links in $G$ in which the $k$th edge has colour $c$ are exactly those that contain the $c$-neighbour in $G$ of
 $u$ or of $v$ as the 2nd or 62nd vertex of the link, respectively, and thus the result follows from \ref{prop:links:throughvertex}. If $2\leq k\leq 61$, then there are $(n-2)$ choices for an edge $e$ in $G$ of colour $c$ not containing $u$ or $v$, so that, again by \ref{prop:links:throughedge} with $\eps$ replaced by $\eps/2$ with $e$ and $k$, we have that the number of  $(u,v,L)$-links in $G$ in which the $k$th edge has colour $c$ is $(n-2)\cdot (1\pm \eps/2)\cdot \Phi_0n^{-2}=(1\pm \eps)\cdot \Phi_0\cdot  n^{-1}$, as required.

\smallskip

\noindent\ref{prop:links:cod:twovertices}:
For \ref{prop:links:cod:twovertices}, let $u,v,x,y\in V(G)$ be distinct. Let $H$ be a $(u,v,L)$-link in $G$ containing $x$ and $y$.
Firstly, there are at most $61\cdot 60$ choices for distinct $2\leq k_x,k_y\leq 62$ which determine, respectively, the position of $x$ and $y$ in the link. Then, partition $H$ into paths $P_1,P_2,P_3,Q_1,Q_2,Q_3$ (some possibly with length 0) such that $H=P_1P_2P_3Q_1Q_2Q_3$, $P_2$ and $Q_2$ have length 7, for each $i\in \{1,3\}$ $P_i$ and $Q_i$ have the same length, and neither $x$ or $y$ is an internal vertex of $P_2$ or $Q_2$.

If both $x$ and $y$ are in $P_3\cup Q_1$, then note that $(P_1,P_3,Q_1,Q_3)$ is determined by $V(P_3\cup Q_1)\setminus \{x,y\}$, so there are at most $n^{\ell(P_3)+\ell(Q_1)+1-2}=n^{31-\ell(P_2)-1}=n^{23}$ choices for $(P_1,P_3,Q_1,Q_3)$. If $x$ is in $P_3\cup Q_1$ and $y$ is in $P_1$, then $(P_1,P_3,Q_1,Q_3)$ is determined by the choices of the vertices $V(P_1)\setminus \{u,y\}$ and $V(Q_3)\setminus \{v\}$, so there are at most $n^{\ell(P_3)+\ell(Q_3)+2-3}=n^{23}$ choices for $(P_1,P_3,Q_1,Q_3)$. If both $x$ and $y$ are in $P_1\cup Q_3$, then $(P_1,P_3,Q_1,Q_3)$ is determined by the choices of the vertices in $V(P_1)\cup V(Q_3)\setminus \{u,v,x,y\}$ and the 31th vertex of $H$, for at most $n^{\ell(P_1)+1+\ell(Q_3)+1-4+1}=n^{23}$ choices for $(P_1,P_3,Q_1,Q_3)$.
Therefore,  after $k_x,k_y$ are chosen, there are at most $n^{23}$ choices for $(P_1,P_3,Q_1,Q_3)$. Thus, using
\ref{prop:7bounds}, the total number of $(u,v,L)$-links containing $x$ and $y$ is at most $61\cdot 62 \cdot (1+\eps/8)n^{5}\cdot n^{23}\leq 10^4\Phi_0\cdot n^{-2}$,
completing the proof of \ref{prop:links:cod:twovertices}.

\smallskip

\noindent\ref{prop:links:cod:1vertex1colour}: Let $u,v,x\in V(G)$ be distinct and let $c\in C$. Let $H$ be a $(u,v,L)$-link in $G$ containing $x$ and using the colour $c$, which does not contain $ux$ or $vx$ if this is a colour-$c$ edge in $G$.
First, there are at most $61\cdot 31$ choices for $2\leq k_x\leq 62$ and $k\in [31]$. Having chosen such $k_x$ and $k$, we count the choices for $H$ with $x$ as the $k_x$th vertex and $c$ as the $k$th colour.
Partition $H$ into paths $P_1,P_2,P_3,Q_1,Q_2,Q_3$ (some possibly with length 0) such that $H=P_1P_2P_3Q_1Q_2Q_3$, $P_2$ and $Q_2$ have length 7, for each $i\in \{1,3\}$ $P_i$ and $Q_i$ have the same length, and $x$ is not an internal vertex of $P_2$ or $Q_2$, and the colour $c$ is not used on $P_2$ (and hence either $Q_2$).

If $x$ is in $V(P_1)$, then $(P_1,Q_3)$ can be determined by choosing all the other colours in $(C(P_1)\cup C(Q_3))\setminus \{c\}$ except for the colour just before $x$ in $H$, or the colour before that if that edge has colour $c$ (which exists as if $ux\in E(H)$ then this is not a colour-$c$ edge).
Therefore, there are at most $n^{\ell(P_1)+\ell(Q_3)-2}$ choices for $(P_1,Q_3)$. As $(P_3,Q_1)$ is then determined by the choice of the 32nd vertex (for example) of $H$ there are at most $n$ choices for $(P_3,Q_1)$, so there are at most $n^{\ell(P_1)+\ell(Q_3)-1}=n^{31-7-1}=n^{23}$ choices for $(P_1,P_3,Q_1,Q_3)$ in this case. Similarly, if $x$ is in $V(Q_3)$, then there are at most $n^{23}$ choices for $(P_1,P_3,Q_1,Q_3)$.
If $x$ is in $V(P_3)\cup V(Q_1)$, then $(P_1,P_3,Q_1,Q_3)$ is determined by choosing the colours in $(C(P_1)\cup C(Q_3))\setminus \{c\}$, so there are at most $n^{\ell(P_1)+\ell(Q_3)-1}=n^{23}$ choices for $(P_1,P_3,Q_1,Q_3)$ in this case as well.

Therefore, counting the possibilities for $(k_x,k)$ and putting this together with \ref{prop:7bounds} for the number of choices for $(P_2,Q_2)$, the total number of  $(u,v,L)$-links in $G$ using $x$ and $c$ in which there is not a colour-$c$ edge $ux$ or $xv$ is at most
$61\cdot 31 \cdot (1+\eps/8)n^{5}\cdot n^{23}\leq 10^4\Phi_0\cdot n^{-2}$,
as required.

\smallskip

\noindent\ref{prop:links:cod:twocolours}:
Let $u,v\in V(G)$ be distinct and let $c,d\in C$ be distinct. Let $H$ be a $(u,v,L)$-link in $G$ using $c$ and $d$. Choose $k_c,k_d\in [31]$ so that $c$ is the $k_c$th colour of $H$ and $d$ is the $k_d$th colour (with $31\cdot 30$ choices). Partition $H$ into paths $P_1,P_2,P_3,Q_1,Q_2,Q_3$ (some possibly with length 0) such that $H=P_1P_2P_3Q_1Q_2Q_3$, $P_2$ and $Q_2$ have length 7, for each $i\in \{1,3\}$ $P_i$ and $Q_i$ have the same length, and neither colour $c$ nor $d$ is used on $P_2$ (and hence neither is used on $Q_2$). Then, $(P_1,P_3,Q_1,Q_3)$ is determined by
choosing the colours in $(C(P_1)\cup C(P_3))\setminus \{c,d\}$ and (for example) the 31th vertex of $H$, so, in total, the number of choices for $(P_1,P_3,Q_1,Q_3)$ is at most $31\cdot 30\cdot n^{31-\ell(P_2)-2}\cdot n\leq 10^3n^{23}$. Therefore, using  \ref{prop:7bounds} for the number of choices for $(P_2,Q_2)$, the number of $(u,v,L)$-links in $G$ using $c$ and $d$ is at most
$10^3\cdot n^{23}\cdot (1+\eps/8)n^{5}\leq 10^4\Phi_0\cdot n^{-2}$, as required.

 \smallskip

 \noindent\ref{prop:links:2edgesofdifferentcoloursanddisjoint}: Let $u,v\in V(G)$ be distinct and let $e,e'\in E(G-\{u,v\})$ have different colours and share no vertices. Let $H$ be a $(u,v,L)$-link in $G$ which contains $e$ and $e'$. Choose $2\leq k_e,k_{e'}\leq 61$ such that $e$ is the $k_e$th edge and $e'$ is the $k_{e'}$th edge of $H$, noting there are at most $60\cdot 59$ choices for $(k_e,k_{e'})$.
 Partition $H$ into paths $P_1,P_2,P_3,Q_1,Q_2,Q_3$ (some possibly with length 0) for which we have $H=P_1P_2P_3Q_1Q_2Q_3$, $P_2$ and $Q_2$ have length 7, for each $i\in \{1,3\}$ $P_i$ and $Q_i$ have the same length, and neither $e$ or $e'$ are used on $P_2$ (and hence either $Q_2$). Assume, without loss of generality, that $e\in E(P_1\cup P_3)$. By looking at different cases, we will show that
 (having chosen $(k_e,k_{e'})$) there are always at most $n^{21}$ choices for $(P_1,P_3,Q_1,Q_3)$.

 If $e$ and $e'$ appear together on $P_3\cup Q_1$, then there are at most $n^{\ell(P_3\cup Q_1)+1-4}=n^{31-\ell(P_2)-3}=n^{21}$ choices for the rest of $P_3\cup Q_1$, which then determines $(P_1,P_3,Q_1,Q_3)$. If $e\in E(P_1)$ and $e'\in E(P_3)$, or vice versa, then there are  at most $n^{\ell(P_1)+1-3}$ choices for the rest of $P_1$ and at most $n^{\ell(P_3)+1-2}$ choices for the rest of $P_3$, for at most $n^{\ell(P_1)+\ell(P_3)-3}=n^{21}$ choices in total, which then determines $(P_1,P_3,Q_1,Q_3)$
If $e\in E(P_1)$ and $e'\in E(Q_3)$, then there are at most $n^{\ell(P_1)+1-3}\cdot n^{\ell(Q_3)+1-3}=n^{21}$ choices for $(P_1,Q_3)$, after which there are at most $n$ choices for $(P_1,P_3,Q_1,Q_3)$, for at most $n^{\ell(P_1)+\ell(P_3)-3}$ choices in total.
If $e,e'\in E(P_1)$, then, as $e$ and $e'$ share no vertices,  there are at most $n^{\ell(P_1)+1-5}$ choices for the rest of $P_1$, after which there are at most $n$ choices for $Q_1$, and at most $n^{\ell(P_3)}$ choices for the colours of $P_3$ in order, which then determines $(P_1,P_3,Q_1,Q_3)$, for at most $n^{\ell(P_1)+\ell(P_3)-3}$ choices in total.
If $e\in E(P_1)$ and $e'\in E(Q_1)$, then, as $e$ and $e'$ are different colours, there are at most $n^{\ell(P_1)-3}$ choices for the rest of $P_1$ (using the colour $c(e')$ in the appropriate place), which then determines $Q_1$, after which there are  at most $n^{\ell(P_3)}$ choices for the colours of $P_3$ in order, which then determines $(P_1,P_3,Q_1,Q_3)$, for at most $n^{\ell(P_1)+\ell(P_3)-3}=n^{21}$ choices in total.
If $e,e'\in E(P_3)$, then, there are at most $n^{\ell(P_3)-3}$ choices for the rest of $P_3$, which determines $Q_3$, and after which there are at most $n^{\ell(P_3)}$ choices for the colours of $P_1$ in order, which then determines $(P_1,P_3,Q_1,Q_3)$, for at most $n^{\ell(P_1)+\ell(P_3)-3}=n^{21}$ choices in total. Thus, there are always at most $n^{21}$ choices for $(P_1,P_3,Q_1,Q_3)$.

Therefore, using  \ref{prop:7bounds} for the number of choices for $(P_2,Q_2)$, the number of $(u,v,L)$-links in $G$ containing $e$ and $e'$ is at most $60\cdot 59\cdot n^{21}\cdot (1+\eps/8)n^{5}\leq 10^4\Phi_0\cdot n^{-4}$,  as required.

   \smallskip

   \noindent\ref{prop:links:throughedgeandvertex}: Let $u,v,w\in V(G)$ be distinct and let $e\in E(G-\{u,v,w\})$.  Let $H$ be a $(u,v,L)$-link in $G$ which contains $w$ and $e$.
    Choose $2\leq k_w\leq 62$ and $2\leq k_{e}\leq 61$ such that $w$ is the $k_w$th vertex of $H$ and $e$ is the $k_e$th edge of $H$, noting there are at most $61\cdot 60$ choices for $(k_w,k_{e})$.
Assume that $k_w\leq 32$, where the case where $k_w>32$ follows similarly. Note that for all except $(k_w,k_e)=(2,3)$ and $(k_w,k_e)=(2,33)$, $w$ is contained in an edge of $H-(\{u,v\}\cup V(e))$ with colour not the same as the colour of $e$. As there are at most $n$ choices for such an edge, by \ref{prop:links:2edgesofdifferentcoloursanddisjoint}, there are at most $10^4\Phi_0\cdot n^{-3}$ choices of $H$ if $(k_w,k_e)\neq (2,3),(2,33)$.

Now, partition $H$ into paths $P_1,P_2,Q_1,Q_2$ for which we have $H=P_1P_2Q_1Q_2$, $P_1$ and $Q_1$ have length $24$ and $P_2$ and $Q_2$ have length 7.
If $(k_w,k_e)=(2,3)$, then, as the first 3 interior vertices of $P_1$ are known, there are at most $n^{21}$ choices for the rest of $P_1$, after which there are at most $n$ choices for $Q_1$, and then, using \ref{prop:7bounds} for the number of choices for $(P_2,Q_2)$, in total there are at most $(1+\eps/8)n^5\cdot n^{21}\cdot n\leq 2n^{27}=2\Phi_0\cdot n^{-3}$ choices for $H$.
If $(k_w,k_e)=(2,33)$, then,  as the first 2 interior vertices of $P_1$ are known, there are at most $n^{22}$ choices for the rest of $P_1$, after which  $Q_1$ is known, and then, using \ref{prop:7bounds} for the number of choices for $(P_2,Q_2)$, in total there are at most $(1+\eps/8)n^5\cdot n^{22}\leq 2n^{27}=2\Phi_0\cdot n^{-3}$ choices for $H$.
Over all the different choices of $k_w$ and $k_e$, we have that the number of $(u,v,L)$-links in $G$ containing $w$ and $e$ is at most $10^8\Phi_0\cdot n^{-3}$, as required.
\end{proof}

\fi

%% file: 6realisation.tex
In this section, we prove Lemma~\ref{keylemma:realisation}. We begin by giving a sketch of the proof in Section~\ref{sec:Bsketch}, after which we outline the rest of this section.


\subsection{Overview of Part~\ref{partB}}\label{sec:Bsketch}
To recap, for each tribe $\tau\in \mathcal{T}$, from Part~\ref{partA} (in particular, from Lemma~\ref{keylemma:absorption}), we will get a collection
\[
\mathcal{I}_\tau\subset \{\{(i,u),(j,v)\}: i,j\in I_\tau,i\neq j, u \in  S_i \setminus (R_i \cup T_j), v \in S_j \setminus (T_i \cup R_j),u\neq v,u{\sim}_{A/B}v\}
\]
which represents a set of instructions. We wish to find, with high probability, near-matchings $\hat{M}_i$, $i\in I_\tau$, in $G\sim G^{\col}_{[n]}$ so that, for each $i\in I_\tau$ the vertices in $T_i$ and $R_i$ have degree $2$ and $0$ in $\hat{M}_i$ respectively, while the vertices in $V(G)\setminus (R_i\cup T_i)$ each have degree $0$ or $1$ in $\hat{M}_i$.
Furthermore, we wish to have the property that, for each $\{(i,u),(j,v)\}\in \mathcal{I}_\tau$, we can reduce the degree of $u$ by 1 and increase the degree of $v$ by 1 in $\hat{M}_i$ and increase the degree of $v$ by 1 and decrease the degree of $u$ by 1 in $\hat{M}_j$, doing so by switching a small number of edges between $\hat{M}_i$ and $\hat{M}_j$ so that a small number of edges are changed and the colours and vertices other than $u$ and $v$ appearing on $\hat{M}_i$ and $\hat{M}_j$ do not change. Crucially, we wish to be able to do this largely independently,
so that for any collection $\mathcal{C}\subset I_\tau$ satisfying \ref{prop:A:correct1}--\ref{prop:A:correct4} these alterations can be made simultaneously for each $\{(i,u),(j,v)\}\in \mathcal{C}$ without interfering with each other.

We split Part \ref{partB} into three sub-parts. Roughly speaking, in Part~\ref{partB1} we find matchings covering $T_i$ for each $i \in [n]$. In Part~\ref{partB2} we find, for each $i \in [n]$ and vertex $u \in S_i\setminus R_i$, a small monochromatic matching consisting of an edge containing $u$ and a small set of other edges $M_{i,u}$. Finally, in Part~\ref{partB3} we find certain even-length paths connecting edges which were found in Part~\ref{partB2}. Of course, each of these substructures is found subject to certain constraints which we describe in more detail below.

The first part, Part~\ref{partB1}, is the simplest and is largely independent of the others. For each near-matching $\hat{M}_i$, we find a rainbow matching that covers $T_i$, by matching each vertex in $T_i$ to a distinct vertex in $X_i$, with each edge of the matching using a distinct colour from $D_i$, so that when we later find another edge to add to $\hat{M}_i$ for each vertex in $T_i$ they will all have degree 2 in $\hat{M}_i$. This we will do using an application of Theorem~\ref{thm:nibble} to an auxiliary hypergraph as described later in this sketch.

For now, we move on to describe Parts~\ref{partB2} and~\ref{partB3}.
For each $\{(i,u),(j,v)\}\in \mathcal{I}_\tau$, our key mechanism to set up the switching property described above are the $L$-links we defined in Section~\ref{sec:rand}. Ideally, for each $\{(i,u),(j,v)\}\in \mathcal{I}_\tau$, we could find a $(u,v,L)$-link, $R$ say, and put the odd edges of $R$ into $\hat{M}_i$ and the even edges of $R$ into $\hat{M}_j$, so that switching these edges between $\hat{M}_i$ and $\hat{M}_j$ would exactly decrease the degree of $u$ in $\hat{M}_i$ by 1 and increase the degree of $v$ in $\hat{M}_i$ by 1 (and vice versa in $\hat{M}_j$) while making no other meaningful changes (i.e., the colours of the matchings and all the other vertex degrees would stay the same, and between them $\hat{M}_i$ and $\hat{M}_j$ would have the same edges). Unfortunately, this is not possible. The reason is that there will be multiple pairs from $\mathcal{I}_\tau$ containing $(i,u)$ so that if we do this for each such pair we will be adding multiple edges in $\hat{M}_i$ next to $u$, where (not counting those added in Part~\ref{partB1}), we want only~1.

Instead, for $i\in [n]$ and $u\in S_i\setminus R_i$, if $J_{i,u}$ is the set of $(j,v)$ for which $\{(i,u),(j,v)\}\in \mathcal{I}_\tau$, in Part~\ref{partB2}, we first find an edge $e_{u,i}$ from $u$ to $Y_i$ along with a monochromatic matching $M_{i,u}$ of the same colour as $e_{u,i}$ which also uses vertices in $Y_i$, with one edge $e_{u,i,j}$ for each $j\in J_{i,u}$ (as depicted on the left in Figure~\ref{fig:B2andB3}). The first edge ($ux_1$ say) we add to $\hat{M}_i$, while for each $j\in J_{i,u}$ we add $e_{u,i,j}$ 
to $\hat{M}_j$. 
Now, fix $j\in J_{i,u}$ and let $v$ be such that $\{(i,u),(j,v)\}\in \mathcal{I}_\tau$ (which will be unique due to \ref{prop:abs:nocodegree}) and suppose $e_{u,i,j}=x_2x_3$ is the edge assigned to $\hat{M}_j$ from $M_{i,u}$. We will have $i\in J_{j,v}$, so we will also have found an edge, $e_{v,j}=vy_1$, next to $v$ for $\hat{M}_j$ as well as an edge, $e_{v,j,i}=y_2y_3$, of the same colour in $M_{j,v}$ which is assigned to $\hat{M}_i$. Furthermore, we will do this 
so that $c(e_{u,i})\neq c(e_{v,j})$ and $\{u,x_1,x_2,x_3,v,y_1,y_2,y_3\}$ are all distinct.

\begin{center}
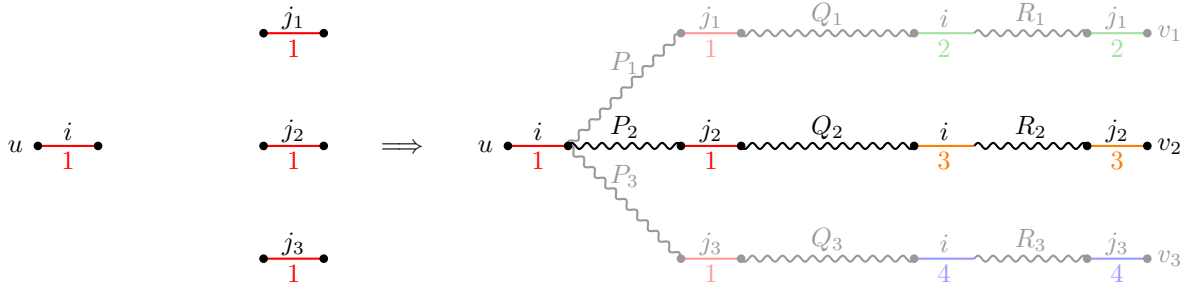
\begin{figure}[t]
\begin{center}\begin{tikzpicture}
\def\setsp{0.8}
\def\sepup{1.5}
\def\sepac{1.5}

\foreach \n in {1,2}
{
\coordinate (u\n) at ($(\n*\setsp,0)$);
}
\foreach \n in {1,2,3}
\foreach \m in {1,2,3}
{
\coordinate (w\n\m) at ($(u2)+(\sepac,0)+\n*(0,\sepup)+\m*(\setsp,0)-(3*\setsp,2*\sepup)$);
}

\foreach \n in {1,2,3}
\foreach \m in {1,2}
{
\coordinate (v\n\m) at ($(w23)+(\sepac,0)+\n*(0,\sepup)+\m*(\setsp,0)-(2*\setsp,2*\sepup)$);
}

\draw [red,thick] (u1) -- (u2);
\draw [red,thick] (v31) -- (v32);
\draw [red,thick] (v21) -- (v22);
\draw [red,thick] (v11) -- (v12);

\draw [red] ($0.5*(u1)+0.5*(u2)+(0,-0.2)$) node { $1$};
\draw [red] ($0.5*(v31)+0.5*(v32)+(0,-0.2)$) node { $1$};
\draw [red] ($0.5*(v11)+0.5*(v12)+(0,-0.2)$) node { $1$};
\draw [red] ($0.5*(v21)+0.5*(v22)+(0,-0.2)$) node { $1$};


\draw ($(u1)-(0.3,0)$) node {$u$};
\draw ($0.5*(v31)+0.5*(v32)+(0,0.2)$) node {$j_1$};
\draw ($0.5*(v21)+0.5*(v22)+(0,0.2)$) node {$j_2$};
\draw ($0.5*(v11)+0.5*(v12)+(0,0.2)$) node {$j_3$};
\draw ($0.5*(u1)+0.5*(u2)+(0,0.2)$) node {$i$};

\foreach \n in {1,2}
{
\draw [fill] (u\n) circle[radius=0.05];
}

\foreach \n in {1,2,3}
\foreach \m in {1,2}
{
\draw [fill] (v\n\m) circle[radius=0.05];
}

\end{tikzpicture}\;\;\;\;\;\begin{tikzpicture}
\draw [white](0,-1.9) -- (0,0);
\draw (0,0) node {$\implies$};
\end{tikzpicture}\;\;\;\;\begin{tikzpicture}
\def\setsp{0.8}
\def\sepup{1.5}
\def\sepac{1.5}

\foreach \n in {1,2}
{
\coordinate (u\n) at ($(\n*\setsp,0)$);
}
\foreach \n in {1,2,3}
\foreach \m in {1,2}
{
\coordinate (w\n\m) at ($(u2)+(\sepac,0)+\n*(0,\sepup)+\m*(\setsp,0)-(\setsp,2*\sepup)$);
}

\foreach \n in {1,2,3}
\foreach \m in {3,4}
{
\coordinate (w\n\m) at ($(u2)+2*(\sepac,0)+\n*(0,\sepup)+\m*(\setsp,0)-(\setsp,2*\sepup)$);
}

\foreach \n in {1,2,3}
\foreach \m in {1,2}
{
\coordinate (v\n\m) at ($(w24)+(\sepac,0)+\n*(0,\sepup)+\m*(\setsp,0)-(\setsp,2*\sepup)$);
}

\draw [red,thick] (u1) -- (u2);
\foreach \n in {1,3}
{
\draw [red!40,thick] (w\n1) -- (w\n2);
}
\foreach \n in {2}
{
\draw [red,thick] (w\n1) -- (w\n2);
}
\draw [red] ($0.5*(u1)+0.5*(u2)+(0,-0.2)$) node { $1$};
\draw [red!40] ($0.5*(w11)+0.5*(w12)+(0,-0.2)$) node { $1$};
\draw [red] ($0.5*(w21)+0.5*(w22)+(0,-0.2)$) node { $1$};
\draw [red!40] ($0.5*(w31)+0.5*(w32)+(0,-0.2)$) node { $1$};

\draw [darkgreen!40,thick] (w33) -- (w34);
\draw [orange,thick] (w23) -- (w24);
\draw [blue!40,thick] (w13) -- (w14);
\draw [darkgreen!40,thick] (v31) -- (v32);
\draw [orange,thick] (v21) -- (v22);
\draw [blue!40,thick] (v11) -- (v12);

\draw [snake=coil,segment aspect=0, segment amplitude=.4mm,segment length=2mm,thick,black!40] (u2) -- (w31);
\draw [snake=coil,segment aspect=0, segment amplitude=.4mm,segment length=2mm,thick] (u2) -- (w21);
\draw [snake=coil,segment aspect=0, segment amplitude=.4mm,segment length=2mm,thick,black!40] (u2) -- (w11);
\draw [snake=coil,segment aspect=0, segment amplitude=.4mm,segment length=2mm,thick,black!40] (w34) -- (v31);
\draw [snake=coil,segment aspect=0, segment amplitude=.4mm,segment length=2mm,thick] (w24) -- (v21);
\draw [snake=coil,segment aspect=0, segment amplitude=.4mm,segment length=2mm,thick,black!40] (w14) -- (v11);
\draw [snake=coil,segment aspect=0, segment amplitude=.4mm,segment length=2mm,thick,black!40] (w33) -- (w32);
\draw [snake=coil,segment aspect=0, segment amplitude=.4mm,segment length=2mm,thick] (w23) -- (w22);
\draw [snake=coil,segment aspect=0, segment amplitude=.4mm,segment length=2mm,thick,black!40] (w13) -- (w12);

\draw [darkgreen!40] ($0.5*(w33)+0.5*(w34)+(0,-0.2)$) node { $2$};
\draw [darkgreen!40] ($0.5*(v31)+0.5*(v32)+(0,-0.2)$) node { $2$};
\draw [blue!40] ($0.5*(w13)+0.5*(w14)+(0,-0.2)$) node { $4$};
\draw [blue!40] ($0.5*(v11)+0.5*(v12)+(0,-0.2)$) node { $4$};
\draw [orange] ($0.5*(w23)+0.5*(w24)+(0,-0.2)$) node { $3$};
\draw [orange] ($0.5*(v21)+0.5*(v22)+(0,-0.2)$) node { $3$};

\draw [black!40] ($0.5*(u2)+0.5*(w31)+(0,0.35)$) node { $P_1$};
\draw [black!40] ($0.5*(w33)+0.5*(w32)+(0,0.25)$) node { $Q_1$};
\draw ($0.5*(u2)+0.5*(w21)+(0,0.25)$) node { $P_2$};
\draw ($0.5*(w23)+0.5*(w22)+(0,0.25)$) node { $Q_2$};
\draw [black!40] ($0.5*(u2)+0.5*(w11)+(0,0.35)$) node { $P_3$};
\draw [black!40] ($0.5*(w13)+0.5*(w12)+(0,0.25)$) node { $Q_3$};
\draw [black!40] ($0.5*(w34)+0.5*(v31)+(0,0.25)$) node { $R_1$};
\draw ($0.5*(w24)+0.5*(v21)+(0,0.25)$) node { $R_2$};
\draw [black!40] ($0.5*(w14)+0.5*(v11)+(0,0.25)$) node { $R_3$};

\draw ($(u1)-(0.3,0)$) node {$u$};
\draw [black!40] ($(v32)+(0.3,0)$) node {$v_1$};
\draw ($(v22)+(0.3,0)$) node {$v_2$};
\draw [black!40] ($(v12)+(0.3,0)$) node {$v_3$};
\draw [black!40] ($0.5*(v31)+0.5*(v32)+(0,0.2)$) node {$j_1$};
\draw ($0.5*(v21)+0.5*(v22)+(0,0.2)$) node {$j_2$};
\draw [black!40] ($0.5*(v11)+0.5*(v12)+(0,0.2)$) node {$j_3$};
\draw [black!40] ($0.5*(w31)+0.5*(w32)+(0,0.2)$) node {$j_1$};
\draw ($0.5*(w21)+0.5*(w22)+(0,0.2)$) node {$j_2$};
\draw [black!40] ($0.5*(w11)+0.5*(w12)+(0,0.2)$) node {$j_3$};
\draw ($0.5*(u1)+0.5*(u2)+(0,0.2)$) node {$i$};
\draw [black!40] ($0.5*(w34)+0.5*(w33)+(0,0.2)$) node {$i$};
\draw ($0.5*(w24)+0.5*(w23)+(0,0.2)$) node {$i$};
\draw [black!40] ($0.5*(w14)+0.5*(w13)+(0,0.2)$) node {$i$};

\foreach \n in {1,2}
{
\draw [fill] (u\n) circle[radius=0.05];
}
\foreach \n in {2}
\foreach \m in {1,2,3}
{
\draw [fill] (w\n\m) circle[radius=0.05];
}

\foreach \n in {2}
\foreach \m in {1,2}
{
\draw [fill] (v\n\m) circle[radius=0.05];
}

\foreach \n in {1,3}
\foreach \m in {1,2,3}
{
\draw [fill,black!40] (w\n\m) circle[radius=0.05];
}

\foreach \n in {1,3}
\foreach \m in {1,2}
{
\draw [fill,black!40] (v\n\m) circle[radius=0.05];
}

\end{tikzpicture}
\end{center}
\caption{In Part~\ref{partB2}, as on the left, for each $i\in [n]$ and $u\in S_i\setminus R_i$ we find an edge from $u$ to $Y_i$ along with a monochromatic matching with the same colour, with one edge for each $j$ such that $\{(i,u),(j,v)\}\in \mathcal{I}$ (depicted for $\{(i,u),(j_1,v_1)\},\{(i,u),(j_2,v_2)\},\{(i,u),(j_3,v_3)\}\in \mathcal{I}$).
Then, in Part~\ref{partB3}, for each pair $\{(i,u),(j,v)\}\in \mathcal{I}$ we find $L$-links which connect together some of the edges found for $(i,u)$ and $(j,v)$, as depicted on the right for $(i,u)$ and $(j_2,v_2)$.
Each path $P_1,P_2,P_3,Q_1,Q_2,Q_3,R_1,R_2,R_3$ has length 62 and forms a link with the pattern $L$, where, for each $j\in [3]$, $P_j$, $Q_j$ and $R_j$ use a disjoint set of colours and new vertices coming from $Z_i$.}\label{fig:B2andB3}
\end{figure}
\end{center}

Now, if we find an $(x_1,x_2,L)$-link $P$, an $(x_3,y_3,L)$-link $Q$ and a $(y_2,y_1,L)$-link $R$ so that
\[
ux_1Px_2x_3Qy_3y_2Ry_1v
\]
is a path in $G$ (see the right of Figure~\ref{fig:B2andB3}), then having added the odd edges of this path to $\hat{M}_i$ and the even edges to $\hat{M}_j$, by switching between these edges we have the change we want in $\hat{M}_i$ and $\hat{M}_j$. Importantly, for each $j\in J_{i,u}$, this uses the same edge, $ux_1$, at $u$, allowing us to only add one edge at $u$ in $\hat{M}_i$, while the conditions on $\mathcal{C}$ (namely \ref{prop:A:correct1} and \ref{prop:A:correct4}) will imply that we only do a switch involving the $ux_1$ edge at most once, as there will be at most one $j\in J_{i,u}$ with $\{(i,u),(j,v)\}\in \mathcal{C}$ for some $v$. Considering pairs $\{(i,u)(j,v)\} \in \mathcal{I}_{\tau}$, we find the paths $P$, $Q$ and $R$ using internal vertices coming from $Z_i~(=Z_j=Z_{\tau})$ and colours from $D_3$.

Of course, in finding all of these edges, matchings and paths, we need to ensure that all of the edges assigned to near-matching $\hat{M}_i$ form a rainbow matching for each $i \in [n]$, and that the matchings are all edge-disjoint. This requires some delicacy, but overall the above sketch gives the key mechanism that we use to find the near-matchings we need.
We will find the small matchings $M_{i,u}$, $i\in [n]$ and $u\in S_i\setminus R_i$, in Part~\ref{partB2}, before finding the switching paths in Part~\ref{partB3}. To recap, then, we will do the following.

\begin{enumerate}[label = \textbf{\Alph{enumi}}]\addtocounter{enumi}{1}
\item Realise the absorption structure in the random colouring by doing the following:
\begin{enumerate}[label = \textbf{\Alph{enumi}.\arabic{enumii}}]
\item \label{partB1} Edge-disjointly matching $T_i$ into $X_i$ for each $i\in [n]$ using colours in $D_1$.
\item \label{partB2} Edge-disjointly finding small matchings for each $i\in [n]$ and $u\in S_i\setminus R_i$, using colours in $D_2$ and vertices in $Y_i$.
\item \label{partB3} Edge-disjointly finding the switching paths using colours in $D_3$ and vertices in $Z_i$.
\end{enumerate}
\end{enumerate}

\smallskip

\noindent\textbf{Finding matchings/paths using Theorem~\ref{thm:nibble}.} Each of the structures mentioned above will be found using the R\"odl nibble, as discussed in Section~\ref{sec:nibble}, via Theorem~\ref{thm:nibble} applied to an auxiliary hypergraph. Using an auxiliary hypergraph in this manner was first done by Kim, K{\"u}hn, Kupavskii, and Osthus~\cite{kim2020rainbow}.
A rough overview goes as follows. Suppose we have in $G$ collections of edge-coloured subgraphs $\mathcal{F}_{i,j}$, $i\in [n]$ and $j\in [m]$ (for some $m$), and we wish to choose subgraphs $F_{i,j}\in \mathcal{F}_{i,j}$ which are all edge-disjoint and, for each $i\in [n]$, $F_{i,j}$ and $F_{i,j'}$ are colour- and vertex-disjoint if $j\neq j'$. We form
a hypergraph $\cH$ with 4 vertex classes
\begin{equation}\label{eqn:Hvx}
\begin{array}{ll}
\textbf{i)}\;\; [n]\times [m]
&\;\;\textbf{ii)}\;\;  \bigcup_{i\in [n]}(\{i\}\times V(G))\\
\textbf{iii)}\;\; \; \bigcup_{i\in [n]}(\{i\}\times C(G))
&\;\;\textbf{iv)}\;\; E(G),
\end{array}
\end{equation}
where, for each $i\in [n]$, $j\in [m]$ and $F\in \mathcal{F}_{i,j}$ we add the edge
\[
\{(i,j)\}\cup (V(F)\times \{i\})\cup (C(F)\times \{i\})\cup E(G)
\]
to $\mathcal{H}$. Observe that if the edges corresponding to $F_{i,j}\in \mathcal{F}_{i,j}$, $i\in [n]$ and $j\in [m]$, form a matching in $\mathcal{H}$, then we have exactly the conditions we want of edge-disjointness and (for each $i\in [n]$) colour- and vertex-disjointness. By setting up $\mathcal{H}$ to be a uniform, almost regular hypergraph with small codegrees, we will be able to use Theorem~\ref{thm:nibble} to find a large matching in $\mathcal{H}$, which would then allow us to find many of the graphs $F_{i,j}$, $i\in [n]$ and $j\in [m]$. The remaining graphs $F_{i,j}$, $i\in [n]$ and $j\in [m]$, will then be found using vertices, colours and edges set aside for this purpose.

We will need further properties from the subgraphs we find. In particular, we will need some delicate conditions from the subgraphs found in Part~\ref{partB2} so that we can then apply Theorem~\ref{thm:nibble} again for Part~\ref{partB3}, as the pairs of vertices we wish to find paths between will depend on the structures found in Part~\ref{partB2}. To do this, we use a careful choice of weight functions for the application of Theorem~\ref{thm:nibble}, and it is here that we need the full power of Theorem~\ref{thm:nibble} over previous implementations of the semi-random method.

The simplest of our weight functions will, for example, ensure that, when finding subgraphs, we will use most of the edges we have allocated for this purpose in the subgraphs found. To record such properties efficiently, we will use the following definition.

\begin{defn}
An edge-coloured graph $G$ is \emph{$m$-bounded} if every $v\in V(G)$ has $d_G(v)\leq m$ and every colour appears on at most $m$ edges in $G$.
\end{defn}

\medskip

\noindent\textbf{Variables and applications of Theorem~\ref{thm:nibble}.}
We now discuss further our applications of the semi-random method via Theorem~\ref{thm:nibble} to auxiliary hypergraphs. We will apply this 3 times, to auxiliary hypergraphs $\mathcal{H}_1$, $\mathcal{H}_2$ and $\mathcal{H}_3$. The application to $\mathcal{H}_1$ is the most straightforward, and so we sketch this briefly. The application to $\mathcal{H}_3$ uses the outcome of the application to $\mathcal{H}_2$, and therefore the order of variables in critical here, and so we also discuss this. The relevant variables from \eqref{eq:hierarchy} are
\[
\frac{1}{n}\llpoly \eps \llpoly \gamma\llpoly \beta.
\]
For the initial properties we need, in Section~\ref{sec:Bsetup}, we use the error term $\eps$. That is, for example, we will have properties like $|X_i|=(1\pm \eps)2p_Xn$ for each $i\in [n]$.

In Part~\ref{partB1}, for the variables in Theorem~\ref{thm:nibble} we will use $r_0=4$, $\delta_0=1/10$ and thus set $\eps_0=\delta_0^2/50r_0$ and have some $\Delta_0$ as in that theorem.
We will construct $\mathcal{H}_1$ which is a 4-uniform hypergraph with $d_{\mathcal{H}_1}(v)=(1\pm \eps)\delta_1$ for each $v\in V(\mathcal{H}_1)$ (see Claim~\ref{clm:H1almostregular}), where we will have $\delta_1=q_1n\leq n$ for some $q_1$ depending on the variables in \eqref{eq:hierarchy}, so that $1/n\llpoly q_1$. In particular, then, we will have $n^{0.9}\leq \delta_1\leq n$.
Set $\Delta=(1+\eps)\delta_1$, so that $\Delta(\mathcal{H}_1)\leq \Delta$ and $\Delta\geq \Delta_0$ as $1/n\ll 1$.
We will have that $\Delta^c(\mathcal{H}_1)\leq 1$ (see Claim~\ref{clm:H1lowcod}), so that $\Delta^c(\mathcal{H}_1)\leq \Delta^{1-\delta_0}$.

Furthermore, there will be a collection $\mathcal{W}_1$ of at most $4n\leq \exp(\Delta^{\eps_0^2})$ weight functions such that, for each $w\in \mathcal{W}_1$, $w(e)\leq 1$ for each $e\in E(\mathcal{H}_1)$ and $w(E(\mathcal{H}_1))\geq n^{3/2}$ for each $w\in \mathcal{W}_1$, so that
\[
w(E(\mathcal{H}_1))\geq \Delta^{1+\delta_0}\cdot \max_{e\in E(\mathcal{H}_1)}w(e).
\]
Then, using the property from Theorem~\ref{thm:nibble}, we can find a matching $\cM_1$ in $\mathcal{H}_1$ such that, for each $w\in \mathcal{W}_1$,
\begin{equation*}
w(\mathcal{M}_1)=(1\pm \Delta^{-\eps_0})\cdot w(E(\mathcal{H}_1))/\Delta=(1\pm 2\eps)\cdot\delta_1^{-1} \cdot w(E(\mathcal{H}_1)),
\end{equation*}
where we have used that $1/n\llpoly \eps$ and $\Delta\geq n^{0.9}$.

In Part~\ref{partB2}, letting $r=24$, we will construct $\mathcal{H}_2$ which is a $(7r+4)$-uniform hypergraph. When bounding the vertex degrees of $\mathcal{H}_2$ we will use $r+1$ different properties, so from our initial error term of $\eps$, we will increase this, showing that $d_{\mathcal{H}_2}(v)=(1\pm 100\eps)\delta_2$ for each $v\in V(\mathcal{H}_2)$ (see Claim~\ref{clm:H2almostregular}), where we will have $\delta_2=q_2n^{r+1}\leq n^{r+1}$ for some $q_2$ depending on the variables in \eqref{eq:hierarchy}, so that $1/n\llpoly q_2$. In particular, then, we will have $n^{r+0.9}\leq \delta_2\leq n^{r+1}$.

We will have that $\Delta^c(\mathcal{H}_2)=O(n^{r+0.5})$ (see Claim~\ref{clm:H2lowcod}), so that $\Delta^c(\mathcal{H}_2)\leq (\Delta(\mathcal{H}_2))^{1-c}$, for some small fixed $c>0$. We will consider $O(n^2)$ weight functions $w$, such that, for each such $w$, we have $w(E(\mathcal{H}_2))\geq n^{r+1.5}$.
Then, from Theorem~\ref{thm:nibble}, we will find a matching $\cM_3$ in $\mathcal{H}_3$ such that, for each of these weight functions $w$,
\begin{equation*}
w(\mathcal{M}_3)=(1\pm (\Delta(\mathcal{H}_3))^{-c'})\cdot w(E(\mathcal{H}_2))/\Delta(\mathcal{H}_3)=(1\pm 101\eps)\cdot\delta_2^{-1} w(E(\mathcal{H}_3))=(1\pm \gamma)\cdot \delta_2^{-1} w(E(\mathcal{H}_3)),
\end{equation*}
for some small fixed $c'>0$, where we have used that $1/n\llpoly \eps\llpoly \gamma$.

The properties produced by this application of Theorem~\ref{thm:nibble} in Part~\ref{partB2} will then give bounds for our last auxiliary hypergraph $\mathcal{H}_3$, allowing us to show that  $d_{\mathcal{H}_3}(v)=(1\pm 10\gamma)\delta_3$ for each $v\in V(\mathcal{H}_3)$ (see Claim~\ref{clm:H2almostregular}), where we will have $\delta_3=q_3n^{30}$ for some $q_3$ depending on the variables in \eqref{eq:hierarchy}, so that $1/n\llpoly q_3$. We will have that $\mathcal{H}_3$ is 247-uniform and $\Delta^c(\mathcal{H}_3)=O(n^{29.5})$ (see Claim~\ref{clm:H3lowcod}) and consider $O(n)$ weight functions $w$, each with $w(E(\mathcal{H}_3))\geq n^{30.25}$.
Similarly to the application of Theorem~\ref{thm:nibble} in Part~\ref{partB2}, we will then be able to find a matching $\cM_3$ in $\mathcal{H}_3$ such that, for each of these weight functions $w$,
\begin{equation*}
w(\mathcal{M}_3)=(1\pm 20\gamma)\cdot \delta_3^{-1} w(E(\mathcal{H}_3)).
\end{equation*}
As $\gamma\llpoly \beta$, this will then allow us to record our final properties using the variable $\beta$ for \ref{prop:real:regularity}.

\medskip

\noindent\textbf{Section outline.}
After some further set-up, in Section~\ref{sec:Bsetup} we will record a long list of properties that hold together with high probability in $G\sim G^\col_{[n]}$. Then, we will carry out Parts~\ref{partB1}--\ref{partB3} in Sections~\ref{sec:partB1}--\ref{sec:partB3} respectively. Finally, in Section~\ref{sec:partBfinal}, we will show that the near-matchings we have found satisfy our desired conditions, completing the proof of Lemma~\ref{keylemma:realisation}.


\subsection{Further set-up and properties for Part~\ref{partB}}\label{sec:Bsetup}
Take the set-up from Lemma~\ref{keylemma:realisation}.
Recall that $\beta_0=1/(1+\beta)$ and, for each $i\in [n]$, we have partitions $X_i=X_{i,0}\cup X_{i,1}$, $Y_i=Y_{i,0}\cup Y_{i,1}$ and $Z_i=Z_{i,0}\cup Z_{i,1}$ and random subsets $C_i\subset C$ and $D_{j,i}\subset D_j$ for each $j\in [3]$.
Furthermore, the edges of $G$ are randomly partitioned as $E^\abs=E^\abs_0\cup E^\abs_1$, and then $E^\abs_1$ is randomly partitioned as $E^\abs_{1,A}\cup E^\abs_{1,B}\cup E^\abs_{1,M}$.

Using Lemma~\ref{keylemma:absorption}, for each $\tau\in \mathcal{T}$, let
\[
\mathcal{I}_\tau\subset \{\{(i,u),(j,v)\}: i,j\in I_\tau,i\neq j, u \in  S_i \setminus (R_i \cup T_j), v \in S_j \setminus (T_i \cup R_j),u\neq v,u{\sim}_{A/B}v\}
\]
such that \ref{prop:abs:regularityout}--\ref{prop:abs:corrections} hold. Let $\mathcal{I}=\cup_{\tau\in \mathcal{T}}I_\tau$. For each $\tau\in \mathcal{T}$, $i\in I_\tau$ and $u\in S_i\setminus R_i$, let $J_{i,u}$ be the set of $j$ for which there is some $v$ with $\{(i,u),(j,v)\}\in \mathcal{I}_\tau$. For each $i\in [n]$ and $u\in S_i\setminus R_i$, let $Y_{i,u,0}=\bigcap_{j\in J_{i,u}\cup\{i\}}Y_{j,0}$.

For each $i,j\in [n]$ with $i\sim_{\mathcal{T}}j$, let $\mathcal{L}_{ij}$ be the set of links $R$ such that, for some distinct $u,v\in S_i$, $R$ is a
$(u,v,L)$-link with colours in $D_3\setminus (C_i\cup C_j\cup D_{3,i}\cup D_{3,j})$, edges in $E^\abs_0$ and internal vertices in $Z_{i,0}\cap Z_{j,0}$.

For any vertex $v$ in $G$, if $i$ and $j$ are in the same tribe then the probability that $v\in Z_{i,0}\cap Z_{j,0}$ is $\beta_0^2p_Z$.  For any  colour $c\in C$, and any distinct $i,j\in [n]$, the probability that $c\in D_3\setminus (C_i\cup C_j\cup D_{3,i}\cup D_{3,j})$ is $\beta_0^2p_{\abs}^2p_3$. For any edge $e\in E(G)$, the probability that $e\in E^\abs_0$ is $\beta_0p_\abs$.
Thus, it will be convenient to set
\[
p_{\mathrm{vx}}=\beta_0^2p_Z,\;\;\; p_{\mathrm{col}}=\beta_0^2p_{\abs}^2p_3,\;\;\text{ and }\;\; p_{\mathrm{edge}}=\beta_0p_\abs.
\]
For each path $P$ in $G$ with length 62 which has 30 colours, and each distinct $i$ and $j$ in the same tribe, the probability that $E(P)\subset E_0^\abs$, and $C(P)\subset D_3\setminus (C_i\cup C_j\cup D_{3,i}\cup D_{3,j})$, and all the interior vertices of $P$ are in $Z_{i,0}\cap Z_{j,0}$ is $p_{\mathrm{vx}}^{61}\cdot p_{\mathrm{col}}^{30}\cdot p_{\mathrm{edge}}^{62}$.
We will let $\Phi$ be (very close to) the expected number of $(u,v,L)$-links in $G\sim G^\col_{[n]}$ for any fixed pair of vertices $u\simAB v$, setting
\begin{equation}\label{eqn:Phidefn}
\Phi:=p_{\mathrm{vx}}^{61}\cdot p_{\mathrm{col}}^{30}\cdot p_{\mathrm{edge}}^{62}\cdot n^{30}.
\end{equation}

Let $\mathcalJJJ=\{\{(i,u),(j,v)\}:i,j\in [n],i\neq j,u,v\in V(G),u\sim_{A/B}v\}$.
For the links that we want that we do not find using an application of Theorem~\ref{thm:nibble}, we will find such an $L$-link for $\{(i,u),(j,v)\}\in \mathcalJJJ$ using colours in $(D_{3,i}\cap D_{3,j})\setminus (C_i\cup C_j)$, edges in $E_1^\abs$ and vertices in $Z_{i,1}\cap Z_{j,1}$, so it will be convenient to set
\[
\beta_1=1-\beta_0,\;\;\; q_{\mathrm{vx}}=\beta_1^2p_Z,\;\;\; q_{\mathrm{col}}=\beta_1^2p_{\abs}^2p_3,\;\;\text{ and }\;\; q_{\mathrm{edge}}=\beta_1p_\abs,
\]
and let
\begin{equation}\label{eqn:Phionedefn}
\Phi_1:=q_{\mathrm{vx}}^{61}\cdot q_{\mathrm{col}}^{30}\cdot q_{\mathrm{edge}}^{62}\cdot n^{30},
\end{equation}
so that the expected number of such links is (very close to) $\Phi_1$.

\begin{claim}\label{clm:properties} With high probability, we have the following properties.
\stepcounter{propcounter}

\smallskip

\textbf{Properties for \ref{partB1}: hypergraph degrees}
\begin{enumerate}[label = {{\textbf{\Alph{propcounter}\arabic{enumi}}}}]
\item \label{prop:B1:1} For each $i\in [n]$ and $u\in V(G)$, $|\{v\in X_{i,0}:c(uv)\in D_1\setminus (C_i\cup D_{1,i}),uv\in E^{\mathrm{abs}}_0\}|=(1\pm \eps)\beta_0^3p_Xp_1p_\abs^2 n$.
\item \label{prop:B1:2} For each $i\in [n]$ and $v\in V(G)$, $|\{u\in T_i:c(uv)\in D_1\setminus (C_i\cup D_{1,i}),uv\in E^\abs_0\}|=(1\pm \eps)\beta_0^2p_Tp_1p_\abs^2 n$.
\item \label{prop:B1:3} For each $i\in [n]$ and $c\in C$, $|\{uv\in E_0^{\mathrm{abs}}:c(uv)=c,u\in T_i,v\in X_{i,0}\}|=(1\pm \eps)2\beta_0^2 p_\abs p_Tp_X n$.
\item \label{prop:B1:4} For each distinct $u,v\in V(G)$ and $c\in D_1$, $|\{i\in [n]:u\in T_i,v\in X_{i,0},c\notin C_i\cup  D_{1,i}\}|=(1\pm \eps)\beta_0^2p_Tp_Xp_\abs n$.

\smallskip

\textbf{Properties for \ref{partB1}: boundedness and missing edges}
\item \label{prop:B1:bounded:1} For each $v\in V(G)$, $|\{i\in [n]:v\in T_i\text{ or }v\in X_{i,0}\}|=(1\pm \eps)2p_Tn$.
\item \label{prop:B1:lastfew:2} For each $i\in [n]$, $|T_i|=(1\pm \eps)p_Tn$, $|X_i|=(1\pm \eps)p_Xn$ and $|X_{i,0}|=(1\pm \eps)\beta_0p_Xn$.
\item \label{prop:B1:bounded:2} For each $v\in V(G)$, $|\{uv\in E_0^\abs:c(uv)\in D_1\}|=(1\pm \eps)\beta_0 p_1p_\abs n$.
\item \label{prop:B1:follow:1} For each $i\in [n]$, $X\in \{A,B\}$ and $u\in X$, $|\{v\in X_{i,1}:c(uv)\in D_{1,i}\setminus C_i,uv\in E^{\mathrm{abs}}_{1,X}\}|=(1\pm 2\eps)(1-\beta_0)^3p_Xp_\abs^2n$.
\item \label{prop:B1:lastfew:1} For each $c\in D_1$ and $\phi\in \mathcal{F}$, $|\{i\in I_\phi:c\in D_1\setminus (C_i\cup D_{1,i})\}|=(1\pm \eps)\beta_0 p_\abs p_\tr p_\fa n$ and $|\{uv\in E_0^\abs:c(uv)=c\}|=(1\pm \eps)\beta_0p_\abs n$.
\item \label{prop:B1:Gabsbounded} Setting $G^\abs_1$ to be the graph with vertex set $V(G)$ and edge set $E^\abs_1$, $G^\abs_1|_{D_1}$ is $(\gamma n)$-bounded.

\smallskip


\textbf{Properties for \ref{partB2}: hypergraph degrees}
\item For each $i\in [n]$ and $u\in S_i\setminus R_i$, there are $(1\pm \eps)\beta_0^{2r+3}p_2p_Yp_\abs^{r+2} n$ choices of $v\in Y_{i,u,0}$ with $uv\in E^\abs_0$, and $c(uv)\in D_{2}\setminus (C_j\cup D_{2,j})$ for each $j\in J_{i,u}\cup \{i\}$.\label{cond:newforv1}
\item For each $\phi\in \mathcal{F}$, distinct $i,j\in \phi$, and $c\in C$, there are $(1\pm \eps)\beta_0^5p_Y^2p_{\abs}n$ choices of a colour-$c$ edge in $E_0^\abs$ with vertices in $Y_{i,0}\cap Y_{j,0}$.\label{cond:newforcedge1}
\item For each $i\in [n]$ and $x\in Y_{i,0}$, there are $(1\pm \eps)\beta_0^{2r+2}p_{S-R}p_2p_\abs^{r+2} n$ choices of $u\in S_i\setminus R_i$ such that $c(ux)\in D_{2}\setminus (C_i\cup D_{2,i})$, $ux\in E^\abs_0$, and, for each $j\in J_{i,u}$, $x\in Y_{j,0}$ and $c(ux)\in D_{2}\setminus (C_j\cup D_j)$.
\label{prop:forB2:deg:1}

\item For each $i\in [n]$ and $x\in Y_{i,0}$, there are $(1\pm \eps)2\beta_0p_{S-R} rn$ choices of $u\in S_i\setminus R_i$ and $j\in J_{i,u}$ such that $x\in Y_{j,0}$.\label{prop:forB2:deg:2}

\item For each $i\in [n]$, $u\in S_i\setminus R_i$, $j\in J_{i,u}$ and $x\in Y_{i,0}\cap Y_{j,0}$, there are
$(1\pm \eps)\beta_0^{2r+6}p_2p_Y^2p_\abs^{r+3} n$ choices of
$c$ for which $c\in D_2\setminus (C_{i'}\cup D_{2,i'})$ for each $i'\in J_{i,u}\cup \{i\}$ and such that there is a colour-$c$ edge in $E^\abs_0$ from $u$ to $Y_{i,u,0}$ and from $x$ to $Y_{i,0}\cap Y_{j,0}$.
\label{prop:forB2:deg:3}

\item For each $j\in [n]$ and $x\in Y_{j,0}$, there are $(1\pm \eps)r\beta_0^{2r+2}p_{S-R}p_2p_\abs^{r+2} n$ choices of $i\in [n]$ and $u\in S_i\setminus R_i$ such that $j\in J_{i,u}$, $x\in Y_{j',0}$ for each  $j'\in J_{i,u}\cup \{i\}$, $ux$ is in $E_0^\abs$ and has colour, $c$ say, which is in $D_2\setminus (C_{j'}\cup D_{2,j'})$ for each $j'\in Y_{i,u}\cup \{i\}$.\label{prop:forB2:deg:4}

\item For each $j\in [n]$ and $x\in Y_{j,0}$, there are $(1\pm \eps)2r\beta_0p_{S-R} n$ choices of $i\in [n]$ and $u\in S_i\setminus R_i$ such that $j\in J_{i,u}$ and $x\in Y_{i,0}$. \label{prop:forB2:deg:5}

\item For each $i\in [n]$ and $c\in D_2\setminus (C_i\cup D_{2,i})$, there are $(1\pm \eps)2\beta_0^{2r+2}p_{S-R}p_Yp_\abs^{r+1} n$ choices of $u\in S_i\setminus R_i$ so that $c\in D_2\setminus (C_{j'}\cup D_{2,j'})$ for each $j'\in J_{i,u}$ and $u$ has a
colour-$c$ edge in $E^\abs_2$ to $Y_{i,u,0}$.\label{prop:forB2:deg:7}

\item For each $j\in [n]$ and $c\in D_2\setminus (C_j\cup D_{2,j})$, there are $(1\pm \eps)2r\beta_0^{2r+2}p_{S-R}p_Yp_\abs^{r+1} n$ choices of $i\in [n]$ and $u\in S_i\setminus R_i$ such that $j\in J_{i,u}$ and $c\in D_2\setminus (C_{j'}\cup D_{2,j'})$ for each $j'\in J_{i,u}\cup \{i\}$ and $u$ has a colour-$c$ edge in $E^\abs_2$ to $Y_{i,u,0}$.
\label{prop:forB2:deg:8}

\item For each $xy\in E(G|_{D_2})$, there are $(1\pm \eps)\beta_0^{2r+2}p_{S-R}p_Yp_\abs^{r+1} n$ choices of $i\in [n]$ such that $x\in S_i\setminus R_i$, $c(xy)\in D_2\setminus (C_{j}\cup D_{2,j})$ for each $j\in J_{i,u}\cup \{i\}$ and $y\in Y_{i,u,0}$.
\label{prop:forB2:deg:9}
\item For each $xy\in E(G|_{D_2})$, there are $(1\pm \eps)2r\beta_0^{2r+7}p_{S-R}p_Y^3p_\abs^{r+2} n^2$ choices for $i\in [n]$, $u\in S_i\setminus R_i$ and $j\in J_{i,u}$ such that $c(xy)\in D_2\setminus (C_{j'}\cup D_{2,j'})$ for each $j'\in J_{i,u}\cup \{i\}$, $x,y\in Y_{i,0}\cap Y_{j,0}$ and $u$ has a colour-$c(xy)$ neighbour in $E_0^\abs$ in $Y_{i,u,0}$.
\label{prop:forB2:deg:10}

\smallskip

\textbf{Properties for \ref{partB2}: weight functions}
\item\label{prop:B2:wgt:1} For each $i\in [n]$ and $x\in Z_{i,0}$, $|\{(u,j):u\in S_i\setminus R_i,j\in J_{i,u},x\in Z_{j,0}\}|=(1\pm \eps)\cdot 2r\beta_0p_{S-R}n$.
\item \label{prop:B2:wgt:3} For each $j\in [n]$ and $x\in Z_{j,0}$, $|\{(i,u):u\in S_i\setminus R_i,j\in J_{i,u},x\in Z_{i,0},u\sim_{A/B}x\}|=(1\pm \eps)\cdot r\beta_0p_{S-R}n$.
\item \label{prop:B2:wgt:5} For each $i\in [n]$ and $x\in Z_{i,0}$, $|\{(j,u):u\in S_i\setminus R_i,j\in J_{i,u},x\in Z_{j,0},u\sim_{A/B}x\}|=(1\pm \eps)\cdot r\beta_0p_{S-R}n$.
\item \label{prop:B2:wgt:6} For each $j\in [n]$ and $x\in Z_{j,0}$, $|\{(i,u):u\in S_i\setminus R_i,j\in J_{i,u},x\in Z_{i,0}\}|=(1\pm \eps)\cdot 2r\beta_0p_{S-R}n$.

\item\label{prop:B2:wgt:2} For each $i\in [n]$, $u\in S_i\setminus R_i$, $j\in J_{i,u}$, $x\in V(G)$ and $c\in C$, there are $(1\pm \eps)\beta_0^8p_Y^2p_3p_\abs^4n$ edges in $E_0^\abs$ with colour $c$ and vertices in $Y_{i,0}\cap Y_{j,0}$ which have an edge to $x$ in $E_0^\abs$ with colour in $D_3\setminus (C_i\cup C_j\cup D_{3,i}\cup D_{3,j})$.

\item \label{prop:B2:wgt:4} For each $i\in [n]$, $u\in S_i\setminus R_i$, $j\in J_{i,u}$, and $x\in V(G)$ with $x\simAB u$, there are
 $(1\pm \eps)\beta_0^{2r+6}p_Yp_2p_3p_\abs^{r+5} n$ vertices $v\in Y_{i,u,0}$ with $uv,xv\in E^\abs_0$, $c(uv)\in D_{2}\setminus (C_{i'}\cup D_{2,i'})$ for each $i'\in J_{j,u}\cup \{j\}$, and $c(xv)\in D_3\setminus (C_i\cup C_j\cup D_{3,i}\cup D_{3,j})$.

\item \label{prop:B2:wgt:7} For each $i\in [n]$ and $c\in D_3\setminus (C_i\cup C_{3,i})$, there are $(1\pm \eps)\beta_0 p_\abs r\cdot p_{S-R}n$ choices of $u\in S_i\setminus R_i$ and $j\in J_{i,u}$ with $c\notin C_{j}\cup D_{3,j}$.

\item \label{prop:B2:wgt:9} For each $j\in [n]$ and $c\in D_3\setminus (C_j\cup C_{3,j})$, there are $(1\pm \eps)\cdot r\beta_0p_\abs p_{S-R}n$ choices for $i\in [n]$ and $u\in S_i\setminus R_i$ with $j\in J_{i,u}$ and $c\notin C_i\cup D_{3,i}$.


\item \label{prop:B2:wgt:8} For each $i\in [n]$, $u\in S_i\setminus R_i$, $j\in J_{i,u}$ and distinct $c,c'\in C$, there are $(1\pm \eps)\beta_0^6p_Y^2p_Zp_\abs^2n$ choices
for $w\in Y_{i,0}\cap Y_{j,0}$ which has a colour-$c$ edge in $E_0^\abs$ to $Z_{j,0}\cap Z_{j',0}$ and a colour-$c'$ edge in $E_0^\abs$ to $Y_{j,0}\cap Y_{j',0}$.

\item \label{prop:B2:wgt:10} For each $i\in [n]$ $u\in S_i\setminus R_i$, $j\in J_{i,u}$, and $c'\in C$ there are $(1\pm \eps)\beta_0^{2r+6}p_Yp_Zp_2p_\abs^{r+3} n$ choices of $v\in Y_{i,u,0}$ with $uv\in E^\abs_0$ and $c(uv)\in D_{2}\setminus (C_{i'}\cup D_{2,i'})$ for each $i'\in J_{i,u}\cup \{i\}$ and there is a colour-$c'$ edge from $v$ to $Z_{i,0}\cap Z_{j,0}$ in $E_0^\abs$.


\item \label{prop:B2:wgt:11} For each $xy\in E(G)$ with $c(xy)\in D_3$, there are $\beta_0^{2r+6}p_Zp_Yp_2p_\abs^{r+3}rp_{S-R}n^2$ choices for $i\in [n]$,  $u\in S_i\setminus R_i$ and $j\in J_{i,u}$ with $x\in Y_{i,u,0}$, $c(ux)\in D_2\setminus (C_{j'}\cup D_{2,j'})$ for each $j'\in J_{i,u}$, $y\in Z_{i,0}\cap Z_{j,0}$ and $c(xy)\notin C_i\cup C_j\cup D_{3,i}\cup D_{3,j}$.

\item \label{prop:B2:wgt:12}
For each $xy\in E(G)$ with $c(xy)\in D_3$, there are $(1\pm \eps)\cdot r\beta_0^{6}p_\abs^2p_Yp_Zp_{S-R}n^2$ choices for $i\in [n]$, $u\in S_i\setminus R_i$ and $j\in J_{i,u}$ with $c(xy)\notin (C_i\cup C_j\cup D_{3,i}\cup D_{3,j})$, $x\in Y_{i,0}\cap Y_{j,0}$ and $y\in Z_{i,0}\cap Z_{j,0}$.

\item \label{prop:B2:wgt:13}
For each $xy\in E(G)$ with $c(xy)\in D_3$, and each $i\in [n]$, $u\in S_i\setminus R_i$ and $j\in J_{i,u}$ there are $(1\pm \eps)\beta_0^{2r+6}p_Yp_2p_\abs^{r+2} n$ choices of $v\in Y_{i,u,0}$ with $uv\in E^\abs_0$ and $c(uv)\in D_{2}\setminus (C_{i'}\cup D_{2,i'})$ for each $i'\in J_{i,u}\cup \{i\}$ such that $x$ has a colour-$c(uv)$ edge in $E_0^\abs$ to $Y_{i,0}\cap Y_{j,0}$ .

\item \label{prop:B2:wgt:14} For each $xy\in E(G)$ with $c(xy)\in D_3$, there are $\beta_0^{6}p_Z^2p_\abs^{2}rp_{S-R}n^2$ choices for $i\in [n]$, $u\in S_i\setminus R_i$ and $j\in J_{i,u}$ with $x,y\in Z_{i,0}\cap Z_{j,0}$ and $c(xy)\notin C_i\cup C_j\cup D_{3,i}\cup D_{3,j}$.

\item For each $i\in [n]$, $|S_i\setminus R_i|=(1\pm \eps)\cdot 2p_{S-R}n$ and $|Y_{i,0}|=(1\pm \eps)\cdot 2\beta_0p_Yn$.\label{prop:B2:sizeofSminusR}

\item For each $u\in V(G)$ and $\phi\in \mathcal{F}$, $|\{i\in I_\phi:u\in S_i\setminus R_i\}|=(1\pm \eps)\cdot p_{S-R}p_\tr p_\fa n$ and $|\{i\in I_\phi:u\in Y_{i,0}\}|=(1\pm \eps)\cdot \beta_0p_Ynp_\tr p_\fa$.\label{prop:B1:lastfew:4}

\item For each $c\in C(G)$ and $\phi\in \mathcal{F}$, $|\{i\in I_\phi:c\in D_2\setminus (C_i\cup D_{2,i})\}|=(1\pm \eps)\beta_0p_\abs p_\tr p_\fa n$.\label{prop:B1:lastfew:3}

\smallskip

\textbf{Properties for \ref{partB2}: missing small matchings}

\item For each $i\in [n]$, $u\in S_i\setminus R_i$ and $X\in \{A,B\}$, there are at least $10\sqrt{\gamma}n$ vertices $v\in \cap_{j\in J_{i,u}\cup \{i\}} Y_{j,1}$ with colour in $\cap_{j\in J_{i,u}\cup \{i\}}(D_{2,j}\setminus C_j)$
and such that $uv\in E_{1,X}^\abs$.\label{prop:B2:missing:uv}

\item For each $i\in [n]$, $u\in S_i\setminus R_i$ and $c\in C$, there are at least $\gamma^{1/3}n$ edges in $E_{1,M}^\abs$ with colour $c$ and vertices in $\cap_{j\in J_{i,u}\cup \{i\}} Y_{j,1}$. \label{prop:B2:missing:Muv}

\smallskip

\textbf{Properties for \ref{partB3}: relevant properties of $\mathcal{I}$}
\item \label{prop:fromB2:5}  For each $i\in [n]$ and $x\in Z_{i,0}$, there are $(1\pm \eps)\beta_0\cdot p_{\mathcal{I}}\cdot n$ labelled\footnote{I.e., we consider the number of choices of $(i,u,j,v)$ for which $\{(i,u),(j,v)\}\in \mathcal{I}$.} choices for $\{(i,u),(j,v)\}\in \mathcal{I}$ with $x\in Z_{j,0}$ and $u\not\simAB x$.

\item \label{prop:fromB2:5b}  For each $i\in [n]$ and $x\in Z_{i,0}$, there are $(1\pm \eps)\beta_0\cdot p_{\mathcal{I}}\cdot n$ labelled choices for $\{(i,u),(j,v)\}\in \mathcal{I}$ with $x\in Z_{j,0}$ and $u\simAB x$.

\item \label{prop:fromB2:8}
For each $i\in [n]$ and $c\in D_3\setminus (C_i\cup D_{3,i})$, there are $(1\pm \eps)
2\beta_0p_\abs\cdot p_{\mathcal{I}}\cdot n$
 labelled choices for $\{(i,u),(j,v)\}\in \mathcal{I}$ with $c\notin C_j\cup D_{3,j}$.

 \item \label{prop:fromB2:13} For each $xy\in E_0^\abs$ with $c(xy)\in D_3$, there are $(1\pm \eps)2{p}_{\mathrm{vx}}^2\cdot {p}_{\mathrm{col}}\cdot p_{\mathcal{I}}\cdot n^2$ labelled choices
 for  $\{(i,u),(j,v)\}\in \mathcal{I}$ for which $c(xy)\in D_3\setminus (C_i\cup C_j\cup D_{3,i}\cup D_{3,j})$ and $x,y\in Z_{i,0}\cap Z_{j,0}$.

\smallskip

\textbf{Properties for \ref{partB3}: link counts for hypergraph degrees}

\item \label{prop:B3:cod} \ref{prop:links:totalnumber}--\ref{prop:links:throughedgeandvertex} hold with $\Phi_0=n^{30}$.

\item \label{prop:B3:1} For each $\{(i,u),(j,v)\}\in \mathcalJJJ$, the number of $(u,v,L)$-links in $\mathcal{L}_{ij}$ is $(1\pm \eps)\Phi$.


\item \label{prop:B3:2} For each $\{(i,u),(j,v)\}\in \mathcalJJJ$ and $x\in V(G)\setminus \{u,v\}$ with $c(ux)\in D_3\setminus (C_i\cup C_j\cup D_{3,i}\cup D_{3,j})$, $x\in Z_{i,0}\cap Z_{j,0}$ and $ux\in E_0^\abs$,
the number of $(u,v,L)$-links in $\mathcal{L}_{ij}$ in which $x$ is the 2nd vertex is
$(1\pm \eps)\cdot \Phi \cdot p_{\mathrm{vx}}^{-1}\cdot p_{\mathrm{col}}^{-1}\cdot p_{\mathrm{edge}}^{-1}\cdot n^{-1}$.

\item \label{prop:B3:6} For each $k$ with $3\leq k\leq 59$ and each
$\{(i,u),(j,v)\}\in \mathcalJJJ$ and  $x\in V(G)\setminus \{u,v\}$ with $x\in Z_{i,0}\cap Z_{j,0}$,
the number of $(u,v,L)$-links in $\mathcal{L}_{ij}$ in which $x$ is the $k$th vertex is $(1\pm \eps)\cdot \Phi \cdot p_{\mathrm{vx}}^{-1}\cdot n^{-1}$.


\item \label{prop:B3:7}\label{prop:B3:10} For each $\{(i,u),(j,v)\}\in \mathcalJJJ$ and $xy\in E_0^\abs$ with $x=u$, $c(xy)\in D_3\setminus (C_i\cup C_j\cup D_{3,i}\cup D_{3,j})$ and $y\in Z_{i,0}\cap Z_{j,0}$, the number of $(u,v,L)$-links in $\mathcal{L}_{ij}$ which have $xy$ as the 1st edge is $(1\pm \eps)\cdot \Phi \cdot  p_{\mathrm{vx}}^{-1}\cdot p_{\mathrm{col}}^{-1}\cdot p_{\mathrm{edge}}^{-1}\cdot n^{-1}$.

\item \label{prop:B3:9} For each $k$ with $2\leq k\leq 59$, and each $\{(i,u),(j,v)\}\in \mathcalJJJ$ and $c'\in D_3\setminus (C_i\cup C_j\cup D_{3,i}\cup D_{3,j})$,
the number of $(u,v,L)$-links in $\mathcal{L}_{ij}$ whose $k$th edge has colour $c'$ is $(1\pm \eps)\cdot \Phi \cdot  p_{\mathrm{col}}^{-1}\cdot n^{-1}$.

\item \label{prop:B3:11}  For each $\{(i,u),(j,v)\}\in \mathcalJJJ$ and $xy\in E_0^\abs$ for which
$c(ux),c(xy)\in D_3\setminus (C_i\cup C_j\cup D_{3,i}\cup D_{3,j})$, $x,y\in Z_{i,0}\cap Z_{j,0}$,
and $ux,xy\in E_0^\abs$, the number of $(u,v,L)$-links in $\mathcal{L}_{ij}$ which have $xy$ as the 2nd edge is
 $(1\pm \eps)\cdot \Phi \cdot  p_{\mathrm{vx}}^{-2}\cdot p_{\mathrm{col}}^{-2}\cdot p_{\mathrm{edge}}^{-2}\cdot n^{-2}$.

 \item \label{prop:B3:14} For each $k$ with $3\leq k\leq 59$, and each  $\{(i,u),(j,v)\}\in \mathcalJJJ$ and $xy\in E_0^\abs$ with $c(xy)\in D_3\setminus (C_i\cup C_j\cup D_{3,i}\cup D_{3,j})$ and $x,y\in Z_{i,0}\cap Z_{j,0}$,
  the number of $(u,v,L)$-links in $\mathcal{L}_{ij}$ which have $xy$ as the $k$th edge is
   $(1\pm \eps)\cdot \Phi \cdot  p_{\mathrm{vx}}^{-2}\cdot p_{\mathrm{col}}^{-1}\cdot p_{\mathrm{edge}}^{-1}\cdot n^{-2}$.

\item \label{prop:B3:assorted:1} For each $c\in C$, $|\{e\in E(G_0^\abs):c(e)=c\}|=(1\pm \eps)\cdot \beta_0p_\abs n$
and $|\{i\in [n]:c\in D_3\setminus (C_i\cup D_{3,i})\}|=(1\pm \eps)\cdot \beta_0p_\abs n$.

\item \label{prop:B3:assorted:2} For each $v\in V(G)$, $|\{i\in [n]:v\in Z_{i,0}\}|=(1\pm \eps)\cdot p_Z\beta_0n$.

\smallskip

\textbf{Property for \ref{partB3}: missing links}

\item For each $\{(i,u),(j,v)\}\in \mathcalJJJ$, the number of $(u,v,L)$-links with colours in $(D_{3,i}\cup D_{3,j})\setminus (C_i\cup C_j)$, edges in $E^\abs_1$ and internal vertices in $Z_{i,1}\cap Z_{j,1}$ is $(1\pm \eps)\Phi_1$.\label{prop:B3:missing}

\textbf{Properties for bounded remainder}

\item For each $v\in V(G)$ and $\phi\in \mathcal{F}$, $|\{i\in I_\phi:v\in X_{i,1}\cup Y_{i,1}\cup Z_{i,1}\}|\leq \gamma p_\tr p_\fa n$.\label{prop:finalB:1}

\item For each $k\in [3]$, $c\in D_k$ and $\phi\in \mathcal{F}$, $|\{i\in I_\phi:c\in D_{k,i}\}|\leq \gamma p_\tr p_\fa  n$.\label{prop:finalB:2}

\item For each $v\in V(G)$, $|\{u:uv\in E^\abs_1\}|=(1\pm \alpha)p_\abs n$.\label{prop:finalB:3}

\item For each $i\in [n]$, $|Z_{i,0}|=(1\pm \eps)\beta_0 p_Zn$. \label{prop:sizeofZi0}
\end{enumerate}
\end{claim}
\begin{proof}[Proof of Claim~\ref{clm:properties}] By Theorem~\ref{thm:Llinks}, with high probability, \ref{prop:links:totalnumber}--\ref{prop:links:throughedgeandvertex} hold with $\Phi_0=n^{30}$. That is, with high probability \ref{prop:B3:cod} holds. Supposing, then, that \ref{prop:B3:cod} holds, we choose our various random vertex, colour and edge partitions, and show that each other property from \ref{prop:B1:1}--\ref{prop:sizeofZi0} will hold with high probability. There is a large overlap in proving that each of the properties hold with high probability. For brevity, we will select only some key properties to prove this for explicitly, as follows, where we have selected a range of properties that cover the different approaches required.

\smallskip

\noindent\ref{prop:B1:1}: Let $i\in [n]$ and $u\in V(G)$, and let $Z=\{v\in X_{i,0}:c(uv)\in D_1\setminus (C_i\cup D_{1,i}),uv\in E^{\mathrm{abs}}_0\}$.
Let $V_u=|\{v\in V(G):c(uv)\in D_1\}|$, so that $|V_u|=|D_1|$, and thus, with probability $1-o(n^{-\omega(1)})$ by Lemma~\ref{chernoff}, $|D_1|=(1\pm \eps/2)n$.
Then, having chosen $D_1$, for each $v\in B_u$, $\P(v\in X_{i,0})=\beta_0p_X$, $\P(c(uv)\notin (C_i\cup D_{i,1}))=\beta_0p_\abs$, and $\P(uv\in E_0^\abs)=\beta_0p_\abs$, so that $\P(v\in Z)=\beta_0^3p_Xp_\abs^2$, and, hence $\E|Z|=(1\pm \eps/2)\beta_0^3p_Xp_\abs^2n$.
The events $\{v\in Z\}$, $v\in V_u$, are independent, so, as $1/n\llpoly \beta_0,p_X,p_\abs,p_1$, by Lemma~\ref{chernoff}, with probability $1-o(n^{-2})$ we have that $|Z|=(1\pm \eps)\beta_0^3p_Xp_1p_\abs^2n$. Taking a union bound then completes the proof that \ref{prop:B1:1} holds with high probability.

\smallskip

\noindent\ref{prop:B1:4}: Let $u,v\in V(G)$ be distinct and $c\in D_1$. Let $Z=\{i\in [n]:u\in T_i,v\in X_{i,0},c\notin C_i\cup  D_{1,i}\}$. A feature here is that the events $\{i\in Z\}$
 and $\{j\in Z\}$ are not independent if $i$ and $j$ are from the same tribe. Let $\mathcal{T}'\subset \mathcal{T}$ be the set of tribes for which $u\in S_\tau$ and $v\in X_{\tau}$, noting that, by Lemma~\ref{chernoff}, with probability $1-o(n^{-\omega(1)})$ we have $|\mathcal{T}'|=(1\pm \eps/4)p_Sp_Xp_{\tr}^{-1}$, where we have used that $p_{tr}\llpoly p_X,p_S,\eps$. For each $\tau\in \mathcal{T}'$, let $\mathcal{F}_\tau'$ be the set of $\phi\in \mathcal{F}_\tau$ for which $u\in U_i$.
By Lemma~\ref{chernoff} and a union bound, with probability $1-o(n^{-\omega(1)})$ we have $|\mathcal{F}'_\tau|=(1\pm \eps/4)(p_U/p_S)p_{\fa}^{-1}$ for each $\tau\in \mathcal{T}$.
 For each $\phi\in \mathcal{F}_\tau$, let $I_\phi'$ be the set of $i\in I_\phi$ such that $u\in S_i$, $v\in X_{i,0}$ and $c\notin C_i\cup D_{1,i}$. By Lemma~\ref{chernoff} and a union bound, with probability $1-o(n^{-\omega(1)})$, we have that $|I_\phi'|=(1\pm \eps/4)\cdot (p_T/p_U)\cdot \beta_0\cdot \beta_0p_\abs \cdot p_{tr}p_{\fa}n$ for each $\phi\in \mathcal{T}$. Putting this altogether, and taking a union bound, we then have, with high probability,
 \[
|Z|=\sum_{\tau\in \mathcal{T}'}\sum_{\phi\in \mathcal{F}_\tau'}I_\phi'=(1\pm \eps)\beta_0^2p_Tp_Xp_\abs n,
 \]
 for every $u,v \in V(G)$ and $c \in D_1$, as required.

\smallskip

\noindent\ref{cond:newforv1}: Let $\phi\in \mathcal{F}$, $i\in I_\phi$, $J\subset I_\phi\setminus \{i\}$ with $|J|=r$, and $u\in V(G)$. Let $V_u=|\{v\in V(G):c(uv)\in D_2\}|$, so that $|V_u|=|D_2|$, and thus, with probability $1-o(n^{-\omega(1)})$ by Lemma~\ref{chernoff}, $|V_u|=(1\pm \eps/2)p_2n$.
Let $Z_{i,J}=\{v\in V_u:v\in \cap_{j\in J\cup \{i\}}Y_j,uv\in E^\abs_0,c(uv)\in D_2\setminus (C_j\cup D_{2,j})\text{ for each }j\in J\cup \{i\}\}$. Note that, for each $v\in V_u$,
\[
\P(v\in Z_{i,J})=\beta_0^{r+1}p_Y\cdot \beta_0p_\abs\cdot (\beta_0 p_\abs)^{r+1},
\]
Thus, by Lemma~\ref{chernoff} and a union bound, with high probability $|Z_{i,J}|=(1\pm \eps)\beta_0^{2r+3}p_2p_Yp_\abs^{r+2} n$ for all such $i$ and $J$, whence we have that \ref{cond:newforv1} holds.

\smallskip

\noindent\ref{cond:newforcedge1}: This holds with high probability similarly to \ref{prop:forB2:deg:1}. Let us note that, for each $c\in C$ and each $\phi\in \mathcal{F}$ and distinct $i,j\in I_\phi$, for each edge $xy$ in $G$ with colour $c$, $\P(x\in Y_i\cap Y_j)=\P(x\in Y_\phi)=p_Y$, and hence $\P(x\in Y_{i,0}\cap Y_{j,0})=\beta_0^2p_Y$, so that $\P(x,y\in  Y_{i,0}\cap Y_{j,0},xy\in E_0^\abs)=\beta_0^5p_Y^2p_\abs$.

\smallskip

\noindent\ref{prop:forB2:deg:1}: Let $i\in [n]$ and $x\in V(G)$. With probability \wvhp we have $|D_2|=(1\pm \eps/6)p_2n$. Let $Z_{\phi,x}$ be the set of $u\in S_i\setminus R_i$ such that $c(ux)\in D_2$, noting that we have with probability \wvhp that $|Z_{\phi,x}|=(1\pm \eps/4)p_{S-R}\cdot |D_2|=(1\pm \eps/2)p_{S-R}p_2n$. Then, after choosing the remaining initial vertex sets, choose $\mathcal{I}$, and then choose the sets $D_{2,j}$, $j\in [n]$, the set of edges $E^\abs_0$, and the sets $Y_{j,0}$, $j\in [n]$.
 Suppose $x\in Y_{i}$ (for otherwise no corresponding bound in \ref{prop:forB2:deg:1} is claimed). Let $Z'_{\phi,x}$ be the set of $u\in Z_{\phi,x}$ for which $c(ux)\notin C_i\cup D_{2,i}$, $ux\in E^\abs_0$, and, for each $j\in J_{i,u}$, $x\in Y_{j,0}$ and $c(ux)\notin C_j\cup D_j$. For each $u\in Z_{\phi,x}$, we have $\P(u\in Z_{\phi,x}')=\beta_0p_\abs\cdot \beta_0p_\abs\cdot \beta_0^r\cdot (\beta_0p_\abs)^r=\beta_0^{2r+2}p_\abs^{r+2}$. The events $\{u\in Z_{\phi,x}'\}$, $u\in Z_{\phi,x}$ are not independent, but,
 by \ref{prop:abs:boundedin}, the dependence is low enough that an application of Lemma~\ref{lem:mcdiarmidchangingc} shows that $|Z_{\phi,x}'|=(1\pm \eps)\beta_0^{2r+2}p_\abs^{r+2}\cdot p_{S-R}p_2n$ with probability \wvhp. Thus, using a union bound, we have that with high probability \ref{prop:forB2:deg:1} holds, where we record this only for $x\in Y_{i,0}$ to match its application.

\smallskip

\noindent\ref{prop:forB2:deg:2}: This follows similarly, though more simply, to \ref{prop:forB2:deg:1}, and we include its proof to emphasise where the term $2r$ in the expression arises. Let $i\in [n]$ and $x\in V(G)$. With probability \wvhp we have that $|S_i\setminus R_i|=(1\pm \eps/3)2p_{S-R}n$ (as $V(G)$ has $2n$ vertices). We can suppose, again, that $x\in Y_i$.
Then, for each of the $|S_i\setminus R_i|\cdot r$ possibilities for $(u,j)$ with $u\in S_i\setminus R_i$ and $j\in J_{i,u}$,  $\P(x\in Y_{j,0})=\beta_0$. For each $j'\in [n]$, the choice of $Y_{j',0}\subset Y_i$ influences whether $x\in Y_{j,0}$ for at most $10^6$ of the pairs $(u,j)$ by \ref{prop:abs:boundedin}, and therefore, by Lemma~\ref{lem:mcdiarmidchangingc}, we have that with probability \wvhp there are, in total, $(1\pm \eps)2\beta_0p_{S-R}n$ choices for $u\in S_i\setminus R_i$ and $j\in J_{i,u}$ with $x\in Y_{j,0}$.

\smallskip

\noindent\ref{prop:forB2:deg:3}: This follows similarly to previous properties, using Lemma~\ref{lem:mcdiarmidchangingc}, but let us comment that that the edge $ux$ may have colour in $D_2$, and this colour may then never be counting among those colours $c$ counted in \ref{prop:forB2:deg:3}. Of course, one colour is comfortably lost in the error terms used.

\smallskip

\noindent\ref{prop:forB2:deg:4}: We discuss this property because, having fixed $j\in [n]$, we count certain choices for $i\in [n]$ with $j\in J_{i,u}$ for some $u$ (rather than, correspondingly, counting certain choices for $i\in [n]$ with $i\in J_{j,u}$ for some $u$, as for \ref{prop:forB2:deg:2}).
Let $j\in [n]$ and $x\in V(G)$, and choose the initial vertex partitions. Let $\phi\in \mathcal{F}$ be such that $i\in I_\phi$. By Lemma~\ref{chernoff}, with probability \wvhp, we have that, for each $i\in I_\phi$, there are $(1\pm \eps/4)2p_{S-R}n$ choices of $u\in S_i\setminus R_i$.
Choose $\mathcal{I}$, and note that then, by \ref{prop:abs:regularityout}, there are $r|S_{j}\setminus R_j|$ choices for $(v,i,u)$ such that $\{(j,v),(i,u)\}\in \mathcal{I}$, and thus $(1\pm \eps/4)2rp_{S-R}n$ choices for $(i,u)$ such that $j\in J_{i,u}$.
Then, choose the colour partition $D_1\cup D_2\cup D_3$. Using \ref{prop:abs:boundedin} and Lemma~\ref{lem:mcdiarmidchangingc}, with probability \wvhp there are $(1\pm \eps/2)2rp_{S-R}p_2n$ choices for $(i,u)$ such that $j\in J_{i,u}$ and $ux$ has colour in $D_2$. As before, we can assume that $x\in Y_j$.
Then, choosing the further random vertex sets, the edges in $E_0^\abs$ and the sets $C_{j'}$ and $D_{2,j'}$, $j'\in [n]$, for each such choice for $(i,u)$ such that $j\in J_{i,u}$ and $c(ux)\in D_2$, the probability that $x\in Y_{j',0}$ for each  $j'\in (J_{i,u}\cup \{i\})\setminus \{j\}$, $ux$ is in $E_0^\abs$ and $c(ux)\notin C_{j'}\cup D_{2,j'}$ for each $j'\in Y_{i,u}\cup \{i\}$ is $\beta_0^r\cdot \beta_0p_\abs \cdot (\beta_0p_\abs)^{r+1}$.
Using \ref{prop:abs:boundedin}, \ref{prop:abs:lowcodegree2}, and Lemma~\ref{lem:mcdiarmidchangingc}, with probability \wvhp  there are $(1\pm \eps)r\beta_0^{2r+2}p_{S-R}p_2p_\abs^{r+2} n$ choices of $i\in [n]$ and $u\in S_i\setminus R_i$ such that $j\in J_{i,u}$, $x\in Y_{j',0}$ for each  $j'\in J_{i,u}\cup \{i\}$, $ux$ is in $E_0^\abs$ and $c(ux)\in D_2\setminus (C_{j'}\cup D_{2,j'})$ for each $j'\in Y_{i,u}\cup \{i\}$.

\smallskip

\noindent \ref{prop:B2:missing:uv}, \ref{prop:B2:missing:Muv}: These properties follows similarly to our others, but note that due to their use we only record a loose lower bound on the vertices/edge counted, using that $\gamma\llpoly \beta,p_Y,p_2,p_\abs,\beta_0$.

\smallskip

\noindent\ref{prop:fromB2:5}--\ref{prop:fromB2:13}: These properties follows similarly to our others, but we comment on them as they are the first to use $p_{\mathcal{I}}=24p_{S-R}=rp_{S-R}$, $q_\col=\beta^2_1p_\abs^2p_3$ and $q_{\mathrm{vx}}=\beta_1^2p_Z$, where $\beta_0=1-\beta_0$.
We set $p_{\mathcal{I}}$ so that we will have $|\mathcal{I}|\approx p_{\mathcal{I}}n^2$.
As each $\mathcal{I}$ is a set of pairs, for each $i\in [n]$, $|\{(u,j,v):\{(i,u),(j,v)\}\in \mathcal{I}\}|\approx 2p_{\mathcal{I}}n$,
and, for each $X\in \{A,B\}$ we will have, for each $i\in [n]$, $|\{(u,j,v):\{(i,u),(j,v)\}\in \mathcal{I},u\in X\}|\approx 2p_{\mathcal{I}}n$. In \ref{prop:fromB2:5}--\ref{prop:fromB2:8}, we require some extra condition, where the expected number of triples $(u,j,v)$ satisfying this as well can be calculated as with our other properties, and shown to be likely concentrated around this expectation using Lemma~\ref{lem:mcdiarmidchangingc}.

As an example, we will do the more complicated property \ref{prop:fromB2:13} more carefully. Let then $xy\in E(G)$ with $c(xy)\in D_3$.
Let $\mathcal{T}'$ be the set of $\tau\in \mathcal{T}$ with $x,y\in Z_{\tau}$, so that, by Lemma~\ref{chernoff}, we have with probability \wvhp that $|\mathcal{T}'|=(1\pm \eps/4)p_Z^2p_{\tr}^{-1}$. Furthermore, with probability \wvhp we have that $|S_i\setminus R_i|=(1\pm \eps/4)2p_{S-R}n$ for each $i\in [n]$.
For each $\tau\in \mathcal{T}'$, $\phi\in \mathcal{F}_\tau$ and $\{(i,u),(j,v)\}\in \mathcal{I}_\tau$,
$\P(c(xy)\notin C_i\cup C_j\cup D_{3,i}\cup D_{3,j}\land x,y\in Z_{i,0}\cap Z_{j,0})=p_\abs^2\beta_0^2\cdot \beta_0^2$. Therefore, using \ref{prop:abs:regularityout}, \ref{prop:abs:boundedin} and Lemma~\ref{lem:mcdiarmidchangingc}, we have that, with probability \wvhp,
 \begin{align*}
|\{(i,u,j,v):\{(i,u),(j,v)\}\in \mathcal{I},c(xy)\in D_3\setminus (C_i\cup C_j&\cup D_{3,i}\cup D_{3,j}),x,y\in Z_{i,0}\cap Z_{j,0}\}|\\
&=(1\pm \eps)\cdot (1\pm \eps/4)p_Z^2p_{\tr}^{-1}\cdot 2p_{S-R}n\cdot p_\abs^2\beta_0^2\cdot \beta_0^2\\
&=(1\pm \eps)\cdot 2{p}_{\mathrm{vx}}^2\cdot {p}_{\mathrm{col}}\cdot p_{\mathcal{I}}\cdot n^2,
 \end{align*}
 as required.

\smallskip

\noindent\ref{prop:B3:1}--\ref{prop:B3:14}, \ref{prop:B3:missing}: As we have assumed that \ref{prop:links:totalnumber}--\ref{prop:links:throughedgeandvertex} hold with $\Phi_0=n^{30}$, these will follow relatively straightforwardly using Lemma~\ref{lem:mcdiarmidchangingc}. As they follow similarly from each other (and more examples are done in detail for \ref{prop:C:forvxbalancing}), we will pick only one example, \ref{prop:B3:11} to do here in detail.

\smallskip

\noindent\ref{prop:B3:11}: Let $\{(i,u),(j,v)\}\in \mathcalJJJ$ and $xy\in E(G)$ with $y\simAB u$.
Let $\mathcal{L}$ be the set of $(u,v,L)$-links in $G$ which have $xy$ as their second edge, so that, by \ref{prop:links:throughedge} $|\mathcal{L}|=(1\pm \eps/4)n^{28}$.
Let $\mathcal{L}'$ be the set of $S\in \mathcal{L}$ with $V(S)\setminus \{u,v,x,y\}\subset Z_i\cap Z_j$, $C(S)\setminus \{c(ux),c(xy)\}\subset D_3\setminus (C_i\cup C_j\cup D_{3,i}\cup D_{3,j})$ and $E(S)\setminus \{ux,xy\}\in E^\abs_0$. Then, using \ref{eqn:Phidefn},
\begin{align*}
\E|\mathcal{L}'|&=(1\pm \eps/4)n^{28}(\beta_0^2p_Z)^{59}(\beta_0^2p_\abs^2p_3)^{28}(\beta_0p_\abs)^{60}=(1\pm \eps/4)p_{\mathrm{vx}}^{59}p_{\col}^{28}p_{\mathrm{edge}}^{28}\\
&=(1\pm \eps/4)\Phi \cdot  p_{\mathrm{vx}}^{-2}\cdot p_{\mathrm{col}}^{-2}\cdot p_{\mathrm{edge}}^{-2}\cdot n^{-2}.
\end{align*}
Now, there are $2n-4$ vertices in $V(G)\setminus \{u,v,x,y\}$, and whether of not each of them is in $Z_i\cap Z_j$ affects $|\mathcal{L}'|$ by at most $O(n^{27})$ by \ref{prop:links:throughedgeandvertex}. There are $n-2$ colours not in $\{c(ux),c(xy)\}$, and whether each of these is in $Z_\tau$ or not affects $|\mathcal{L}'|$ by at most $O(n^{27})$ by \ref{prop:links:2edgesofdifferentcoloursanddisjoint} (summed over the edges of that colour in $G-\{u,v,x,y\}$) and \ref{prop:links:throughedgeandvertex} (summed over the neighbours of $\{u,v,x,y\}$ of that colour in $G$).
There are at most $n^2$ edges in $G-\{u,v,x,y\}$, and whether each of these is in $E^\abs_0$ or not affects $|\mathcal{L}'|$ by at most $O(n^{26})$ by \ref{prop:links:2edgesofdifferentcoloursanddisjoint}. There are at most $4n$ edges in $G$ with exactly one vertex in $\{u,v,x,y\}$, and whether or not each of these is in $E^\abs_0$ or not affects $|\mathcal{L}'|$ by at most $O(n^{27})$ by \ref{prop:links:throughedgeandvertex}.

Therefore, by Lemma~\ref{lem:mcdiarmidchangingc}, with $t=(\eps/4)\cdot\Phi \cdot  p_{\mathrm{vx}}^{-2}\cdot p_{\mathrm{col}}^{-2}\cdot p_{\mathrm{edge}}^{-2}\cdot n^{-2}$, we have
\begin{align*}
\P(|\mathcal{L}'|\neq (1\pm \eps)\Phi \cdot&  p_{\mathrm{vx}}^{-2}\cdot p_{\mathrm{col}}^{-2}\cdot p_{\mathrm{edge}}^{-2}\cdot n^{-2})\\
&
\leq 2\exp\left(-\frac{2t^2}{O(n\cdot (n^{27})^2+n\cdot (n^{27})^2+n^2\cdot (n^{26})^2+n\cdot (n^{27})^2}\right)\\
&\leq 2\exp\left(-\Omega(\Phi^2\cdot n^{-55})\right)=n^{-\omega(1)},
\end{align*}
where we have used that $1/n\llpoly p_{\mathrm{vx}},p_{\mathrm{col}},p_{\mathrm{edge}}$. Therefore, taking a union bound, we have that \ref{prop:B3:11} holds with high probability.
\end{proof}


\subsection{Part~\ref{partB1}: Matching into $X_i$}\label{sec:partB1}
In Part~\ref{partB1}, we will find matchings $\hat{M}_{i,1}$, $i\in [n]$, satisfying the following properties.
\stepcounter{propcounter}
\begin{enumerate}[label = {{\textbf{\Alph{propcounter}\arabic{enumi}}}}]
\item For each $i\in [n]$, $\hat{M}_{i,1}$ is a rainbow matching from $T_i$ into $X_i$ which covers $T_i$ and has colours in $D_1\setminus C_i$.\label{prop:B1final:2}
\item The matchings $\hat{M}_{i,1}$, $i\in [n]$, are edge-disjoint and their edges are all in $E^\abs$.\label{prop:B1final:3}
\item $G^\abs|_{D_1}-\hat{M}_{1,1}-\hat{M}_{2,1}-\ldots -\hat{M}_{n,1}$ is $(4\gamma n)$-bounded.\label{prop:B1final:1}
\end{enumerate}
For this, define a 4-partite 4-uniform hypergraph $\cH_1$ with vertex classes

\begin{equation}\label{eqn:H1vx}
\begin{array}{ll}
\textbf{i)}\;\; \mathcal{V}_T:=\bigcup_{i\in [n]}(\{i\}\times T_i)
&\;\;\textbf{ii)}\;\;  \mathcal{V}_X:=\bigcup_{i\in [n]}(\{i\}\times X_{i,0})\\
\textbf{iii)}\;\; \; \mathcal{C}_1:=\bigcup_{i\in [n]}(\{i\}\times (D_1\setminus (C_i\cup D_{1,i}))
&\;\;\textbf{iv)}\;\; \mathcal{E}_1:=E(G|_{D_1})\cap E^{\mathrm{abs}}_0
\end{array}
\end{equation}
where, for each $i\in [n]$, and each edge $uv\in E^{\mathrm{abs}}_0$ with colour $c\in D_1\setminus (C_i\cup D_{1,i})$, with $u\in T_i$ and $v\in X_{i,0}$, we add the edge
\begin{equation}\label{eqn:edgestoH1}
\{(i,u),(i,v),(i,c),uv\}
\end{equation}
to $\mathcal{H}_1$.
Let
\[
\delta_1=\beta_0^3p_Xp_1p_\abs^2 n.
\]
We will now show that $\mathcal{H}_1$ is almost $\delta_1$-regular, as follows.

\begin{claim}\label{clm:H1almostregular}
For each $v\in V(\mathcal{H}_1)$, we have $d_{\cH_1}(v)=(1\pm \eps)\delta_1$.
\end{claim}
\begin{proof}[Proof of Claim~\ref{clm:H1almostregular}]
We check this for the vertices in each of the 4 classes in the order at \eqref{eqn:H1vx}.

\noindent\textbf{i)} Let $(i,u)\in \mathcal{V}_T$, so that $i\in [n]$ and $u\in T_i$. Then,
\[
d_{\mathcal{H}_1}((i,u))=|\{v\in X_{i,0}:c(uv)\in D_1\setminus (C_i\cup D_{1,i}),uv\in E^\abs_0\}|\overset{\ref{prop:B1:1}}{=} (1\pm\eps )\delta_1.
\]
\noindent\textbf{ii)} Let $(i,v)\in \mathcal{V}_X$, so that $i\in [n]$ and $v\in X_{i,0}$. Then, as $p_T=\beta_0p_X$,
\[
d_{\mathcal{H}_1}((i,v))=|\{u\in T_i:c(uv)\in D_1\setminus (C_i\cup D_{1,i}),uv\in E^\abs_0\}|\overset{\ref{prop:B1:2}}{=} (1\pm\eps )\delta_1.
\]
\noindent\textbf{iii)}
Let $(i,c)\in \mathcal{C}_1$, so that $i\in [n]$ and $c\in D_1\setminus (C_i\cup D_{1,i})$. Then, as $p_T=\beta_0p_X$ and $2p_X=p_\abs p_1$,
\[
d_{\mathcal{H}_1}((i,c))=|\{uv\in E_0^\abs:c(uv)=c,u\in T_i,v\in X_{i,0}\}|\overset{\ref{prop:B1:3}}{=} (1\pm\eps )\delta_1
\]
\noindent\textbf{iv)} Let $uv\in \mathcal{E}_1$, so that $c(uv)\in D_1$ and $uv\in E^\abs_0$. Then, as $p_T=\beta_0p_X$ and $2p_X=p_\abs p_1$,\renewcommand{\qedsymbol}{$\boxdot$}
\begin{align*}
d_{\mathcal{H}_1}(uv)&=|\{i\in [n]:u\in T_i,v\in X_{i,0},c(uv)\in D_1\setminus (C_i\cup D_{1,i})\}|\\
&\hspace{2cm}+|\{i\in [n]:u\in X_{i,0},v\in T_i,c(uv)\in D_1\setminus (C_i\cup D_{1,i})\}|\overset{\ref{prop:B1:4}}{=}(1\pm\eps)\delta_1.\qedhere
\end{align*}
\end{proof}
\renewcommand{\qedsymbol}{$\square$}

Moreover, $\mathcal{H}_1$ has codegrees at most 1, as follows.

\begin{claim}\label{clm:H1lowcod}
$\Delta^c(\cH_1)\leq 1$.
\end{claim}
\begin{proof}[Proof of Claim~\ref{clm:H1lowcod}] This follows as any edge $e=\{(i,u),(i,v),(i,c),uv\}\in E(\cH_1)$ is uniquely determined by any two of its vertices. Indeed, any two vertices from $e$ determines $i$ and at least 2 of $u$, $v$ and $c$.  Knowing $u$ and $v$ determines $c=c(uv)$ while knowing $c$ and $u$ determines $v$ (the neighbour of $u$ along a colour-$c$ edge), and, similarly, knowing $c$ and $v$ determines $u$. Thus, any two vertices from $e$ determines all of $i$, $u$, $v$, and $c$, and hence $e$.
\claimproofend

We now set up the weight functions we will use in the application of Theorem~\ref{thm:nibble} to $\mathcal{H}_1$.
For each $i\in [n]$, $v\in V(G)$, and $c\in D_1$, let $w_i$, $w_v^T$, $w_v^X$ and $w_c$ be the indicator function for whether an edge in $E(\cH_1)$ uses $i$, $v$ in the vertex from $\mathcal{V}_T$, $v$ in the vertex from $\mathcal{V}_X$, and $c$, respectively, i.e., for each $e=\{(i',u),(i',v'),(i',c'),uv'\}\in E(\cH_1)$ with $u\in T_i$ and $v'\in X_{i,0}$, we set
\[
w_i(e)=\mathbf{1}_{\{i'=i\}},\;\;\; w_v^T(e)=\mathbf{1}_{\{v=u\}},\;\;\; w_v^X(e)=\mathbf{1}_{\{v=v'\}}\;\;\text{and}\;\; w_c(e)=\mathbf{1}_{\{c'=c\}}.
\]
Furthermore, for each $v\in V(G)$, $c\in D_1$, and $\phi\in \mathcal{F}$, for each $e=\{(i',u),(i',v'),(i',c'),uv'\}\in E(\cH_1)$, let
\[
w_{v,\phi}^X(e)=\mathbf{1}_{\{v'=v,i'\in I_\phi\}}\;\;\;\text{ and }\;\;\; w_{c,\phi}(e)=\mathbf{1}_{\{c'=c,i'\in I_\phi\}}.
\]
Let $\mathcal{W}_1=\{w_i:i\in [n]\}\cup \{w_v^T:v\in V(G)\}\cup \{w_v^X:v\in V(G)\}\cup \{w_c:c\in D_1\}\cup \{w_{v,\phi}^X:v\in V(G),\phi\in \mathcal{F}\}\cup \{w_{c,\phi}:c\in C(G),\phi\in \mathcal{F}\}$.

For each $i\in [n]$, as each  $e\in E(\mathcal{H}_1)$ with $w_i(e)=1$ contains exactly one vertex $(i,u)$ with $u\in T_i$, we have, using Claim~\ref{clm:H1almostregular}, that
\begin{equation}\label{eqn:W1:itotalweight}
w_i(E(\mathcal{H}_1))=\sum_{u\in T_i}d_{\mathcal{H}_1}((i,u))=(1\pm \eps)\delta_1\cdot |T_i|.
\end{equation}
Similarly, for each $c\in D_1$, we have
\begin{equation}\label{eqn:W1:ctotalweight}
w_c(E(\mathcal{H}_1))=\sum_{uv\in \mathcal{E}_1:c(uv)=c}d_{\mathcal{H}_1}(uv)=(1\pm \eps)\delta_1\cdot |\{uv\in E_0^\abs:c(uv)=c\}|.
\end{equation}
Furthermore, for each $c\in D_1$ and $\phi\in \mathcal{F}$, we have
\begin{equation}\label{eqn:W1:cphitotalweight}
w_{c,\phi}(E(\mathcal{H}_1))=\sum_{i\in I_\phi:(i,c)\in V(\mathcal{H}_1)}d_{\mathcal{H}_1}((i,c))=(1\pm \eps)\delta_1\cdot |\{i\in I_\phi:c\notin C_i\cup D_{i,1}\}|.
\end{equation}
For each $v\in V(G)$, we have
\begin{equation}\label{eqn:W1:vtotalweight0}
w^T_v(E(\mathcal{H}_1))=\sum_{i\in [n]:v\in T_i}d_{\mathcal{H}_1}((i,u))=(1\pm \eps)\delta_1\cdot |\{i\in [n]:v\in T_i\}|,
\end{equation}
\begin{equation}\label{eqn:W1:vtotalweight}
w_v^X(E(\mathcal{H}_1))=\sum_{i\in [n]:v\in X_{i,0}}d_{\mathcal{H}_1}((i,u))=(1\pm \eps)\delta_1\cdot |\{i\in [n]:v\in X_{i,0}\}|,
\end{equation}
and, for each $\phi\in \mathcal{F}$,
\begin{equation}\label{eqn:W1:vphitotalweight}
w_{v,\phi}^X(E(\mathcal{H}_1))=\sum_{i\in I_\phi:v\in X_{i,0}}d_{\mathcal{H}_1}((i,u))=(1\pm \eps)\delta_1\cdot |\{i\in I_\phi:v\in X_{i,0}\}|.
\end{equation}
In particular, \eqref{eqn:W1:itotalweight}--\eqref{eqn:W1:vphitotalweight} imply that, for each $w\in \mathcal{W}_1$, $w(E(\mathcal{H}_1))\geq n^{3/2}$.
Therefore, by Claims~\ref{clm:H1almostregular},~\ref{clm:H1lowcod} and Theorem~\ref{thm:nibble}, we can find a matching $\cM_1$ in $\mathcal{H}_1$ such that, for each $w\in \mathcal{W}_1$,
\begin{equation}\label{eqn:M1:balanced}
w(\mathcal{M}_1)=(1\pm \gamma)\cdot\delta_1^{-1} w(E(\mathcal{H}_1)).
\end{equation}
For each $i\in [n]$, setting $\hat{M}_{i,1,0}=\{uv:\{(i,u),(i,v),(i,c(uv)),uv\}\in \mathcal{M}_1\}$, we have the following properties.

\begin{claim}\label{clm:M1props}
\begin{enumerate}[label = {{\textbf{\alph{enumi})}}}]\item The matchings $\hat{M}_{i,1,0}$, $i\in [n]$, are edge-disjoint.\label{prop:hatM:b}
\item For each $i\in [n]$, $\hat{M}_{i,1,0}$ is a rainbow matching from $T_i$ to $X_{i,0}$ with colours in $D_1\setminus (C_i\cup D_{1,i})$.\label{prop:hatM:a}
\item For each $i\in [n]$, $|T_i\setminus V(\hat{M}_{i,1,0})|\leq 2\gamma |T_i|$ and $|X_{i,0}\setminus V(\hat{M}_{i,1,0})|\leq \gamma n$.\label{prop:hatM:c}
\item For each $v\in V(G)$, $|\{i\in [n]:v\in T_i\setminus V(\hat{M}_{i,1,0})\}|\leq 2\gamma n$.\label{prop:hatM:cb}
\item For each $v\in V(G)$ and $\phi\in \mathcal{F}$, $|\{i\in I_\phi:v\in X_{i,0}\setminus V(\hat{M}_{i,1,0}\})|\leq 2\gamma p_\tr p_\fa n$.\label{prop:hatM:cbb}
\item For each $c\in D_1$ and $\phi\in \mathcal{F}$, $|\{i\in I_\phi:c\notin C(\hat{M}_{i,1,0})\cup C_i\cup D_{1,i}\}|\leq 2\gamma p_\tr p_\fa n$.\label{prop:hatM:cbbb}
\item Setting $G^\abs_0$ to be the graph with vertex set $V(G)$ and edge set $E^\abs_0$, we have $G^\abs_0|_{D_1}-\hat{M}_{1,1,0}-\hat{M}_{2,1,0}-\ldots -\hat{M}_{n,1,0}$ is $(3\gamma n)$-bounded.\label{prop:hatM:d}
\end{enumerate}
\end{claim}
\begin{proof} \ref{prop:hatM:b}: This follows as each edge $uv\in \mathcal{E}_1$ appears at most once in the edges of $\mathcal{M}_1$.

\smallskip

\noindent\ref{prop:hatM:a}: For each $i\in [n]$, that $\hat{M}_{i,1,0}$ is rainbow follows from the fact that the vertices $\{i\}\times \{c\}$, $c\in D_1\setminus (C_i\cup D_{i,1})$, appear at most once in the edges of the matching $\mathcal{M}_1$, while the other properties follow from the choice of the edges of $\mathcal{H}_1$ at \eqref{eqn:edgestoH1}.

\smallskip

\noindent\ref{prop:hatM:c}: For each $i\in [n]$,
\[
|T_i\cap V(\hat{M}_{i,1,0})|=w_i(\mathcal{M}_1)\overset{\eqref{eqn:M1:balanced}}{=}(1\pm \gamma)\cdot\delta_1^{-1} w_i(E(\mathcal{H}_1))\overset{\eqref{eqn:W1:itotalweight}}{=}(1\pm 2\gamma)|T_i|,
\]
so that $|T_i\setminus \hat{M}_{i,1,0}|\leq 2\gamma |T_i|\leq \gamma n$. Then,
\[
|X_{i,0}\setminus V(\hat{M}_{i,1,0})|=|T_i\setminus \hat{M}_{i,1,0}|+|X_{i,0}|-|T_i|\overset{\ref{prop:B1:lastfew:2}}{\leq} \gamma n+(1+\eps)\beta_0p_Xn-(1-\eps)p_Tn\leq 2\gamma n,
\]
where we have used \ref{prop:B1:lastfew:2} and that $p_X=(1+\beta)p_T=p_T/\beta_0$.

\smallskip

\noindent\ref{prop:hatM:cb}, \ref{prop:hatM:cbb}: For each $v\in V(G)$,
\begin{align*}
|\{i\in [n]:v\in T_i\setminus \hat{M}_{i,1,0}\}|&=|\{i\in [n]:v\in T_i\}|-w_v^T(\mathcal{M}_1)\\
&\overset{\eqref{eqn:M1:balanced}}{\leq}|\{i\in [n]:v\in T_i\}|-(1-\gamma)\cdot\delta_1^{-1} w_v^T(E(\mathcal{H}_1))\\
&\overset{\eqref{eqn:W1:vtotalweight0}}{\leq}2\gamma|\{i\in [n]:v\in T_i\}|\leq 2\gamma n,
\end{align*}
and therefore \ref{prop:hatM:cb} holds.
Similarly, using instead $w_{v,\phi}^X$, \eqref{eqn:M1:balanced} and \eqref{eqn:W1:vphitotalweight}, we have that \ref{prop:hatM:cbb} holds.

\smallskip

\noindent\ref{prop:hatM:cbbb}: For each $c\in D_1$ and $\phi\in \mathcal{F}$,
\begin{align*}
|\{i\in I_\phi:c\notin C(\hat{M}_{i,1,0})\cup C_i\cup D_{1,i}\}|&=|\{i\in I_\phi:c\in D_1\setminus (C_i\cup D_{1,i})\}|-w_{c,\phi}(\mathcal{M}_1)\\
&\overset{\ref{prop:B1:lastfew:1},\eqref{eqn:M1:balanced}}{\leq}(1+ \eps)\beta_0 p_\abs p_\tr p_\fa n-(1-\gamma)\cdot\delta_1^{-1} w_{c,\phi}(E(\mathcal{H}_1))\\
&\overset{\eqref{eqn:W1:cphitotalweight}}{\leq}2p_\tr p_\fa \gamma n,
\end{align*}
and therefore \ref{prop:hatM:cbbb} holds.

\smallskip

\noindent\ref{prop:hatM:d}: For each $c\in D_1$,
\begin{align*}
|\{e\in E(\hat{M}_{1,1,0}\cup \ldots \cup \hat{M}_{n,1,0}):c(e)=e\}|&=w_c(\mathcal{M}_1)\overset{\eqref{eqn:M1:balanced}}{=}(1\pm \gamma)\cdot\delta_1^{-1} w_c(E(\mathcal{H}_1))
\\
&\overset{\eqref{eqn:W1:ctotalweight}}{=}(1\pm 2\gamma)\cdot |\{uv\in E_0^\abs:c(uv)=c\}|.
\end{align*}
Thus, the number of edges of colour $c$ in $G^\abs_0|_{D_1}-\hat{M}_{1,1,0}-\hat{M}_{2,1,0}-\ldots -\hat{M}_{n,1,0}$  is at most $2\gamma\cdot |\{uv\in E_0^\abs:c(uv)=c\}|\leq 2\gamma n$.

Furthermore, for each $v\in V(G)$, we have that
\begin{align}
|\{e\in E(\hat{M}_{1,1,0}\cup \ldots \cup \hat{M}_{n,1,0}):v\in V(e)\}|&=w_v^T(\mathcal{M}_1)+w_v^X(\mathcal{M}_1)\\
&\overset{\eqref{eqn:M1:balanced}}{=}(1\pm \gamma)\cdot\delta_1^{-1} (w_v^T(E(\mathcal{H}_1)v+w_v^X(E(\mathcal{H}_1))\nonumber
\\
&\overset{\eqref{eqn:W1:vtotalweight0},\eqref{eqn:W1:vtotalweight}}{=}(1\pm 2\gamma)|\{i\in [n]:v\in T_i\text{ or }v\in X_{i,0}\}|.\label{eqn:M1:coversvertices}
\end{align}
Thus, the degree of $v$ in $G^\abs|_{D_1}-\hat{M}_{1,1,0}-\hat{M}_{2,1,0}-\ldots -\hat{M}_{n,1,0}$  is at most
\begin{align*}
|\{uv\in E_0^\abs:&c(uv)\in D_1\}|-|\{e\in E(\hat{M}_{1,1,0}\cup \ldots \cup \hat{M}_{n,1,0}):v\in V(e)\}|\\
&\overset{\ref{prop:B1:bounded:1},\eqref{eqn:M1:coversvertices}}{\leq} (1+ \eps)\beta_0 p_1p_\abs n-
(1- 2\gamma)|\{i\in [n]:v\in T_i\text{ or }v\in X_{i,0}\}|\\
&\overset{\ref{prop:B1:bounded:2}}{\leq} (1+ \eps)\beta_0 p_1p_\abs n-
(1- 2\gamma)(1- \eps)2p_Tn\\
&\leq 3\gamma \beta_0 p_1p_\abs n\leq 2\gamma n,
\end{align*}
where we have used that $2p_T=\beta_0 p_1p_\abs$.
Therefore, \ref{prop:hatM:d} holds.
\claimproofend

We will now find matchings $\hat{M}_{i,1,1}$, $i\in [n]$, to cover the uncovered vertices in $T_i\setminus \hat{M}_{i,1,0}$, using from Claim~\ref{clm:M1props} \ref{prop:hatM:c}, that these are small sets.

For this, take matchings $\hat{M}_{i,1,1}$, $i\in [n]$, which maximise $\sum_{i\in [n]}|\hat{M}_{i,1,1}|$ subject to the following properties.
\begin{enumerate}[label = \textbf{\arabic{enumi})}]
\item For each $i\in [n]$, $\hat{M}_{i,1,1}$ is a rainbow matching between $T_i\setminus V(\hat{M}_{i,1,0})$ and $X_{i,1}$ of edges in $E_1^{\abs}$ with colour in $D_{1,i}\setminus C_i$.\label{prop:forB1follow1}
\item For each $i\in [n]$, and $xy\in \hat{M}_{i,1,1}$ with $x\in T_i$, if $X\in \{A,B\}$ is such that $x\in X$, then $xy\in E_{1,X}^\abs$.\label{prop:interim}
\item The matchings  $\hat{M}_{i,1,1}$, $i\in [n]$, are edge-disjoint from each other.\label{prop:forB1follow2}
\end{enumerate}

We will show that the matchings  $\hat{M}_{i,1,1}$, $i\in [n]$, have the properties we need, as follows.

\begin{claim}\label{clm:mi01big} For each $i\in [n]$,
$|\hat{M}_{i,1,1}|=|T_i|-|\hat{M}_{i,1,0}|$.
\end{claim}
\begin{proof}[Proof of Claim~\ref{clm:mi01big}]
Suppose to the contrary, there is some $i\in [n]$ with $|\hat{M}_{i,1,1}|\neq|T_i|-|\hat{M}_{i,1,0}|$ so that, from \ref{prop:forB1follow1}, $|\hat{M}_{i,1,1}|<|T_i|-|\hat{M}_{i,1,0}|$, and there is some $u\in T_i\setminus V(\hat{M}_{i,1,0})$.
Suppose that $u\in A$, where the case where $u\in B$ follows similarly. By the maximality of $\sum_{i\in [n]}|\hat{M}_{i,1,1}|$, every neighbouring edge from $u$ to $X_{i,1}$ with colour in $D_{1,i}\setminus C_i$ must have its non-$u$ vertex in $V(\hat{M}_{i,1,1})$ or its colour in $C(\hat{M}_{i,1,1})$, or be in a matching $\hat{M}_{j,1,1}$ for some $j\in [n]$ with $u\in T_j\setminus V(\hat{M}_{j,1,0})$.
Therefore, using Claim~\ref{clm:M1props} \ref{prop:hatM:c} and \ref{prop:hatM:cb},
\begin{align*}
|\{v\in X_{i,1}:c(uv)\in D_{1,i}\setminus C_i,uv\in E^{\mathrm{abs}}_{1,X}\}|&\leq 2|\hat{M}_{i,1,1}|+|\{j\in [n]:u\in T_j\setminus \hat{M}_{i,1,0}\}|\leq 4\gamma n,
\end{align*}
which, as $\gamma\llpoly p_x,p_\abs$, contradicts \ref{prop:B1:follow:1}.
\claimproofend

For each $i\in [n]$, let $\hat{M}_{i,1,1}=\{uv_{i,u}:u\in T_i\setminus V(\hat{M}_{i,1,0})\}$ and set $\hat{M}_{i,1}=\hat{M}_{i,1,0}\cup \hat{M}_{i,1,1}$. We show that $\hat{M}_{i,1}$, $i\in [n]$, satisfy \ref{prop:B1final:2}--\ref{prop:B1final:1}.
Indeed, \ref{prop:B1final:2} follows from Claim~\ref{clm:M1props} \ref{prop:hatM:a}, Claim~\ref{clm:mi01big} and \ref{prop:forB1follow1}.
Furthermore, \ref{prop:B1final:3} follows from Claim~\ref{clm:M1props} \ref{prop:hatM:b} and \ref{prop:forB1follow2}.
Finally, \ref{prop:B1final:1} follows from Claim~\ref{clm:M1props} \ref{prop:hatM:d} and \ref{prop:B1:Gabsbounded}.


\subsection{Part~\ref{partB2}: Small matchings with $Y_i$}\label{sec:partB2}
We now embark on Part~\ref{partB2}, where we will find edge-disjoint rainbow matchings $\hat{M}_{i,2}$, $i\in [n]$, and a set
\[
\mathcal{J}\subset \{\{(i,u),(j,v)\}:i,j\in [n],i\neq j,u,v\in V(G),u\neq v,u,v\in Y_{i}\cap Y_{j}\},
\]
satisfying a range of properties later stated as \ref{prop:fromB2:1}--\ref{prop:forBfinalfromB2:4}. To do this,  as discussed in Section~\ref{sec:Bsketch}, we will find, for each $i\in [n]$ and $u\in S_i\setminus R_i$ a tuple $(v_{i,u}, c_{i,u},M_{i,u},\omega_{i,u})$, so that these tuples will have various desirable properties, including that $M_{i,u}$ is a small monochromatic matching in $G$ with vertices in $Y_i$ and the colour of its edges will be in $D_2$, and $uv_{i,u}$ is a disjoint edge of the same colour with $v_{i,u}$ also in $Y_i$. Recalling that $r=24$, this matching will have $r$ edges.
We will then divide, for each $i\in [n]$ and $u\in S_i\setminus R_i$, the edges $\{uv_{i,u}\cup M_{i,u}\}$ among the matchings $\hat{M}_{j,2}$, $j\in \{i\}\cup J_{i,u}$, while taking, for each $\{(i,u),(j,v)\}$ three particular pairs $\{(i,u'),(j,v')\}$ with $u',v'\in V(M_{i,u})\cup V(M_{j,v}) \cup \{v_{i,u}, v_{j,v}\}$ into $\mathcal{J}$ (see Section~\ref{sec:part2choosefinally}).

For each $i\in [n]$ and $u\in S_i\setminus R_i$, recall that $J_{i,u}$ is the set of $j$ for which there is some $v$ with $\{(i,u),(j,v)\}\in \mathcal{I}$, that $|J_{i,u}|=r$, and that $Y_{i,u,0}=\bigcap_{j\in J_{i,u}\cup \{i\}}Y_{j,0}$.
For each $i\in [n]$ and $u\in S_i\setminus R_i$, let $\mathcal{R}_{i,u}$ be the set of tuples $(v,M,c,\omega)$, where
\begin{itemize}
\item $M\cup \{uv\}$ is a matching of $r+1$ edges in $E^\abs_0$ with colour $c$ and $\omega:J_{i,u}\to M$ is a bijection,
\item $v\in Y_{i,u,0}=\bigcap_{j\in J_{i,u}\cup \{i\}}Y_{j,0}$,
\item $V(M)\subset Y_{i,0}$,
\item for each $j\in J_{i,u}$, $V(\omega(j))\subset Y_{j,0}$, and
\item for each $j\in \{i\}\cup J_{i,u}$, $c\in D_2\setminus (C_j\cup D_{2,j})$.
\end{itemize}

Define an auxiliary hypergraph $\cH_2$ with 4 vertex classes
\begin{equation}\label{eqn:H2vx}
\begin{array}{ll}
\textbf{i)}\;\; \mathcal{V}_{S-R}:=\cup_{i\in [n]}(\{i\}\times (S_i\setminus R_i))
&\;\;\textbf{ii)}\;\;  \mathcal{V}_Y:=\cup_{i\in [n]}(\{i\}\times Y_{i,0})\\
\textbf{iii)}\;\;  \mathcal{C}_2:=\cup_{i\in [n]}(\{i\}\times (D_{2}\setminus (C_i\cup D_{2,i})))
&\;\;\textbf{iv)}\;\; \mathcal{E}_2:=E(G|_{D_2})\cap E^{\mathrm{abs}}_0
\end{array}
\end{equation}
where, for each $i\in [n]$, $u\in S_i\setminus R_i$, and $(v,M,c,\omega)\in \mathcal{R}_{i,u}$, we add the edge
\begin{equation}\label{eq:edgesofH2}
E_{(i,u,v,M,c,\omega)}:=\{(i,u)\}\cup \left((J_{i,u}\cup \{i\})\times \{c,v\}\right)\cup \left(\cup_{j\in J_{i,u}}\{i,j\}\times V(\omega(j))\right)
\cup M\cup \{uv\}
\end{equation}
to $\mathcal{H}_2$. Each edge, then, has $1+(r+1)\cdot 2+r\cdot 4+r+1=7r+4$ vertices, and, hence, $\mathcal{H}_2$ is a $(7r +4)$-uniform hypergraph. We will now show, in Sections~\ref{sec:H2:almostregular} and~\ref{sec:H2:lowcod} respectively, that $\mathcal{H}_2$ is almost regular with low codegrees.


\subsubsection{Vertex degrees of $\mathcal{H}_2$}\label{sec:H2:almostregular}
Setting
\[
\delta_2=\beta_0^{7r+3}p_2p_Y^{2r+1}p_\abs^{2r+2} n^{r+1},
\]
we will show that $\mathcal{H}_2$ is almost $\delta_2$-regular, as follows.

\begin{claim}\label{clm:H2almostregular}
For each $v\in V(\mathcal{H}_2)$, we have $d_{\cH_2}(v)=(1\pm 100\eps)\cdot \delta_2$.
\end{claim}
\begin{proof}[Proof of Claim~\ref{clm:H2almostregular}]
We check this for the vertices in each of the 4 classes in the order at \eqref{eqn:H2vx}.

\noindent\textbf{i)} Let $(i,u)\in \mathcal{V}_{S-R}$, so that $i\in [n]$ and $u\in S_i\setminus R_i$. Then, by \ref{cond:newforv1} and \ref{cond:newforcedge1}, we have
\[
d_{\mathcal{H}_2}((i,u))=  (1\pm \eps)\cdot \beta_0^{2r+3}p_2p_Yp_\abs^{r+2} n\cdot  \left((1\pm 2\eps)\cdot (\beta_0^5p_Y^2p_{\abs}n\right)^r{=} (1\pm 100\eps )\cdot \delta_2.
\]

\smallskip


\noindent\textbf{ii)} Let $(j,x)\in \mathcal{V}_Y$, so that $j\in [n]$ and $x\in Y_{j,0}$. There are four different ways $(j,x)$ can arise in an edge $E_{(i,u,v,M,c,\omega)}$ for some $i\in [n]$, $u\in S_{i}\setminus R_{i}$, and $(v,M,c,\omega)\in \mathcal{R}_{i,u}$:
\textbf{a)} $j=i$ and $x=v$, \textbf{b)} $j=i$ and $x\neq v$, \textbf{c)} $j\neq i$ and $x=v$, and \textbf{d)} $j\neq i$ and $x\neq v$. We count these in turn.

\smallskip

\textbf{a)} $j=i$ and $x=v$:
Let $i=j$. Pick $u\in S_i\setminus R_i$ such that $c(ux)\in D_{2}\setminus (C_i\cup D_{2,i})$, $ux\in E^\abs_0$, and, for each $j\in J_{i,u}$, $x\in Y_{j,0}$ and $c(ux)\in D_{2}\setminus (C_j\cup D_j)$
(with $(1\pm \eps)\beta_0^{2r+3}p_2p_\abs^{r+1} n$ choices by \ref{prop:forB2:deg:1}). Then, iteratively pick $r$ edges of colour $c$ in $E^\abs_0$ disjointly from within, respectively, $Y_{i,0}\cap Y_{j',0}$ for each $j'\in J_{i,u}$ (each time having $(1\pm 2\eps)\cdot \beta_0^5p_Y^2p_{\abs}n$ choices by \ref{cond:newforcedge1}). Thus, recalling certain relationships between variables from Section \ref{sec:variables}, as $\beta_0p_Y=121p_{S-R}$, the total number of choices is
\[
(1\pm \eps)\cdot \beta_0^{2r+2}p_{S-R}p_2p_\abs^{r+2} n\cdot \left((1\pm 2\eps)\cdot \beta_0^5p_Y^2p_{\abs}n\right)^r{=} (1\pm 100\eps)\cdot \delta_2/121.
\]

\smallskip

\textbf{b)} $j=i$ and $x\neq v$: Let $i=j$. Pick $u\in S_i\setminus R_i$ and $j'\in J_{i,u}$ such that $x\in Y_{j',0}$
 (with $(1\pm \eps)2\beta_0p_{S-R} rn$ choices by \ref{prop:forB2:deg:2}).
Pick $c$ for which $c\in D_2\setminus (C_{i'}\cup D_{2,i'})$ for each $i'\in J_{i,u}\cup \{i\}$ and such that there is a colour-$c$ edge in $E^\abs_0$ from $u$ to $Y_{i,u,0}$ and from $x$ to $Y_{i,0}\cap Y_{j',0}$
(with $(1\pm 2\eps)\beta_0^{2r+6}p_2p_Y^2p_\abs^{r+3}n$ choices by \ref{prop:forB2:deg:3}).
Then, iteratively pick $r-1$ edges of colour $c$ disjointly from within, respectively, $Y_{i,0}\cap Y_{i',0}$ for each $i'\in J_{i,u}\setminus\{j'\}$ (each time having $(1\pm 2\eps)\beta_0^5p_Y^2p_{\abs}n$  choices by \ref{cond:newforcedge1}). Thus, as $\beta_0p_Y=121p_{S-R}$, the total number of choices is
\[
(1\pm \eps)\cdot 2\beta_0p_{S-R} rn\cdot (1\pm 2\eps)\cdot \beta_0^{2r+6}p_2p_Y^2p_\abs^{r+3}n\cdot \left((1\pm 2\eps)\cdot \beta_0^5p_Y^2p_{\abs}n\right)^{r-1}{=} (1\pm 100\eps )\cdot \delta_2\cdot 2r/121.
\]

\smallskip

\textbf{c)} $j\neq i$ and $x= v$: Pick $i\in [n]$ and $u\in S_i\setminus R_i$ such that $j\in J_{i,u}$, $x\in Y_{j',0}$ for each  $j'\in J_{i,u}\cup \{i\}$, $ux$ is in $E_0^\abs$ and has colour, $c$ say, in $D_2\setminus (C_{j'}\cup D_{2,j'})$ for each $j'\in Y_{i,u}\cup \{i\}$
(with $(1\pm \eps)r\beta_0^{2r+2}p_{S-R}p_2p_\abs^{r+2} n$ choices by \ref{prop:forB2:deg:4}). Then, iteratively pick $r$ edges of colour $c$ in $E^\abs_0$ disjointly from within, respectively, $Y_{i,0}\cap Y_{j',0}$ for each $j'\in J_{i,u}$ (each time having $(1\pm 2\eps)\beta_0^5p_Y^2p_{\abs}n$ choices by \ref{cond:newforcedge1}). Thus, as $\beta_0p_Y=121p_{S-R}$, the total number of choices is
\[
(1\pm \eps)r\beta_0^{2r+2}p_{S-R}p_2p_\abs^{r+2} n\cdot \left((1\pm 2\eps)\beta_0^5p_Y^2p_{\abs}n\right)^r{=} (1\pm 100\eps)\cdot \delta_2\cdot r/121.
\]

\smallskip

\textbf{d)} $j\neq i$ and $x\neq v$: Pick $i\in [n]$ and $u\in S_i\setminus R_i$ such that $j\in J_{i,u}$ and $x\in Y_{i,0}$ (with $(1\pm \eps)2\beta_0p_{S-R} rn$ choices by \ref{prop:forB2:deg:5}).
Pick $c$ for which $c\in D_2\setminus (C_{i'}\cup D_{2,i'})$ for each $i'\in J_{i,u}\cup \{i\}$ and such that there is a colour-$c$ edge in $E^\abs_0$ from $u$ to $Y_{i,u,0}$ and from $x$ to $Y_{i,0}\cap Y_{j,0}$
(with $(1\pm 2\eps)\beta_0^{2r+6}p_2p_Y^2p_\abs^{r+3}n$ choices by \ref{prop:forB2:deg:3}).
Then, iteratively pick $r-1$ edges of colour $c$ disjointly from within, respectively, $Y_{i,0}\cap Y_{i',0}$ for each $i'\in J_{i,u}\setminus\{j'\}$ (each time having $(1\pm 2\eps)\beta_0^5p_Y^2p_{\abs}n$  choices by \ref{cond:newforcedge1}). Thus, as $\beta_0p_Y=121p_{S-R}$, the total number of choices is
\[
(1\pm \eps)\cdot 2r\beta_0p_{S-R} n\cdot (1\pm 2\eps)\cdot \beta_0^{2r+6}p_2p_Y^2p_\abs^{r+3}n\cdot \left((1\pm 2\eps)\cdot \beta_0^5p_Y^2p_{\abs}n\right)^{r-1}{=} (1\pm 100\eps )\cdot \delta_2\cdot 2r/121.
\]

\smallskip

Therefore, recalling that $r=24$, in total for \textbf{ii)}, for each $(j,x)\in \mathcal{V}_Y$ we have
\[
d_{\mathcal{H}_2}((j,x))= (1\pm 100\eps)\cdot \frac{5r+1}{121}\cdot \delta_2 {=} (1\pm 100\eps )\cdot \delta_2.
\]

\smallskip


\noindent\textbf{iii)}
Let $(j,c)\in \mathcal{C}_2$, so that $j\in [n]$ and $c\in D_{2}\setminus (C_j\cup D_{2,j})$. There are two different ways $(j,c)$ can arise in an edge $E_{(i,u,v,M,c,\omega)}$ for some $i\in [n]$, $u\in S_{i}\setminus R_{i}$, and $(v,M,c,\omega)\in \mathcal{R}_{i,u}$:
\textbf{a)} $j=i$ and \textbf{b)} $j\neq i$. We count these in turn.

\smallskip

\textbf{a)} $j=i$. Pick $u\in S_i\setminus R_i$ so that $c\notin C_{j'}\cup D_{2,j'}$ for each $j'\in J_{i,u}$ and $u$ has a colour-$c$ edge in $E^\abs_0$ to $Y_{i,u,0}$
(with $(1\pm \eps)2\beta_0^{2r+2}p_{S-R}p_Yp_\abs^{r+1} n$ choices by \ref{prop:forB2:deg:7}).
Then, iteratively pick $r$ edges of colour $c$ in $E^\abs_0$ disjointly from within, respectively, $Y_{i,0}\cap Y_{j',0}$ for each $j'\in J_{i,u}$ (each time having $(1\pm 2\eps)\beta_0^5p_Y^2p_{\abs}n$ choices by \ref{cond:newforcedge1}). Thus, as $\beta_0p_2p_\abs=50p_{S-R}$, the total number of choices is
\[
(1\pm \eps)\cdot 2\beta_0^{2r+2}p_{S-R}p_Yp_\abs^{r+1} n\cdot \left((1\pm 2\eps)\cdot \beta_0^5p_Y^2p_{\abs}n\right)^r{=} (1\pm 100\eps )\cdot \delta_2\cdot 2/50.
\]

\smallskip

\textbf{b)} $j\neq i$. Pick $i\in [n]$ and $u\in S_i\setminus R_i$ such that $j\in J_{i,u}$ and $c\in D_2\setminus (C_{j'}\cup D_{2,j'})$ for each $j'\in J_{i,u}\cup \{i\}$ and $u$ has a colour-$c$ edge in $E^\abs_0$ to $Y_{i,u,0}$
(with $(1\pm \eps)2r\beta_0^{2r+2}p_{S-R}p_Yp_\abs^{r+1} n$ choices by \ref{prop:forB2:deg:8}).
Then, iteratively pick $r$ edges of colour $c$ disjointly from within, respectively, $Y_{i,0}\cap Y_{i',0}$ for each $i'\in J_{i,u}$ (each time having $(1\pm 2\eps)\beta_0^5p_Y^2p_{\abs}n$ choices by \ref{cond:newforcedge1}). Thus, as $\beta_0p_2p_\abs=50p_{S-R}$, the total number of choices is
\[
(1\pm \eps)\cdot 2r\beta_0^{2r+2}p_{S-R}p_Yp_\abs^{r+1} n\cdot \left((1\pm 2\eps)\cdot \beta_0^5p_Y^2p_{\abs}n\right)^{r}{=} (1\pm 100\eps )\cdot \delta_2\cdot 2r/50.
\]

\smallskip

Therefore, in total for \textbf{iii)}, for each $(j,c)\in \mathcal{C}_2$ we have
\[
d_{\mathcal{H}_2}((j,c))= (1\pm 100\eps)\cdot \frac{2r+2}{50}\cdot \delta_2 {=} (1\pm 100\eps )\cdot \delta_2.
\]


\noindent\textbf{iv)} Let $xy\in \mathcal{E}_2$, so that $c(xy)\in D_{2}$ and $xy\in E^\abs_0$.
There are two different ways $xy$ can arise in an edge $E_{(i,u,v,M,c,\omega)}$ for some $i\in [n]$, $u\in S_{i}\setminus R_{i}$, and $(v,M,c,\omega)\in \mathcal{R}_{i,u}$:
\textbf{a)} $xy=uv$ and \textbf{b)} $xy\in M$. We count these in turn.

\smallskip

\textbf{a)} $xy=uv$. Pick $i\in [n]$ such that $x\in S_i\setminus R_i$, $c(xy)\in D_2\setminus (C_{j}\cup D_{2,j})$ for each $j\in J_{i,u}\cup \{i\}$ and $y\in Y_{i,u,0}$
(with $(1\pm \eps)\beta_0^{2r+2}p_{S-R}p_Yp_\abs^{r+1} n$ choices by \ref{prop:forB2:deg:9}).
Then, iteratively pick $r$ edges of colour $c$ in $E^\abs_0$ disjointly from within, respectively, $Y_{i,0}\cap Y_{j,0}$ for each $j\in J_{i,u}$ (each time having $(1\pm 2\eps)\beta_0^5p_Y^2p_{\abs}n$ choices by \ref{cond:newforcedge1}).
Thus, as $\beta_0p_2p_\abs=50p_{S-R}$, the total number of choices is
\[
(1\pm\eps)\cdot\beta_0^{2r+2}p_{S-R}p_Yp_\abs^{r+1}n\cdot \left((1\pm \eps)\cdot \beta_0^5p_Y^2p_{\abs}n\right)^r{=} (1\pm 100\eps )\cdot \delta_2/50.
\]
Counting similarly with $x$ and $y$ interchanged, we get another $(1\pm 100\eps)\cdot \delta_2/50$ choices.

\smallskip

\textbf{b)} $xy\in M$. Pick $i\in [n]$, $u\in S_i\setminus R_i$ and $j\in J_{i,u}$ such that $c(xy)\in D_2\setminus (C_{j'}\cup D_{2,j'})$ for each $j'\in J_{i,u}\cup \{i\}$, $x,y\in Y_{i,0}\cap Y_{j,0}$ and $u$ has a colour-$c(xy)$ neighbour in $E_0^\abs$ in $Y_{i,u,0}$
(with $(1\pm \eps)2r\beta_0^{2r+7}p_{S-R}p_Y^3p_\abs^{r+2} n^2$ choices by \ref{prop:forB2:deg:10}).
Then, iteratively pick $r-1$ edges of colour $c$ in $E^\abs_0$ disjointly from within, respectively, $Y_{i,0}\cap Y_{j',0}$ for each $j'\in J_{i,u}\setminus \{j\}$ (each time having $(1\pm 2\eps)\beta_0^5p_Y^2p_{\abs}n$ choices by \ref{cond:newforcedge1}). Thus, as $\beta_0p_2p_\abs=50p_{S-R}$, the total number of choices is
\[
(1\pm \eps)\cdot 2r\beta_0^{2r+7}p_{S-R}p_Y^3p_\abs^{r+2}n^2\cdot \left((1\pm 2\eps)\cdot \beta_0^5p_Y^2p_{\abs}n\right)^{r-1}{=} (1\pm 100\eps )\cdot \delta_2\cdot 2r/50.
\]

Therefore,  in total for \textbf{iv)}, for each $xy\in \mathcal{E}_2$ we have\renewcommand{\qedsymbol}{$\boxdot$}
\[
d_{\mathcal{H}_2}(xy)= (1\pm 100\eps)\cdot \frac{1+1+2r}{50}\cdot \delta_2 {=} (1\pm 100\eps )\cdot \delta_2.\qedhere
\]
\end{proof}
\renewcommand{\qedsymbol}{$\square$}


\subsubsection{Codegrees of $\mathcal{H}_2$}\label{sec:H2:lowcod}
We will now show that the codegrees of $\mathcal{H}_2$ are all $O(n^{r+0.5})$, where $r=24$.
As the vertex degrees of $\mathcal{H}_2$ are (by Claim~\ref{clm:H2almostregular}) all around $\delta_2$, where $\delta_2=\beta_0^{7r+3}p_2p_Y^{2r+1}p_\abs^{2r+2} n^{r+1}$, and $1/n\llpoly \beta_0,p_2,p_Y,p_\abs$, these codegrees are all much smaller than the vertex degrees in $\mathcal{H}_2$.

\begin{claim}\label{clm:H2lowcod}
$\Delta^c(\cH_2)=O(n^{r+0.5})$.
\end{claim}
\begin{proof}[Proof of Claim~\ref{clm:H2lowcod}] Let
$i\in [n]$, $u\in S_i\setminus R_i$, and $(v,M,c,\omega)\in \mathcal{R}_{i,u}$, and consider the edge $e=E_{(i,u,v,M,c,\omega)}$ (using \eqref{eq:edgesofH2}), so that
\[
e=\{(i,u)\}\cup \left((J_{i,u}\cup \{i\})\times \{c,v\}\right)\cup \left(\cup_{j\in J_{i,u}}\{i,j\}\times V(\omega(j))\right)
\cup M\cup \{uv\}.
\]
Let ${v}_1$ and $v_2$ be two vertices in $e$. We will check the codegree in cases \textbf{i)}--\textbf{xvii)} as follows.

If from ${v}_1$ and $v_2$ we do not know $i$ or any $j\in J_{i,u}$, then (as $v_1,v_2\in M\cup \{uv\}$) from ${v}_1$ and $v_2$ we can write down a triple
\[
\text{\textbf{i)}\;\;}(u,w,c) \;\;\;\text{ or }\;\;\; \text{\textbf{ii)}\;\;}(w,w',c)
\]
where $w,w'\in V(M)$ are not in the same edge in $M$. If from ${v}_1$ and $v_2$ we know $i$ but no $j\in J_{i,u}$, then (as knowing any 2 of $u,v$ and $c$ determines all of them) from ${v}_1$ and $v_2$ we can write down one of the triples
\[
\text{\textbf{iii)}\;}(i,u,c) \;\;\;\text{\textbf{iv)}\;}(i,u,w) \;\;\; \text{\textbf{v)}\;}(i,v,w) \;\;\; \text{\textbf{vi)}\;}(i,c,w) \;\;\; \text{\textbf{vii)}\;}(i,w,w'),
\]
where $w,w'\in V(M)$ are not in the same edge in $M$. If from ${v}_1$ and $v_2$ we know $i$ and one $j\in J_{i,u}$, then from ${v}_1$ and $v_2$ we can write down one of the triples
\[
\text{\textbf{viii)}\;}(i,j,v) \;\;\;\text{\textbf{ix)}\;}(i,j,c) \;\;\;
\text{\textbf{x)}\;}(i,j,w),
\]
where $w\in V(M)$. If from ${v}_1$ and $v_2$ we do not know $i$ and know exactly one $j\in J_{i,u}$, then from ${v}_1$ and $v_2$ we can write down one of the triples
\[
\text{\textbf{xi)}\;}(j,v,c) \;\;\;\text{\textbf{xii)}\;}(j,v,w) \;\;\;\text{\textbf{xiii)}\;}(j,c,w) \;\;\;
\text{\textbf{xiv)}\;}(j,w,w'),
\]
where $w,w'\in V(M)$ are not in the same edge in $M$. Finally, if from ${v}_1$ and $v_2$ we know distinct $j,j'\in J_{i,u}$, then from ${v}_1$ and $v_2$ we can write down one of the triples
\[
\text{\textbf{xv)}\;}(j,j',v) \;\;\;\text{\textbf{xvi)}\;}(j,j',c) \;\;\;\text{\textbf{xvii)}\;}(j,j',w),
\]
where $w\in V(M)$.

We now show that, in each case \textbf{i)}--\textbf{xvii)}, we have $d_{\mathcal{H}_3}(v_1,v_2)=O(n^{r+0.5})$.
For \textbf{i)} and \textbf{ii)}, we know the colour of the edges in $M\cup \{uv\}$ and vertices from 2 different edges, so there are at most $n^{r-1}$ choices for the other edges in $M\cup \{uv\}$, and then at most $n$ choices for $i$, $O(1)$ choices for $u$ and $O(1)$ choices for $\omega$, so that $d_{\mathcal{H}_3}(v_1,v_2)=O(n^r)=O(n^{r+0.5})$.

For \textbf{iii)} and \textbf{vi)}, we know $i$, $c$ and one vertex in one of the edges in $M\cup \{uv\}$, and therefore there are at most $n^r$ ways to choose the remaining edges of $M\cup\{uv\}$, after which there are $O(1)$ choices for $u$ and then $\omega$, so that $d_{\mathcal{H}_3}(v_1,v_2)=O(n^r)=O(n^{r+0.5})$.
For \textbf{iv)}, \textbf{v)} and \textbf{vii)}, we know $i$ and two vertices in different edges of $M\cup\{uv\}$, and therefore, after choosing $c$ with at most $n$ choices, there are then at most $n^{r-1}$ ways to choose the remaining edges of $M\cup\{uv\}$, after which there are $O(1)$ choices for $u$ and then $\omega$, so that $d_{\mathcal{H}_3}(v_1,v_2)=O(n^r)=O(n^{r+0.5})$.

For \textbf{xi)}, we know $u$ from $v$ and $c$, and therefore from \ref{prop:abs:lowcodegree3} have at most $\sqrt{n}$ choices for $i$, after which there are at most $n^r$ possibilities for $M$, so that $d_{\mathcal{H}_3}(v_1,v_2)=O(n^{r+0.5})$. For \textbf{xii)}--\textbf{xiv)}, as $j$ is known, there are at most $n^{1.5}$ choices for $(i,u)$ by \ref{prop:abs:lowcodegree1}, after which we know either the colour of the edges in $M\cup \{uv\}$ and two vertices from different edges in $M\cup\{uv\}$, or and three vertices from different edges in $M\cup\{uv\}$.
In either case, we have at most $n^{r-1}$ choices for the edges in $M\cup \{uv\}$, and therefore  $d_{\mathcal{H}_3}(v_1,v_2)=O(n^{r+0.5})$.

For \textbf{xv)}--\textbf{xvii)}, by \ref{prop:abs:lowcodegree2}, there are at most $\sqrt{n}$ choices for $(i,u)$ with $j,j'\in J_{i,u}$, after which, as either the colour of the edges in $M\cup\{uv\}$ is known (from the triple directly or from $uv$) or a vertex in an edge of $M$, there are at most $n^r$ choices for edges $M\cup \{uv\}$, and therefore $d_{\mathcal{H}_3}(v_1,v_2)=O(n^{r+0.5})$.
\claimproofend


\subsubsection{Weight functions for properties for Part~\ref{partB3}}\label{sec:complicatedweight}

Before considering some of the simpler weight functions needed to complete Part~\ref{partB2}, we address some important weight functions that are needed to prove certain properties required to complete Part~\ref{partB3}. In particular, these relate to properties \ref{prop:fromB2:1}-\ref{prop:fromB2:10} which are only stated later, in Section \ref{sec:part2choosefinally}, and are important for proving the properties required of $\mathcal{J}$ for $\mathcal{H}_3$.

For \ref{prop:fromB2:1}: For each $j\in [n]$ and $x\in Z_{j,0}$, define $w^{\ref{prop:fromB2:1}:\mathrm{same}}_{j,x},w^{\ref{prop:fromB2:1}:\mathrm{other}}_{j,x}:E(\mathcal{M}_2)\to \mathbb{N}$ by,
for each $i\in [n]$, $u\in S_i\setminus R_i$, and $(v,M,c,\omega)\in \mathcal{R}_{i,u}$ 
\begin{itemize}
\item if $j=i$, then letting $w^{\ref{prop:fromB2:1}:\mathrm{same}}_{j,x}(E_{(i,u,v,M,c,\omega)})$ be the number of $j'\in J_{i,u}$ with $x\in Z_{j',0}$ such that there is an edge from $x$ to $V(\omega(j'))$ in $E_0^\abs$ with colour in $D_3\setminus (C_i\cup C_{j'}\cup D_{3,i}\cup D_{3,j'})$, and 0 otherwise.
\item $j\in J_{i,u}$, then let $w^{\ref{prop:fromB2:1}:\mathrm{other}}_{j,x}(E_{(i,u,v,M,c,\omega)})$ be 1 if there is an edge from $x$ to $v$ in $E_0^\abs$ with colour in $D_3\setminus (C_i\cup C_j\cup D_{3,i}\cup D_{3,j})$, and 0 otherwise.
\end{itemize}

\ifabbrev\else
For $w^{\ref{prop:fromB2:1}:\mathrm{same}}_{j,x}(E(\mathcal{H}_2))$: There are $(1\pm \eps)\cdot 2r\beta_0p_{S-R}n$ choices for $u\in S_j\setminus R_j$ and $j'\in J_{j,u}$ with $x\in Z_{j',0}$ by \ref{prop:B2:wgt:1}. After this, there are $(1\pm \eps)\beta_0^{2r+3}p_Yp_2p_\abs^{r+2} n$ choices of $v\in Y_{j,u,0}$ with $uv\in E^\abs_0$ and $c(uv)\in D_{2}\setminus (C_{i'}\cup D_{2,i'})$ for each $i'\in J_{j,u}\cup \{j\}$ by \ref{cond:newforv1}.
After this, there are $(1\pm \eps)\beta_0^8p_Y^2p_3p_\abs^4n$ choices for $\omega(j')$
such that there is an edge from $x$ to $V(\omega(j'))$ in $E_0^\abs$ with colour in $D_3\setminus (C_j\cup C_{j'}\cup D_{3,j}\cup D_{3,j'})$ by \ref{prop:B2:wgt:2}.
After this, there are $(1\pm \eps)\beta_0^5p_Y^2p_{\abs}n$ choices for each edge in $M\setminus \{\omega(j')\}$ by \ref{cond:newforcedge1}.

Therefore, in total, we have
\begin{align*}
w^{\ref{prop:fromB2:1}:\mathrm{same}}_{j,x}(E(\mathcal{H}_2))&=(1\pm 10\eps)\cdot 2r\beta_0p_{S-R}n\cdot \beta_0^{2r+3}p_Yp_2p_\abs^{r+2} n
\cdot
\beta_0^8p_Y^2p_3p_\abs^2n\cdot \left(\beta_0^5p_Y^2p_{\abs}n\right)^{r-1}
\\
&=(1\pm 10\eps)2(p_{S}-p_R)n\cdot r\cdot \delta_2\cdot \beta_0^4\cdot \beta_0p_\abs^3\cdot p_3
\\
&=(1\pm 10\eps)\cdot \frac{2}{3}p_{\mathcal{J}}n\cdot \delta_2\cdot \beta_0\cdot p_{\mathrm{edge}}\cdot p_{\mathrm{col}}.
\end{align*}

For $w^{\ref{prop:fromB2:1}:\mathrm{other}}_{j,x}(E(\mathcal{H}_2))$: There are $(1\pm \eps)\cdot r\beta_0p_{S-R}n$ choices for $i\in [n]$ and $u\in S_i\setminus R_i$ with $j\in J_{i,u}$, $x\in Z_{i,0}$ and $u\sim_{A/B}x$ by \ref{prop:B2:wgt:3}.
After this, there are $(1\pm \eps)\beta_0^{2r+6}p_Yp_2p_3p_\abs^{r+5} n$ choices of $v\in Y_{i,u,0}$ with $uv\in E^\abs_0$ and  $c(uv)\in D_{2}\setminus (C_{i'}\cup D_{2,i'})$ for each $i'\in J_{j,u}\cup \{j\}$ and $xv$ is an edge of $E^\abs_0$ with colour in $D_3\setminus (C_i\cup C_j\cup D_{3,i}\cup D_{3,j})$ by \ref{prop:B2:wgt:4}.
After this, there are $(1\pm \eps)\beta_0^5p_Y^2p_{\abs}n$ choices for each edge in $M$ by \ref{cond:newforcedge1}.

Therefore, in total, we have
\begin{align*}
w^{\ref{prop:fromB2:1}:\mathrm{other}}_{j,x}(E(\mathcal{H}_2))&=(1\pm 10\eps)\cdot 2r\beta_0p_{S-R}n\cdot \beta_0^{2r+4}p_Yp_2p_3p_\abs^{r+5} n\cdot \left(\beta_0^5p_Y^2p_{\abs}n\right)^{r}
\\
&=(1\pm 10\eps)(p_{S}-p_R)n\cdot r\cdot \delta_2\cdot \beta_0^4\cdot \beta_0p_\abs^3\cdot p_3
\\
&=(1\pm 10\eps)\cdot \frac{1}{3}p_{\mathcal{J}}n\cdot \delta_2\cdot \beta_0\cdot p_{\mathrm{edge}}\cdot p_{\mathrm{col}}.
\end{align*}
\fi
Letting $w^{\ref{prop:fromB2:1}}_{j,x}=w^{\ref{prop:fromB2:1}:\mathrm{same}}_{j,x}+w^{\ref{prop:fromB2:1}:\mathrm{other}}_{j,x}$, we therefore have that
\begin{equation}\label{eqn:W2:complicated1:totalweight}
w^{\ref{prop:fromB2:1}}_{j,x}(E(\mathcal{H}_2))=(1\pm 10\eps)p_{\mathcal{J}}n\cdot \delta_2\cdot \beta_0\cdot p_{\mathrm{edge}}\cdot p_{\mathrm{col}}.
\end{equation}


\medskip

For \ref{prop:fromB2:2}: For each $j\in [n]$ and $x\in Z_{j,0}$, define $w^{\ref{prop:fromB2:2}:\mathrm{same}}_{j,x},w^{\ref{prop:fromB2:2}:\mathrm{other}}_{j,x}:E(\mathcal{M}_2)\to \mathbb{N}$ by,
for each $i\in [n]$, $u\in S_i\setminus R_i$, and $(v,M,c,\omega)\in \mathcal{R}_{i,u}$ 
\begin{itemize}
\item if $j=i$, then let $w^{\ref{prop:fromB2:1}:\mathrm{same}}_{j,x}(E_{(i,u,v,M,c,\omega)})$ be the number of $j'\in J_{i,u}$ for which there is an edge from $x$ to $v$ in $E_0^\abs$ with colour in $D_3\setminus (C_i\cup C_j\cup D_{3,i}\cup D_{3,j})$, and 0 otherwise.
\item if $j\in J_{i,u}$,  let $w^{\ref{prop:fromB2:1}:\mathrm{other}}_{j,x}$ be 1 if there is an edge from $x$ to $V(\omega(j))$ in $E_0^\abs$ with colour in $D_3\setminus (C_i\cup C_{j}\cup D_{3,i}\cup D_{3,j})$, and 0 otherwise.
\end{itemize}

\ifabbrev\else
For $w^{\ref{prop:fromB2:2}:\mathrm{same}}_{j,x}(E(\mathcal{H}_2))$: There are $(1\pm \eps)\cdot r\beta_0p_{S-R}n$ choices for $u\in S_j\setminus R_j$ and $j'\in J_{j,u}$ with $x\sim_{A/B} u$ and $x\in Z_{j',0}$ by \ref{prop:B2:wgt:5}, which is a fraction around $1/r$ of the possibilities in the first choice for examining $w^{\ref{prop:fromB2:1}:\mathrm{other}}_{j,x}$.
Then, continuing as did for $w^{\ref{prop:fromB2:1}:\mathrm{other}}_{j,x}$, using \ref{prop:B2:wgt:4} and \ref{cond:newforcedge1}, we have, as $w^{\ref{prop:fromB2:2}:\mathrm{same}}_{j,x}$ takes values in $\{0,r\}$, that,
\begin{align*}
w^{\ref{prop:fromB2:2}:\mathrm{same}}_{j,x}(E(\mathcal{H}_2))&=(1\pm 10\eps)\cdot \frac{1}{3}p_{\mathcal{J}}n\cdot \delta_2\cdot \beta_0\cdot p_{\mathrm{edge}}\cdot p_{\mathrm{col}}.
\end{align*}

For $w^{\ref{prop:fromB2:2}:\mathrm{other}}_{j,x}(E(\mathcal{H}_2))$: There are $(1\pm \eps)\cdot 2r\beta_0p_{S-R}n$ choices for $i\in [n]$ and $u\in S_i\setminus R_i$ with $j\in J_{i,u}$ and $x\in Z_{i,0}$ by \ref{prop:B2:wgt:6}, which is around twice as many as the possibilities in the first choice for examining $w^{\ref{prop:fromB2:1}:\mathrm{same}}_{j,x}$.
Continuing as did for $w^{\ref{prop:fromB2:1}:\mathrm{same}}_{j,x}$, using \ref{cond:newforv1}, \ref{prop:B2:wgt:2} and  \ref{cond:newforcedge1}, we get
\begin{align*}
w^{\ref{prop:fromB2:2}:\mathrm{other}}_{j,x}(E(\mathcal{H}_2))&=(1\pm 10\eps)\cdot \frac{2}{3}p_{\mathcal{J}}n\cdot \delta_2\cdot \beta_0\cdot p_{\mathrm{edge}}\cdot p_{\mathrm{col}}.
\end{align*}
\fi
Letting $w^{\ref{prop:fromB2:2}}_{j,x}=w^{\ref{prop:fromB2:2}:\mathrm{same}}_{j,x}+w^{\ref{prop:fromB2:2}:\mathrm{other}}_{j,x}$, we therefore have that
\begin{equation}\label{eqn:W2:complicated2:totalweight}
w^{\ref{prop:fromB2:2}}_{j,x}(E(\mathcal{H}_2))=(1\pm 10\eps)\cdot p_{\mathcal{J}}n\cdot \delta_2\cdot \beta_0\cdot p_{\mathrm{edge}}\cdot p_{\mathrm{col}}.
\end{equation}


For \ref{prop:fromB2:6}: For each $j\in [n]$ and $c'\in D_3\setminus (C_j\cup D_{3,j})$, define $w^{\ref{prop:fromB2:6}:\mathrm{same}}_{j,c'},w^{\ref{prop:fromB2:6}:\mathrm{other}}_{j,c'}:E(\mathcal{M}_2)\to \mathbb{N}$ by,
for each $i\in [n]$, $u\in S_i\setminus R_i$, and $(v,M,c,\omega)\in \mathcal{R}_{i,u}$ 
\begin{itemize}
\item if $j=i$, then letting $w^{\ref{prop:fromB2:1}:\mathrm{same}}_{j,c}(E_{(i,u,v,M,c,\omega)})$ be the number of pairs $(j',w)$ with $j'\in J_{i,u}$ and $w\in V(\omega(j'))$ for which $c'\notin C_{j'}\cup D_{3,j'}\}$
 and $w$ has a colour-$c'$ edge to $Z_{i,0}\cap Z_{j',0}$ in $E_0^\abs$, and 0 otherwise, and
\item if $j\in J_{i,u}$, then let $w^{\ref{prop:fromB2:1}:\mathrm{other}}_{j,c'}$ be 1 if $c'\notin C_{i}\cup D_{3,i}$ and there is a colour-$c'$ edge from $v$ to $Z_{i,0}\cap Z_{j,0}$ in $E_0^\abs$, and 0 otherwise.
\end{itemize}

\ifabbrev\else
For $w^{\ref{prop:fromB2:6}:\mathrm{same}}_{j,c'}(E(\mathcal{H}_2))$: There are $2\beta_0 p_\abs r\cdot p_{S-R}n$ choices of $u\in S_j\setminus R_j$ and $j'\in J_{j,u}$ with $c'\notin C_{j'}\cup D_{3,j'}$ by \ref{prop:B2:wgt:7}.
After this, there are $(1\pm \eps)\beta_0^{2r+3}p_Yp_2p_\abs^{r+2} n$ choices of $v\in Y_{j,u,0}$ with $uv\in E^\abs_0$ and $c(uv)\in D_{2}\setminus (C_{i'}\cup D_{2,i'})$ for each $i'\in J_{j,u}\cup \{j\}$ by \ref{cond:newforv1}.
Then, by \ref{prop:B2:wgt:8}, there are $(1\pm \eps)2\beta_0^8p_Y^2p_Zp_\abs^2n$ choices for $w\in Y_{j,0}\cap Y_{j',0}$ which has a colour-$c'$ edge in $E_0^\abs$ to $Z_{j,0}\cap Z_{j',0}$ and a colour-$c(uv)$ edge in $E_0^\abs$ to $Y_{j,0}\cap Y_{j',0}$. As this determines the edge with colour $c=c(uv)$ in the matching $M$, we then have
$(1\pm \eps)\beta_0^5p_Y^2p_{\abs}n$ choices for each of the $r-1$ other edges in $M\setminus \{\omega(j')\}$ by \ref{cond:newforcedge1}.

Therefore, in total, we have
\begin{align*}
w^{\ref{prop:fromB2:6}:\mathrm{same}}_{j,c'}(E(\mathcal{H}_2))&=(1\pm 10\eps)\cdot
2\beta_0 p_\abs r\cdot p_{S-R}n
\cdot
\beta_0^{2r+3}p_Yp_2p_\abs^{r+2} n
\cdot
2\beta_0^8p_Yp_Zp_\abs^2n
\cdot
\left(\beta_0^5p_Y^2p_{\abs}n\right)^{r-1}
\\
&=(1\pm 10\eps)2(p_{S}-p_R)n\cdot r\cdot \delta_2\cdot \beta_0^4p_\abs^2\cdot p_Z
\\
&=(1\pm 10\eps)\cdot \frac{4}{3}{p}_{\mathrm{vx}}\cdot \hat{p}_{\mathrm{col}}\cdot p_{\mathrm{edge}}\cdot p_{\mathcal{J}}n\cdot \delta_2.
\end{align*}

For $w^{\ref{prop:fromB2:6}:\mathrm{other}}_{j,c'}(E(\mathcal{H}_2))$: There are $(1\pm \eps)\cdot 2r\beta_0p_\abs p_{S-R}n$ choices for $i\in [n]$ and $u\in S_i\setminus R_i$ with $j\in J_{i,u}$ and $c'\notin C_i\cup D_{3,i}$ by \ref{prop:B2:wgt:9}.
After this, there are $(1\pm \eps)\beta_0^{2r+6}p_Yp_Zp_2p_\abs^{r+3} n$ choices of $v\in Y_{i,u,0}$ with $uv\in E^\abs_0$ and $c(uv)\in D_{2}\setminus (C_{i'}\cup D_{2,i'})$ for each $i'\in J_{i,u}\cup \{i\}$ and there is a colour-$c'$ edge from $v$ to $Z_{i,0}\cap Z_{j,0}$ in $E_0^\abs$ by \ref{prop:B2:wgt:10}.
After this, there are $(1\pm \eps)\beta_0^5p_Y^2p_{\abs}n$ choices for each edge in $M$ by \ref{cond:newforcedge1}.

Therefore, in total, we have
\begin{align*}
w^{\ref{prop:fromB2:6}:\mathrm{other}}_{j,c'}(E(\mathcal{H}_2))&=(1\pm 10\eps)\cdot 2r\beta_0p_\abs p_{S-R}n\cdot \beta_0^{2r+6}p_Yp_Zp_2p_\abs^{r+3} n\cdot \left(\beta_0^5p_Y^2p_{\abs}n\right)^{r}
\\
&=(1\pm 10\eps)(p_{S}-p_R)n\cdot r\cdot \delta_2\cdot \beta_0^4p_\abs^2\cdot p_Z
\\
&=(1\pm 10\eps)\cdot \frac{2}{3}{p}_{\mathrm{vx}}\cdot \hat{p}_{\mathrm{col}}\cdot p_{\mathrm{edge}}\cdot p_{\mathcal{J}}n\cdot \delta_2.
\end{align*}
\fi
Letting $w^{\ref{prop:fromB2:6}}_{j,c'}=\frac{1}{2}w^{\ref{prop:fromB2:6}:\mathrm{same}}_{j,c'}+\frac{1}{2}w^{\ref{prop:fromB2:6}:\mathrm{other}}_{j,c'}$, we therefore have that
\begin{equation}\label{eqn:W2:complicated3:totalweight}
w^{\ref{prop:fromB2:6}}_{j,c'}(E(\mathcal{H}_2))=(1\pm 10\eps)\cdot p_{\mathcal{J}}n\cdot \delta_2\cdot \beta_0\cdot p_{\mathrm{edge}}\cdot p_{\mathrm{col}}.
\end{equation}


For \ref{prop:fromB2:7}: For each $j\in [n]$ and $c'\in D_3\setminus (C_j\cup D_{3,j})$, define $w^{\ref{prop:fromB2:7}:\mathrm{same}}_{j,c'},w^{\ref{prop:fromB2:7}:\mathrm{other}}_{j,c'}:E(\mathcal{M}_2)\to \mathbb{N}$ by,
for each $i\in [n]$, $u\in S_i\setminus R_i$, and $(v,M,c,\omega)\in \mathcal{R}_{i,u}$ 
\begin{itemize}
\item if $j=i$, then letting $w^{\ref{prop:fromB2:7}:\mathrm{same}}_{j,c}(E_{(i,u,v,M,c,\omega)})$ be
the number of $j'\in J_{i,u}$ for which $c'\notin C_{j'}\cup D_{3,j'}$ and there is a colour-$c'$ edge from $v$ to $Z_{j,0}\cap Z_{j',0}$ in $E_0^\abs$, and 0 otherwise, and
\item if $j\in J_{i,u}$, then letting $w^{\ref{prop:fromB2:7}:\mathrm{other}}_{j,c'}(E_{(i,u,v,M,c,\omega)})$ be the number of vertices in $\omega(i)$ with a colour-$c'$ edge in $E_0^\abs$ to $Z_{i,0}\cap Z_{j,0}$ if $c'\notin C_{i}\cup D_{3,i}$, and 0 otherwise.
\end{itemize}

\ifabbrev\else
For $w^{\ref{prop:fromB2:7}:\mathrm{same}}_{j,c'}(E(\mathcal{H}_2))$:
There are $2r\beta_0p_\abs p_{S-R}n$ choices of $u\in S_j\setminus R_j$ and $j'\in J_{j,u}$ for which $c'\notin C_{j'}\cup D_{3,j'}$ by \ref{prop:B2:wgt:7}.
After this, there are $(1\pm \eps)\beta_0^{2r+6}p_Yp_Zp_2p_\abs^{r+3} n$ choices of $v\in Y_{i,u,0}$ with $uv\in E^\abs_0$ and $c(uv)\in D_{2}\setminus (C_{i'}\cup D_{2,i'})$ for each $i'\in J_{i,u}\cup \{i\}$ and there is a colour-$c'$ edge from $v$ to $Z_{j,0}\cap Z_{j',0}$ in $E_0^\abs$ by \ref{prop:B2:wgt:10}.
After this, there are $(1\pm \eps)\beta_0^5p_Y^2p_{\abs}n$ choices for each edge in $M$ by \ref{cond:newforcedge1}.

Therefore, in total, we have
\begin{align*}
w^{\ref{prop:fromB2:7}:\mathrm{other}}_{j,c'}(E(\mathcal{H}_2))&=(1\pm 10\eps)\cdot 2r\beta_0p_\abs p_{S-R}n\cdot \beta_0^{2r+6}p_Yp_Zp_2p_\abs^{r+3} n\cdot \left(\beta_0^5p_Y^2p_{\abs}n\right)^{r}
\\
&=(1\pm 10\eps)2(p_{S}-p_R)n\cdot r\cdot \delta_2\cdot \beta_0^4p_\abs^2\cdot p_Z
\\
&=(1\pm 10\eps)\cdot \frac{2}{3}{p}_{\mathrm{vx}}\cdot \hat{p}_{\mathrm{col}}\cdot p_{\mathrm{edge}}\cdot p_{\mathcal{J}}n\cdot \delta_2.
\end{align*}

For $w^{\ref{prop:fromB2:7}:\mathrm{other}}_{j,c'}(E(\mathcal{H}_2))$:
There are $(1\pm \eps)\cdot 2r\beta_0p_\abs p_{S-R}n$ choices for $i\in [n]$ with $j\in J_{i,u}$ and $c'\notin C_i\cup D_{3,i}$ by \ref{prop:B2:wgt:9}.
After this, there are $(1\pm \eps)\beta_0^{2r+3}p_Yp_2p_\abs^{r+2} n$ choices of $v\in Y_{j,u,0}$ with $uv\in E^\abs_0$ and $c(uv)\in D_{2}\setminus (C_{i'}\cup D_{2,i'})$ for each $i'\in J_{j,u}\cup \{j\}$ by \ref{cond:newforv1}.
Then, by \ref{prop:B2:wgt:8}, there are $(1\pm \eps)2\beta_0^8p_Y^2p_Zp_\abs^2n$ choices for $w\in Y_{j,0}\cap Y_{j',0}$ which has a colour-$c'$ edge in $E_0^\abs$ to $Z_{j,0}\cap Z_{j',0}$ and a colour-$c(uv)$ edge in $E_0^\abs$ to $Y_{j,0}\cap Y_{j',0}$. As this determines the edge with colour $c=c(uv)$ in the matching $M$, we then have
$(1\pm \eps)\beta_0^5p_Y^2p_{\abs}n$ choices for each of the $r-1$ other edges in $M\setminus \{\omega(j')\}$ by \ref{cond:newforcedge1}.

Therefore, in total, we have
\begin{align*}
w^{\ref{prop:fromB2:7}:\mathrm{other}}_{j,c'}(E(\mathcal{H}_2))&=(1\pm 10\eps)\cdot
4\beta_0 p_\abs r\cdot p_{S-R}n
\cdot
\beta_0^{2r+3}p_Yp_2p_\abs^{r+2} n
\cdot
\beta_0^8p_Yp_Zp_\abs^2n
\cdot
\left(\beta_0^5p_Y^2p_{\abs}n\right)^{r-1}
\\
&=(1\pm 10\eps)4(p_{S}-p_R)n\cdot r\cdot \delta_2\cdot \beta_0^4p_\abs^2\cdot p_Z
\\
&=(1\pm 10\eps)\cdot \frac{4}{3}{p}_{\mathrm{vx}}\cdot \hat{p}_{\mathrm{col}}\cdot p_{\mathrm{edge}}\cdot p_{\mathcal{J}}n\cdot \delta_2.
\end{align*}

\fi
Letting $w^{\ref{prop:fromB2:7}}_{j,c'}=\frac{1}{2}w^{\ref{prop:fromB2:7}:\mathrm{same}}_{j,c'}+\frac{1}{2}w^{\ref{prop:fromB2:7}:\mathrm{other}}_{j,c'}$,
we therefore have that
\begin{equation}\label{eqn:W2:complicated4:totalweight}
w^{\ref{prop:fromB2:6}}_{j,c'}(E(\mathcal{H}_2))=(1\pm 10\eps){p}_{\mathrm{vx}}\cdot \hat{p}_{\mathrm{col}}\cdot p_{\mathrm{edge}}\cdot p_{\mathcal{J}}n\cdot \delta_2.
\end{equation}


For \ref{prop:fromB2:9}: For each $xy\in E_0^\abs$ with $c(xy)\in D_3$, define $w^{\ref{prop:fromB2:9}:\mathrm{first}}_{xy},w^{\ref{prop:fromB2:9}:\mathrm{match}}_{xy}:E(\mathcal{M}_2)\to \mathbb{N}$ by,
for each $i\in [n]$, $u\in S_i\setminus R_i$, and $(v,M,c,\omega)\in \mathcal{R}_{i,u}$ 
\begin{itemize}
\item if $v\in \{x,y\}$, $y\in Z_{i,0}$ and $c(xy)\notin C_i\cup D_{3,i}$, then letting $w^{\ref{prop:fromB2:9}:\mathrm{first}}_{xy}(E_{(i,u,v,M,c,\omega)})$ be the number of $j\in J_{i,u}$ for which $y\in Z_{j,0}$ and
$c(xy)\notin C_j\cup D_{3,j}$, and 0 otherwise, and
\item if $v\notin \{x,y\}$, $x,y\in Z_{i,0}$ and $c(xy)\notin C_i\cup D_{3,i}$, then letting $w^{\ref{prop:fromB2:9}:\mathrm{match}}_{xy}(E_{(i,u,v,M,c,\omega)})$ be the number of $j\in J_{i,u}$ with $x,y\in Z_{j,0}$ and
$c(xy)\notin C_j\cup D_{3,j}$, and 0 otherwise.
\end{itemize}

\ifabbrev\else
For $w^{\ref{prop:fromB2:9}:\mathrm{first}}_{xy}(E(\mathcal{H}_2))$: There are $(1\pm\eps)\beta_0^{2r+7}p_Zp_Yp_2p_\abs^{r+4}rp_{S-R}n^2$ choices for $i\in [n]$,  $u\in S_i\setminus R_i$ and $j\in J_{i,u}$ with $x\in Y_{i,u,0}$, $ux\in E_0^\abs$, $c(ux)\in D_2\setminus (C_{j'}\cup D_{2,j'})$ for each $j'\in J_{i,u}$, $y\in Z_{i,0}\cap Z_{j,0}$ and $c(xy)\notin C_i\cup C_j\cup D_{3,i}\cup D_{3,j}$ by \ref{prop:B2:wgt:11}.
After this, there are $(1\pm\eps)\beta_0^5p_Y^2p_{\abs}n$ choices for each edge in $M$ by \ref{cond:newforcedge1}.

Therefore, in total, counting similarly with $x$ and $y$ switched, we have
\begin{align*}
w^{\ref{prop:fromB2:1}:\mathrm{first}}_{xy}(E(\mathcal{H}_2))&=(1\pm 10\eps)\cdot 2\beta_0^{2r+7}p_Zp_Yp_2p_\abs^{r+4}rp_{S-R}n^2
\cdot \left(\beta_0^5p_Y^2p_{\abs}n\right)^{r}
\\
&=(1\pm 10\eps)2(p_{S}-p_R)n\cdot r\cdot \delta_2\cdot \beta_0^4\cdot p_\abs^2 p_Z
\\
&=(1\pm 10\eps)\cdot \frac{2}{3}{p}_{\mathrm{vx}}^2\cdot p_3\beta_0^4p_\abs^4\cdot p_{\mathrm{edge}}\cdot p_{\mathcal{I}}\cdot n^2.
\end{align*}

For $w^{\ref{prop:fromB2:9}:\mathrm{match}}_{xy}(E(\mathcal{H}_2))$: There are $(1\pm \eps)\cdot r\beta_0^{6}p_\abs^2p_Yp_Zp_{S-R}n^2$ choices for $i\in [n]$, $u\in S_i\setminus R_i$ and $j\in J_{i,u}$ with $c(xy)\notin (C_i\cup C_j\cup D_{3,i}\cup D_{3,j})$, $x\in Y_{i,0}\cap Y_{j,0}$, $y\in Z_{i,0}\cap Z_{j,0}$ by \ref{prop:B2:wgt:12}.
After this, there are $(1\pm \eps)\beta_0^{2r+6}p_Yp_2p_\abs^{r+2} n$ choices of $v\in Y_{i,u,0}$ with $uv\in E^\abs_0$ and $c(uv)\in D_{2}\setminus (C_{i'}\cup D_{2,i'})$ for each $i'\in J_{i,u}\cup \{i\}$ such that $x$ has a colour-$c(uv)$ edge in $E_0^\abs$ to $Y_{i,0}\cap Y_{j,0}$ by \ref{prop:B2:wgt:13}.
After this, there are $(1\pm \eps)\beta_0^5p_Y^2p_{\abs}n$ choices for each of the remaining $r-1$ edges in $M$ by \ref{cond:newforcedge1}.
Therefore, in total, we have
\begin{align*}
w^{\ref{prop:fromB2:9}:\mathrm{match}}_{xy}(E(\mathcal{H}_2))&=(1\pm 10\eps)\cdot 2r\beta_0^{6}p_\abs^2p_Yp_Zp_{S-R}n^2
\cdot \beta_0^{2r+6}p_Yp_2p_\abs^{r+2} n\cdot
\left(\beta_0^5p_Y^2p_{\abs}n\right)^{r-1}
\\
&=(1\pm 10\eps)(p_{S}-p_R)n\cdot r\cdot \delta_2\cdot \beta_0^4\cdot p_\abs^2
\\
&=(1\pm 10\eps)\cdot \frac{1}{3}p_{\mathcal{J}}n\cdot \delta_2\cdot \beta_0\cdot p_{\mathrm{edge}}\cdot \beta_0^2\cdot p_\abs^2.
\end{align*}\fi
Letting
$w^{\ref{prop:fromB2:9}}_{xy}=w^{\ref{prop:fromB2:9}:\mathrm{first}}_{xy}+w^{\ref{prop:fromB2:9}:\mathrm{match}}_{xy}$, we therefore have that
\begin{equation}\label{eqn:W2:complicated5:totalweight}
w^{\ref{prop:fromB2:9}}_{xy}(E(\mathcal{H}_2))=(1\pm 10\eps){p}_{\mathrm{vx}}\cdot \beta_0^2\cdot p_\abs^2\cdot p_{\mathcal{J}}\cdot n.
\end{equation}

For \ref{prop:fromB2:10}: For each $xy\in E_0^\abs$ with $c(xy)\in D_3$, define $w^{\ref{prop:fromB2:10}:\mathrm{first}}_{xy},w^{\ref{prop:fromB2:10}:\mathrm{match}}_{xy}:E(\mathcal{M}_2)\to \mathbb{N}$ by,
for each $i\in [n]$, $u\in S_i\setminus R_i$, and $(v,M,c,\omega)\in \mathcal{R}_{i,u}$ 
\begin{itemize}
\item if $vx\in E^\abs_0$, $c(uv)\in D_3\setminus (C_i\cup D_{3,i})$, $x,y\in Z_i$, then letting $w^{\ref{prop:fromB2:10}:\mathrm{first}}_{xy}(E_{(i,u,v,M,c,\omega)})$ be the number of $j\in J_{i,u}$ for which $x,y\in Z_{j,0}$ and $c(xy)\notin C_j\cup D_{3,j}$, and 0 otherwise, and
\item if $x,y\in Z_{i,0}$ and $c(xy)\notin C_i\cup D_{3,i}$, then letting $w^{\ref{prop:fromB2:10}:\mathrm{match}}_{xy}(E_{(i,u,v,M,c,\omega)})$ be the number of $j\in J_{i,u}$ with $x,y\in Z_{j,0}$ and $c(xy)\notin C_j\cup D_{3,j}$ for which there is an edge in $E^\abs_0$ from $x$ to $\omega(j)$ with colour in $D_3\setminus (C_i\cup C_j\cup D_{3,i}\cup D_{3,j})$, and 0 otherwise.
\end{itemize}

\ifabbrev\else
For $w^{\ref{prop:fromB2:10}:\mathrm{first}}_{xy}(E(\mathcal{H}_2))$: There are $(1\pm\eps)\beta_0^{6}p_Z^2p_\abs^{2}rp_{S-R}n^2$ choices for $i\in [n]$,  $u\in S_i\setminus R_i$ and $j\in J_{i,u}$ with $x,y\in Z_{i,0}\cap Z_{j,0}$ and $c(xy)\notin C_i\cup C_j\cup D_{3,i}\cup D_{3,j}$ by \ref{prop:B2:wgt:14}.
After this, there are $(1\pm \eps)\beta_0^{2r+6}p_Yp_2p_3p_\abs^{r+5} n$ choices of $v\in Y_{i,u,0}$ with $uv\in E^\abs_0$ and  $c(uv)\in D_{2}\setminus (C_{i'}\cup D_{2,i'})$ for each $i'\in J_{j,u}\cup \{j\}$ and $xv$ is an edge of $E^\abs_0$ with colour in $D_3\setminus (C_i\cup C_j\cup D_{3,i}\cup D_{3,j})$ by \ref{prop:B2:wgt:4}.
After this, there are $(1\pm \eps)\beta_0^5p_Y^2p_{\abs}n$ choices for each edge in $M$ by \ref{cond:newforcedge1}.

Therefore, in total, counting similarly with $x$ and $y$ switched, we have
\begin{align*}
w^{\ref{prop:fromB2:1}:\mathrm{first}}_{xy}(E(\mathcal{H}_2))&=(1\pm 10\eps)\cdot \beta_0^{6}p_Z^2p_\abs^{2}rp_{S-R}n^2
\cdot \beta_0^{2r+6}p_Yp_2p_3p_\abs^{r+5}n
\cdot \left(\beta_0^5p_Y^2p_{\abs}n\right)^{r}
\\
&=(1\pm 10\eps)2(p_{S}-p_R)n\cdot r\cdot \delta_2\cdot \beta_0^9\cdot p_\abs^5 p_Z^2\cdot p_3
\\
&=(1\pm 10\eps)\cdot \frac{2}{3}{p}_{\mathrm{vx}}\cdot \beta_0^2\cdot p_\abs^2\cdot p_{\mathcal{J}}\cdot n.
\end{align*}

For $w^{\ref{prop:fromB2:10}:\mathrm{match}}_{xy}(E(\mathcal{H}_2))$: There are $(1\pm \eps)\cdot r\beta_0^{6}p_\abs^2p_Z^2p_{S-R}n^2$ choices for $i\in [n]$, $u\in S_i\setminus R_i$ and $j\in J_{i,u}$ with  $x,y\in Z_{i,0}\cap Z_{j,0}$ and $c(xy)\notin (C_i\cup C_j\cup D_{3,i}\cup D_{3,j})$, by \ref{prop:B2:wgt:14}.
After this, there are $(1\pm \eps)\beta_0^{2r+3}p_Yp_2p_\abs^{r+2} n$ choices of $v\in Y_{j,u,0}$ with $uv\in E^\abs_0$ and $c(uv)\in D_{2}\setminus (C_{i'}\cup D_{2,i'})$ for each $i'\in J_{j,u}\cup \{j\}$ by \ref{cond:newforv1}.
Then, by \ref{prop:B2:wgt:2}, there are $(1\pm \eps)\beta_0^8p_Y^2p_3p_\abs^4n$ choices for the edge $\omega(j)$.
After this, there are $(1\pm \eps)\beta_0^5p_Y^2p_{\abs}n$ choices for each of the remaining $r-1$ edges in $M$ by \ref{cond:newforcedge1}.
Therefore, in total, counting similarly with $x$ and $y$ switched, we have
\begin{align*}
w^{\ref{prop:fromB2:10}:\mathrm{match}}_{xy}(E(\mathcal{H}_2))&=(1\pm 10\eps)\cdot 2r\beta_0^{6}p_\abs^2p_Z^2p_{S-R}n^2
\cdot \beta_0^{2r+3}p_Yp_2p_\abs^{r+2} n
\cdot \beta_0^8p_Y^2p_3p_\abs^4n\cdot
\left(\beta_0^5p_Y^2p_{\abs}n\right)^{r-1}
\\
&=(1\pm 10\eps)2(p_{S}-p_R)n\cdot r\cdot \delta_2\cdot \beta_0^9\cdot p_\abs^5 p_Z^2\cdot p_3
\\
&=(1\pm 10\eps)\cdot \frac{2}{3}{p}_{\mathrm{vx}}\cdot \beta_0^2\cdot p_\abs^2\cdot p_{\mathcal{J}}\cdot n.
\end{align*}\fi
Letting $w^{\ref{prop:fromB2:10}}_{xy}=w^{\ref{prop:fromB2:10}:\mathrm{first}}_{xy}+w^{\ref{prop:fromB2:10}:\mathrm{match}}_{xy}$, we therefore have that
\begin{equation}\label{eqn:W2:complicated6:totalweight}
w^{\ref{prop:fromB2:10}}_{xy}(E(\mathcal{H}_2))=(1\pm 10\eps){p}_{\mathrm{vx}}\cdot \beta_0^2\cdot p_\abs^2\cdot p_{\mathcal{J}}\cdot n.
\end{equation}

Let
\begin{align}
\mathcal{W}_2=\{w^{\ref{prop:fromB2:1}}_{j,x},w^{\ref{prop:fromB2:2}}_{j,x}:j\in [n],x\in Z_{j,0}\}\cup
\{w_{j,c}^{\ref{prop:fromB2:6}},w_{j,c}^{\ref{prop:fromB2:7}}:&j\in [n],c\in D_3\setminus (C_j\cup D_{3,j})\}\nonumber\\
&\cup \{w_{xy}^{\ref{prop:fromB2:9}},w_{xy}^{\ref{prop:fromB2:10}}:xy\in E_0^{\abs}\text{ s.t.\ }c(xy)\in D_3\},\label{eq:defW2nonprime}
\end{align}
and note that from \eqref{eqn:W2:complicated1:totalweight}--\eqref{eqn:W2:complicated6:totalweight}, we have that $w(E(\mathcal{H}_2))\geq \sqrt{n}\delta_2$ for each $w\in \mathcal{W}_2$.


\subsubsection{Simpler weight functions for Part~\ref{partB2}}\label{sec:simplerweight}
We now define the weight functions that will allow us to control the edges in $E_0^\abs$ with colour in $D_2$ that we do no end up using in the matchings found in Part~\ref{partB2}.
For each $i\in [n]$, $\phi\in \mathcal{F}$, $v\in V(G)$, and $c\in D_2$, and each $i'\in [n]$, $u\in S_i\setminus R_i$, and $(v',M,c',\omega)\in R_{i',u}$, recalling the edge $E_{(v',M,c',\omega)}$ of $\mathcal{H}_2$ from \eqref{eq:edgesofH2}, we define
\[
w^{\mathrm{same}}_v(E_{(i',u,v',M,c',\omega)})=\mathbf{1}_{\{v=v'\}},\;\;\; w^{\mathrm{diff}}_v(E_{(i',u,v',M,c',\omega)})=\mathbf{1}_{\{v\in V(M)\}}
\]
\[
w^{\mathrm{same}}_{v,\phi}(E_{(i',u,v',M,c',\omega)})=\mathbf{1}_{\{v=v',i'\in I_\phi\}},\;\;\;w_{c,\phi}(E_{(i',u',v,M,c',\omega)})=\mathbf{1}_{\{c'=c,i'\in I_\phi\}}
\]
\[
w^{\mathrm{same}}_i(E_{(i',u,v',M,c',\omega)})=\mathbf{1}_{\{i=i'\}},\;\;\;\;\;\text{and}\;\;
w^{\mathrm{diff}}_i(E_{(i',u,v',M,c,\omega)})=\mathbf{1}_{\{i\in J_{i,u}\}},
\]
and, if $v\in Y_{i,0}$, we define
\[
w^{\textrm{cod}}_{i,v}(E_{(i',u,v',M,c',\omega)})=r\cdot \mathbf{1}_{\{i=i'\}}\cdot \mathbf{1}_{\{v=v'\}}+\mathbf{1}_{\{i\in J_{i,u}\}}\cdot \mathbf{1}_{\{v\in V(M)\}}.
\]

Let $\mathcal{W}'_2=\{w^{\mathrm{same}}_i,w^{\mathrm{diff}}_i:i\in [n]\}\cup \{w^{\mathrm{same}}_v,w_v^{\mathrm{diff}}:v\in V(G)\}\cup \{w_{c,\phi}:c\in D_2,\phi\in \mathcal{F}\}$, and let $\mathcal{W}''_2=\{w^{\textrm{cod}}_{i,v}:i\in [n],v\in Y_{i,0}\}$.

For each $i\in [n]$,
we have, using Claim~\ref{clm:H2almostregular}, that
\begin{equation}\label{eqn:W2:isametotalweight}
w_i^{\mathrm{same}}(E(\mathcal{H}_2))=\sum_{u\in S_i\setminus R_i}d_{\mathcal{H}_2}((i,u))=(1\pm \eps)|S_i\setminus R_i|\delta_2\overset{\ref{prop:B2:sizeofSminusR}}{=}(1\pm 2\eps)\cdot p_{S-R}n\cdot \delta_2.
\end{equation}
Furthermore,
\begin{equation}\label{eqn:W2:idifftotalweight}
w_i^{\mathrm{diff}}(E(\mathcal{H}_2))=|\{(i',u):i\in J_{i',u}\}|\cdot d_{\mathcal{H}_2}((i',u))=(1\pm \eps)\cdot r\cdot |S_i\setminus R_i|\delta_2\overset{\ref{prop:B2:sizeofSminusR}}{=}(1\pm 2\eps)\cdot r\cdot p_{S-R}n\cdot \delta_2.
\end{equation}
For each $v\in V(G)$,
using the proof of Claim~\ref{clm:H2almostregular}, and in particular counting part \textbf{ii)a)} over each possible $i\in [n]$ and $i\in I_\phi$, $\phi\in \mathcal{F}$, respectively, we have
\begin{equation}\label{eqn:W2:vsametotalweight}
w_v^{\mathrm{same}}(E(\mathcal{H}_2))=(1\pm \eps)\cdot n\cdot \delta_2/121,
\end{equation}
and, for each $\phi\in \mathcal{F}$,
\begin{equation}\label{eqn:W2:vsamewithtotalweight}
w_{v,\phi}^{\mathrm{same}}(E(\mathcal{H}_2))=(1\pm \eps)\cdot p_\tr p_\fa n\cdot \delta_2/121.
\end{equation}
Furthermore, using the proof of Claim~\ref{clm:H2almostregular}, and in particular counting part \textbf{ii)b)} over each possible $i\in [n]$, we have
\begin{equation}\label{eqn:W2:vdifftotalweight}
w_v^{\mathrm{diff}}(E(\mathcal{H}_2))=(1\pm \eps)\cdot n\cdot \delta_2\cdot 2r/121.
\end{equation}
Finally, for each $c\in D_2$ and $\phi\in \mathcal{F}$, from Claim~\ref{clm:H3almostregular}, and as each edge in $E(H_2)$ which uses $c$ contains $r+1$ different pairs $(i,c)$ for some $i$, we have
\begin{equation}\label{eqn:W2:ctotalweight}
w_{c,\phi}(E(\mathcal{H}_2))=\frac{1}{r+1}\sum_{i\in I_\phi:c\in D_2\setminus (C_i\cup D_{2,i})}d_{\mathcal{H}_2}((i,c))\overset{\ref{prop:B1:lastfew:3}}=(1\pm \eps)\cdot \frac{1}{r+1}\cdot \beta_0p_\abs p_\tr p_\fa n\cdot \delta_2.
\end{equation}

In particular, \eqref{eqn:W2:isametotalweight}--\eqref{eqn:W2:ctotalweight} imply that, for each $w\in \mathcal{W}_2'$, $w(E(\mathcal{H}_2))\geq n^{1/2}\cdot \delta_2$. Furthermore, for each $i\in [n]$ and $v\in Y_{i,0}$, we have, using Claim~\ref{clm:H2lowcod}, that
\begin{equation}\label{eq:forivcod}
w^{\textrm{cod}}_{i,v}(E(\mathcal{H}_2))\leq \sum_{u\in S_i\setminus R_i}\sum_{j\in J_{i,u}} |\{e\in E(\mathcal{H}_2):(i,u),(j,v)\in V(e)\}|=O(n\cdot r\cdot n^{r+0.5})\leq \delta_2\cdot n^{0.6}.
\end{equation}


\subsubsection{Choice of $\mathcal{M}_2$ and its properties}\label{sec:choiceofM2}

Recalling $\mathcal{W}_2$, $\mathcal{W}_2'$, $\mathcal{W}_2''$ from Sections~\ref{sec:complicatedweight} and \ref{sec:simplerweight}, we have shown that $w(E(\mathcal{H}_2))\geq n^{1/2}\cdot \delta_2$ holds for each $w\in \mathcal{W}_2\cup \mathcal{W}_2'$.
Therefore, by Claims~\ref{clm:H2almostregular},~\ref{clm:H2lowcod} and Theorem~\ref{thm:nibble}, we can find a matching $\cM_2$ in $\mathcal{H}_2$ such that, for each $w\in \mathcal{W}_2\cup \mathcal{W}_2'$,
\begin{equation}\label{eqn:M2:balanced}
w(\mathcal{M}_2)=(1\pm 200\eps)\cdot\delta_2^{-1} w(E(\mathcal{H}_2)),
\end{equation}
and, for each $w\in \mathcal{W}_2''$, we have\footnote{More formally, for each such $w \in \mathcal{W}_2''$ we add an arbitrary function $w'$ with total weight $n^{0.1}\cdot \delta_2$, say, to ensure that $(w+w')((E(\mathcal{H}_2)))\geq n^{0.1}$ before the application of Theorem~\ref{thm:nibble} using the function $w+w'$, from which we only take the upper bound for $w$.}
 $w(\mathcal{M}_2)\leq 2n^{0.6}$.

Let $J=\{(i,u)\in V(\mathcal{M}_2):i\in [n],u\in S_i\setminus R_i\}$. For each $(i,u)\in J$, let $(v_{i,u},M_{i,u},c_{i,u},\omega_{i,u})\in \mathcal{R}_{i,u}$ be the such that $E_{(i,u,v_{i,u},M_{i,u},c_{i,u},\omega_{i,u})}$ is the edge in $\mathcal{H}_2$ containing $(i,u)$. For each $(i,u)\in J$ and $j\in J_{i,u}$, let $x_{i,u,j}$ be the vertex of $\omega_{i,u}(j)$ with $x_{i,u,j}\simAB v_{i,u}$ and let $y_{i,u,j}$ be the vertex of $\omega_{i,u}(j)$ with $x_{i,u,j}\simAB v_{i,u}$.

For each $i\in [n]$, let
\begin{equation}\label{eqn:calFidefn}
\mathcal{F}_i=\left(\cup_{u\in S_i\setminus R_i:(i,u)\in J}\cup_{j\in J_{i,u}}\{(j,v_{i,u}),(j,y_{i,u,j})\}\right)\cup \left(\cup_{(j,u)\in J:i\in J_{j,u}}(j,x_{j,u,i})\right),
\end{equation}
which is a set of tuples which we will eventually add to $\mathcal{J}$ in a pair with $(i,v)$ for some $v\in V(G)$.
For each $i\in [n]$, let
\begin{equation}\label{eqn:calGidefn}
\mathcal{G}_i=\left(\cup_{u\in S_i\setminus R_i:(i,u)\in J}\cup_{j\in J_{i,u}}\{(j,x_{i,u,j})\}\right)\cup \left(\cup_{(j,u)\in J:i\in J_{j,u}}\{(j,v_{j,u}),(j,y_{j,u,i})\}\right)
\end{equation}
which is a set of pairs $(j,v)$ for which we will add some pair $\{(i,v),(j,v')\}$ to $\mathcal{J}$ for some $v'$.
For each $v\in V(G)$, let $\mathcal{K}_v$ be the multi-set
\begin{equation}\label{eqn:calKxdefn}
\left(\cup_{(i,u)\in J:v_{i,u}=v}\cup_{j\in J_{i,u}}(j,i)\}\right)\cup \left(\cup_{(i,u)\in J}\cup_{j\in J_{i,u}:x_{i,u,j}=v}(i,j)\right)
\cup \left(\cup_{(i,u)\in J}\cup_{j\in J_{i,u}:y_{i,u,j}=v}(j,i)\right).
\end{equation}

\begin{claim}\label{clm:propsfromBtwopart1} The following hold.

\stepcounter{propcounter}
\begin{enumerate}[label = {{\textbf{\Alph{propcounter}\arabic{enumi}}}}]

\item  \label{prop:fromB2:1:forJhat} For each $i\in [n]$ and $x\in Z_{i,0}$, there are $(1\pm 2\gamma)\beta_0\cdot p_{\mathrm{col}}\cdot p_{\mathrm{edge}}\cdot p_{\mathcal{J}}\cdot n$ choices for $(j,u)\in \mathcal{G}_i$ such that
 $x\in Z_{j,0}$, $c(ux)\in D_3\setminus (C_{i}\cup C_j\cup D_{3,i}\cup D_{3,j})$ and $ux\in E^\abs_0$.

\item  \label{prop:fromB2:2:forJhat} For each $i\in [n]$ and $x\in Z_{i,0}$, there are $(1\pm 2\gamma)\beta_0\cdot p_{\mathrm{col}}\cdot p_{\mathrm{edge}}\cdot p_{\mathcal{J}}\cdot n$ choices for $(j,v)\in \mathcal{F}_i$ for which $x\in Z_{j,0}$, $c(xv)\in D_3\setminus (C_{i}\cup C_j\cup D_{3,i}\cup D_{3,j})$ and $xv\in E^\abs_0$.

\item  \label{prop:fromB2:6:forJhat} For each $i\in [n]$ and $c\in D_3\setminus (C_i\cup D_{3,i})$, there are $(1\pm 2\gamma){p}_{\mathrm{vx}}\cdot \beta_0p_\abs\cdot p_{\mathrm{edge}}\cdot p_{\mathcal{J}}\cdot n$ choices for  $(j,u)\in \mathcal{G}_i$
 for which $c\in D_3\setminus (C_j\cup D_{3,j})$ and there is a colour-$c$ edge from $u$ to $Z_{i,0}\cap Z_{j,0}$ in $E_0^\abs$.

\item  \label{prop:fromB2:7:forJhat} For each $i\in [n]$ and $c\in D_3\setminus (C_i\cup D_{3,i})$, there are $(1\pm 2\gamma){p}_{\mathrm{vx}}\cdot \beta_0p_\abs\cdot p_{\mathrm{edge}}\cdot p_{\mathcal{J}}\cdot n$ choices for  $(j,v)\in \mathcal{F}_i$  for which $c\in D_3\setminus (C_j\cup D_{3,j})$ and there is a colour-$c$ edge from $v$ to $Z_{i,0}\cap Z_{j,0}$ in $E_0^\abs$.

\item \label{prop:fromB2:9:forJhat} For each $xy\in E_0^\abs$ with $c(xy)\in D_3$, there are $(1\pm 2\gamma){p}_{\mathrm{vx}}\cdot \beta_0^2\cdot p_\abs^2\cdot p_{\mathcal{J}}\cdot n$
choices of $(i,j)\in \mathcal{K}_x$
for which $c(xy)\in D_3\setminus (C_i\cup C_j\cup D_{3,i}\cup D_{3,j})$ and $y\in Z_{i,0}\cap Z_{j,0}$.
\item  \label{prop:fromB2:10:forJhat} For each $xy\in E_0^\abs$ with $c(xy)\in D_3$, there are $(1\pm 2\gamma){p}_{\mathrm{vx}}^2\cdot {p}_{\mathrm{col}}\cdot \beta_0^2p_\abs^2\cdot p_{\mathrm{edge}}\cdot p_{\mathcal{J}}\cdot n^2$
 choices for  $u\in V(G)$ and $(i,j)\in \mathcal{K}_x$ for which $c(ux),c(xy)\in D_3\setminus (C_i\cup C_j\cup D_{3,i}\cup D_{3,j})$, $x,y\in Z_{i,0}\cap Z_{j,0}$ and $ux\in E_0^\abs$.

\item\label{prop:from:B2:vertexinJ:forJhat} For each $v\in V(G)$, there are $(1\pm \gamma)p_{\mathcal{J}}n$ choices of $(i,j)$ with $(j,v)\in \mathcal{G}_i$.

\item \label{prop:from:B2:11a:forJhat} For each $i\in [n]$, there are at most $\gamma n$ vertices $u\in S_i\setminus R_i$ with $(i,u)\notin V(\mathcal{M}_2)$.

\item \label{prop:from:B2:11b:forJhat} For each $u\in V(G)$, there are at most $\gamma n$ values of $i\in [n]$ for which $u\in S_i\setminus R_i$ but $(i,u)\notin V(\mathcal{M}_2)$.

\item \label{prop:from:B2:11b:forJhat:new} For each $u\in V(G)$ and $\phi\in \mathcal{F}$, there are at most $p_\tr p_\fa \gamma n$ values of $i\in I_\phi$ for which $u\in Y_{i,0}$ but $(i,u)\notin V(\mathcal{M}_2)$.

\item \label{prop:from:B2:11ab:forJhat} For each $i\in [n]$, there are at most $\gamma n$ pairs $(j,u)$ with $u\in S_j\setminus R_j$, $i\in J_{i,u}$ with $(j,u)\notin V(\mathcal{M}_2)$.

\item \label{prop:forB3cod:1:forJhat}  For each $i\in [n]$ and $v\in V(G)$, there are at most $2rn^{0.6}$ choices for $j$ with $(j,v)\in \mathcal{F}_i$.
\end{enumerate}

\end{claim}
\begin{proof}[Proof of Claim~\ref{clm:propsfromBtwopart1}]
\ref{prop:fromB2:1:forJhat}, \ref{prop:fromB2:2:forJhat}: For \ref{prop:fromB2:1:forJhat}, note that, for each $i\in [n]$ and $x\in Z_{i,0}$, we have
\begin{align*}
|\{(j,u)\in \mathcal{G}_i:&x\in Z_{j,0},c(ux)\in D_3\setminus (C_{i}\cup C_j\cup D_{3,i}\cup D_{3,j}),ux\in E^\abs_0\}|=w^{\ref{prop:fromB2:1}}_{i,x}(\mathcal{M}_2)\\
&\overset{\eqref{eqn:M2:balanced}}=(1\pm \gamma)\cdot\delta_2^{-1}\cdot w^{\ref{prop:fromB2:1}}_{i,x}(E(\mathcal{H}_2))\\
&\overset{\eqref{eqn:W2:complicated1:totalweight}}{=} (1\pm 2\gamma)\cdot \beta_0\cdot p_{\mathrm{col}}\cdot p_{\mathrm{edge}}\cdot p_{\mathcal{J}}\cdot n,
\end{align*}
as claimed. Similarly, but using $w^{\ref{prop:fromB2:2}}_{j,x}$, \eqref{eqn:M2:balanced}, and \eqref{eqn:W2:complicated2:totalweight}, we have that \ref{prop:fromB2:2:forJhat} holds.

\smallskip

\noindent \ref{prop:fromB2:6:forJhat},\ref{prop:fromB2:7:forJhat}: For \ref{prop:fromB2:6:forJhat}, note that, for each $i\in [n]$ and $c\in D_3\setminus (C_i\cup D_{3,i})$, we have
\begin{align*}
&|\{(j,u)\in \mathcal{G}_i:c\in D_3\setminus (C_j\cup D_{3,j}),u\text{ has a colour-$c$ neighbour in }Z_{i,0}\cap Z_{j,0}\text{ in }E_0^\abs\}|\\
&=w^{\ref{prop:fromB2:6}}_{i,c}(\mathcal{M}_2)\\
&\overset{\eqref{eqn:M2:balanced}}=(1\pm \gamma)\cdot\delta_2^{-1}\cdot w^{\ref{prop:fromB2:6}}_{i,c}(E(\mathcal{H}_2))\\
&\overset{\eqref{eqn:W2:complicated3:totalweight}}{=} (1\pm 2\gamma)\cdot {p}_{\mathrm{vx}}\cdot \beta_0p_\abs\cdot p_{\mathrm{edge}}\cdot p_{\mathcal{J}}\cdot n,
\end{align*}
as claimed. Similarly, but using $w^{\ref{prop:fromB2:7}}_{i,c}$, \eqref{eqn:M2:balanced}, and \eqref{eqn:W2:complicated4:totalweight}, we have that \ref{prop:fromB2:7:forJhat} holds.

\smallskip

\noindent\ref{prop:fromB2:9:forJhat}: Note that, for each $xy\in E_0^\abs$ with $c(xy)\in D_3$, we have
\begin{align*}
&|\{(i,j)\in \mathcal{K}_x:c(xy)\in D_3\setminus (C_i\cup C_j\cup D_{3,i}\cup D_{3,j})\text{ and }y\in Z_{i,0}\cap Z_{j,0}\}|\\
&=w^{\ref{prop:fromB2:9}}_{xy}(\mathcal{M}_2)\\
&\overset{\eqref{eqn:M2:balanced}}=(1\pm \gamma)\cdot\delta_2^{-1}\cdot w^{\ref{prop:fromB2:9}}_{xy}(E(\mathcal{H}_2))\\
&\overset{\eqref{eqn:W2:complicated5:totalweight}}{=} (1\pm 2\gamma)\cdot {p}_{\mathrm{vx}}\cdot \beta_0^2\cdot p_\abs^2\cdot p_{\mathcal{J}}\cdot n,
\end{align*}
as claimed.

\smallskip

\noindent\ref{prop:fromB2:10:forJhat}: For each $xy\in E_0^\abs$ with $c(xy)\in D_3$, there are $(1\pm 2\gamma){p}_{\mathrm{vx}}^2\cdot {p}_{\mathrm{col}}\cdot \beta_0^2p_\abs^2\cdot p_{\mathrm{edge}}\cdot p_{\mathcal{J}}\cdot n^2$
 choices for  $u\in V(G)$ and $(i,j)\in \mathcal{K}_x$ for which $c(ux),c(xy)\in D_3\setminus (C_i\cup C_j\cup D_{3,i}\cup D_{3,j})$, $x,y\in Z_{i,0}\cap Z_{j,0}$ and $ux\in E_0^\abs$.

\smallskip

\noindent\ref{prop:from:B2:vertexinJ:forJhat}: Let $v\in V(G)$.  Then,
\begin{align*}
|\{(i,j):(j,v)\}\in \mathcal{F}_i\}|&=r\cdot |\{(i,u)\in J:v_{i,u}=v\}|+|\{(i,u)\in J:v\in V(M_{i,u})\}|
\\
&=r\cdot w^{\mathrm{same}}_u(\mathcal{M}_2)+w^{\mathrm{diff}}_u(\mathcal{M}_2)\\
&\overset{\eqref{eqn:M2:balanced},\eqref{eqn:W2:vsametotalweight},\eqref{eqn:W2:vdifftotalweight}}{=}
r\cdot (1\pm \gamma)\cdot p_{S-R}n\cdot \delta_2+(1\pm \gamma)\cdot r\cdot p_{S-R}n\cdot \delta_2=(1\pm \gamma)p_{\mathcal{J}}n,
\end{align*}
as required.

\smallskip

\noindent\ref{prop:from:B2:11a:forJhat}, \ref{prop:from:B2:11b:forJhat}, \ref{prop:from:B2:11b:forJhat:new}: Let $i\in [n]$. Then,
\begin{align*}
|\{u\in S_i\setminus R_i:(i,u)\notin V(\mathcal{M}_2)\}|&=|S_i\setminus R_i|-|\{u\in S_i\setminus R_i:(i,u)\in V(\mathcal{M}_2)\}|\\
&\overset{\ref{prop:B2:sizeofSminusR}}{\leq}(1+ \eps)\cdot p_{S-R}n -w_i^{\mathrm{same}}(\mathcal{M}_2)\\
&\overset{\eqref{eqn:M2:balanced},\eqref{eqn:W2:isametotalweight}}{\leq} 2\gamma p_{S-R}n\leq \gamma n,
\end{align*}
so that \ref{prop:from:B2:11a:forJhat} holds. Similarly, \ref{prop:from:B2:11b:forJhat} follows for each $u\in V(G)$ using $w_u^{\mathrm{same}}$, using \ref{prop:B2:sizeofSminusR}, \eqref{eqn:M2:balanced}, and \eqref{eqn:W2:vsametotalweight}. Furthermore, \ref{prop:from:B2:11b:forJhat:new} follows for each $u\in V(G)$ and $\phi\in \mathcal{F}$ using $w_{u,\phi}^{\mathrm{same}}$, \ref{prop:B1:lastfew:4}, \eqref{eqn:M2:balanced}, and \eqref{eqn:W2:vsamewithtotalweight}.

\smallskip

\noindent\ref{prop:from:B2:11ab:forJhat}: Let $i\in [n]$. Then,
\begin{align*}
|\{(j,u):u\in S_j\setminus R_j,i\in J_{i,u},&(j,u)\notin V(\mathcal{M}_2)\}|\\
&=r\cdot |S_i\setminus R_i|-|\{(j,u):u\in S_j\setminus R_j,i\in J_{i,u},(j,u)\in V(\mathcal{M}_2)\}|\\
&\overset{\ref{prop:B2:sizeofSminusR}}{\leq}r\cdot (1+\eps)\cdot p_{S-R}n-w_i^{\mathrm{diff}}(\mathcal{M}_2)\\
&\overset{\eqref{eqn:M2:balanced},\eqref{eqn:W2:idifftotalweight}}{\leq} r\cdot 2\gamma p_{S-R}n\leq \gamma n,
\end{align*}
as required.

\smallskip

\noindent\ref{prop:forB3cod:1:forJhat}:  For each $i\in [n]$ and $v\in V(G)$, we have,
\begin{align*}
|\{j:(j,v)\in \mathcal{F}_i\}|= w_{i,v}^{\mathrm{cod}}(E(\mathcal{M}_2))\leq 2rn^{0.6},
\end{align*}
as required.
\claimproofend

\subsubsection{Missing matchings}
Let $\bar{J}=\{(i,u):i\in [n],u\in S_i\setminus R_i\}\setminus J$, where for each $(i,u)\in \bar{J}$ we have $(i,u)\notin V(\mathcal{M}_2)$.
We will now find, for each $(i,u)\in \bar{J}$, a tuple $(v_{i,u},M_{i,u},c_{i,u},\omega_{i,u})$ similar to the one found for each $(i',u')\in J$ in Section~\ref{sec:choiceofM2}, except using, for example, vertices in $Y_{i,1}$ instead of $Y_{i,0}$. We start by choosing the vertices $v_{i,u}$ and colours $c_{i,u}$, for each $(i,u)\in \bar{J}$.

Take a maximal set $\bar{J}'\subset \bar{J}$ for which there are $v_{i,u}$ and $c_{i,u}$, $(i,u)\in \bar{J}'$ such that the following hold.
\stepcounter{propcounter}
\begin{enumerate}[label = {{\textbf{\Alph{propcounter}\arabic{enumi}}}}]
\item For each $(i,u)\in \bar{J}'$, $v_{i,u}\in \cap_{j\in J_{i,u}\cup \{i\}}Y_{j,1}$, $c_{i,u}\in \cap_{j\in J_{i,u}\cup \{i\}} \cap_{j\in J_{i,u}\cup \{i\}}(D_{2,j}\setminus C_j)$, and $uv_{i,u}$ is a colour-$c_{i,u}$ edge which is in $E_{1,X}^\abs$ where $X\in \{A,B\}$ is such that $u\in X$.\label{prop:missingviuciu:1}
\item For each $i\in [n]$, the vertices $v_{i,u}$, $(i,u)\in \bar{J}'$, and $v_{j,u'}$, $(j,u')\in \bar{J}'$ with $i\in J_{j,u'}$ are all distinct.\label{prop:missingviuciu:2}
\item For each $i\in [n]$, the colours $c_{i,u}$, $(i,u)\in \bar{J}'$, and $c_{j,u'}$, $(j,u')\in \bar{J}'$ with $i\in J_{j,u'}$ are all distinct.\label{prop:missingviuciu:3}
\item The edges $uv_{i,u}$, $(i,u)\in \bar{J}'$, are all distinct.\label{prop:missingviuciu:4}
\item For each $v\in V(G)$, there are at most $\sqrt{\gamma} n$ pairs $(i,u)\in \bar{J}'$ for which $v=v_{i,u}$.\label{prop:missingviuciu:5a}
\item For each $i\in [n]$ and $v\in V(G)$, there are at most $r\cdot n^{0.6}/4$ choices for $(u,j)$ for which $(i,u)\in \bar{J}'$, $i\in J_{j,u}$ and $v_{j,u}=v$.\label{prop:missingviuciu:5b}
\item For each $c\in C$, there are at most $\sqrt{\gamma} n$ $(i,u)\in \bar{J}'$ with $c_{i,u}=c$.\label{prop:missingviuciu:5}
\end{enumerate}

We now infer that we can find a suitable $v_{i,u}$ and $c_{i,u}$ for every $(i,u)\in \bar{J}$.

\begin{claim}\label{clm:allfoundviuciu}
$\bar{J}'=\bar{J}$.
\end{claim}
\begin{proof}[Proof of Claim~\ref{clm:allfoundviuciu}]
Suppose otherwise, so that, in particular, we can choose some $(i,u)\in \bar{J}\setminus \bar{J}'$. Let $v_{i,u}$ and $c_{i,u}$, $(i,u)\in \bar{J}'$, be such that \ref{prop:missingviuciu:1}--\ref{prop:missingviuciu:5} hold.
Suppose $u\in A$, where the case where $u\in B$ follows similarly.
Let $V_1^{\mathrm{forb}}=\{v_{i,u'}:(i,u')\in \bar{J}'\}$, $V_2^{\mathrm{forb}}=\{v:|\{(i,u)\in \bar{J}':v_{i,u}=v\}|\geq \sqrt{\gamma}n/2\}$, $V_3^{\mathrm{forb}}=\{v\in V(G):|\{(u,j):(i,u)\in \bar{J}',i\in J_{j,u},v_{j,u}=v\}|\geq r\cdot n^{0.6}/8\}$,
$C^{\mathrm{forb}}_1=\{v_{i,u'}:(i,u')\in \bar{J}'\}$, $C^{\mathrm{forb}}_2=\{c\in C:|\{(i,u)\in \bar{J}':c_{i,u}=c\}|\geq \sqrt{\gamma}n/2\}|$
 and $E^{\mathrm{forb}}_A=\{u'v_{i',u'}:(i',u')\in \bar{J}'\}$.

By \ref{prop:from:B2:11a:forJhat} and \ref{prop:from:B2:11ab:forJhat}, we have that $|V^{\mathrm{forb}}_1|,|C^{\mathrm{forb}}_1|\leq 2\gamma n$. Then, as, by \ref{prop:from:B2:11a:forJhat}, $|\bar{J}'|\leq \gamma n^2$, we have $|V^{\mathrm{forb}}_2|,|C^{\mathrm{forb}}_2|\leq 2\sqrt{\gamma}n$.
Furthermore, the number of edges in $E^{\mathrm{forb}}_{A}$ containing $u$ is, by \ref{prop:from:B2:11b:forJhat}, at most $\gamma n$. Therefore, by \ref{prop:B2:missing:uv}, there is some choice for $v_{i,u}\in \left(\cap_{j\in J_{i,u}\cup \{i\}} Y_{j,1}\right)\setminus V^{\mathrm{forb}}$ such that $uv_{i,u}\in E_{1,X}^\abs\setminus E^{\mathrm{forb}}$ and $c(uv_{i,u})\in \left(\cap_{j\in J_{i,u}\cup \{i\}})D_{3,j}\right)\setminus (C^{\mathrm{forb}}_1\cup C^{\mathrm{forb}}_2)$. Letting $c_{i,u}=c(uv_{i,u})$, the pair $c_{i,u},v_{i,u}$ show that $\bar{J}'$ contradicts the maximality of $\bar{J}$.
\claimproofend

Let, then, $v_{i,u}$ and $c_{i,u}$, $(i,u)\in \bar{J}$, be such that \ref{prop:missingviuciu:1}--\ref{prop:missingviuciu:5} hold.
For each $(i,u)\in \bar{J}$, we will now find $M_{i,u}$ and $\omega_{i,u}$. For this, let $\bar{J}''\subset J_{i,u}$ be a maximal set for which there are $M_{i,u}$ and $\omega_{i,u}$, $(i,u)\in \bar{J}''$ such that the following hold.
\stepcounter{propcounter}
\begin{enumerate}[label = {{\textbf{\Alph{propcounter}\arabic{enumi}}}}]
\item For each $(i,u)\in \bar{J}''$, $M_{i,u}$ is a colour-$c_{i,u}$ matching in $E_0^\abs$ and $\omega_{i,u}:J_{i,u}\to M_{i,u}$ is a bijection.\label{prop:missingMiuomiu:1}
\item For each $(i,u)\in \bar{J}''$, and each $j\in J_{i,u}$, $V(\omega_{i,u}(j))\subset Y_{i,1}\cap Y_{j,1}$ and $\omega_{i,u}(j)\in E^\abs_{1,M}$.\label{prop:missingMiuomiu:2}
\item For each $i\in [n]$, the sets $V(M_{i,u})$, $(i,u)\in \bar{J}''$, and $V(M_{j,u'})$, $(j,u')\in \bar{J}''$ and $i\in J_{i,u}$, are all disjoint from each other and
from $\{v_{i,u}:(i,u)\in \bar{J}\}$ and $\{v_{j,u'}:(j,u')\in \bar{J},i\in J_{j,u'}\}$.\label{prop:missingMiuomiu:3}
\item The matchings $M_{i,u}$, $(i,u)\in \bar{J}''$, are all edge-disjoint.\label{prop:missingMiuomiu:4}
\item For each $v\in V(G)$, there are at most $4\sqrt{\gamma} n$ pairs $(i,u)\in \bar{J}''$ for which $v\in V(M_{i,u})\cup \{v_{i,u}\}$.\label{prop:missingMiuomiu:5}
\item For each $i\in [n]$ and $v\in V(G)$, there are at most $r\cdot n^{0.6}$ choices for $(u,j)$ for which $(i,u)\in \bar{J}''$, $j\in J_{i,u}$ and $v\in V(M_{i,u})$ or for which $(j,u)\in \bar{J}''$, $i\in J_{j,u}$ and $v_{j,u}=v$.\label{prop:missingMiuomiu:6}
\end{enumerate}

We now infer that we can find a suitable matching $M_{i,u}$ and function $\omega_{i,u}$ for each $(i,u)\in \bar{J}$.
\begin{claim}\label{clm:allfoundMiuomiu}
$\bar{J}''=\bar{J}$.
\end{claim}
\begin{proof}[Proof of Claim~\ref{clm:allfoundMiuomiu}] Suppose otherwise, so that, in particular, we can choose some $(i,u)\in \bar{J}\setminus \bar{J}'$. Let $M_{i',u'}$ and $\omega_{i',u'}$, $(i',u')\in \bar{J}''$, be such that \ref{prop:missingMiuomiu:1}--\ref{prop:missingMiuomiu:6} hold.

For each $j\in J_{i,u}\cup \{i\}$, let
\[
V^{\mathrm{forb}}_j=\{v_{j,u'}:(j,u')\in \bar{J}\}\cup \left(\cup_{(j,v)\in \bar{J}''}V(M_{j,v}))\right)\cup\left(\cup_{(j',v)\in \bar{J}'':j\in J_{j',v}}V(M_{j',v})\right),
\]
so that, from \ref{prop:from:B2:11a:forJhat} and \ref{prop:from:B2:11ab:forJhat},
we have $|V^{\mathrm{forb}}|\leq (2r+2)\cdot 2\gamma n)$. Let $V^{\mathrm{forb}}=\cup_{j\in J_{i,u}\cup \{i\}}V^{\mathrm{forb}}_j$, so that $|V^{\mathrm{forb}}|\leq (2r+2)^2\gamma n$.

Let $W^{\mathrm{forb}}_0=\{v\in V(G):|\{(i',u')\in \bar{J}'':v\in V(M_{i',u'})\cup \{v_{i',u'}\}\}|\geq 2\sqrt{\gamma}n\}$ and
$W^{\mathrm{forb}}_1=\{v\in V(G):|\{(u',j):(i,u')\in \bar{J}'',j\in J_{i,u'},v\in V(M_{i,u'})\text{ or }(j,u')\in \bar{J}'',i\in J_{j,u'},v_{j,u'}=v\}|\geq r\cdot n^{0.6}/2\}$. From \ref{prop:missingviuciu:5a}, and as $|\bar{J}|\leq \gamma n^2$ by \ref{prop:from:B2:11a:forJhat}, we have that $|W^{\mathrm{forb}}_0|\leq 4\sqrt{\gamma} n$.
From \ref{prop:missingviuciu:5b}, we have $|W^{\mathrm{forb}}_1|\cdot r\cdot n^{0.6}/8\geq 2n\cdot 2r+2r\cdot n$, so that $|W^{\mathrm{forb}}_1|\leq \sqrt{n}$.

Let $E^{\mathrm{forb}}=\{e\in \cup_{(j,u')\in \bar{J}''}(M_{j,u'}\cup \{u'v_{j,u'}\},:c(e)=c_{i,u}\}$, so that, by \ref{prop:missingviuciu:5}, we have $|E^{\mathrm{forb}}|\leq (r+1)\sqrt{\gamma}n$.
Now, using \ref{prop:B2:missing:Muv}, for each $j\in J_{i,u}$, let $\omega_{i,u}(j)$ be an edge of colour $c$ in $E^\abs_{1,M}\setminus E^{\mathrm{forb}}$ with vertices in $(Y_{i,1}\cap Y_{j,1})\setminus (V^{\mathrm{forb}}\cup W^{\mathrm{forb}}_0\cup W^{\mathrm{forb}}_1)$, so that $\omega_{i,u}(j)$ are distinct. Let $M_{i,u}=\{\omega_{i,u}(j):j\in J_{i,u}\}$.
Noting that $\bar{J}''\cup \{(i,u)\}$ satisfies \ref{prop:missingMiuomiu:1}--\ref{prop:missingMiuomiu:6} with $\bar{J}''$ replaced by $\bar{J}''\cup \{(i,u)\}$ contradicts the maximality of $\bar{J}''$.
\claimproofend

For each $(i,u)\in \bar{J}$ and $j\in J_{i,u}$, let $x_{i,u,j}$ be the vertex of $\omega_{i,u}(j)$ with $x_{i,u,j}\simAB v_{i,u}$ and let $y_{i,u,j}$ be the vertex of $\omega_{i,u}(j)$ with $x_{i,u,j}\simAB v_{i,u}$.
Note that we now have for each $i\in [n]$, $v_{i,u},M_{i,u},c_{i,u},\omega_{i,u}$, where $M_{i,u}=\{x_{i,u,j}y_{i,u,j}:j\in J_{i,u}\}$.


\subsubsection{Choice of the matchings $\hat{M}_{i,2}$, $\mathcal{J}$, and their properties}\label{sec:part2choosefinally}
For each $i\in [n]$, let
\[
\hat{M}_{i,2}=\left(\cup_{u\in S_i\setminus R_i}\{uv_{i,u}\}\right)\cup \left(\cup_{j\in [n],u\in S_j\setminus R_j:i\in J_{i,u}}\{\omega_{j,u}(i)\}\right).
\]
Let
\begin{equation}\label{eqn:Jdefn}
\mathcal{J}=\cup_{\{(i,u),(j,v)\}\in \mathcal{I}}\{\{(j,v_{i,u}),(i,x_{i,u,j})\},\{(j,y_{i,u,j}),(i,y_{j,v,i})\},\{(i,v_{j,v}),(j,x_{j,v,i})\}\}.
\end{equation}
We now record the properties of $\hat{M}_{i,2}$ and $\mathcal{J}$ that we need.
\begin{claim}\label{claim:outcomeofB2}
\stepcounter{propcounter}
The following hold.
\begin{enumerate}[label = {{\textbf{\Alph{propcounter}\arabic{enumi}}}}]
\item  \label{prop:fromB2:1} For each $i\in [n]$ and $x\in Z_{i,0}$, there are $(1\pm 4\gamma)\beta_0\cdot p_{\mathrm{col}}\cdot p_{\mathrm{edge}}\cdot p_{\mathcal{J}}\cdot n$ choices for
$\{(i,u),(j,v)\}\in \mathcal{J}$ for which $x\in Z_{j,0}$, $c(ux)\in D_3\setminus (C_{i}\cup C_j\cup D_{3,i}\cup D_{3,j})$ and $ux\in E^\abs_0$.

\item  \label{prop:fromB2:2} For each $i\in [n]$ and $x\in Z_{i,0}$, there are $(1\pm 4\gamma)\beta_0\cdot p_{\mathrm{col}}\cdot p_{\mathrm{edge}}\cdot p_{\mathcal{J}}\cdot n$ choices for $\{(i,u),(j,v)\}\in \mathcal{J}$ for which $x\in Z_{j,0}$, $c(xv)\in D_3\setminus (C_{i}\cup C_j\cup D_{3,i}\cup D_{3,j})$ and $xv\in E^\abs_0$.

\item  \label{prop:fromB2:6} For each $i\in [n]$ and $c'\in D_3\setminus (C_i\cup D_{3,i})$, there are $(1\pm 4\gamma){p}_{\mathrm{vx}}\cdot \beta_0p_\abs\cdot p_{\mathrm{edge}}\cdot p_{\mathcal{J}}\cdot n$ choices for  $\{(i,u),(j,v)\}\in \mathcal{J}$ for which $c'\in D_3\setminus (C_j\cup D_{3,j})$ and there is a colour-$c'$ edge from $u$ to $Z_{i,0}\cap Z_{j,0}$ in $E_0^\abs$.

\item  \label{prop:fromB2:7} For each $i\in [n]$ and $c'\in D_3\setminus (C_i\cup D_{3,i})$, there are $(1\pm 4\gamma){p}_{\mathrm{vx}}\cdot \beta_0p_\abs\cdot p_{\mathrm{edge}}\cdot p_{\mathcal{J}}\cdot n$ choices for  $\{(i,u),(j,v)\}\in \mathcal{J}$ for which $c'\in D_3\setminus (C_j\cup D_{3,j})$ and there is a colour-$c'$ edge from $v$ to $Z_{i,0}\cap Z_{j,0}$ in $E_0^\abs$.

\item \label{prop:fromB2:9} For each $xy\in E_0^\abs$ with $c(xy)\in D_3$, there are $(1\pm 4\gamma){p}_{\mathrm{vx}}\cdot \beta_0^2\cdot p_\abs^2\cdot p_{\mathcal{J}}\cdot n$
choices of $\{(i,u),(j,v)\}\in \mathcal{J}$ with $u=x$
for which $c(xy)\in D_3\setminus (C_i\cup C_j\cup D_{3,i}\cup D_{3,j})$ and $y\in Z_{i,0}\cap Z_{j,0}$.

\item  \label{prop:fromB2:10} For each $xy\in E_0^\abs$ with $c(xy)\in D_3$, there are $(1\pm 4\gamma){p}_{\mathrm{vx}}^2\cdot {p}_{\mathrm{col}}\cdot \beta_0^2p_\abs^2\cdot p_{\mathrm{edge}}\cdot p_{\mathcal{J}}\cdot n^2$
 choices for  $\{(i,u),(j,v)\}\in \mathcal{J}$ for which $c(ux),c(xy)\in D_3\setminus (C_i\cup C_j\cup D_{3,i}\cup D_{3,j})$, $x,y\in Z_{i,0}\cap Z_{j,0}$ and $ux\in E_0^\abs$.

\item\label{prop:from:B2:vertexinJ} For each $u\in V(G)$, there are $(1\pm \beta)p_{\mathcal{J}}n$ triples $(i,j,v)$ with $\{(i,u),(j,v)\}\in \mathcal{J}$.

\item \label{prop:from:B2:11a} For each $i\in [n]$, there are $(1\pm \beta) 2p_{\mathcal{J}}n$ triples $(u,j,v)$ with  $\{(i,u),(j,v)\}\in \mathcal{J}$.

\item \label{prop:from:B2:11} For each distinct $i,j\in [n]$ there are at most $3\sqrt{n}$ pairs $(u,v)$ with $\{(i,u),(j,v)\}\in \mathcal{J}$.

\item \label{prop:forB3cod:1}  For each $i\in [n]$ and $v\in V(G)$, there are at most $3r\cdot n^{0.6}$ pairs $(j,u)$ with $\{(i,u),(j,v)\}\in \mathcal{J}$.
\item \label{prop:forB3cod:2}  For each $i\in [n]$ and $u\in V(G)$, there is at most $1$ pair $(j,v)$ with $\{(i,u),(j,v)\}\in \mathcal{J}$.

\item \label{prop:forBfinalfromB2} For each $u\in V(G)$ and $\phi\in \mathcal{F}$,
\[
|\{i\in I_\phi:v\in Y_{i,0}\setminus (\cup_{u,j,u':\{(i,u),(j,u')\}\in \mathcal{I}}\{v_{i,u},x_{i,u,j},y_{i,u,j},v_{j,u'},x_{j,u',i},y_{j,u',i}\})\}|\leq \gamma p_\tr p_\fa n.
\]

\item \label{prop:forBfinalfromB2:2} For each $c\in D_2$ and $\phi\in \mathcal{F}$,
\[
|\{i\in I_\phi:c \notin (C(\hat{M}_{i,2}) \cup C_i\cup D_{2,i})\}|\leq \gamma p_\tr p_\fa n.
\]

\item \label{prop:forBfinalfromB2:3} For each $i\in [n]$,
\[
|Y_{i,0}\setminus (\cup_{u,j,u':\{(i,u),(j,u')\}\in \mathcal{I}}\{v_{i,u},x_{i,u,j},y_{i,u,j},v_{j,u'},x_{j,u',i},y_{j,u',i}\})\}|\leq \gamma n.
\]
\item For each $i\in [n]$, $\hat{M}_{i,2}$ is a rainbow matching with colours in $D_2\setminus C_i$.\label{prop:forBfinalfromB2:4}
\end{enumerate}
\end{claim}
\begin{proof}[Proof of Claim~\ref{claim:outcomeofB2}]
\ref{prop:fromB2:1}, \ref{prop:fromB2:6}: Let $i\in [n]$ and $x\in Z_{i,0}$. Note that if $\{(i,u'),(j,v')\}\in \mathcal{J}$, then, from \eqref{eqn:Jdefn}
there is some $\{(i,u),(j,v)\}\in \mathcal{I}$ with $(u',v')\in \{(x_{i,u,j},v_{i,u}),(y_{j,v,i},y_{i,u,j}),(v_{j,v},x_{j,v,i})\}$.
If $(i,u)$ and $(j,v)$ are both in $J$, then $(j,v_{i,u}),(j,y_{i,u,j}),(i,x_{j,v,i})\in \mathcal{G}_i$, and thus $(j,v')\in \mathcal{G}_i$, while there are at most $\gamma n$ vertices $u\in S_i\setminus R_i$ for which $(i,u)\in \bar{J}$ by \ref{prop:from:B2:11a:forJhat}. Therefore, by \ref{prop:fromB2:1:forJhat},
\begin{align*}
|\{\{(i,u),(j,v)\}\in &\mathcal{J}:x\in Z_{j,0},c(ux)\in D_3\setminus (C_{i}\cup C_j\cup D_{3,i}\cup D_{3,j}),ux\in E^\abs_0\}|\\
&=(1\pm 4\gamma)\beta_0\cdot p_{\mathrm{col}}\cdot p_{\mathrm{edge}}\cdot p_{\mathcal{J}}\cdot n+2\gamma n=(1\pm 4\gamma)\beta_0\cdot p_{\mathrm{col}}\cdot p_{\mathrm{edge}}\cdot p_{\mathcal{J}}\cdot n,
\end{align*}
so that \ref{prop:fromB2:1} holds, while, by \ref{prop:fromB2:6:forJhat},
\begin{align*}
|\{\{(i,u),(j,v)\}\in &\mathcal{J}:c'\in D_3\setminus (C_j\cup D_{3,j}),\exists\text{ a colour-$c'$ edge from $u$ to $Z_{i,0}\cap Z_{j,0}$ in $E_0^\abs$}\}|\\
&=(1\pm 4\gamma){p}_{\mathrm{vx}}\cdot \beta_0p_\abs\cdot p_{\mathrm{edge}}\cdot p_{\mathcal{J}}\cdot n,
\end{align*}
so that \ref{prop:fromB2:6} holds.

\smallskip

\noindent\ref{prop:fromB2:2}, \ref{prop:fromB2:7}: Similarly, it follows that if $\{(i,u'),(j,v')\}\in \mathcal{J}$, then, if this is $\mathcal{J}$ due to $\{(i,u),(j,v)\}\in \mathcal{I}$, we have that if $(i,u)$ and $(j,v)$ are both in $J$, then $(j,u')\in \mathcal{F}_i$. Thus, by \ref{prop:from:B2:11a:forJhat} and, respectively, \ref{prop:fromB2:2:forJhat} and \ref{prop:fromB2:7:forJhat}, we can conclude that \ref{prop:fromB2:2} and \ref{prop:fromB2:7} hold.

\smallskip

\noindent\ref{prop:fromB2:9}: Let $xy\in E_0^\abs$ with $c(xy)\in D_3$. If $(i,j)\in \mathcal{K}_x$, then for some $v'$ we have $\{(i,x),(j,v')\}\in \mathcal{J}$ from \eqref{eqn:calKxdefn}. Thus, from \eqref{eqn:Jdefn}, \ref{prop:fromB2:9:forJhat} and \ref{prop:missingMiuomiu:5}, we have that \ref{prop:fromB2:9} holds.

\smallskip

\noindent\ref{prop:fromB2:10}: Let $xy\in E_0^\abs$ with $c(xy)\in D_3$. Similarly to \ref{prop:fromB2:9}, but using \eqref{eqn:Jdefn}, \ref{prop:fromB2:10:forJhat}  and \ref{prop:missingMiuomiu:5} (summed over all $v\in V(G)$), we have that \ref{prop:fromB2:10} holds.

\smallskip

\noindent\ref{prop:from:B2:vertexinJ}: This follows from \eqref{eqn:Jdefn}, \ref{prop:from:B2:vertexinJ:forJhat} and \ref{prop:missingMiuomiu:5}.

\smallskip

\noindent\ref{prop:from:B2:11a}: Let $i\in [n]$. For each $(u,j,v)$ with $\{(i,u),(j,v)\}\in \mathcal{I}$, by \eqref{eqn:Jdefn}, there are 3 triples $(u',j',v')$ with $\{(i,u'),(j',v')\}\in \mathcal{J}$. Therefore, by \ref{prop:abs:regularityout} and \ref{prop:B2:sizeofSminusR}, we have that the number of triple $(u,j,v)$ with  $\{(i,u),(j,v)\}\in \mathcal{J}$ is $3r\cdot |S_i\setminus R_i|=3r\cdot (1\pm \eps)\cdot 2p_{S-R}n=(1\pm \eps)\cdot 2p_\mathcal{J}n$, and thus \ref{prop:from:B2:11a} holds.

\smallskip

\noindent\ref{prop:from:B2:11}: Let $i,j\in [n]$ be distinct. Then, as for each $(u,v)$ with $\{(i,u),(j,v)\}\in \mathcal{I}$ there are 2 triples $(u',v')$ with $\{(i,u'),(j,v')\}\in \mathcal{J}$, we have that \ref{prop:from:B2:11} follows from \ref{prop:abs:lowcodegree1}.

\smallskip

\noindent\ref{prop:forB3cod:1}: Let $i\in [n]$ and $v\in V(G)$. Then, by \ref{prop:forB3cod:1:forJhat} and \ref{prop:missingMiuomiu:6} we have that \ref{prop:forB3cod:1} holds.

\smallskip

\noindent\ref{prop:forB3cod:2}: Let $i\in [n]$ and $u\in V(G)$, and note that if there is some $(j,v)$ with $\{(i,u),(j,v)\}\in \mathcal{J}$, then $u\in Y_{i}$. If $u\in Y_{i,0}$, then there is a unique such $(j,v)$, coming from the unique edge of $\mathcal{M}_2$, while if $u\in Y_{i,1}$, then there is a unique such $(j,v)$ by \ref{prop:missingMiuomiu:3}.

\smallskip

\noindent\ref{prop:forBfinalfromB2}: For each $u\in V(G)$ and $\phi\in \mathcal{F}$,
\begin{align*}
|\{i\in I_\phi:v\in Y_{i,0}\setminus &(\cup_{u,j,u':\{(i,u),(j,u')\}\in \mathcal{I}}\{v_{i,u},x_{i,u,j},y_{i,u,j},v_{j,u'},x_{j,u',i},y_{j,u',i}\})\}|\\
&=|\{i\in I_\phi:v\in Y_{i,0}\}|-|\{i\in I_\phi:v\in Y_{i,0},(i,v)\notin V(\mathcal{M}_2)\}|\\
&\overset{\ref{prop:from:B2:11b:forJhat:new}}{\leq} 2\gamma p_\tr p_\fa n,
\end{align*}
and therefore \ref{prop:forBfinalfromB2} holds.

\smallskip

\noindent\ref{prop:forBfinalfromB2:2}: Let $c\in D_2$ and $\phi\in \mathcal{F}$. For each $i\in [n]$, if $(i,c)\in V(\mathcal{M}_2)$, then $\hat{M}_{i,2}$ has an edge of colour $c$, and when this occurs there are $r+1$ pairs $(j,c)\in V(\mathcal{M}_2)$ with different values of $j$. Thus,
\begin{align*}
|\{i\in [n]:c \notin (C(\hat{M}_{i,2}) \cup C_i\cup D_{2,i})\}|&=|\{i\in [n]:c\in D_2\setminus (C_i\cup D_{2,i})\}|-(r+1)\cdot w_c(\mathcal{M}_2)\\
&
\overset{\ref{prop:B1:lastfew:3},\eqref{eqn:M2:balanced}}{\leq}(1\pm \eps)\beta_0p_\abs n-(1+\gamma)\cdot \delta_2^{-1}\cdot w_c(E(\mathcal{H}_2))\\
&\overset{\eqref{eqn:W2:ctotalweight}}
\leq 2\gamma n,
\end{align*}
as required.

\smallskip

\noindent\ref{prop:forBfinalfromB2:3}: For each $i\in [n]$,
\begin{align*}
|Y_{i,0}\setminus &(\cup_{u,j,u':\{(i,u),(j,u')\}\in \mathcal{I}}\{v_{i,u},x_{i,u,j},y_{i,u,j},v_{j,u'},x_{j,u',i},y_{j,u',i}\})\}|\\
&=|Y_{i,0}|-|\{u:\exists(j,u')\text{ s.t.\ }\{(i,u),(j,u')\}\in \mathcal{I}\}|-5|\{(u,j,u'):\{(i,u),(j,u')\}\in \mathcal{I}\}|\\
&=|Y_{i,0}|-(5r+1)|S_i\setminus R_i|
\overset{\ref{prop:B2:sizeofSminusR}}{\leq} (1+\eps)\beta_0p_Yn-(1-\eps)121p_{S-R}n
\leq 2\beta n,
\end{align*}
as required.
\smallskip

\noindent\ref{prop:forBfinalfromB2:4}: Finally, note that \ref{prop:forBfinalfromB2:4} follows from \ref{prop:missingviuciu:2}, \ref{prop:missingviuciu:3} and \ref{prop:missingMiuomiu:3}, and the construction of $\mathcal{H}_2$ and choices of $\mathcal{M}_2$.
\claimproofend


\subsection{Part~\ref{partB3}: Switching paths with $Z_i$}\label{sec:partB3}
Recalling the set $\mathcal{J}$ from Part~\ref{partB2}, for each $f=\{(i,u),(j,v)\}\in \mathcal{J}$, let $\mathcal{R}_f$ be the set of $(u,v,L)$-links with internal vertices in $Z_{i,0}\cap Z_{j,0}$, colours in $D_3\setminus (C_i\cup C_j\cup D_{i,0}\cup D_{j,0})$ and edges in $E^{\mathrm{abs}}_{0}$.

Define an auxiliary hypergraph $\mathcal{H}_3$ with 4 vertex classes
\begin{equation}\label{eqn:H3vx}
\begin{array}{ll}
\textbf{i)}\;\; \mathcal{J} &\;\;\textbf{ii)}\;\;  \mathcal{V}_Z:=\cup_{i\in [n]}(\{i\}\times Z_{i,0})\\
\textbf{iii)}\;\; \mathcal{C}_3:=\cup_{i\in [n]}(\{i\}\times (D_{3}\setminus (C_i\cup D_{3,i})))
&\;\;\textbf{iv)}\;\; \mathcal{E}_3:=E(G|_{D_3})\cap E^{\mathrm{abs}}_0
\end{array}
\end{equation}
where, for each $f=\{(i,u),(j,v)\}\in \mathcal{J}$ and $S\in \mathcal{R}_{f}$, we add the edge
\[
\{f\}\cup (\{i,j\}\times ((V(S)\setminus\{u,v\})\cup C(S))\cup E(S).
\]
As for each $f=\{(i,u),(j,v)\}\in \mathcal{J}$ and $S\in \mathcal{R}_{f}$, $S$ is a path with 31 colours, 62 edges and 61 internal vertices, each edge of $\mathcal{H}_3$ has $1+2(61+31)+62=247$ vertices, and thus $\mathcal{H}_3$ is a 247-uniform hypergraph. We will now show that $\mathcal{H}_3$ is almost regular (in Section~\ref{sec:H3:almostregular}) with low codegrees (in Section~\ref{sec:H3:lowcod}).


\subsubsection{Vertex degrees in $\mathcal{H}_3$}\label{sec:H3:almostregular}
Recalling $\Phi$ from \eqref{eqn:Phidefn}, let $\delta_3=\Phi=p_{\mathrm{vx}}^{61}\cdot p_{\mathrm{col}}^{30}\cdot p_{\mathrm{edge}}^{62}\cdot n^{30}$. 
We now show that $\mathcal{H}_3$ is almost $\delta_3$-regular.
\begin{claim}\label{clm:H3almostregular}
For each $v\in V(\mathcal{H}_3)$, we have $d_{\cH_3}(v)=(1\pm 10\gamma)\delta_3$.
\end{claim}
\begin{proof}[Proof of Claim~\ref{clm:H3almostregular}]
We check this for vertices in each of the 4 classes in the order at \eqref{eqn:H3vx}.

\smallskip

\noindent\textbf{i)} Let $f=\{(i,u),(j,v)\}\in \mathcal{J}$. Then, using \ref{prop:B3:1}, we have $d_{\mathcal{H}_3}(f)=(1\pm \eps)\Phi=(1\pm 10\gamma)\delta_3$.

\smallskip

\noindent\textbf{ii)} Let $(i,x)\in \mathcal{V}_Z$ so that $i\in [n]$ and $x\in Z_{i,0}$. We count the number of $f=\{(i,u),(j,v)\}\in \mathcal{J}$ and $S\in \mathcal{R}_f$ for which $x$ is an internal vertex of $V(S)$ in cases depending on the position of $x$ in $S$.
\begin{itemize}
\item $x$ is the 2nd vertex of $S$: From \ref{prop:fromB2:1} we can bound the number of choices for $\{(i,u),(j,v)\}\in \mathcal{J}$ for which $x\in Z_{i,0}\cap Z_{j,0}$, $c(ux)\in D_3\setminus (C_{i}\cup C_j\cup D_{3,i}\cup D_{3,j})$ and $ux\in E^\abs_0$,
and thus, by \ref{prop:B3:2}, the number of choices for $S$ in this case is
\[
(1+5\gamma)\cdot \beta_0\cdot p_{\mathrm{col}}\cdot p_{\mathrm{edge}}\cdot p_{\mathcal{J}}\cdot n\cdot (1\pm \eps)\cdot \Phi \cdot p_{\mathrm{vx}}^{-1}\cdot p_{\mathrm{col}}^{-1}\cdot p_{\mathrm{edge}}^{-1}\cdot n^{-1}=(1\pm 10\gamma)\cdot \beta_0\cdot p_{\mathcal{J}}\cdot \Phi \cdot p_{\mathrm{vx}}^{-1}.
\]
\item $x$ is the 62nd vertex of $S$: By \ref{prop:fromB2:2}, we can bound the number of choices for $f=\{(i,u),(j,v)\}\in \mathcal{J}$ for which $x\in Z_{i,0}\cap Z_{j,0}$, $c(xv)\in D_3\setminus (C_{i}\cup C_j\cup D_{3,i}\cup D_{3,j})$ and $xv\in E^\abs_0$,
and thus, by \ref{prop:B3:2}, the number of choices for $S$ in this case is
\[
(1+5\gamma)\cdot \beta_0\cdot p_{\mathrm{col}}\cdot p_{\mathrm{edge}}\cdot p_{\mathcal{J}}\cdot n\cdot (1\pm \eps)\cdot \Phi \cdot p_{\mathrm{vx}}^{-1}\cdot p_{\mathrm{col}}^{-1}\cdot p_{\mathrm{edge}}^{-1}\cdot n^{-1}=(1\pm  10\gamma)\cdot \beta_0\cdot p_{\mathcal{J}}\cdot \Phi \cdot p_{\mathrm{vx}}^{-1}.
\]

\item $x$ is the $k$th vertex of $S$, with $3\leq k\leq 62$:
By \ref{prop:fromB2:5} or \ref{prop:fromB2:5b}, we can bound the number of choices for $f=\{(i,u),(j,v)\}\in \mathcal{J}$ for which $x\in Z_{i,0}\cap Z_{j,0}$ and $x\simAB u$ if $k$ is odd and $x\not\simAB u$ if $k$ is even,
and thus, by \ref{prop:B3:6}, the number of choices for $S$ in this case is
\[
(1+5\gamma)\cdot 3\cdot \beta_0\cdot p_{\mathcal{I}}\cdot n\cdot (1\pm \eps)\cdot \Phi \cdot p_{\mathrm{vx}}^{-1}\cdot n^{-1}=(1\pm 10\gamma)\cdot \beta_0\cdot p_{\mathcal{J}}\cdot \Phi \cdot p_{\mathrm{vx}}^{-1}.
\]
\end{itemize}
Thus, as $61\beta_0p_{\mathcal{J}}=\beta _0^2p_Z=p_{\mathrm{vx}}$, in total we have $d_{\mathcal{H}_3}((i,v))=(1\pm  10\gamma)\cdot 61 \cdot \beta_0\cdot p_{\mathcal{J}}\cdot \Phi \cdot p_{\mathrm{vx}}^{-1}=(1\pm  10\gamma)\delta_3$.

\smallskip

\noindent\textbf{iii)} Let $(i,c)\in \mathcal{C}_3$, so that $i\in [n]$ and $c\in D_3\setminus (C_i\cup D_{3,i})$. We count the number of $f=\{(i,u),(j,v)\}\in \mathcal{J}$ and $S\in \mathcal{R}_f$ for which $c$ is used as a colour by considering the following cases:
\begin{itemize}
\item $c$ is the colour of the first edge of $S$: By \ref{prop:fromB2:6}, we can bound the number of choices for $f=\{(i,u),(j,v)\}\in \mathcal{J}$ for which $c\in D_3\setminus (C_i\cup C_j\cup D_{3,i}\cup D_{3,j})$ and there is a colour-$c$ edge from $u$ to $Z_{i,0}\cap Z_{j,0}$ in $E_0^\abs$, and thus, by \ref{prop:B3:7}, the number of choices for $S$ in this case is
\[
(1+5\gamma)\cdot 2\cdot  {p}_{\mathrm{vx}}\cdot \beta_0p_\abs\cdot p_{\mathrm{edge}}\cdot p_{\mathcal{J}}\cdot n\cdot (1\pm \eps)\cdot \Phi \cdot p_{\mathrm{vx}}^{-1}\cdot p_{\mathrm{col}}^{-1}\cdot p_{\mathrm{edge}}^{-1}\cdot n^{-1}=(1\pm  10\gamma)\cdot 2\cdot   \beta_0p_\abs\cdot p_{\mathcal{J}}\cdot \Phi \cdot p_{\mathrm{col}}^{-1}.
\]
\item $c$ is the colour of the last edge of $S$: By \ref{prop:fromB2:7}, we can bound the number of choices for $f=\{(i,u),(j,v)\}\in \mathcal{J}$ for which $c\in D_3\setminus (C_i\cup C_j\cup D_{3,i}\cup D_{3,j})$ and there is a colour-$c$ edge from $v$ to $Z_{i,0}\cap Z_{j,0}$ in $E_0^\abs$, and thus, by \ref{prop:B3:7}, the number of choices for $S$ in this case is
\[
(1+5\gamma)\cdot 2\cdot {p}_{\mathrm{vx}}\cdot \beta_0p_\abs\cdot p_{\mathrm{edge}}\cdot p_{\mathcal{J}}\cdot n\cdot (1\pm \eps)\cdot \Phi \cdot p_{\mathrm{vx}}^{-1}\cdot p_{\mathrm{col}}^{-1}\cdot p_{\mathrm{edge}}^{-1}\cdot n^{-1}=(1\pm  10\gamma)\cdot 2\cdot   \beta_0p_\abs\cdot p_{\mathcal{J}}\cdot \Phi \cdot p_{\mathrm{col}}^{-1}.
\]
\item $c$ is not the colour of the first or last edge of $S$: By \ref{prop:fromB2:8}, we can bound the number of choices for $f=\{(i,u),(j,v)\}\in \mathcal{J}$ for which $c\in D_3\setminus (C_i\cup C_j\cup D_{3,i}\cup D_{3,j})$, and thus, by \ref{prop:B3:9}, the number of choices for $S$ in this case is
\[
(1+5\gamma)\cdot 3\cdot 2\cdot  \beta_0p_\abs\cdot p_{\mathcal{I}}\cdot (1\pm \eps)\cdot \Phi \cdot  p_{\mathrm{col}}^{-1}\cdot n^{-1}=(1\pm  10\gamma)\cdot   2\cdot \beta_0p_\abs\cdot p_{\mathcal{J}}\cdot \Phi \cdot p_{\mathrm{col}}^{-1}.
\]
\end{itemize}
Thus, as $p_\col=\beta_0^2p_\abs^2p_3=62p_{\mathcal{J}}$ in total, we have $d_{\mathcal{H}_3}((i,c))=(1\pm  10\gamma)\cdot 62\cdot \beta_0p_\abs\cdot p_{\mathcal{J}}\cdot \Phi \cdot p_{\mathrm{col}}^{-1}=(1\pm  10\gamma)\delta_3$.

\smallskip

\noindent\textbf{iv)} Let $xy\in \mathcal{E}_3$, so that $xy\in E(G|_{D_3})\cap E_0^\abs$.  We count the number of $f=\{(i,u),(j,v)\}\in \mathcal{J}$ and $S\in \mathcal{R}_f$ for which $xy$ is an edge of $f$ by considering the following cases:
\begin{itemize}
\item $xy$ is the first or last edge of $S$: By \ref{prop:fromB2:9},  we can bound the number of choices for $f=\{(i,u),(j,v)\}\in \mathcal{J}$ with $u=x$ for which $c(xy)\in D_3\setminus (C_i\cup C_j\cup D_{3,i}\cup D_{3,j})$
and $y\in Z_{i,0}\cap Z_{j,0}$, and thus, by \ref{prop:B3:10}, the number of choices for $S$ where $xy$ is the first edge and $u=x$ is
\[
(1+5\gamma)\cdot {p}_{\mathrm{vx}}\cdot \beta_0^2p_\abs^2\cdot p_{\mathcal{J}}\cdot n\cdot (1\pm \eps)\cdot \Phi \cdot  p_{\mathrm{vx}}^{-1}\cdot p_{\mathrm{col}}^{-1}\cdot p_{\mathrm{edge}}^{-1}\cdot n^{-1}=(1\pm  10\gamma)\cdot p_{\mathcal{J}}\cdot \Phi \cdot  p_{\mathrm{edge}}^{-1}\cdot p_3^{-1}.
\]
As a similar bound holds with $x$ and $y$ switched, there are  $(1\pm  10\gamma)\cdot 2\cdot p_{\mathcal{J}}\cdot \Phi \cdot  p_{\mathrm{edge}}^{-1}\cdot p_3^{-1}$ choices for $S$ in total in this case.

\item $xy$ is the 2nd or 61st edge of $S$:  By \ref{prop:fromB2:10}, we can bound the number of choices for $f=\{(i,u),(j,v)\}\in \mathcal{J}$ for which $c(ux),c(xy)\in D_3\setminus (C_i\cup C_j\cup D_{3,i}\cup D_{3,j})$, $x,y\in Z_{i,0}\cap Z_{j,0}$,
and $ux\in E_0^\abs$, and thus, by \ref{prop:B3:11}, the total number of choices for $S$ in which $xy$ is the second vertex and $xy$ is the second edge is at most
\begin{align*}
(1+5\gamma)\cdot {p}_{\mathrm{vx}}^2\cdot & {p}_{\mathrm{col}}\cdot \beta_0^2p_\abs^2\cdot p_{\mathrm{edge}}\cdot p_{\mathcal{J}}\cdot n^2\cdot (1\pm \eps)\cdot \Phi \cdot  p_{\mathrm{vx}}^{-2}\cdot p_{\mathrm{col}}^{-2}\cdot p_{\mathrm{edge}}^{-2}\cdot n^{-2}\\
&=(1\pm  10\gamma)\cdot p_{\mathcal{J}}\cdot  \Phi \cdot  p_{\mathrm{edge}}^{-1}\cdot p_3^{-1}.
\end{align*}
As a similar bound holds with $x$ and $y$ switched, in total there are $(1\pm  10\gamma)\cdot 2\cdot p_{\mathcal{J}}\cdot  \Phi \cdot  p_{\mathrm{edge}}^{-1}\cdot p_3^{-1}$ choices for $S$ in this case (as when $xy$ is the 2nd edge of a $(u,v,L)$-link, it is the 61st edge of that subgraph considered as a $(v,u,L)$-link).

\item $xy$ is the $k$th edge of $S$, for $3\leq k\leq 60$: By \ref{prop:fromB2:13}, we can bound the number of choices  $f=\{(i,u),(j,v)\}\in \mathcal{J}$ for which $c(uv)\in D_3\setminus (C_i\cup C_j\cup D_{3,i}\cup D_{3,j})$
and $u,v\in Z_{i,0}\cap Z_{j,0}$, and thus, by \ref{prop:B3:14}, the number of choices for $S$ in which $xy$ is the $k$th edge is
\begin{align*}
(1+5\gamma)\cdot 3\cdot 2{p}_{\mathrm{vx}}^2\cdot& {p}_{\mathrm{col}}\cdot p_{\mathcal{I}}\cdot n^2\cdot (1\pm \eps)\cdot \Phi \cdot  p_{\mathrm{vx}}^{-2}\cdot p_{\mathrm{col}}^{-1}\cdot p_{\mathrm{edge}}^{-1}\cdot n^{-2}\\
&=(1\pm  10\gamma)\cdot 2p_{\mathcal{J}}\cdot  \Phi \cdot p_{\mathrm{edge}}^{-1}\cdot p_3^{-1}
\end{align*} choices for $S$.
\end{itemize}
As when $xy$ is the $k$th edge of a $(u,v,L)$-link, it is the $(62-k)th$ edge of that subgraph considered as a $(v,u,L)$-link, we have, in total, that $d_{\mathcal{H}_3}(xy)=(1\pm  10\gamma)\cdot 62\cdot p_{\mathcal{J}}\cdot  \Phi \cdot p_{\mathrm{edge}}^{-1}\cdot p_3^{-1}=(1\pm  10\gamma)\delta_3$, where we have used that $p_3p_{\mathrm{edge}}=p_3\beta_0p_{\abs}=62p_{\mathcal{J}}$.
\claimproofend


\subsubsection{Codegrees in $\mathcal{H}_3$}\label{sec:H3:lowcod}

We will now show that the codegrees of $\mathcal{H}_3$ are all $O(n^{29.5})$.
As the vertex degrees of $\mathcal{H}_3$ are (by Claim~\ref{clm:H3almostregular}) all around $\delta_3$, where $\delta_3=\Phi=p_{\mathrm{vx}}^{61}\cdot p_{\mathrm{col}}^{30}\cdot p_{\mathrm{edge}}^{62}\cdot n^{30}$, and $1/n\llpoly p_{\mathrm{vx}},p_{\mathrm{col}},p_{\mathrm{edge}}$,
these codegrees are all much smaller than the vertex degrees in $\mathcal{H}_3$.

\begin{claim}\label{clm:H3lowcod}
$\Delta^c(\cH_3)=O(n^{29.5})$.
\end{claim}
\begin{proof}[Proof of Claim~\ref{clm:H3lowcod}] Let $f=\{(i,u),(j,v)\}\in \mathcal{J}$ and $S\in \mathcal{R}_f$, and consider the edge
\[
e=\{f\}\cup (\{i,j\}\times ((V(S)\setminus\{u,v\})\cup C(S))\cup E(S).
\]
Let ${v}_1$ and $v_2$ be two vertices in $e$.

If one of $v_1$ or $v_2$ is $f$, then note that from the other vertex we will know an internal vertex of $S$ (possibly by knowing an edge) or a colour of $S$, so therefore, by \ref{prop:B3:cod} (in particular \ref{prop:links:throughvertex} or \ref{prop:links:throughcolour}), we have that $d_{\mathcal{H}_3}(v_1,v_2)= O(n^{29})=O(n^{29.5})$.
Assume, then, that neither $v_1$ or $v_2$ is $f$.

Note that we have one of the following cases \ref{case:a}--\ref{case:c} (up to relabelling $i$ and $j$).
\begin{enumerate}[label = \textbf{\alph{enumi})}]
\item We know $i$ and either \label{case:a}
\begin{enumerate}[label = \textbf{\roman{enumii})}]
\item two internal vertices of $S$, or\label{case:a:1}
\item two colours of $S$, or\label{case:a:2}
\item an internal vertex and a colour of $S$.\label{case:a:3}
\end{enumerate}
\item We know $i$ and $j$ and either \label{case:b}
\begin{enumerate}[label = \textbf{\roman{enumii})}]
\item an internal vertex of $S$, or\label{case:b:1}
\item a colour of $S$.\label{case:b:2}
\end{enumerate}
\item We know two edges of $S$.\label{case:c}
\end{enumerate}

\smallskip

\ref{case:a} Knowing $i$, we have at most $3p_{\mathcal{J}}n\leq n$ choices for $\{(i,u),(j,v)\}\in \mathcal{J}$ by \ref{prop:from:B2:11a}, after which there are $O(n^{28})$ choices for a $(u,v,L)$-link in $G$ in \ref{case:a}\ref{case:a:1} and \ref{case:a:2} by \ref{prop:B3:cod} (in particular \ref{prop:links:cod:twovertices} or \ref{prop:links:cod:twocolours}), and thus $d_{\mathcal{H}_3}(v_1,v_2)=O(n^{29})$ in these cases. In \ref{case:a}\ref{case:a:3}, if colour $c$ and internal vertex $x$ are known, then, similarly by \ref{prop:from:B2:11a} and \ref{prop:B3:cod} (in particular \ref{prop:links:cod:1vertex1colour}), there are at most $O(n^{29})$ choices for $\{(i,u),(j,v)\}\in \mathcal{J}$ and a $(u,v,L)$-link $S$ in $G$ using $c$ and $x$ so that neither $ux$ or $vx$ is a colour-$c$ edge in $S$. To count the remaining choices $ux$ or $vx$ is a colour-$c$ edge in $S$, note that we have 2 choices for whether $u$ or $v$ is the colour-$c$ neighbour of $x$, after which, by \ref{prop:forB3cod:1} or \ref{prop:forB3cod:2}, there are $O(1)$ choices for $\{(i,u),(j,v)\}\in \mathcal{J}$ and then, by \ref{prop:B3:cod} (in particular \ref{prop:links:throughvertex}), $O(n^{29})$ choices of a $(u,v,L)$-link $S$ in $G$ using $x$ as an interior vertex. Thus, we also have $d_{\mathcal{H}_3}(v_1,v_2)=O(n^{29})$ in \ref{case:a}\ref{case:a:3},

\smallskip

\ref{case:b} Knowing $i$ and $j$, we have at most $3n^{1/2}$ choices for $\{(i,u),(j,v)\}\in \mathcal{J}$ by \ref{prop:from:B2:11}, after which there are $O(n^{29})$ choices for a $(u,v,L)$-link in $G$ in \ref{case:b}\ref{case:b:1} and \ref{case:b:2} by \ref{prop:B3:cod} (in particular \ref{prop:links:throughvertex} or \ref{prop:links:throughcolour}). Thus, $d_{\mathcal{H}_3}(v_1,v_2)=O(n^{29.5})$ in this case.

\smallskip

\ref{case:c} Suppose the known edges are $xy$ and $x'y'$. We count separately the possibilities for the link where \textbf{i)} there are at least three internal vertices of $S$ among these edges, or \textbf{ii)} where the edges lie at each end of the link.
For \textbf{i)}, after choosing $i\in [n]$ (with $n$ choices), there are  at most $3p_{\mathcal{J}}n\leq n$ choices for $\{(i,u),(j,v)\}\in \mathcal{J}$ by \ref{prop:from:B2:11a} for which $u\notin\{x,y,x',y'\}$, after which there are $O(n^{27})$ possibilities for a link containing $xy$ and $x'y'$ by \ref{prop:B3:cod} (in particular  \ref{prop:links:throughedgeandvertex}), for $O(n^{29})$ choices in total, so that $d_{\mathcal{H}_3}(v_1,v_2)= O(n^{29})$.
For \textbf{ii)}, we have at most 2 choices to pick $u\in \{x,y\}$ and $v\in \{x',y'\}$ with $u\sim_{A/B}v$. After choosing $i\in [n]$ (with at most $n$ choices), we have at most $O(1)$ choices for $\{(i,u),(j,v)\}\in \mathcal{J}$
by \ref{prop:forB3cod:1}, and then at most $O(n^{28})$ choices for a link with $xy$ as the first edge and $x'y'$ as the last edge by \ref{prop:links:cod:twovertices}. Therefore,
$d_{\mathcal{H}_3}(v_1,v_2)=O(n^{29})$ in this case as well.
\claimproofend


\subsubsection{Weight functions for Part~\ref{partB3} and the choice of $\mathcal{M}_3$}
We now define the weight functions we use with our application of Theorem~\ref{thm:nibble} to $\mathcal{H}_3$.
For each $i\in [n]$, $v\in V(G)$, $c\in D_3$, $\phi\in \mathcal{F}$, and each
\begin{equation}\label{eqn:edefn}
e=\{\{(i',u'),(j',v')\}\cup (\{i',j'\}\times ((V(S)\setminus\{u',v'\})\cup C(S))\cup E(S)\in E(\cH_3),
\end{equation}
set $w_i(e)=\mathbf{1}_{\{i\in \{i',j'\}\}}$,
\[
w^{\mathrm{link:end}}_v(e)=\mathbf{1}_{\{v\in \{u',v'\}\}},
\;\;\; w^{\mathrm{link:mid}}_v(e)=\mathbf{1}_{\{v\in V(S)\setminus\{u',v'\}\}},
\;\;\; w^{\mathrm{link:mid}}_{v,\phi}(e)=\mathbf{1}_{\{v\in V(S)\setminus\{u',v'\},i\in I_\phi\}},
\]
\[
 w_{c,\phi}(e)=\mathbf{1}_{\{c\in C(S),i\in I_\phi\}}\;\;\text{and}\;\; w_c(e)=\mathbf{1}_{\{c\in C(S)\}}.
\]
Let $\mathcal{W}_3=\{w_i:i\in [n]\}\cup \{w_v^{\mathrm{link:end}},w_v^{\mathrm{link:mid}}:v\in V(G)\}\cup \{w_c:c\in D_3\}\cup \{w_{v,\phi}^{\mathrm{link:mid}}:v\in V(G),\phi\in \mathcal{F}\}\cup \{w_{c,\phi}:c\in D_3,\phi\in \mathcal{F}\}$.

For each $i\in [n]$,
we have, using Claim~\ref{clm:H3almostregular}, that
\begin{equation}\label{eqn:W3:itotalweight}
w_i(E(\mathcal{H}_3))=\sum_{(u,j,v):\{(i,u),(j,v)\}\in \mathcal{J}}d_{\mathcal{H}_3}(\{(i,u),(j,v)\})=
|\{(u,j,v):\{(i,u),(j,v)\}\in \mathcal{J}\}|\cdot (1\pm 10\gamma)\cdot\delta_3.
\end{equation}
For each $v\in V(G)$,
using Claim~\ref{clm:H3almostregular}, we have
\begin{equation}\label{eqn:W3:vlinktotalweight}
w_v^{\mathrm{link:end}}(E(\mathcal{H}_3))=|\{(i,j,u):\{(i,u),(j,v)\}\in \mathcal{J}\}|\cdot (1\pm 10\gamma)\cdot\delta_3.
\end{equation}
Furthermore, for each $v\in V(G)$,
as each $e\in E(\mathcal{H}_3)$ corresponding to a link containing a vertex $v$ contains two vertices containing $v$ (i.e., $(i',v)$ and $(j',v)$ in the notation at \eqref{eqn:edefn}), we have by Claim~\ref{clm:H3almostregular} that
\begin{equation}\label{eqn:W3:vtotalweight}
w_v^{\mathrm{link:mid}}(E(\mathcal{H}_3))=\frac{1}{2}\cdot |\{i\in [n]:v\in Z_{i,0}\}|\cdot(1\pm 10\gamma)\cdot \delta_3,
\end{equation}
while, for each $\phi\in \mathcal{F}$,
\begin{equation}\label{eqn:W3:midvphitotalweight}
w_{v,\phi}^{\mathrm{link:mid}}(E(\mathcal{H}_3))=\frac{1}{2}\cdot |\{i\in I_\phi:v\in Z_{i,0}\}|\cdot(1\pm 10\gamma)\cdot \delta_3.
\end{equation}

Finally, for each $c\in D_3$, as each $e\in E(\mathcal{H}_3)$  corresponding to a link using the colour $c$ contains two vertices containing $c$ (i.e., $(i',c)$ and $(j',c)$ in the notation at \eqref{eqn:edefn})
we have by from Claim~\ref{clm:H3almostregular} that
\begin{equation}\label{eqn:W3:ctotalweight}
w_c(E(\mathcal{H}_3))=\frac{1}{2}\cdot |\{i\in [n]:c\in D_3\setminus (C_i\cup D_{3,i})\}|\cdot(1\pm 10\gamma)\cdot \delta_3,
\end{equation}
while, for each $\phi\in \mathcal{F}$,
\begin{equation}\label{eqn:W3:midcphitotalweight}
w_{c,\phi}(E(\mathcal{H}_3))=\frac{1}{2}\cdot |\{i\in I_\phi:c\in D_3\setminus (C_i\cup D_{3,i})\}|\cdot(1\pm 10\gamma)\cdot \delta_3.
\end{equation}

In particular, \eqref{eqn:W3:itotalweight}--\eqref{eqn:W3:ctotalweight} along with \ref{prop:from:B2:11a}, \ref{prop:from:B2:vertexinJ}, \ref{prop:B3:assorted:2}, and \ref{prop:B3:assorted:1} imply that, for each $w\in \mathcal{W}_3$, $w(E(\mathcal{H}_3))\geq \sqrt{n}\cdot \delta_3$.
Therefore, by Claims~\ref{clm:H3almostregular} and~\ref{clm:H3lowcod}, and Theorem~\ref{thm:nibble}, we can find a matching $\cM_3$ in $\mathcal{H}_3$ such that, for each $w\in \mathcal{W}_3$,
\begin{equation}\label{eqn:M3:balanced}
w(\mathcal{M}_3)=(1\pm 20\gamma)\cdot\delta_3^{-1} \cdot w(E(\mathcal{H}_3)).
\end{equation}

Now, for each $f=\{(i,u),(j,v)\}\in \mathcal{J}\cap V(\mathcal{M}_3)$, let $S_f$ be the $(u,v,L)$-link corresponding to the edge containing $f$ in $\mathcal{M}_3$, let $M_{f,i}$ be the path of the odd edges of this link and let $M_{f,j}$ be the path of the even edges of the link, noting that these are both rainbow matchings and that $C(M_{f,i})=C(M_{f,j})$. For each $i\in [n]$, let
\[
\hat{M}_{i,3,0}=\bigcup_{f=\{(i,u),(j,v)\}\in \mathcal{J}\cap V(\mathcal{M}_3)}M_{f,i}.
\]
Let $\mathcal{J}^-=\mathcal{J}\setminus V(\mathcal{M}_3)$, the set of instructions we have not found a link for. We now show the following properties of $\mathcal{J}^-$ and the matchings $\hat{M}_{i,3,0}$, $i\in [n]$.

\begin{claim}\label{clm:propsofMi30} \begin{enumerate}[label = \textbf{\alph{enumi})}]
\item The matchings $\hat{M}_{i,3,0}$, $i\in [n]$, are edge-disjoint.\label{prop:ofMi30:1}
\item For each $i\in [n]$, $\hat{M}_{i,3,0}$ is a rainbow matching with colours in $D_3\setminus (C_i\cup D_{3,i})$ and vertices in $S_i\cup Z_{i,0}$.\label{prop:ofMi30:2}
\item Setting $G_0^\abs$ to be the graph with vertex set $V(G)$ and edge set $E_0^\abs$, we have that $G_0^\abs|_{D_3}-\hat{M}_{1,3,0}-M_{2,3,0}-\ldots-\hat{M}_{n,3,0}$ is $(3\gamma n)$-bounded.\label{prop:ofMi30:3}
\item For each $i\in [n]$, $|\{(u,j,v):\{(i,u),(j,v)\}\in \mathcal{J}^-\}|{\leq} \gamma n$.\label{prop:ofJminus:1}
\item For each $u\in V(G)$, $|\{(i,j,v):\{(i,u),(j,v)\}\in \mathcal{J}^-\}|{\leq} \gamma n$.\label{prop:ofJminus:2}
\item For each $v\in V(G)$ and $\phi\in \mathcal{F}$, $|\{i\in I_\phi:v\in Z_{i,0}\setminus V(\hat{M}_{i,3,0})\}|\leq 2\gamma p_\tr p_\fa n$.\label{prop:finalB:fromB3}
\item For each $c\in D_3$ and $\phi\in \mathcal{F}$, $|\{i\in I_\phi:c \notin (C(\hat{M}_{i,3,0}) \cup C_i\cup D_{3,i})\}|\leq \gamma p_\tr p_\fa n$.\label{prop:finalB:fromB3:2}
\item For each $i\in [n]$, $|Z_{i,0}\setminus V(\hat{M}_{i,3,0})|\leq \gamma n$. \label{prop:finalB:fromB3:3}
\end{enumerate}
\end{claim}
\begin{proof}[Proof of Claim~\ref{clm:propsofMi30}]
\ref{prop:ofMi30:1}: This follows as each edge in $\mathcal{E}_3$ appears in at most one of the edges of $\mathcal{M}_3$.

\smallskip

\noindent \ref{prop:ofMi30:2}: For each $i\in [n]$, as each pair $\{(i,x)\}$ or $\{(i,c)\}$, with $x\in V(G)$ and $c\in C$, appears at most once in the edges of $\mathcal{M}_3$, we have that the rainbow matchings $M_{f,i}$ corresponding to the edges of $\mathcal{M}_3$ involving $i$ are vertex- and colour-disjoint. That their vertices are in $S_i\cup Z_{i,0}$ and colours are in $D_3\setminus (C_i\cup D_{3,i})$ follows directly from the definition of the edges of $\mathcal{M}_3$.

\smallskip

\noindent \ref{prop:ofMi30:3}:
For each $c\in D_3$,
\begin{align}
|\{e\in E(G_0^\abs)\setminus &(\cup_{i\in [n]}\hat{M}_{i,3,0}):c(e)=c)\}|\nonumber\\
&=
|\{e\in E(G_0^\abs):c(e)=c\}|-|\{i\in [n]:(i,c)\in V(\mathcal{M}_3)\}|\nonumber\\
&\overset{\ref{prop:B3:assorted:1}}{\leq}
(1+ \eps)\cdot \beta_0p_\abs n-2\cdot w_c(E(\mathcal{H}_3))\nonumber\\
&\hspace{-0.3cm}\overset{\eqref{eqn:W3:ctotalweight},\eqref{eqn:M3:balanced}}{\leq}  (1+ \eps)\cdot \beta_0p_\abs n-(1- 20\gamma)\cdot (1-10\gamma)\cdot |\{i\in [n]:c\in D_3\setminus (C_i\cup D_{3,i})\}|\nonumber\\
&\overset{\ref{prop:B3:assorted:1}}{\leq} (1+ \eps)\cdot \beta_0p_\abs n-(1- 200\gamma)\cdot (1-\eps)^2\cdot \beta_0p_\abs n
\leq 3\gamma \beta_0p_\abs n\leq 10^3\gamma n.\label{eq:forboundednessofrestofG0:1}
\end{align}
Furthermore, for each $v\in V(G)$,
\begin{align}
|\{e&\in E(G_0^\abs|_{D_3})\setminus(\cup_{i\in [n]}\hat{M}_{i,3,0}):v\in V(e)\}|\nonumber \\
&=|\{e\in E(G_0^\abs|_{D_3}):v\in V(e)\}|-|\{i\in [n]:(i,v)\in V(\mathcal{M}_3)\}|\nonumber\\
&\overset{\ref{prop:B3:assorted:2}}{\leq}
(1+ \eps)\cdot \beta_0p_\abs p_3 n-w^{\mathrm{link:end}}_v(E(\mathcal{H}_3))-2\cdot w^{\mathrm{link:mid}}_v(E(\mathcal{H}_3))\nonumber\\
&\hspace{-0.6cm}\overset{\eqref{eqn:W3:vlinktotalweight},\eqref{eqn:W3:vtotalweight},\eqref{eqn:M3:balanced}}{\leq}  (1+ \eps)\cdot \beta_0p_\abs p_3 n-(1- 20\gamma)\cdot (1-10\gamma)\cdot (|\{(i,j,u):\{(i,u),(j,v)\}\in \mathcal{J}\}|\nonumber \\
&\hspace{4cm}+|\{i\in [n]:v\in Z_{i,0}\}|)\nonumber\\
&\hspace{-0.1cm}\overset{\ref{prop:from:B2:vertexinJ},\ref{prop:B3:assorted:2}}{\leq}  (1+ \eps)\cdot \beta_0p_\abs p_3 n-(1- 200\gamma)\cdot (1-\eps)\cdot ((1- \beta)p_{\mathcal{J}}n+(1- \eps)p_Z\beta_0n)\nonumber
\end{align}
\begin{align}
& =(1+ \eps)\cdot \beta_0p_\abs p_3 n-(1- 2\gamma)\cdot (1-\eps)\cdot (1- \beta)\cdot 62p_{\mathcal{J}}n
\nonumber\\
&\leq 2\beta \cdot \beta_0p_\abs p_3 n\leq 2\beta n.\label{eq:forboundednessofrestofG0:2}
\end{align}
In combination, \eqref{eq:forboundednessofrestofG0:1} and \eqref{eq:forboundednessofrestofG0:2} show that \ref{prop:ofMi30:3} holds.

\smallskip

\noindent\ref{prop:ofJminus:1}: For each $i\in [n]$,
\begin{align*}
|\{(u,j,v):\{(i,u),(j,v)\}\in \mathcal{J}^-\}|&\leq |\{(u,j,v):\{(i,u),(j,v)\}\in \mathcal{J}\}|-w_i(E(\mathcal{M}_3))\\
&\overset{\eqref{eqn:W3:itotalweight},\eqref{eqn:M3:balanced}}{\leq} 10^3\gamma\cdot |\{(u,j,v):\{(i,u),(j,v)\}\in \mathcal{J}\}|
\overset{\ref{prop:from:B2:11a}}\leq  \gamma n.
\end{align*}

\smallskip

\noindent\ref{prop:ofJminus:2}: For each $u\in V(G)$,\renewcommand{\qedsymbol}{$\boxdot$}
\begin{align*}
|\{(i,j,v):\{(i,u),(j,v)\}\in \mathcal{J}^-\}|&=|\{(i,j,v):\{(i,u),(j,v)\}\in \mathcal{J}\}|-w_v^{\mathrm{link:end}}(E(\mathcal{M}_3))\\
&\overset{\eqref{eqn:W3:vlinktotalweight},\eqref{eqn:M3:balanced}}{\leq} 10^3\gamma\cdot |\{(i,j,v):\{(i,u),(j,v)\}\in \mathcal{J}\}|\overset{\ref{prop:from:B2:vertexinJ}}{\leq}
\gamma n.\qedhere
\end{align*}

\smallskip

\noindent\ref{prop:finalB:fromB3}: For each $v\in V(G)$ and $\phi\in \mathcal{F}$,
\begin{align*}
|\{i\in [n]:v\in Z_{i,0}\setminus V(\hat{M}_{i,3,0})\}|&=|\{i\in I_\phi:v\in Z_{i,0},(i,v)\notin V(\mathcal{M}_3)\}|\\
&=|\{i\in I_\phi:v\in Z_{i,0}\}|-2\cdot w_{v,\phi}^{\mathrm{link:mid}}(E(\mathcal{M}_3))
\overset{\eqref{eqn:M3:balanced},\eqref{eqn:W3:midvphitotalweight}}{\leq} 2\gamma p_\tr p_\fa n.
\end{align*}

\smallskip

\noindent\ref{prop:finalB:fromB3:2}: For each $c\in D_3$ and $\phi\in \mathcal{F}$, $|\{i\in I_\phi:c \notin (C(\hat{M}_{i,3,0}) \cup C_i\cup D_{3,i})\}|\leq \gamma p_\tr p_\fa n$.
\begin{align*}
|\{i\in I_\phi:c \notin (C(\hat{M}_{i,3,0}) \cup C_i\cup D_{3,i})\}|&=|\{i\in I_\phi:c\in D_3\setminus (C_i\cup D_{3,i})\}|-|\{i\in I_\phi:(i,c)\in V(\mathcal{M}_3)\}|\\
&=|\{i\in I_\phi:c\in D_3\setminus (C_i\cup D_{3,i})\}|-2\cdot w_{c,\phi}(E(\mathcal{M}_3))\\
&\overset{\eqref{eqn:M3:balanced},\eqref{eqn:W3:midcphitotalweight}}{\leq} 2\gamma n.
\end{align*}

\smallskip

\noindent\ref{prop:finalB:fromB3:3}: For each $i\in [n]$,
\begin{align*}
|Z_{i,0}\setminus V(\hat{M}_{i,3,0})|&=|\{v\in Z_{i,0}:(i,v)\notin V(\mathcal{M}_3)\}|
=|Z_{i,0}|-2\cdot w_i(E(\mathcal{M}_3))\\
&\overset{\eqref{eqn:M3:balanced},\eqref{eqn:W3:itotalweight}}{\leq}
|Z_{i,0}|-(1-2\gamma)\cdot |\{(u,j,v):\{(i,u),(j,v)\}\in \mathcal{J}\}|\\
&=|Z_{i,0}|-(1-2\gamma)\cdot (1+5r)\cdot |S_i\setminus R_i|\\
&\overset{\ref{prop:sizeofZi0},\ref{prop:B2:sizeofSminusR}}\leq (1+\eps)2\beta_0p_Zn-(1-2\gamma)\cdot 121\cdot 2p_{S-R}n\leq 4\gamma n,
\end{align*}
as required.
\claimproofend


\subsubsection{Missing links and the choice of $\hat{M}_{i,3}$}
For each $f=\{(i,u),(j,v)\}\in \mathcal{J}^-$, let $\mathcal{R}^+_{f}$ be the set of $(u,v,L)$-links with colours in $(D_{3,i}\cup D_{3,j})\setminus (C_i\cup C_j)$, edges in $E^\abs_{1}$ and internal vertices in $Z_{i,1}\cap Z_{j,1}$,
so that $|\mathcal{R}^+_f|=(1\pm \eps)\Phi_1=q_{\mathrm{vx}}^{61}\cdot q_{\mathrm{col}}^{30}\cdot q_{\mathrm{edge}}^{62}\cdot n^{30}$ by \ref{prop:B3:missing}.

For each $f\in \mathcal{J}^-$, form $\mathcal{R}^0_{f}$ by selecting elements of $\mathcal{R}^+_f$ independently at random with probability $q_0=20\log^8n/\Phi_1$. By a simple application of Lemma~\ref{chernoff} and a union bound, then, we get that with high probability, for each $f\in \mathcal{J}^-$, $|\mathcal{R}^0_{f}|\geq 10\log^8n$. We now show the following claim.

\begin{claim}\label{clm:intermezzo} With high probability, for each $f=\{(i,u),(j,v)\}\in \mathcal{J}^-$ and $S\in \mathcal{R}^+_f$, the following hold.
\begin{enumerate}[label = \textbf{\arabic{enumi})}]
\item There are at most $\log^2n$ links $f'\in \mathcal{J}^-\setminus \{f\}$ and $S'\in \mathcal{R}^0_{f'}$ with $E(S)\cap E(S')\neq \emptyset$.\label{prop:inter:forB3follow1}
\item There are at most $\log^2n$ tuples $(u',j',v',f',S')$ with $f'=\{(i,u'),(j',v')\}\in \mathcal{J}^-\setminus\{f\}$, $S'\in \mathcal{R}^0_{f'}$ and either $(V(S)\cap V(S'))\setminus (\{u,v\}\cap \{u',v'\}) \neq \emptyset$ or $C(S)\cap C(S')\neq \emptyset$.\label{prop:inter:forB3follow3}
\end{enumerate}
\end{claim}
\begin{proof}[Proof of Claim~\ref{clm:intermezzo}] Let $f=\{(i,u),(j,v)\}\in \mathcal{J}^-$ and $S\in \mathcal{R}^+_f$. By Claim~\ref{clm:propsofMi30} \ref{prop:ofJminus:1} and \ref{prop:ofJminus:2}, there are at most $\gamma n^2$ choices
 for $\{(i',u'),(j',v')\}\in \mathcal{J}^-\setminus \{f\}$ and at most $100\gamma n$ choices for $\{(i',u'),(j',v')\}\in \mathcal{J}^-\setminus \{f\}$ such that $\{u',v'\}\cap V(S)\neq \emptyset$,
 after which there are at most $100n^{28}$ and $100n^{29}$ choices respectively for $S'\in \mathcal{R}^+_{f'}$ with $E(S)\cap E(S')\neq \emptyset$ by \ref{prop:B3:cod} (and in particular \ref{prop:links:throughedge} and \ref{prop:links:throughvertex}, respectively).

 Thus, there are at most $10^5\gamma n^{30}$ different $(f',S')$ with $f'\in \mathcal{J}^-\setminus \{f\}$, $S'\in \mathcal{R}^+_{f'}$, and $E(S)\cap E(S')\neq \emptyset$. As $10^5\gamma n^{30}<\Phi_1/(20\log^{8}n)=1/q_0$ (using that $\gamma\llpoly q_{\mathrm{vx}},q_{\mathrm{col}},q_{\mathrm{edge}},\log^{-1}n$), the expected number of such $(f',S')$ with $S'\in \mathcal{R}^0_{f'}$ is less than 1. Thus, by a simple application of Lemma~\ref{chernoff}, we have that \ref{prop:inter:forB3follow1} holds with probability $1-n^{-\omega(1)}$. Taking a union bound shows that, with high probability, \ref{prop:inter:forB3follow1} holds
for all $f=\{(i,u),(j,v)\}\in \mathcal{J}^-$ and $S\in \mathcal{R}^+_f$.

Now, let $f=\{(i,u),(j,v)\}\in \mathcal{J}^-$ and $S\in \mathcal{R}^+_f$ again. By  Claim~\ref{clm:propsofMi30} \ref{prop:ofJminus:1}
there are at most $\gamma n$ choices for a tuple $(u',j',v')$ such that, setting $f'=\{(i,u'),(j',v')\}$, $f'\in \mathcal{J}^-\setminus \{f\}$.
After this, if $\{u',v'\}\cap \{u,v\}=\emptyset$, then, by  \ref{prop:B3:cod} (and in particular \ref{prop:links:throughvertex} and \ref{prop:links:throughcolour}) there are at most $200n^{29}$ choices for $S'\in \mathcal{R}^+_{f'}$ with $V(S)\cap V(S')\neq \emptyset$ or $C(S)\cap C(S')\neq \emptyset$, for at most $200\gamma n^{30}$ choices in total.
On the other hand, by \ref{prop:forB3cod:1} and \ref{prop:forB3cod:2}, there are at most $10^7$ triples
$(u',j',v')$ with $\{u,v\}\cap \{u',v'\}\neq \emptyset$ and, setting $f'=\{(i,u'),(j',v')\}$, $f'\in \mathcal{J}^-\setminus \{f\}$. Then, as $u,v,u',v'\notin Z_{i,0}$, by \ref{prop:B3:cod} (and in particular \ref{prop:links:throughvertex} and \ref{prop:links:throughcolour}) there are at most $200n^{29}$ choices for $S'\in \mathcal{R}^+_{f'}$ with $(V(S)\cap V(S'))\setminus (\{u,v\}\cap \{u',v'\}) \neq \emptyset$ or $C(S)\cap C(S')\neq \emptyset$.

Thus, there are at most $200\gamma n^{30}$ different $(f',S')$ with $f'\in \mathcal{J}^-\setminus \{f\}$, $S'\in \mathcal{R}^+_{f'}$, and
$(V(S)\cap V(S'))\setminus (\{u,v\}\cap \{u',v'\}) \neq \emptyset$ or $C(S)\cap C(S')\neq \emptyset$. As $10^5\gamma n^{30}<\Phi_1$ the expected number of such $(f',S')$ with $S'\in \mathcal{R}^0_{f'}$ is less than 1. Therefore, similarly to as we did for  \ref{prop:inter:forB3follow1}, we can show that, with high probability, \ref{prop:inter:forB3follow3} holds
for all $f=\{(i,u),(j,v)\}\in \mathcal{J}^-$ and $S\in \mathcal{R}^+_f$.
\claimproofend

Therefore, as they hold together with high probability, we can take a choice of $\mathcal{R}^0_f$, $f\in \mathcal{J}'$ for which  $|\mathcal{R}^0_{f}|\geq 10\log^8n$ for each $f\in \mathcal{J}^-$ and, for each $f\in \mathcal{J}'$ and $S\in \mathcal{R}^0_f$, \ref{prop:inter:forB3follow1} and \ref{prop:inter:forB3follow3} holds. Then, for each $f\in \mathcal{J}'$, form $\mathcal{R}_f$ by selecting elements of $\mathcal{R}^0_f$ independently at random with probability $q_1=1/\log^2 n$,
and let $\mathcal{R}^-_f$ be the set of $S\in \mathcal{R}^0_f$ such that the following both hold.
\begin{enumerate}[label = \textbf{\roman{enumi})}]
\item There is no $f'\in \mathcal{J}^-\setminus \{f\}$ and $S'\in \mathcal{R}^0_{f'}$ with $E(S)\cap E(S')\neq \emptyset$.\label{prop:forB3follow1}
\item There is no $(u',j',v',f')$ with $f'=\{(i,u'),(j',v')\}\in \mathcal{J}^-\setminus \{f\}$ for which there is an $S'\in \mathcal{R}_{f'}^0$ with $(V(S)\cap V(S'))\setminus (\{u,v\}\cap \{u',v'\})\neq \emptyset$ or $C(S)\cap C(S')\neq \emptyset$.\label{prop:forB3follow3}
\end{enumerate}

\begin{claim}\label{clm:H3findsremainder}
With high probability, for each $f\in \mathcal{J}^-$, $\mathcal{R}^-_f\neq \emptyset$.
\end{claim}
\begin{proof}[Proof of Claim~\ref{clm:H3findsremainder}] 
Let $f\in \mathcal{J}^-$ and $X_f=|\mathcal{R}^-_f|$. For each $S\in \mathcal{R}^0_f$, by the properties \ref{prop:inter:forB3follow1} and \ref{prop:inter:forB3follow3} from
Claim~\ref{clm:intermezzo}, we have
$\P(S\in \mathcal{R}^-)\geq q_1(1-q_1)^{2\log^2n}\geq 1/(10\log^2n)$, and, hence, $\E|X_f|\geq |\mathcal{R}^0_f|/(10\log^2n)$.

Now, for each $f'=\{(i',u'),(j',v')\}\in \mathcal{J}^-$ and $S'\in \mathcal{R}_{f'}^0$, there are at most $2\log^2n$ choices for $S\in \mathcal{R}^0_f$ which $E(S)\cap E(S')\neq \emptyset$, $(V(S)\cap V(S'))\setminus (\{u,v\}\cap \{u',v'\})\neq \emptyset$ or $C(S)\cap C(S')\neq \emptyset$.
Thus, there are at most $2\log^2n\cdot |\mathcal{R}^0_{f}|$ tuples $(f',S')\in \mathcal{R}_{f'}^0$ for which the event $\{S'\in\mathcal{R}^0_{f'}\}$ influences $X_f$, and each such event on its own can change $X_f$ by at most $2\log^2 n$. Therefore, by Lemma~\ref{lem:mcdiarmidchangingc} with $t=\E|X_f|\geq |\mathcal{R}^0_{f}|/(10\log^2n)$, we have, as $|\mathcal{R}^0_f|\geq 10\log^8n$,
\[
\P(X_f=0)\leq 2\exp\left(-\frac{2(|\mathcal{R}^0_{f}|/(10\log^2n))^2}{2\cdot 2\log^2 n\cdot |\mathcal{R}^0_{f}| \cdot (2\log n)^2}\right)
=2\exp\left(-\frac{|\mathcal{R}^0_{f}|}{800\log^6 n}\right)=n^{-\omega(1)},
\]
and thus, by a union bound, with high probability we have $\mathcal{R}^-_f\neq\emptyset$ for each $f\in \mathcal{J}'$.
\claimproofend

For each $f=\{(i,u),(j,v)\}\in \mathcal{J}^-$, using Claim~\ref{clm:H3findsremainder}, arbitrarily pick $S_f\in \mathcal{R}^-_f$, and, considering $S_f$ as a $(u,v,L)$-link,
let $M_{f,i}$ be the matching of the odd edges of this link and let $M_{f,j}$ be the matching of the even edges of the link, noting that these are both rainbow matchings and that $C(M_{f,i})=C(M_{f,j})$. For each $i\in [n]$, let
\[
\hat{M}_{i,3,1}=\bigcup_{f=\{(i,u),(j,v)\}\in \mathcal{J}^-}M_{f,i}.
\]
and let $\hat{M}_{i,3}=\hat{M}_{i,3,0}\cup \hat{M}_{i,3,1}$, noting that this is a rainbow matching with colours in $D_{3}\setminus C_{i}$.

For each $f=\{(i,u),(j,v)\}\in \mathcal{J}$, then, we have found a $(u,v,L)$-link $S_f$, which has edges in $E^\abs$, colours in $D_3\setminus (C_i\cup C_j)$ and interior vertices in $Z_i\cap Z_j$, and matchings $M_{f,i}$ and $M_{f,j}$ such that $S_f=M_{f,i}\cup M_{f,j}$.


\subsection{Proof of Lemma~\ref{keylemma:realisation}: properties of the absorption structure}\label{sec:partBfinal}
For each $i\in [n]$, let $\hat{M}_i=\hat{M}_{i,1}\cup \hat{M}_{i,2}\cup \hat{M}_{i,3}$.
As, for each $j\in [3]$, the matchings $\hat{M}_{1,j}$, \ldots, $\hat{M}_{n,j}$ use colours in $D_j$ and edges in $E^\abs_j$, and are edge-disjoint, we have that $\hat{M}_1,\ldots,\hat{M}_n$ are edge-disjoint subgraphs of $G[E^\abs]$.
We now confirm that \ref{prop:real:regularity}--\ref{prop:real:absorption} hold, completing the proof of Lemma~\ref{keylemma:realisation}.

\smallskip

\ref{prop:real:regularity} \ref{prop:B1b}: Let $v\in V(G)$ and $\phi\in \mathcal{F}$. As $\hat{M}_{i,1}$ is a matching from $S_i\setminus R_i$ into $X_i$ covering $S_i\setminus R_i$, we have
\begin{align*}
|\{i\in I_\phi:v\notin V(\hat{M}_i)\cup R_i\}|&\leq |\{i\in I_\phi:v\in X_{i,1}\cup Y_{i,1}\cup Z_{i,1}\}|+|\{i\in [n]:v\in X_{i,0}\setminus V(\hat{M}_{i,1,0}\}|\\
&\hspace{-1cm}+|\{i\in I_\phi:v\in Y_{i,0}\setminus (\cup_{u,j,u':\{(i,u),(j,u')\}\in \mathcal{I}}\{v_{i,u},x_{i,u,j},y_{i,u,j},v_{j,u'},x_{j,u',i},y_{j,u',i}\})\}|\\
&+|\{i\in I_\phi:v\in Z_{i,0}\setminus V(\hat{M}_{i,3,0})\}|\\
&\hspace{-2cm}\leq 4\beta p_\tr p_\fa n,
\end{align*}
where we have used \ref{prop:finalB:1}, Claim~\ref{clm:M1props} \ref{prop:hatM:cbb}, \ref{prop:forBfinalfromB2} and Claim~\ref{clm:propsofMi30} \ref{prop:finalB:fromB3}. Thus, \ref{prop:real:regularity} \ref{prop:B1b} holds.

\smallskip

\ref{prop:real:regularity} \ref{prop:B1c}: Let $c\in C(G)$, and let $k\in [3]$ be such that $c\in D_k$.
Then, by \ref{prop:finalB:2}, Claim~\ref{clm:M1props} \ref{prop:hatM:cbbb} (if $k=1$), \ref{prop:forBfinalfromB2:2} (if $k=2$) and Claim~\ref{clm:propsofMi30} \ref{prop:finalB:fromB3:2} (if $k=3$),
\begin{align*}
|\{i\in I_\phi:c \notin C(\hat{M}_{i}) \cup C_i\}|&\leq |\{i\in I_\phi:c\in D_{k,i}\}|+|\{i\in I_\phi:c \notin (C(\hat{M}_{i,j,0}) \cup C_i\cup D_{k,i})\}|\leq 2\beta p_\tr p_\fa n,
\end{align*}
and thus \ref{prop:real:regularity} \ref{prop:B1c} holds.

\smallskip

\ref{prop:real:regularity} \ref{prop:B1d}: Let $v \in V(G)$. Then, using Claim~\ref{clm:M1props} \ref{prop:hatM:cb} and \ref{prop:hatM:cbb}, \ref{prop:forBfinalfromB2}, and Claim~\ref{clm:propsofMi30} \ref{prop:finalB:fromB3},
\begin{align*}
\Big|\big\{u:uv\in &(\bigcup_{i \in [n]} \hat{M}_i)\big\}\Big|\geq |\{i\in [n]:v\in \hat{M}_{i,1}\}|+|\{i\in [n]:v\in \hat{M}_{i,2}\}|+|\{i\in [n]:v\in \hat{M}_{i,3}\}|\\
&\geq |\{i\in [n]:v\in T_i\cup X_{i,0}\}|-4\gamma n+|\{i\in [n]:v\in S_i\setminus R_i\}|\\
&\hspace{2cm}+|\{i\in [n]:v\in Y_{i,0}\}|-\gamma n+|\{i\in [n]:v\in Z_{i,0}\}|-\gamma n\\
&= n-|\{i\in [n]:v\in X_{i,1}\cup Y_{i,1}\cup Z_{i,1}\}|-|\{i\in [n]:v\in R_i\}|+|\{i\in [n]:v\in T_i\}|-10^5\gamma n\\
&\geq n-(1+\eps)(1-\beta_0)(p_X+p_Y+p_Y)n-(1+2\eps)\alpha p_Tn-10^5\gamma n \\
&\overset{\eqref{eq:newnewnew}}{=}n-(1+\eps)(1-\beta_0)(p_X+p_Y+p_Y)n-(1+2\eps)(p_{\mathrm{pt}}-\beta(p_T+p_S-p_R))\alpha p_Tn -10^5\gamma n\\
&\geq n-(1+\eps)\beta (p_X+p_Y+p_Y)n-(1+2\eps)(p_{\mathrm{pt}}-\beta(p_T+p_S-p_R))\alpha p_Tn -10^5\gamma n\\
&\geq n-(1+2\eps)\beta (p_S+p_X+p_Y+p_Y)n-(1+2\eps)(p_{\mathrm{pt}}-\beta(p_T-p_R))\alpha p_Tn-10^5\gamma n \\
&\geq n-1.5\beta n-(1+2\eps)p_{\mathrm{pt}}-10^5\gamma n.
\end{align*}
Thus, by \ref{prop:finalB:3}, the degree of $v$ in $G[E^\abs] \setminus (\bigcup_{i \in [n]} \hat{M}_i)$ is at most $2\beta n$.

\smallskip

\ref{prop:real:regularity} \ref{prop:B1e}: Let $i\in [n]$. Every vertex in $S_i\setminus R_i$ is in an edge in $\hat{M}_{i,1}\subset \hat{M}_i$, so that
\begin{align*}
|V(G)\setminus (R_i\cup V(\hat{M}_i)|&\leq |X_{i,0}\setminus V(\hat{M}_{i,1,0})|\\
&\;\;+|Y_{i,0}\setminus (\cup_{u,j,u':\{(i,u),(j,u')\}\in \mathcal{I}}\{v_{i,u},x_{i,u,j},y_{i,u,j},v_{j,u'},x_{j,u',i},y_{j,u',i}\})\}|\\
&\;\;+|Z_{i,0}\setminus V(\hat{M}_{i,3,0})|\leq 4\beta n,
\end{align*}
where we have used Claim~\ref{clm:M1props} \ref{prop:hatM:c}, \ref{prop:forBfinalfromB2:3} and Claim~\ref{clm:propsofMi30} \ref{prop:finalB:fromB3:3}.
 Thus, there are at most $4\beta n$ vertices in $V(G)\setminus R_i$ that have degree $0$ in $\hat{M}_i$, so that \ref{prop:real:regularity} \ref{prop:B1e} holds.

\smallskip

\ref{prop:real:vertices}: Let $i\in [n]$. We have $T_i\subset V(\hat{M}_{i,1})$, $T_i\subset S_i\setminus R_i\subset V(\hat{M}_{i,2})$ and, as $V(\hat{M}_{i,2})$ and $V(\hat{M}_{i,1})$ can be seen to be disjoint, we have that \ref{prop:real:vertices} holds.

\smallskip

\ref{prop:real:colours}: Let $i\in [n]$. As $\hat{M}_{i,1}$ is a rainbow matching with colours in $D_1\setminus C_i$ by \ref{prop:B1final:2}, $\hat{M}_{i,2}$  is a rainbow matching with colours in $D_2\setminus C_i$ by \ref{prop:forBfinalfromB2:4},
 and $\hat{M}_{i,3}$ is a rainbow matching with colours in $D_3\setminus C_i$, we have that \ref{prop:real:colours} holds.

\smallskip

\ref{prop:real:absorption}: Suppose there are edge-disjoint matchings  $\tilde{M}_1,\ldots,\tilde{M}_n$ in $G-\hat{M}_1-\ldots-\hat{M}_n$
such that \ref{prop:real:absorption} \ref{prop:Bi}--\ref{prop:Biii} all hold.
For each $i\in [n]$, let $R_i'=R_i\setminus V(\tilde{M}_i)=V(G)\setminus (V(\hat{M}_i)\cup V(\tilde{M}_i))$. Then, for each $i\in [n]$, by \ref{prop:real:absorption} \ref{prop:Bii} and \ref{prop:Biii},  as $\hat{M}_i\cup \tilde{M}_i$ has $n$ edges, and every vertex outside of $T_i$ has degree 0 or 1 in $\hat{M}_i\cup \tilde{M}_i$ and every vertex in $T_i$ has degree 2 in $\hat{M}_i\cup \tilde{M}_i$, we have that $|R_i'|=|T_i|$, while, for each $\tau\in \mathcal{T}$ and $\phi\in \mathcal{F}$, by \ref{prop:real:absorption} \ref{prop:Biii} we have
$\bigcup_{i \in I_\phi} R_i' =_{\mult} \bigcup_{i \in I_\phi} T_i$.

Then, by the property of $\mathcal{I}=\cup_{\tau\in \mathcal{T}}\mathcal{I}_\tau$ from \ref{prop:abs:corrections} in Lemma~\ref{keylemma:absorption}, we have that there exists $\mathcal{C} \subseteq \mathcal{I}$ satisfying the following.
\begin{enumerate}[label = {{\textbf{\alph{enumi})}}}]
\item \labelinthm{prop:inBfromA:correct1} For every $i \in [n]$ and $u \in T_i$, there is exactly one  $(j,v)$ such that $\{(i,u),(j,v)\} \in \mathcal{C}$.
\item \labelinthm{prop:inBfromA:correct2} For every $i \in [n]$ and $u \in R_i'$ there is exactly one  $(v,j)$ such that $\{(i,v),(j,u)\} \in \mathcal{C}$.
\item \labelinthm{prop:inBfromA:correct3} For every $i\in [n]$ and $u\in R_i\setminus R_i'$ there is no $(v,j)$ such that $\{(i,v),(j,u)\}\in \mathcal{C}$.
\item \labelinthm{prop:inBfromA:correct4} For every $i \in [n]$ and $u \in S_i\setminus (R_i \cup T_i)$, $(i,u)$ is $(\leq 1)$-balanced in $\mathcal{C}$.
\end{enumerate}

Now, for each $i\in [n]$, take $\hat{M}_i$ and, for each $\cup_{\{(i,u),(j,v)\}\in \mathcal{C}}$, remove the edges in
\[
\{uv_{i,u},x_{j,v,i}y_{j,v,i}\}\cup M_{\{(i,x_{i,u,j}),(j,v_{i,u})\},i}\cup M_{\{(i,y_{j,v,i}),(j,y_{i,u,j})\},i}\cup M_{\{(i,v_{j,v}),(j,x_{j,v,i})\},i}
\]
and add the edges
\[
\{vv_{j,v},x_{i,u,j}y_{i,u,j}\}\cup M_{\{(i,x_{i,u,j}),(j,v_{i,u})\},j}\cup M_{\{(i,y_{j,v,i}),(j,y_{i,u,j})\},j}\cup M_{\{(i,v_{j,v}),(j,x_{j,v,i})\},j},
\]
calling the final result $\hat{M}_i'$.
Note that, for each $\cup_{\{(i,u),(j,v)\}\in \mathcal{C}}$, this operation decreases the degree of $u$ in $\hat{M}_i$ by 1 and increases the degree of $u$ by $\hat{M}_j$ by 1, while increasing the degree of $u$ in $\hat{M}_i$ by 1 and decreasing the degree of $u$ by $\hat{M}_j$ by 1, while making no other changes in the colours or vertex degrees of $\hat{M}_i$ or $\hat{M}_j$, and only moving edges between $\hat{M}_i$ and $\hat{M}_j$.
By careful construction, we have that these alterations do not interfere with each other, and, therefore, it can be seen that $\tilde{M}_i\cup \hat{M}_i'$, $i\in [n]$, is a decomposition of $G$ into perfect rainbow matchings. This completes the proof of \ref{prop:real:absorption}, and hence \ref{keylemma:realisation}.

%% file: 7coverandbalance.tex
Throughout this section, and as our last task in this paper, we will prove Lemma~\ref{keylemma:completion}.
\ifsecsevenout\label{sec:discussiondeletionmethod}\label{sec:C4}
\else We start by giving an overview of its proof in Section~\ref{sec:Coverview} which, after recalling the key parts of the set-up, divides the proof into 4 subparts, Parts \ref{partC1}--\ref{partC4}, which are then carried out in Sections~\ref{sec:C1}--\ref{sec:C4} respectively after some additional set-up in Section~\ref{sec:Csetup}.


\subsection{Overview of Part~\ref{partC}}\label{sec:Coverview}

Take the set-up detailed in Sections~\ref{sec:variables} and~\ref{sec:choosevxsets}, where, in particular, we have $G\sim G^\col_{[n]}$ and that the edges of $G$ appear in $E^{{\bal}}\subset E(G)$ independently at random with probability $p_{{\bal}}$, while, for each $i\in [n]$, $C_i\subset C$ is a random set of colours where each colour is included independently at random with probability $p_{{\bal}}$ and $R_i$ is a random subset of $V(G)$ such that, for each $v\in R_i$, $\P(v\in R_i)=p_R$. We wish to show that, with high probability, we can do the following.
Suppose we have any edge set $\hat{E}\subset E(G)$, and any sets $\hat{V}_i\subset V(G)$ and $\hat{C}_i\subset C(G)$, $i\in [n]$, which satisfy \ref{key3:need0}--\ref{key3:need4}. Then, $\hat{E}$ can be partitioned into matchings $\tilde{M}_1,\ldots,\tilde{M}_n$ such that \ref{key3:outcome1}--\ref{key3:outcome3} hold, where in particular these properties require that, for each $i\in [n]$, $\tilde{M}_i$ is a rainbow matching with colour set $\hat{C}_i$ which covers all the vertices in $\hat{V}_i$ which are not in $R_i$, and so that vertices unused in $R_i$ are balanced among each family (that is, \ref{key3:outcome3} holds).

The conditions \ref{key3:need0}--\ref{key3:need4} variously make sure that such a partition is feasible based on edge and colour degrees (\ref{key3:need3} and \ref{key3:need4}) or make sure the colours, vertices, and edges involved are sufficiently random-looking to make this task achievable. To find $\tilde{M}_1,\ldots,\tilde{M}_n$, we will further split Part~\ref{partC} into 4 subparts, as follows.

\begin{enumerate}[label = \textbf{\Alph{enumi}}]\addtocounter{enumi}{2}
\item Covering, balancing, and splitting the final edges.
\begin{enumerate}[label = \textbf{\Alph{enumi}.\arabic{enumii}}]
\item Making sure the matching $\tilde{M}_i$, $i\in [n]$, will cover the vertices in $\hat{V}_i\setminus R_i$ and use the colours in $\hat{C}_i\setminus C_i$.\label{partC1}

\smallskip

\emph{We find matchings $\tilde{M}_{i,1}$, $i\in [n]$, which cover $\hat{V}_i\setminus R_i$ and use the colours in $\hat{C}_i\setminus C_i$ and whose inclusion in $\tilde{M}_i$ will thus ensure this property for $\tilde{M}_i$.}

\smallskip

\item Balancing colours between families.\label{partC3}

\smallskip

\emph{From an initial partition of the remaining edges (i.e., those not in $\tilde{M}_{i,1}$, $i\in [n]$), into $\hat{E}'_\phi$, $\phi\in \mathcal{F}$, we adjust this to give a partition $\hat{E}_\phi$, $\phi\in \mathcal{F}$, where each family $\phi\in \mathcal{F}$ has the right number of edges of each colour in $\hat{E}_\phi$ to
complete the matchings $\tilde{M}_i$, $i\in [n]$, while following the $\hat{C}_i$-rainbow conditions.}

\smallskip

\item \label{partC3b} Balancing vertex degrees between families.

\smallskip

\emph{Similarly to Part~\emph{\ref{partC3}}, we adjust the edge partition $\hat{E}_\phi$, $\phi\in \mathcal{F}$, (without changing the number of edges of each colour in each set in this partition) to give a partition $\hat{E}^*_\phi$, $\phi\in \mathcal{F}$, so that, for each vertex $v$, each family $\phi\in \mathcal{F}$ has the right number of edges at $v$ in $\hat{E}^*_\phi$ to
complete the matchings $\tilde{M}_i$, $i\in [n]$, in order that \emph{\ref{key3:outcome3}} holds.}

\smallskip

\item Partitioning the remaining edges allocated to each family $\phi\in \mathcal{F}$ to complete $\tilde{M}_i$, $i\in I_\phi$.\label{partC4}

\smallskip

\emph{For each family $\phi\in \mathcal{F}$, we partition the edges of $\hat{E}^*_\phi$ into $\tilde{M}_{i,2}$, $i\in I_\phi$, so that each $\tilde{M}_{i,1}\cup \tilde{M}_{i,2}$ is a rainbow matching using exactly the colours in $\hat{C}_i$ (and thus \emph{\ref{key3:outcome1}} holds) and whose vertices are in $R_i$, which then ensures that \emph{\ref{key3:outcome2}} holds.}
\end{enumerate}
\end{enumerate}

Parts~\ref{partC1} and~\ref{partC3} will be straightforward to carry out. Essentially, having set aside some random vertices, colours and edges for the task in the set-up in Section~\ref{sec:Csetup}, the matchings $\tilde{M}_{i,1}$ in Part~\ref{partC1} can be found greedily, while the initial partition of edges in Part~\ref{partC3} will be random and thus only require a small adjustment, which can be made by switching some small number of edges between the parts of the partition (relying on \ref{key3:need4}).

For Part~\ref{partC3b}, we are fortunate in that we have already done much of the required work, in Section~\ref{sec:rand}, where we showed the likely existence of many $(u,v,L)$-links for each distinct $u,v\in V(G)$ with $u\simAB v$ (where $L$ is the link defined in Theorem~\ref{thm:Llinks}). In the partition $\hat{E}_{\phi}$, $\phi\in \mathcal{F}$, each vertex $v$ will be in too many edges in some of these sets, and too few in some others, but in total it will be in the right number (due to \ref{key3:need3}). Not too disimilarly to some of our easier work in Section~\ref{sec:absorb}, we will be able to decompose the changes we need to make so that the problem is reduced to, for each pair of vertices $u,v$ with $u\simAB v$ and each distinct $\phi,\phi'\in \mathcal{F}$ with $u,v\in S_\phi\cap S_{\phi'}$, being able to swap edges between $\hat{E}_{\phi}$ and $\hat{E}_{\phi'}$ to reduce the degree of $u$ by 1 in $\hat{E}_{\phi}$ and increase it by 1 in $\hat{E}_{\phi'}$, and reduce the degree of $v$ by 1 in $\hat{E}_{\phi'}$ and increase it by~1 in $\hat{E}_{\phi}$, without making any changes to the number of edges of each colour in $\hat{E}_{\phi}$ and $\hat{E}_{\phi'}$ or to any of the other vertex degrees in these sets.

Suppose we could find a $(u,v,L)$-link $S$ in $G$ such that the odd edges are in  $\hat{E}_{\phi}$ and the even edges of $S$ are in $\hat{E}_{\phi'}$. Then, switching out the odd edges of $S$ from $\hat{E}_{\phi}$ for the even edges of $S$, and vice versa for $\hat{E}_{\phi'}$, we get exactly the change we want (cf.\ Figure~\ref{fig:simpleswitcher}).
Showing the existence of many such links is straight-forward using Theorem~\ref{thm:Llinks}. We will not need to set aside any set of links to do these alterations, and instead can show that there are sufficiently many of them, which will moreover be sufficiently well spread out, that enough can be found edge-disjointly to make all the corrections we require.

In Part~\ref{partC4}, following all of our work so far, we will finally arrive at an edge partitioning problem where we perfectly partition a set of edges into rainbow matchings with specific colours (as for the original problem solved by Theorem~\ref{thm:main}), except here the matchings in the partition will be less restricted in their vertex sets -- each $\tilde{M}_{i,2}$ will have a relatively small number of edges compared to the size of $R_i$. This is the key relaxation that will allow us to partition the remaining edges. For each colour $c$, we will have exactly the right number of edges of colour $c$ remaining to assign one to each $\tilde{M}_{i,2}$ for which we want an edge with colour $c$. The challenge is to do this so that each $\tilde{M}_{i,2}$ is a matching. We first sparsify an accompanying auxiliary graph  (see $L_c$ in Section~\ref{sec:C4}) by forbidding most of the possible assignments randomly. Then, we do a similar sparsification, but keep only the assignment of an edge $e$ to $\tilde{M}_{i,2}$ if we do not forbid this, but do forbid the assignment to $\tilde{M}_{i,2}$ of any edge intersecting with $e$. Then, we show that it is very likely we can use the remaining non-forbidden assignment possibilities to assign the remaining edges of colour $c$ to the required $\tilde{M}_{i,2}$, where we now have that this will give a matching.


\subsection{Set-up for Part~\ref{partC}}\label{sec:Csetup}

For each $i\in [n]$, partition $R_i=R_{i,1} \cup R_{i,2}$ by taking each $v\in R_i$ and independently at random allocating it so that $\P(v\in R_{i,j})=1/4$ for each $j\in [2]$. For each $i\in [n]$, similarly partition $C_i=C_{i,1}\cup C_{i,2}$ so that the location of each $c\in C_i$ is independent and such that $\P(v\in C_{i,1})=p_{\mathrm{cov}}/p_\bal$.
Similarly, partition $E^\partit=E_1\cup E_2\cup E_3\cup E_4$ so that, for each $e\in E^\partit$, $\P(e\in E_1)=p_{\mathrm{cov}}/p_{{\bal}}$, $\P(e\in E_2)=p_{\balcol}/p_{{\bal}}$ and $\P(e\in E_3)=p_{\balvx}/p_{{\bal}}$.
Then, partition $E_1=E^A_1\cup E^B_1\cup E^C_1$, by, for each $e\in E_1$, choosing the set for $e$ independently and uniformly at random.
 Partition $E_2=\bigcup_{\phi\in \mathcal{F}}E_{2,\phi}$ by, for each edge $e\in E_2$, independently and uniformly at random assigning $e$ to some $E_{2,\phi}$ for which $V(e)\subset S_\phi$. Similarly, partition $E_3=\bigcup_{\phi\in \mathcal{F}}E_{3,\phi}$ and $E_4=\bigcup_{\phi\in \mathcal{F}}E_{4,\phi}$.

For \ref{propforLc:2}--\ref{propforLc:1} later, let
\begin{equation}\label{eqn:nzeroetc}
n_0=1.01p_{\mathrm{pt}}p_{\tr}p_{\fa}n, \;\;\;D_0=p_R^2p_{\mathrm{pt}}p_\tr p_\fa n/8p_S^2,\;\;\text{ and }\;\;q_0=p_R^2/8p_S^2.
\end{equation}

\begin{claim}\label{clm:forpartC} With high probability, we have the following properties.
\stepcounter{propcounter}
\begin{enumerate}[label = {{\textbf{\Alph{propcounter}\arabic{enumi}}}}]
\item For each $i\in [n]$, $v\in V(G)$ and $X\in \{A,B\}$, there are at least $\frac{p_\cov^2p_Rn}{8}$ edges in $E_1^X$ between $v$ and $R_{i,1}$ with colour in $C_{i,1}$.\label{prop:C:forvxcovering}
\item For each $i\in [n]$ and $c\in C$, there are at least $\frac{p_\cov p_R^2n}{20}$ edges in $E_1^C$ with vertices in $R_{i,1}$ and colour $c$.\label{prop:C:forcolcovering}

\item For each $u\in V(G)$, $|\{\phi\in \mathcal{F}:u\in S_\phi\}|=(1\pm \eps)p_Sp_{\tr}^{-1}p_\fa^{-1}$.\label{prop:C:uplentyofphi}
\item For each distinct $u,v\in V(G)$, $|\{\phi\in \mathcal{F}:u,v\in S_\phi\}|=(1\pm \eps)p_S^2p_{\tr}^{-1}p_\fa^{-1}$.\label{prop:C:uvplentyofphi}
\item For each distinct $u,v,w\in V(G)$, $|\{\phi\in \mathcal{F}:u,v,w\in S_\phi\}|=(1\pm \eps)p_S^3p_{\tr}^{-1}p_\fa^{-1}$.\label{prop:partC:triplesofverticesinSphi}
\item For each $c\in C$ and $\phi\in \mathcal{F}$, \label{prop:C:colsalreadyquitebalanced}
$|\{e\in E_{2,\phi}\cup E_{3,\phi}\cup E_{4,\phi}:c(e)=c\}|=(1\pm \sqrt{p_\cov})p_{\bal}p_\fa p_\tr n$.
\item For each $c\in C$,
$|\{e\in E_{1}:c(e)=c\}|=(1\pm \eps)p_{\cov}n$.\label{prop:C:colsinE1}
\item For each $v\in V(G)$ and $\phi\in \mathcal{F}$,
$|\{e\in E_{1}:v\in V(e), V(e)\setminus \{v\}\subset S_\phi\}|=(1\pm \eps)p_{\cov}p_Sn$. \label{prop:C:verticesinE1}
\item For each $\phi\in \mathcal{F}$ and $v\in S_\phi$, $|\{e\in E_{2,\phi}\cup E_{3,\phi}\cup E_{4,\phi}:v\in V(e)\}|=(1\pm \sqrt{p_{\cov}})p_\bal p_S^{-1}p_\tr p_\fa n$ and $|\{e\in E_{2,\phi}:v\in V(e)\}|\leq 2 p_\balcol p_S^{-1}p_\tr p_\fa n$.
\label{prop:C:verticesalreadyquitebalanced}
\item For each $\phi\in \mathcal{F}$, $i\in I_\phi$ and $v\in S_i$,\label{prop:C:newnew} \[|\{e\in E_{2,\phi}\cup E_{3,\phi}\cup E_{4,\phi}:v\in V(e),c(e)\in C_i\}|=(1\pm p_{\cov})p_{\mathrm{pt}}^2p_S^{-1}p_\tr p_\fa n.\]
\item For each $\phi\in \mathcal{F}$, $i\in I_\phi$ and $v\in S_i$,
 $|\{e\in E_{1}:v\in V(e)\subset S_i,c(e)\in C_i\}|\leq (1\pm \eps)p_\cov p_{\mathrm{pt}}p_S n$.\label{prop:C:new}
\item For each $c\in C$ and $\phi\in \mathcal{F}$,
$|\{i\in I_\phi:c\in C_i\}|=(1\pm \eps)p_{\bal}p_\tr p_\fa n$ and $|\{i\in I_\phi:c\in C_{i,2}\}|=(1\pm p_{\balcol})p_{\bal}p_\tr p_\fa n$.\label{prop:C:colsinfamilies}
\item For each $c\in C$ and distinct $\phi,\phi'\in \mathcal{F}$,
$|\{e\in E_{2,\phi}:V(e)\subset S_\phi\cap S_{\phi'},c(e)=c\}|\geq p_\balcol p_S^{2}p_\tr p_\fa n/2$.\label{prop:C:forcolcoveringNEW}

\item For each $\phi\in \mathcal{F}$ and $v\in S_\phi$,
$\big||\{i\in I_\phi:v\in T_i\}|-|\{i\in I_\phi:v\in R_i\}|\big|=(1\pm \eps)\alpha p_Tp_S^{-1}p_{\tr}p_\fa n$.\label{prop:forC3:1}

\item For each $\phi\in \mathcal{F}$, $|S_\phi|\leq (2+\eps)p_Sn$. \label{prop:partC:Sphibound}
\item For each distinct $\phi,\phi'\in \mathcal{F}$ and $v\in S_\phi\cap S_{\phi'}$,\label{prop:forC3:maxdeg}
 \[|\{e\in E_{3,\phi}:V(e)\subset S_\phi\cap S_{\phi'},v\in V(e)\}|\leq 2p_{\balvx}p_S^{-1}p_{\tr}p_\fa n.\]

\item For each distinct $\phi,\phi'\in \mathcal{F}$ and distinct $x,y\in S_\phi\cap S_{\phi'}$ with $x\sim_{A/B}y$, letting $\mathcal{L}$ be the set of $(x,y,L)$-links (as defined in Theorem~\ref{thm:Llinks})
whose odd edges are in $E_{3,\phi}$ and whose even edges are in $E_{3,\phi'}$, and whose vertices are in $S_\phi\cap S_{\phi'}$, we have
\begin{enumerate}[label = \textbf{\alph{enumii}}]
\item $|\mathcal{L}|\geq p_\balvx(p_\balvx p_{\tr}p_{\fa})^{62}n^{30}$,\label{prop:C:linksa}
\item for each $e\in E(G)$, there are at most $(p_{\tr}p_{\fa})^{61}n^{28}$ links in $\mathcal{L}$ which use $e$, but not as the $k$th edge for any $k\in \{1,2,61,62\}$,\label{prop:C:linksb}
\item for each $e\in E(G)$, there are at most $(p_{\tr}p_{\fa})^{60}n^{28}$ links in $\mathcal{L}$ which use $e$ as either the 2nd or 61st edge,\label{prop:C:linksc}
\item for each $v\in V(G)\setminus \{x,y\}$, there are at most $(p_{\tr}p_{\fa})^{62}n^{29}$ links in $\mathcal{L}$ which use $v$ not as a neighbour of $x$ or $y$, and\label{prop:C:linksd}
\item for each $v\in V(G)\setminus \{x,y\}$, there are at most $(p_{\tr}p_{\fa})^{61}n^{29}$ links in $\mathcal{L}$ which use $v$ as a neighbour of $x$ or $y$.\label{prop:C:linkse}
\end{enumerate}\label{prop:C:forvxbalancing}

\item \label{propforLc:2} For each $\phi\in \mathcal{F}$, $c\in C$ and $I\subset I_\phi$ with $|I|\leq n_0/2D_0$, there are at least $D_0|I|$ edges $e\in E_{4,\phi}$ with colour $c$ such that $V(e)\subset R_{i,2}$ for some $i\in I$.
\item \label{propforLc:3} For each $\phi\in \mathcal{F}$, $c\in C$ and $E\subset \{e\in E(G):V(e)\subset S_\phi,c(e)=c\}$ with $|E|\leq n_0/2D_0$, there are at least $D_0|E|$ values of $i\in I_\phi$ such that $c\in C_{i,2}$ and there is some $e\in E$ such that $V(e)\subset R_{i,2}$.
\item \label{propforLc:1} For each $\phi\in \mathcal{F}$, $c\in C$, $I\subset I_\phi$ and $E\subset \{e\in E(G):V(e)\subset S_\phi,c(e)=c\}$ with $|I|,|E|\geq n_0/2D_0$, there are at least $q_0|I||E|$ pairs
$i\in I$ and $e\in E$ with $V(e)\subset R_{i,2}$.
\end{enumerate}
\end{claim}
\begin{proof}[Proof of Claim~\ref{clm:forpartC}]
To see that \ref{prop:C:forvxcovering}--\ref{propforLc:1} hold with high probability, we first observe that \ref{prop:C:forvxcovering}--\ref{prop:forC3:maxdeg} hold with high probability, each by a simple application of Lemma~\ref{chernoff} and a union bound. This leaves us to show, in turn that \ref{prop:C:forvxbalancing}--\ref{propforLc:1} hold with high probability.

\smallskip
\noindent\ref{prop:C:forvxbalancing}: By Theorem~\ref{thm:Llinks}, with high probability, we can assume that, setting $\Phi_0=n^{30}$,  \ref{prop:links:totalnumber}--\ref{prop:links:throughedgeandvertex} hold.

\smallskip
\noindent\ref{prop:C:forvxbalancing}\ref{prop:C:linksa}: Let $\phi,\phi'\in \mathcal{F}$ be distinct, let $x,y\in S_\phi\cap S_{\phi'}$ be distinct with $x\sim_{A/B}y$,
and let $\mathcal{L}_0$ be the set of $(x,y,L)$-links in $G$ and $\mathcal{L}$ be the set of $(x,y,L)$-links
whose odd edges are in $E_{3,\phi}$, whose even edges are in $E_{3,\phi'}$, and whose vertices are in $S_\phi\cap S_{\phi'}$.
By \ref{prop:links:totalnumber} and \ref{prop:C:uvplentyofphi}, we have $\E|\mathcal{L}|\geq (1-\eps)((1-2\eps)p_\balvx p_{\tr}p_{\fa}p_S^{-2})^{62}p_S^{122}\Phi_0\geq 2p_\balvx(p_\balvx p_{\tr}p_{\fa})^{62}\Phi_0$.
Now, for each $v\in V(G)\setminus \{x,y\}$, we have by \ref{prop:links:throughvertex} that $|\{H\in \mathcal{L}_0:v\in V(H)\}|\leq 100\Phi_0\cdot n^{-1}$.
For each $e\in E(G-\{x,y\})$, we have by \ref{prop:links:throughedge} that $|\{H\in \mathcal{L}_0:e\in E(H)\}|\leq 100\Phi_0\cdot n^{-2}$. For each $e\in E(G)$ with $\{x,y\}\cap V(e)\neq\emptyset$, we have by \ref{prop:links:throughvertex} that $|\{H\in \mathcal{L}_0:e\in E(H)\}|\leq 4\Phi_0\cdot n^{-1}$.

Therefore, by Lemma~\ref{lem:mcdiarmidchangingc}, we have
\begin{align*}
\P(|\mathcal{L}|<p_\balvx(p_\balvx p_{\tr}p_{\fa})^{62}&\Phi_0)\\
&\leq 2\exp\left(-\frac{2(p_\balvx(p_\balvx p_{\tr}p_{\fa})^{62}\Phi_0)^2}{2n\cdot (100\Phi_0\cdot n^{-1})^2+n^2\cdot (100\Phi_0\cdot n^{-2})^2+2n\cdot (4\Phi_0\cdot n^{-1})^2}\right)\\
&\leq 2\exp\left(-\Omega\left({p_\balvx^2(p_\balvx p_{\tr}p_{\fa})^{124}}n\right)\right)=n^{-\omega(1)}.
\end{align*}
Thus, taking a union bound, we have that with high probability \ref{prop:C:forvxbalancing}\ref{prop:C:linksa} always holds.

\smallskip
\noindent\ref{prop:C:forvxbalancing}\ref{prop:C:linksb}: Fix $e\in E(G-\{x,y\})$ and $3\leq k\leq 60$. Let $\mathcal{L}_{e,k}=\{H\in \mathcal{L}:e\text{ is the }k\text{th edge of }H\}$, so that, by  \ref{prop:links:throughedge},
 $|\mathcal{L}_{e,k}|\leq (1+\eps)\Phi_0\cdot n^{-2}$. Let $\mathcal{L}'_{e,k}$ be the set of links in $\mathcal{L}_{e,k}$ whose odd edges are in $E_{3,\phi}$ and whose even edges are in $E_{3,\phi'}$, and whose vertices are in $S_\phi\cap S_{\phi'}$. Let $\mathcal{E}_{e,k}$ be the event that $V(e)\subset S_\phi\cap S_{\phi'}$ and $e\in E_{3,\phi}$ if $k$ is odd and $e\in E_{3,\phi'}$ if $k$ is even. Then,
\[
\E(|\mathcal{L}_{e,k}||\mathcal{E}_{e,k})\leq ((1+2\eps)p_\balvx p_{\tr}p_{\fa}p_S^{-2})^{61}\cdot (1+\eps)\Phi_0\cdot n^{-2}\leq (p_{\tr}p_{\fa})^{61}\Phi_0\cdot n^{-2}/2.
\]

Now, for each $w\in V(G)\setminus (\{x,y\}\cup V(e))$, we have by \ref{prop:links:throughedgeandvertex} that $|\{H\in \mathcal{L}_{e,k}:v\in V(H)\}|\leq 2\cdot 10^8\Phi_0\cdot n^{-3}$.
For each $e'\in E(G-(\{x,y\}\cup V(e))$, we have by \ref{prop:links:2edgesofdifferentcoloursanddisjoint} that $|\{H\in \mathcal{L}_{e,k}:e'\in E(H)\}|\leq 10^4\Phi_0\cdot n^{-4}$.
For each $e'\in E(G)$ with $(\{x,y\}\cup V(e))\cap V(e')\neq\emptyset$ and $V(e')\not\subset \{x,y\}\cup V(e)$, we have by \ref{prop:links:throughedgeandvertex} that $|\{H\in \mathcal{L}_{e,k}:e'\in E(H)\}|\leq 10^8\Phi_0\cdot n^{-3}$.

Therefore, by Lemma~\ref{lem:mcdiarmidchangingc}, we have
\begin{align*}
\P(|\mathcal{L}_{e,k}'|<(p_{\tr}p_{\fa})^{61}&\Phi_0\cdot n^{-2}|\mathcal{E}_{e,k})\\
&\leq 2\exp\left(-\frac{2((p_{\tr}p_{\fa})^{61}\Phi_0\cdot n^{-2}/2)^2}{2n\cdot (2\cdot 10^8\Phi_0\cdot n^{-3})^2+n^2\cdot (10^4\Phi_0\cdot n^{-4})^2+4n\cdot (10^8\Phi_0\cdot n^{-3})^2}\right)\\
&\leq 2\exp\left(-\Omega\left({(p_{\tr}p_{\fa})^{122}}n\right)\right)=n^{-\omega(1)}.
\end{align*}
Thus, taking a union bound over all $e\in E(G-\{x,y\})$ and $3\leq k\leq 60$, we have that with high probability \ref{prop:C:forvxbalancing}\ref{prop:C:linksb} always holds.

\smallskip
\noindent\ref{prop:C:forvxbalancing}\ref{prop:C:linksc}: Fix $e\in E(G-\{x,y\})$ and suppose $k=2$ (where the case where $k=61$) follows similarly. Let $\mathcal{L}_{e,k}=\{H\in \mathcal{L}:e\text{ is the }k\text{th edge of }H\}$, so that, by  \ref{prop:links:throughedge},
$|\mathcal{L}_{e,k}|\leq (1+\eps)\Phi_0\cdot n^{-2}$. Let $\mathcal{L}'_{e,k}$ be the set of links in $\mathcal{L}_{e,k}$ whose odd edges are in $E_{3,\phi}$ and whose even edges are in $E_{3,\phi'}$, and whose vertices are in $S_\phi\cap S_{\phi'}$. Let $f$ be the edge between $x$ and $V(e)$, and let $\mathcal{E}_{e,k}$ be the event that $V(e)\subset S_\phi\cap S_{\phi'}$ and $f\in E_{3,\phi}$ and $e\in E_{3,\phi'}$. (The difference here to \ref{prop:C:forvxbalancing}\ref{prop:C:linksb} is that when $f\notin E_{3,\phi}$ then $\mathcal{L}'_{e,k}$ is always empty.)
Then,
\[
\E(|\mathcal{L}_{e,k}'||\mathcal{E}_{e,k})\leq (p_\balvx p_{\tr}p_{\fa}p_S^{-2})^{60}\cdot (1+\eps)\Phi_0\cdot n^{-2}\leq (p_{\tr}p_{\fa})^{60}\cdot \Phi_0\cdot n^{-2}/2.
\]
Working very similarly to \ref{prop:C:forvxbalancing}\ref{prop:C:linksb} with only the difference of $p_\balvx p_{\tr}p_{\fa}p_S^{-2}$ in the upper bound on the expectation, we have, by Lemma~\ref{lem:mcdiarmidchangingc}, that
\begin{align*}
\P(|\mathcal{L}_{e,k}'|<& (p_{\tr}p_{\fa})^{60}\cdot \Phi_0\cdot n^{-2}|\mathcal{E}_{e,k})\\
&\leq 2\exp\left(-\frac{2( (p_{\tr}p_{\fa})^{60}\cdot \Phi_0\cdot n^{-2}/2)^2}{2n\cdot (2\cdot 10^8\Phi_0\cdot n^{-3})^2+n^2\cdot (10^4\Phi_0\cdot n^{-4})^2+4n\cdot (10^8\Phi_0\cdot n^{-3})^2}\right)\\
&\leq 2\exp\left(-\Omega\left({(p_{\tr}p_{\fa})^{120}}n\right)\right)=n^{-\omega(1)}.
\end{align*}
Thus, taking a union bound over all $e\in E(G-\{x,y\})$, and considering also the case $k=61$, we have that with high probability \ref{prop:C:forvxbalancing}\ref{prop:C:linksc} always holds.

\smallskip
\noindent\ref{prop:C:forvxbalancing}\ref{prop:C:linksd}: Fix $v\in V(G)\setminus \{x,y\}$ and $3\leq k\leq 61$. Let $\mathcal{L}_{v,k}=\{H\in \mathcal{L}:e\text{ is the }k\text{th vertex of }H\}$, so that, by  \ref{prop:links:throughvertex},
 $|\mathcal{L}_{v,k}|\leq (1+\eps)\Phi_0\cdot n^{-1}$. Let $\mathcal{L}'_{v,k}$ be the set of links in $\mathcal{L}_{v,k}$ whose odd edges are in $E_{3,\phi}$ and whose even edges are in $E_{3,\phi'}$, and whose vertices are in $S_\phi\cap S_{\phi'}$. Let $\mathcal{E}_{v}$ be the event that $v\in S_\phi\cap S_{\phi'}$. Then,
\begin{equation}\label{eqn:expupp}
\E(|\mathcal{L}_{e,k}||\mathcal{E}_{v})\leq (p_\balvx p_{\tr}p_{\fa}p_S^{-2})^{61}\cdot (1+\eps)\Phi_0\cdot n^{-1}\leq (p_{\tr}p_{\fa})^{61}\cdot \Phi_0\cdot n^{-1}/2.
\end{equation}

Now, for each $w\in V(G)\setminus (\{x,y,v\})$, we have by \ref{prop:links:cod:twovertices} that $|\{H\in \mathcal{L}_{v,k}:w\in V(H)\}|\leq 10^4\Phi_0\cdot n^{-2}$.
For each $e\in E(G-\{x,y,v\})$, we have by \ref{prop:links:throughedgeandvertex} that $|\{H\in \mathcal{L}_{v,k}:e\in E(H)\}|\leq 10^8\Phi_0\cdot n^{-3}$.
For each $e\in E(G)$ with $\{x,y,v\}\cap V(e)\neq\emptyset$, we have by \ref{prop:links:cod:twovertices} that $|\{H\in \mathcal{L}_{v,k}:e\in E(H)\}|\leq 10^4\Phi_0\cdot n^{-2}$.

Therefore, by Lemma~\ref{lem:mcdiarmidchangingc}, we have
\begin{align*}
\P(|\mathcal{L}_{v,k}'|<&(p_{\tr}p_{\fa})^{61}\cdot \Phi_0\cdot n^{-1}|\mathcal{E}_{v})\\
&\leq 2\exp\left(-\frac{2((p_{\tr}p_{\fa})^{61}\cdot \Phi_0\cdot n^{-1}/2)^2}{2n\cdot (10^4\Phi_0\cdot n^{-2})^2+n^2\cdot (10^8\Phi_0\cdot n^{-3})^2+3n\cdot (10^4\Phi_0\cdot n^{-2})^2}\right)\\
&\leq 2\exp\left(-\Omega\left({(p_{\tr}p_{\fa})^{122}}n\right)\right)=n^{-\omega(1)}.
\end{align*}
Thus, taking a union bound over all $v\in V(G)\setminus \{x,y\}$ and $3\leq k\leq 61$, we have that with high probability \ref{prop:C:forvxbalancing}\ref{prop:C:linksd} always holds.

\smallskip
\noindent\ref{prop:C:forvxbalancing}\ref{prop:C:linkse}: Here, \ref{prop:C:forvxbalancing}\ref{prop:C:linkse} follows very similarly to \ref{prop:C:forvxbalancing}\ref{prop:C:linksd} in the same way that \ref{prop:C:forvxbalancing}\ref{prop:C:linksc} follows similarly to \ref{prop:C:forvxbalancing}\ref{prop:C:linksb}. The difference to \ref{prop:C:forvxbalancing}\ref{prop:C:linksd} is that if $v$ is to be a neighbour of $x$ in the link, then, there are no such links if $xv\notin E_{3,\phi}$, so we condition on this, and therefore save a factor of $(p_\balvx p_{\tr}p_{\fa}p_S^{-2})$ in the corresponding version of \eqref{eqn:expupp}, which we use to make a saving of $p_{\tr}p_{\fa}$
 in the bound in \ref{prop:C:forvxbalancing}\ref{prop:C:linkse} compared to that in \ref{prop:C:forvxbalancing}\ref{prop:C:linksd}. Note that the difference in the two bounds in \ref{prop:C:forvxbalancing}\ref{prop:C:linksc} and \ref{prop:C:forvxbalancing}\ref{prop:C:linksb} is the same saving.

\smallskip
\noindent\ref{propforLc:2}: Let $\phi\in \mathcal{F}$, $c\in C$ and $I\subset I_\phi$ with $|I|\leq n_0/2D_0$. For each $e\in E_{4,\phi}$ with colour $c$ such that $V(e)\subset S_\phi$, the probability that $V(e)\subset R_{i,2}$ is $(p_R/2p_S)^2$ for each $i\in I$, and this is independent across $i\in I$ and such $e$. Therefore, as there is in expectation at least $p_{\mathrm{pt}}p_\tr p_\fa n/2$ edges  $e\in E_{4,\phi}$ with colour $c$ such that $V(e)\subset S_\phi$, and $D_0=p_R^2p_{\mathrm{pt}}p_\tr p_\fa n/8p_S^2$ we have by Lemma~\ref{chernoff} with probability at least $1-\exp(-\omega(|I|\log n))$ there are at least $D_0|I|$ edges $e\in E_{4,\phi}$ with colour $c$ such that $V(e)\subset R_{i,2}$ for some $i\in I$. Thus, by a union bound, \ref{propforLc:2} holds with high probability.

\smallskip
\noindent\ref{propforLc:3}: Let $\phi\in \mathcal{F}$, $c\in C$, and $E\subset \{e\in E(G):V(e)\subset S_\phi,c(e)=c\}$ with $|E|\leq n_0/2D_0$. Then, with high probability, we have $c\in C_{i,2}$ for at least $0.99p_{\mathrm{pt}}p_\tr p_\fa n$ values of $i\in I_\phi$. 
Then, for each $i\in I_\phi$, the probability that there is some $e\in E$ with $V(e)\subset R_{i,2}$, is $1-(1-(p_R/2p_S)^2)^{|E|}\geq 0.99|E|(p_R/2p_S)^2$. Thus, by Lemma~\ref{chernoff}, with probability at least $1-\exp(-\omega(|E|\log n))$
there are at least $D_0|E|$ values of $i\in I_\phi$ such that $c\in C_{i,2}$ and there is some $e\in E$ such that $V(e)\subset R_{i,2}$. Thus, by a union bound, \ref{propforLc:3} holds with high probability.

\smallskip
\noindent\ref{propforLc:1}: Let $\phi\in \mathcal{F}$, $c\in C$, $I\subset I_\phi$ and $E\subset \{e\in E(G):V(e)\subset S_\phi,c(e)=c\}$ with $|I|,|E|\geq n_0/4D_0$. Now, the events $\{V(e)\subset R_{i,2}\}$, $e\in E$ and $i\in I$, are independent, and each occur with probability $p_R^2/4p_S^2=2q_0$, so the property follows by a simple application of Lemma~\ref{chernoff} and a union bound.
\claimproofend


\subsection{Part \ref{partC1}: Vertex and colour covering}\label{sec:C1}
Assuming now the properties \ref{prop:C:forvxcovering}--\ref{propforLc:1}, we will show we have the required property in Lemma~\ref{keylemma:completion}. For this, suppose we have any edge set $\hat{E}\subset E(G)$, and any sets $\hat{V}_i\subset V(G)$ and $\hat{C}_i\subset C(G)$, $i\in [n]$, which satisfy \ref{key3:need0}--\ref{key3:need4}. By carrying out Parts~\ref{partC1}--\ref{partC4}, we will partition $\hat{E}$ into matchings $\tilde{M}_1,\ldots,\tilde{M}_n$ such that \ref{key3:outcome1}--\ref{key3:outcome3} hold

For Part~\ref{partC1}, we will use edges from $E_1$ to, edge-disjointly, find for each $i\in [n]$ a rainbow matching $\tilde{M}_{i,1}$ using colours in $C_{i,1}$ and vertices in $R_{i,1}$ as well as every vertex in $\hat{V}_i\setminus R_i$ and colour in $\hat{C}_i\setminus C_i$, as follows.

\begin{claim}\label{claim:partC1}
There are edge-disjoint rainbow matchings $\tilde{M}_{i,1}$, $i\in [n]$, in $E_1$ such that,
\stepcounter{propcounter}
\begin{enumerate}[label = {{\textbf{\Alph{propcounter}\arabic{enumi}}}}]
\item For each $i\in [n]$, $\hat{V}_i\setminus R_i\subset V(\tilde{M}_{i,1})\subset R_{i,1}\cup (\hat{V}_i\setminus R_i)$.\label{prop:partc1:1}
\item For each $i\in [n]$,  $\hat{C}_i\setminus C_i\subset C(\tilde{M}_{i,1})\subset C_{i,1}\cup (\hat{C}_i\setminus C_i)$.\label{prop:partc1:1b}
\item For each $\phi\in \mathcal{F}$ and $v\in S_\phi$, $|\{i\in I_\phi:v\in V(\tilde{M}_{i,1})\}|\leq p_{\balcol}p_\tr p_\fa n$.\label{prop:partc1:2}
\end{enumerate}
\end{claim}
\begin{proof}[Proof of Claim~\ref{claim:partC1}]
Let $\tilde{M}_{i,1}$, $i\in [n]$, be a set of edge-disjoint rainbow matchings of edges of $E_1=E_1^A\cup E_1^B\cup E_1^C$ such that, for each $i\in [n]$, $V(\tilde{M}_{i,1})\subset R_{i,1}\cup (\hat{V}_i\setminus R_{i})$,
$C(\tilde{M}_{i,2})\subset C_{i,1}\cup (\hat{C}_i\setminus C_i)$, and each edge $e$ of $\tilde{M}_{i,1}$ either
\begin{itemize}
\item contains exactly one vertex in $\hat{V}_i\setminus R_i$, which is in $A$, and $e\in E_1^A$, or
\item contains exactly one vertex in $\hat{V}_i\setminus R_i$, which is in $B$, and $e\in E_1^B$, or
\item contains no vertices in $\hat{V}_i\setminus R_i$ and $e\in E_1^C$,
\end{itemize}
for each $\phi\in \mathcal{F}$ and $v\in S_\phi$, $|\{i\in I_\phi:v\in V(\tilde{M}_{i,1})\cap R_i\}|\leq p_{\balcol}p_\tr p_\fa n/2$, and, subject to all this, such that $\sum_{i\in [n]}|\tilde{M}_{i,1}|$ is maximised.

Suppose, first, for contradiction, that there is some $i\in [n]$ such that $\hat{V}_i\setminus R_i\not\subset V(\tilde{M}_{i,1})$. Let $v\in \hat{V}_i\setminus (R_i\cup V(\tilde{M}_{i,1}))$.
Assume that $v\in A$, where the case where $v\in B$ follows similarly.

Let $\phi\in \mathcal{F}$ be such that $i\in I_\phi$.
Let $Z$ be the set of vertices $w\in S_\phi$ for which $|\{i\in I_\phi:v\in V(\tilde{M}_{i,1})\cap R_i\}|> p_{\balcol}p_\tr p_\fa n/4$, and note that
\[
|Z|\leq \frac{\sum_{i\in I_\phi}2|\tilde{M}_{i,1}|}{p_{\balcol}p_\tr p_\fa n/4}
\overset{\ref{key3:need1},\ref{key3:need2}}{\leq}
\frac{2p_\tr p_\fa n\cdot 2\cdot 6\beta n}{p_{\balcol}p_\tr p_\fa n/4}\leq \sqrt{\beta} n.
\]

By \ref{prop:C:forvxcovering}, if $E^A_{1,v,i}$ is the set of edges in $E_1^A$ between $v$ and $R_{i,1}$ with colour in $C_{i,1}$, then $|E^A_{1,v,i}|\geq p_\cov^2p_Rn/8$.
By \ref{key3:need1} and \ref{key3:need2}, we have that $|\tilde{M}_{i,1}|\leq 6\beta n$, so at most $6\beta n$ of the edges in $E^A_{1,v,i}$ share their colour with $\tilde{M}_{i,1}$, and at most $6\beta n$ of the edges in $E^A_{1,v,i}$ share their vertex which is not in $\hat{V}_i\setminus R_i$ with any edge in $\tilde{M}_{i,1}$. Furthermore, if $vx\in E^A_{1,v,i}$ is in some $\tilde{M}_{j,1}$ with $j\neq i$, then, as $v\in A$ we must have that $v\in \hat{V}_j\setminus R_j$.
Thus, by \ref{prop:novertexmissingtoomuch}, there are at most $4\beta n$ $j\in [n]$ with $v\in \hat{V}_j\setminus R_j$, and hence at most $4\beta n$ edges in $E^A_{1,v,i}$ which are in some $\tilde{M}_{j,1}$ with $j\neq i$.

Therefore, as $\beta\llpoly p_\cov,p_R$, there is some edge $e\in E_{1,v,i}^A\setminus\bigcup_{j\in [n]}\tilde{M}_{j,1}$ whose colour is not used on $\tilde{M}_{i,1}$ and whose non-$v$ vertex is in $R_{i,1}\setminus (V(\tilde{M}_{i,1})\cup Z)$.
 Adding $e$ to $\tilde{M}_{i,1}$ would increase $\sum_{j\in [n]}|\tilde{M}_{j,1}|$, contradicting the choice of $\tilde{M}_{j,1}$, $j\in [n]$. Thus, there was no such $i\in [n]$ for which $\hat{V}_i\setminus R_i\not\subset V(\tilde{M}_{i,1})$.

Suppose, instead, again for contradiction, that there is some $i\in [n]$ for which $\hat{C}_i\setminus C_i\not \subset C(\tilde{M}_{i,1})$. Let $c\in \hat{C}_i\setminus (C_i\cup C(\tilde{M}_{i,1}))$. Let $\phi\in \mathcal{F}$ be such that $i\in I_\phi$ and, again, let $Z$ be the set of vertices $w\in S_\phi$ for which $|\{i\in I_\phi:v\in V(\tilde{M}_{i,1})\cap R_i\}|> p_{\balcol}p_\tr p_\fa n/4$, so that, as before, we have $|Z|\leq \sqrt{\beta} n$.

Again by \ref{key3:need1} and \ref{key3:need2}, we have that $|\tilde{M}_{i,1}|\leq 6\beta n$.
By \ref{prop:nocolourmissingtoomuch}, there are at most $2\beta n$ edges in $E^C_1\cap(\cap_{j\in [n]}\tilde{M}_{j,1})$ with colour $c$. However, by \ref{prop:C:forcolcovering}, there are at least $\frac{p_\cov p_R^2n}{20}$ edges in $E^C_1$ with colour $c$ and vertices in $R_{i,1}$.

Therefore, as $\beta\llpoly p_\cov,p_R$, there is some edge $e\in E_1^C\setminus\bigcup_{j\in [n]}\tilde{M}_{j,1}$ with colour $c$ and vertices in $R_{i,1}$ which has no vertices in $V(\tilde{M}_{i,1})$.
Adding $e$ to $\tilde{M}_{i,1}$ would increase $\sum_{j\in [n]}|\tilde{M}_{j,1}|$, contradicting the choice of $\tilde{M}_{j,1}$, $j\in [n]$. Thus, there was no such $i\in [n]$ for which $\hat{C}_i\setminus C_i\not \subset C(\tilde{M}_{i,1})$.

Thus, we have that \ref{prop:partc1:1} and \ref{prop:partc1:1b} hold. For each $\phi\in \mathcal{F}$ and $v\in S_\phi$,
\[
|\{i\in I_\phi:v\in V(\tilde{M}_{i,1})\}|\leq |\{i\in I_\phi:v\in V(\tilde{M}_{i,1}\cap R_i)\}|+|\{i\in I_\phi:v\in \hat{V}_i\setminus R_i\}|
\overset{\ref{key3:need1}}{\leq} p_{\balcol}p_\tr p_\fa n,
\]
and therefore \ref{prop:partc1:2} holds as well.
\claimproofend


\subsection{Part \ref{partC3}: Balancing colours between families}\label{sec:C2}

Taking the matchings $\tilde{M}_{i,1}$, $i\in [n]$, from Part~\ref{partC1}, we now partition the rest of the edges in $\hat{E}$ between the families, so that each family receives the right number of edges of each colour, as follows.

\begin{claim}\label{claim:partC3} $\hat{E}\setminus E(\bigcup_{i\in [n]}\tilde{M}_{i,1})$ can be partitioned into $\hat{E}_\phi$, $\phi\in \mathcal{F}$,
so that the following hold.
\stepcounter{propcounter}
\begin{enumerate}[label = {{\textbf{\Alph{propcounter}\arabic{enumi}}}}]
\item For each $\phi\in \mathcal{F}$, $E_{3,\phi}\cup E_{4,\phi}\subset \hat{E}_\phi$.\label{prop:fromcolbalance:edge}
\item For each $\phi\in \mathcal{F}$ and $c\in C$, $|\{e\in \hat{E}_\phi:c(e)=c\}|=|\{i\in I_\phi:c\in \hat{C}_i\setminus C(\tilde{M}_{i,1})\}|$.\label{prop:fromcolbalance:colour}
\item For each $\phi\in \mathcal{F}$, every edge in $\hat{E}_\phi$ is contained within $S_\phi$.\label{prop:fromcolbalance:stillinS}
\item For each $\phi\in \mathcal{F}$ and $v\in S_\phi$, \label{prop:fromcolbalance:vxdegree}
$|\{e\in \hat{E}_\phi:v\in V(e)\}|=(1\pm 3p_\balcol)p_{\bal}p_S^{-1}p_{\tr}p_{\fa}n$.
\item For each $\phi\in\mathcal{F}$, $i\in I_\phi$ and $v\in S_i$, \label{prop:frompartC2:roundvertex}
$|\{e\in \hat{E}_\phi:v\in V(e), c(e)\in C_i\}|\leq p_{\mathrm{pt}}^{3/2}p_\tr p_\fa n/2$.
\end{enumerate}
\end{claim}
\begin{proof}[Proof of Claim~\ref{claim:partC3}]
Partition $\hat{E}\setminus ((\bigcup_{\phi\in \mathcal{F}}E_{2,\phi}\cup E_{3,\phi}\cup E_{4,\phi} )\cup E(\bigcup_{i\in [n]}\tilde{M}_{i,1}))$
into sets $E'_{\phi}$, $\phi\in \mathcal{F}$, by, for each $e\in \hat{E}\setminus ((\bigcup_{\phi\in \mathcal{F}}E_{2,\phi}\cup E_{3,\phi}\cup E_{4,\phi} )\cup E(\bigcup_{i\in [n]}\tilde{M}_{i,1}))$,
placing $e$ into a set $E'_{\phi}$, $\phi\in \mathcal{F}$, independently and uniformly at random subject to $V(e)\subset S_\phi$.

By \ref{prop:C:uvplentyofphi} and Lemma~\ref{chernoff}, and, respectively, \ref{prop:C:colsalreadyquitebalanced}, \ref{prop:C:colsinE1}, \ref{prop:nocolourmissingtoomuch}, \ref{key3:need4} and \ref{prop:C:colsinfamilies},
 \ref{prop:C:verticesinE1} and \ref{prop:novertexintoomanyedges},  and \ref{prop:C:new}, with high probability, we have the following properties.
\begin{enumerate}[label = {{\textbf{\Alph{propcounter}\arabic{enumi}}}}]\addtocounter{enumi}{5}
\item For each $c\in C$ and $\phi\in \mathcal{F}$, $|\{e\in E'_\phi:c(e)=c\}|\leq 2p_\cov p_S^{-2}p_\tr p_\fa n$. \label{prop:roughbalancecol}
\item For each $\phi\in \mathcal{F}$ and $v\in S_\phi$, $|\{e\in E'_\phi:v\in V(e)\}|\leq 2p_\cov p_S^{-1}p_\tr p_\fa n$. \label{prop:roughbalancevx}
\item \label{prop:roughbalancevxnew} For each $\phi\in \mathcal{F}$, $i\in I_\phi$ and $v\in S_i$, $|\{e\in E'_\phi:v\in V(e),c(e)\in C_i\}|\leq 2p_\cov p_\mathrm{pt}p_S^{-1}p_\tr p_\fa n$.
\end{enumerate}
Thus, we can assume that \ref{prop:roughbalancecol}--\ref{prop:roughbalancevxnew} hold.

\ifsecsevenout
\else
For each $\phi\in \mathcal{F}$, let $E_{\phi}^+=E'_\phi\cup E_{2,\phi}\cup E_{3,\phi}\cup E_{4,\phi}$.
For each $\phi\in \mathcal{F}$ and $c\in C$, let
\begin{equation}\label{eqn:lambdaphic}
\lambda_{\phi,c}=|\{e\in E_{\phi}^+:c(e)=c\}|-|\{i\in I_\phi:c\in \hat{C}_i\setminus C(\tilde{M}_{i,1})\}|,
\end{equation}
and note that, by \ref{prop:C:colsalreadyquitebalanced}, \ref{prop:C:colsinfamilies}, \ref{prop:nocolourmissingtoomuch}, and \ref{prop:roughbalancecol},
\begin{align}
|\lambda_{\phi,c}|&\leq \big||\{e\in E_{2,\phi}\cup E_{3,\phi}\cup E_{4,\phi}:c(e)=c\}|-|\{i\in I_\phi:c\in C_i\}|\big|\nonumber\\
&\hspace{3cm}+|\{i\in I_\phi:c\in \hat{C}_i\setminus C_i\}|+|\{i\in I_\phi:c\in C(\tilde{M}_{i,1})\}|\nonumber
\\
&
\leq 2\sqrt{p_\cov} p_\tr p_\fa n+2\beta p_\tr p_\fa n+2p_\cov p_S^{-2}p_\tr p_\fa n\leq 3\sqrt{p_\cov} p_\tr p_\fa n.\label{eqn:lambdabound}
\end{align}
Furthermore, we have, for each $c\in C$, that
\begin{align}
\sum_{\phi\in \mathcal{F}}&\lambda_{\phi,c}=\sum_{\phi\in \mathcal{F}}(|\{e\in E_{\phi}^+:c(e)=c\}|-|\{i\in I_\phi:c\in \hat{C}_i\setminus C(\tilde{M}_{i,1})\}|)
\nonumber\\
&=\Big|\Big\{e\in \hat{E}\setminus \Big(\bigcup_{i\in [n]}\tilde{M}_{i,1}\Big):c(e)=c\Big\}\Big|-|\{i\in [n]:c\in \hat{C}_i\setminus C(\tilde{M}_{i,1})\}|
\nonumber\\
&=|\{e\in \hat{E}:c(e)=c\}|-|\{i\in [n]:c\in \hat{C}_i\}|
\overset{\ref{key3:need4}}{=}0.\label{eqn:sumto0}
\end{align}

For each $c\in C$, let $\mathcal{F}^+_c=\{\phi\in \mathcal{F}:\lambda_{\phi,c}>0\}$ and $\mathcal{F}^-_c=\{\phi\in \mathcal{F}:\lambda_{\phi,c}<0\}$.
Using \eqref{eqn:sumto0}, take integers $\lambda_{\phi,\phi',c}\geq 0$, $\phi\in \mathcal{F}^+_c$ and $\phi'\in \mathcal{F}^-_c$, such that for each $\phi\in \mathcal{F}^+_c$ and $\phi'\in \mathcal{F}^-_c$ we have
\[
\sum_{\phi''\in \mathcal{F}^-_c}\lambda_{\phi,\phi'',c}=\lambda_{\phi,c}\;\text{ and }\;\;
\sum_{\phi''\in \mathcal{F}^+_c}\lambda_{\phi'',\phi',c}=-\lambda_{\phi',c}.
\]
For each distinct $\phi,\phi'\in \mathcal{F}$ and $c\in C$ for which $\lambda_{\phi,\phi',c}$ is not defined, that is, when at least one of $\lambda_{\phi, c}$ and $\lambda_{\phi', c}$ is equal to $0$, let $\lambda_{\phi,\phi',c}=0$, and note that, for each $\phi\in \mathcal{F}$ and $c\in C$,
\begin{equation}\label{eqn:lambdas}
\sum_{\phi'\in \mathcal{F}:\phi'\neq \phi}\lambda_{\phi,\phi',c}-\sum_{\phi'\in \mathcal{F}:\phi'\neq \phi}\lambda_{\phi',\phi,c}=\lambda_{\phi,c}.
\end{equation}

Let $\mathcal{E}=\{(\phi,\phi',c):\phi,\phi'\in \mathcal{F},\phi\neq\phi',c\in C\}$.
Take edge disjoint matchings $M_{\phi,\phi',c}$, $(\phi,\phi',c)\in \mathcal{E}$, such that, for each $(\phi,\phi',c)\in \mathcal{E}$,
\begin{enumerate}[label = \roman{enumi})]
\item $|M_{\phi,\phi',c}|\leq \lambda_{\phi,\phi',c}$,\label{prop:C3frommax:1}
\item $V(M_{\phi,\phi',c})\subset S_\phi\cap S_{\phi'}$,\label{prop:C3frommax:2}
\item every edge in $M_{\phi,\phi',c}$ is a colour-$c$ edge in $E_{2,\phi}$,\label{prop:C3frommax:3}
\item and, for each $\phi\in \mathcal{F}$ and $v\in S_\phi$, $v$ is in at most $p_{\mathrm{pt}}^{3/2}p_\tr p_\fa n/4$
edges in the matchings $M_{\phi',\phi,c}$, $\phi'\in \mathcal{F}\setminus\{\phi\}$ and $c\in C$,\label{prop:C3frommax:new}
\end{enumerate}
and, subject to all this, $\sum_{(\phi,\phi',c)\in \mathcal{E}}|M_{\phi,\phi',c}|$ is maximised.

Suppose, for contradiction, that there is some $(\phi,\phi',c)\in \mathcal{E}$ such that $|M_{\phi,\phi',c}|<\lambda_{\phi,\phi',c}$. Let $W$ be the set of vertices in at least $p_{\mathrm{pt}}^{3/2}p_\tr p_\fa n/8$
edges in the matchings $M_{\phi',\phi,c}$, $\phi'\in \mathcal{F}\setminus\{\phi\}$ and $c\in C$,
and note that
\begin{equation}\label{eqn:Wbound}
|W|\leq \frac{2\sum_{c\in C}|\lambda_{\phi,c}|}{p_{\mathrm{pt}}^{3/2}p_\tr p_\fa n/8}\overset{\eqref{eqn:lambdabound}}{\leq} \frac{2n\cdot 3\sqrt{p_\cov} p_\tr p_\fa n}{p_{\mathrm{pt}}^{3/2}p_\tr p_\fa n/8}\leq p_\balcol^2 p_S^2p_\fa p_\tr n.
\end{equation}
As $|M_{\phi,\phi',c}|<\lambda_{\phi,\phi',c}$, every edge in $E_{2,\phi}\setminus \bigcup_{\phi''\neq \phi}M_{\phi,\phi'',c}$ of colour $c$ with vertices in $S_\phi\cap S_\phi'$ has a vertex in $W$.
However, by \ref{prop:C:forcolcoveringNEW}, there are at least $p_\balcol p_S^2p_\fa p_\tr n/2$ such edges in $E_{2,\phi}$.
Thus, as
\[
\sum_{\phi''\in \mathcal{F}:\phi''\neq \phi}|M_{\phi,\phi'',c}|\leq \sum_{\phi''\in \mathcal{F}:\phi''\neq \phi}\lambda_{\phi,\phi'',c}\overset{\eqref{eqn:lambdas}}{\leq} |\lambda_{\phi,c}|\overset{\eqref{eqn:lambdabound}}{\leq} 3\sqrt{p_\cov}p_\tr p_\fa n,
\]
and \eqref{eqn:Wbound} holds, this is a contradiction as $p_\cov \llpoly p_\balcol,p_S$.
Therefore, we have $|M_{\phi,\phi',c}|=\lambda_{\phi,\phi',c}$ for each $(\phi,\phi',c)\in \mathcal{E}$.

Now, for each $\phi\in \mathcal{F}$,
let
\[
\hat{E}_\phi=\left(\left(\hat{E}'_{\phi}\cup E_{2,\phi}\cup E_{3,\phi}\cup E_{4,\phi}\right)\setminus \left(\bigcup_{\phi'\in \mathcal{F}\setminus\{\phi\}}\bigcup_{c\in C}M_{\phi,\phi',c}\right)\right)
\bigcup  \left(\bigcup_{\phi'\in \mathcal{F}\setminus \{\phi\}}\bigcup_{c\in C}M_{\phi',\phi,c}\right).
\]
We now show that \ref{prop:fromcolbalance:edge}--\ref{prop:frompartC2:roundvertex} hold.
That \ref{prop:fromcolbalance:edge} holds follows from \ref{prop:C3frommax:3}.
That \ref{prop:fromcolbalance:colour} holds follows from \ref{prop:C3frommax:3} and the fact that we have equality in \ref{prop:C3frommax:1}, so that, for each $\phi\in \mathcal{F}$ and $c\in C(G)$,
\begin{align*}
|\{e\in \hat{E}_\phi:c(e)=c\}|&=|\{e\in \hat{E}'_{\phi}\cup E_{2,\phi}\cup E_{3,\phi}\cup E_{4,\phi}:c(e)=c\}|-\sum_{\phi'\in \mathcal{F}\setminus \{\phi\}}\lambda_{\phi,\phi',c}+\sum_{\phi'\in \mathcal{F}\setminus \{\phi\}}\lambda_{\phi',\phi,c}
\\
&\overset{\eqref{eqn:lambdas}}{=}|\{e\in \hat{E}'_{\phi}\cup E_{2,\phi}\cup E_{3,\phi}\cup E_{4,\phi}:c(e)=c\}|-\lambda_{\phi,c}
\\
&\overset{\eqref{eqn:lambdaphic}}{=}|\{i\in I_\phi:c\in \hat{C}_i\setminus C(\tilde{M}_{i,1})\}|.
\end{align*}
That \ref{prop:fromcolbalance:stillinS} holds follows from \ref{prop:C3frommax:2}, as each edge in $\hat{E}_\phi$ which is not in $\hat{E}'_\phi\cup E_{2,\phi}\cup E_{3,\phi}\cup E_{4,\phi}$ is in some matching $M_{\phi,\phi',c}$.
Note that \ref{prop:fromcolbalance:vxdegree} follows from (both parts of) \ref{prop:C:verticesalreadyquitebalanced} and \ref{prop:roughbalancevx}.
Finally, for each $\phi\in \mathcal{F}$, $i\in I_\phi$ and $v\in S_i$, by \ref{prop:C:newnew}, \ref{prop:roughbalancevxnew} and \ref{prop:C3frommax:new},
\begin{align*}
|\{e\in \hat{E}_\phi:v\in V(e), c(e)\in C_i\}|&\leq 2p_{\mathrm{pt}}^2p_S^{-1}p_\tr p_\fa n+2p_\cov p_\mathrm{pt}p_S^{-1}p_\tr p_\fa n+p_{\mathrm{pt}}^{3/2}p_\tr p_\fa n/4\\
&\leq p_{\mathrm{pt}}^{3/2}p_\tr p_\fa n/2
\end{align*}
and thus \ref{prop:frompartC2:roundvertex} holds.
Thus,  \ref{prop:fromcolbalance:edge}--\ref{prop:frompartC2:roundvertex} hold, as required.
\fi
\end{proof}


\subsection{Part \ref{partC3b}: Balancing vertex degrees between families}\label{sec:C3}
\ifsecsevenout
\else
We now take the partition $\hat{E}_\phi$, $\phi\in \mathcal{F}$, from Claim~\ref{claim:partC3} and alter it slightly to achieve the right number of edges at each vertex in each part, as per the following claim. Having set up the changes we wish to make at each vertex in each part (see \eqref{eqn:lambdaphiv}), we confirm these changes are small (see \eqref{eqn:lambdaphivbound}) and are balanced within each family (see~\eqref{balancewithinfamily}) and at each vertex (see~\eqref{balanceatvertex}). The proof then takes two stages. In stage \textbf{I}, we decompose the changes we wish to make into pairs of changes we will be able to make together (in a similar approach to that in Section~\ref{sec:partA1}), before making these changes in stage \textbf{II}.
\fi
\begin{claim}\label{claim:partC33} $\hat{E}\setminus E(\bigcup_{i\in [n]}\tilde{M}_{i,1})$ can be partitioned into $\hat{E}_\phi^*$, $\phi\in \mathcal{F}$,
so that the following hold.
\stepcounter{propcounter}
\begin{enumerate}[label = {{\textbf{\Alph{propcounter}\arabic{enumi}}}}]
\item For each $\phi\in \mathcal{F}$, $E_{4,\phi}\subset \hat{E}^*_\phi$.\label{prop:frompartC3:edge}
\item For each $\phi\in \mathcal{F}$ and $c\in C$, $|\{e\in \hat{E}^*_\phi:c(e)=c\}|=|\{i\in I_\phi:c\in \hat{C}_i\setminus C(\tilde{M}_{i,1})\}|$.\label{prop:frompartC3:colour}
\item For each $\phi\in \mathcal{F}$, every edge in $\hat{E}^*_\phi$ is contained within $S_\phi$.\label{prop:fromcolbalance:stillstillinS}
\item For each $\phi\in \mathcal{F}$ and $v\in S_\phi$,\label{prop:frompartC3:vertex}
\[
|\{e\in \hat{E}_\phi^*:v\in V(e)\}|=|\{i\in I_\phi:v\in \hat{V}_i\setminus V(\tilde{M}_{i,1})\}|-|\{i\in I_\phi:v\in T_i\}|.
\]
\item For each $\phi\in\mathcal{F}$, $i\in I_\phi$ and $v\in S_i$, \label{prop:frompartC3:roundvertex}
$|\{e\in \hat{E}_\phi^*:v\in V(e), c(e)\in C_i\}|\leq p_{\mathrm{pt}}^{3/2}p_\tr p_\fa n$.
\end{enumerate}
\end{claim}

\ifsecsevenout
\else
\begin{proof}[Proof of Claim~\ref{claim:partC33}]
For each $\phi\in \mathcal{F}$ and $v\in S_\phi$, let
\begin{align}\label{eqn:lambdaphiv}
\lambda_{\phi,v}&=|\{e\in \hat{E}_\phi:v\in V(e)\}|-|\{i\in I_\phi:v\in \hat{V}_i\setminus V(\tilde{M}_{i,1})\}|+|\{i\in I_\phi:v\in T_i\}|
\end{align}
which represents the change in degree of $v$ from $\hat{E}_\phi$ to $\hat{E}^*_\phi$ we wish to make in order for~\ref{prop:frompartC3:vertex} to hold.
Note that, by \ref{prop:partc1:2}, \ref{prop:fromcolbalance:vxdegree} and \ref{prop:forC3:1}, for each $\phi\in \mathcal{F}$ and $v\in S_\phi$,
\begin{align}\label{eqn:lambdaphivbound}
|\lambda_{\phi,v}|&\leq 2|\{i\in I_\phi:v\in V(\tilde{M}_{i,1})\}|+\big||\{e\in \hat{E}_\phi:v\in V(e)\}|-|\{i\in I_\phi:v\in R_i\}|+|\{i\in I_\phi:v\in T_i\}|\big|\nonumber\\
&\leq 2p_\balcol p_{\tr}p_{\fa}n+|(1\pm 3p_\balcol)p_{\bal}-(1\pm \eps)\alpha p_T|\cdot p_S^{-1}p_{\tr}p_\fa n
\overset{\eqref{eqn:ppt}}{\leq} 10p_{\balcol} p_{\tr}p_\fa n,
\end{align}
so that each of these adjustments is relatively small (compared to the degree of $v$ in $E_{3,\phi}$).

Now, note that, for each $\phi\in \mathcal{F}$ the adjustments $\lambda_{\phi,v}$ to be made at each vertex $v\in S_\phi$ sum to 0, as
\begin{align}
\sum_{v\in S_\phi}\lambda_{\phi,v}&=2|\hat{E}_{\phi}|
-\sum_{i\in I_\phi}(|\hat{V}_i\setminus V(\tilde{M}_{i,1})|-|T_i|)
\nonumber\\
&{=}2\sum_{c\in C}|\{e\in \hat{E}_\phi:c(e)=c\}|
-\sum_{i\in I_\phi}(|\hat{V}_i|-|V(\tilde{M}_{i,1})|-|T_i|)
\nonumber\\
&\overset{\ref{prop:fromcolbalance:colour}}{=}2\sum_{c\in C}|\{i\in I_\phi:c\in \hat{C}_i\setminus C(\tilde{M}_{i,1})\}|
-\sum_{i\in I_\phi}(|\hat{V}_i|-|V(\tilde{M}_{i,1})|-|T_i|)
\nonumber\\
&=2\sum_{i\in I_\phi}|\hat{C}_i\setminus C(\tilde{M}_{i,1})|
-\sum_{i\in I_\phi}(|\hat{V}_i|-|V(\tilde{M}_{i,1})|-|T_i|)
\nonumber\\
&=2\sum_{i\in I_\phi}|\hat{C}_i|
-\sum_{i\in I_\phi}(|\hat{V}_i|-|T_i|)
\overset{\ref{key3:extra}}{=}0.\label{balancewithinfamily}
\end{align}
Furthermore, for each $v\in V(G)$, the adjustments $\lambda_{\phi,v}$ to be made at $v$ for each $\phi\in \mathcal{F}$ sum to 0, as
\begin{align}
\sum_{\phi\in \mathcal{F}}\lambda_{\phi,v}&=\Big|\Big\{e\in \bigcup_{\phi\in \mathcal{F}}\hat{E}_\phi:v\in V(e)\Big\}\Big|
-\sum_{\phi\in \mathcal{F}}(|\{i\in I_\phi:v\in \hat{V}_i\setminus V(\tilde{M}_{i,1})\}|-|\{i\in I_\phi:v\in T_i\}|)\nonumber\\
&=\Big|\Big\{e\in \bigcup_{\phi\in \mathcal{F}}\hat{E}_\phi:v\in V(e)\Big\}\Big|
-|\{i\in [n]:v\in \hat{V}_i\setminus V(\tilde{M}_{i,1})\}|+|\{i\in [n]:v\in T_i\}|\nonumber\\
&=\Big|\Big\{e\in \Big(\bigcup_{\phi\in \mathcal{F}}\hat{E}_\phi\Big)\cup \Big(\bigcup_{i\in [n]}\tilde{M}_{i,1}\Big):v\in V(e)\Big\}\Big|
-|\{i\in [n]:v\in \hat{V}_i\}|+|\{i\in [n]:v\in T_i\}|\nonumber\\
&=\Big|\Big\{e\in \hat{E}:v\in V(e)\Big\}\Big|
-|\{i\in [n]:v\in \hat{V}_i\}|+|\{i\in [n]:v\in T_i\}|\overset{\ref{key3:need3}}{=}0.\label{balanceatvertex}
\end{align}
We will now use \eqref{balancewithinfamily} and \eqref{balanceatvertex} to show that we can decompose the changes we need to make into a collection of directed 2-cycles, similarly to the approach in Section~\ref{sec:partA1}.


\medskip

\noindent\textbf{I. Decomposing the necessary changes into directed 2-cycles.}
For each $\phi\in \mathcal{F}$, let $V_\phi^+=\{v\in S_\phi:\lambda_{\phi,v}>0\}$ and $V_\phi^-=\{v\in S_\phi:\lambda_{\phi,v}<0\}$.
Form an auxiliary directed multigraph $\Psi$ with vertex set $V(G)$ whose edges have symbols (chosen from $\mathcal{F}$) fixed to them (where between any two vertices we allow multiple edges with the same attached symbol), by adding edges under the following rule.
\begin{itemize}
\item For each $\phi\in \mathcal{F}$, using \eqref{balancewithinfamily}, take a matching $K_\phi$ between $\{(v,j):v\in V_\phi^+,j\in [\lambda_{\phi,v}]\}$ and $\{(v,j):v\in V_\phi^-,j\in [-\lambda_{\phi,v}]\}$, and, for each $(v,j)(v',j')\in K_\phi$, add an edge $e$ to $\Psi$ directed from $v$ to $v'$ which has $\phi$ attached as a symbol to it. We set $\mathrm{symb}(e)=\phi$.
\end{itemize}
Note that, for each $\phi\in \mathcal{F}$ and $v\in S_\phi$, we have
\begin{equation}\label{eqn:balanceroundvertex1}
|\{\vec{uv}\in \Psi:u\in S_\phi,\mathrm{symb}(\vec{uv})=\phi\}|-|\{\vec{vu}\in \Psi:u\in S_\phi,\mathrm{symb}(\vec{vu})=\phi\}|=\lambda_{\phi,u}.
\end{equation}
Furthermore, for each $v\in V(G)$,
\begin{align*}
d_\Psi^+(v)-d_\Psi^-(v)=\sum_{\phi\in \mathcal{F}:v\in V_\phi^+}\lambda_{\phi,v}-\sum_{\phi\in \mathcal{F}:v\in V_\phi^+}(-\lambda_{\phi,v})=\sum_{\phi\in \mathcal{F}}\lambda_{\phi,v}\overset{\eqref{balanceatvertex}}{=}0.
\end{align*}
Therefore, $E(\Psi)$ can be decomposed into a set of directed cycles, $\mathcal{C}$ say.

Subject to the stated restrictions on the choice of $K_{\phi}$, $\phi\in \mathcal{F}$, and $\mathcal{C}$, minimise the number of pairs of edges from $E(\Psi)$ which are in the same cycle in $\mathcal{C}$ and have the same symbol. Suppose, for contradiction, that there is some $S\in \mathcal{C}$ which contains two edges, $\vec{v_1v_2}$ and $\vec{v_3v_4}$ say, which have the same symbol, $\phi$ say. Note that, as the edges with the symbol $\phi$ are all directed from $V_\phi^+$ to $V_\phi^-$, these vertices must be distinct. Then, let $j_1$, $j_2$, $j_3$ and $j_4$ be such that
$\vec{v_1v_2}$ and $\vec{v_3v_4}$ were added to $\Psi$ because $(v_1,j_1)(v_2,j_2)$ and $(v_3,j_4)(v_4,j_4)$ belong to $K_\phi$. Note that removing $(v_1,j_1)(v_2,j_2)$ and $(v_3,j_4)(v_4,j_4)$ from $K_\phi$ and adding  $(v_1,j_1)(v_4,j_4)$ and $(v_3,j_4)(v_2,j_2)$
would have formed $\Psi$ so that it had a cycle decomposition $\mathcal{C}':=(\mathcal{C}\setminus S)\cup \{S_1,S_2\}$, where $S_1$ and $S_2$ are the two disjoint directed cycles formed from $S$ by removing $\vec{v_1v_2}$ and $\vec{v_3v_4}$ and adding $\vec{v_1v_4}$ and $\vec{v_2v_3}$, both with symbol $\phi$ attached. As $\mathcal{C}'$ has fewer pairs of edges in the same cycle with the same symbol, this is a contradiction. Thus, there is no cycle in $\mathcal{C}$ which has two edges with the same symbol.

Let $r=|\mathcal{C}|$, and enumerate $\mathcal{C}$ as $L_1,\ldots,L_r$. For each $j\in [r]$, let $\ell_j$ be the length of $L_j$ and, if $\ell_j\geq 4$, let $F_j$ be a set of undirected pairs from $V(L_j)$ so that the underlying graph of $L_j+F_j$ is a triangulation of $L_j$ with maximum degree at most 4 (cf.\ Figure~\ref{fig:reducecoloursurplus}), and otherwise let $F_j=\emptyset$.

Form $\Psi'$ by starting with $\Psi$ and, greedily in some arbitrary order, for each $j\in [r]$ and each $\{u,v\}\in F_j$, choosing $\phi\in \mathcal{F}$ with $u,v\in S_\phi$ and adding $\vec{uv}$ and $\vec{vu}$, both with the symbol $\phi$ attached so that the following holds.
\stepcounter{propcounter}
\begin{enumerate}[label = {{\textbf{\Alph{propcounter}\arabic{enumi}}}}]
\item For each $\phi\in \mathcal{F}$ and $v\in S_\phi$, there are $\leq \sqrt{p_{\balcol}}p_\tr p_\fa n$ edges in $\Psi'$ around $v$ with symbol $\phi$.\label{prop:upperboundinPhiprime}
\end{enumerate}
Note that this is possible, as, initially, for each $\phi\in \mathcal{F}$ and $u,v\in S_\phi$, the number of edges in $\Psi'$ around $u$ or $v$ with the symbol $\phi$ is at most $|\lambda_{\phi,v}|\leq 10p_\balcol p_{\tr}p_{\fa}n\leq \sqrt{p_{\balcol}}p_\tr p_\fa n$ by \eqref{eqn:lambdaphivbound}.
Then, when for $j\in [r]$ and $\{u,v\}\in F_j$ we look to select $\phi$, the number of $\phi\in \mathcal{F}$ with $u,v\in S_\phi$ is, by \ref{prop:C:uvplentyofphi}, at least $p_S^{2}p_\tr^{-1} p_\fa^{-1}/2$.
If none of these possibilities for $\phi$ can be added without violating \ref{prop:upperboundinPhiprime}, then the number of edges of $\Psi'$ containing $u$ or $v$ must be at least
\begin{equation}\label{eqn:uppboundedgepsiprime}
(p_S^{2}p_\tr^{-1} p_\fa^{-1} /2)\cdot \sqrt{p_{\balcol}}p_\tr p_\fa n\geq p_{\balcol}^{2/3}n.
\end{equation}
On the other hand, for each $w\in V(G)$, by \ref{prop:C:uplentyofphi} and \eqref{eqn:lambdaphivbound}, the number of cycles in $\mathcal{C}$ containing $w$ is at most
\[
\sum_{\phi\in \mathcal{F}:w\in S_\phi}|\lambda_{\phi,w}|\leq 2p_Sp_{\tr}^{-1}p_{\fa}^{-1} \cdot 10p_\balcol p_{\tr}p_{\fa}n= 20p_S p_\balcol n.
\]
Therefore, the number of edges in $\Psi'$ around $u$ and $v$ with any symbol is (as, by the choice of the $F_j$, $j\in [r]$, the underlying undirected graph of $L_j+F_j$ has degree at most 4) at most
\begin{equation}\label{eqn:uppboundsymbolroundvertex}
2\cdot 4\cdot 20p_S p_\balcol n\leq p_{\balcol}^{3/4}n,
\end{equation}
which contradicts \eqref{eqn:uppboundedgepsiprime}. Thus, we can take $\Psi'$ as claimed, so that \ref{prop:upperboundinPhiprime} holds.

Note that, similarly again to our work in Section~\ref{sec:partA1}, due to the choice of the $F_j$, $j\in [r]$, we can now let $\mathcal{C}'$ be a directed cycle decomposition of $\Psi'$ into directed 2-cycles and triangles. Form $\Psi''$ by starting with $\Psi'$ and, for each directed triangle $L\in \mathcal{C}$ with $V(L)=\{x,y,z\}$ and $\vec{xy},\vec{yz},\vec{zy}\in E(L)$, choose a symbol $\phi$ with $x,y,z\in S_\phi$ and add $\vec{yx},\vec{zy},\vec{xz}$, each with the symbol $\phi$, so that the following holds.
\begin{enumerate}[label = {{\textbf{\Alph{propcounter}\arabic{enumi}}}}]\addtocounter{enumi}{1}
\item For each $\phi$ and each $w\in S_\phi$, there are $\leq \sqrt{p_{\balcol}}p_\tr p_\fa n$ edges in $\Psi''$ around $w$ with symbol $\phi$.\label{prop:upperboundinPhiprimeprime}
\end{enumerate}
Note that this is possible, as when for $L\in \mathcal{C}'$  with $V(L)=\{x,y,z\}$ and $\vec{xy},\vec{yz},\vec{zy}\in E(L)$ we look to select $\phi$, the number of edges in $\Psi''$ around $x$, $y$ or $z$ with any symbol is (noting we are trying to add at most the same number of edges again around each vertex to get from $\Psi'$ to $\Psi''$),
by \eqref{eqn:uppboundsymbolroundvertex}, at most $2p_{\balcol}^{3/4}n$.
On the other hand, by \ref{prop:partC:triplesofverticesinSphi}, there are at least $p_S^3p_\tr^{-1} p_{\fa}^{-1}/2$ possibilities for $\phi\in \mathcal{F}$ with $x,y,z\in S_\phi$, so if no such $\phi\in \mathcal{F}$ is such that $\vec{yx},\vec{zy},\vec{xz}$ can be added to $\Phi''$, each with the symbol $\phi$, without violating \ref{prop:upperboundinPhiprimeprime}, then there must at least
$(p_S^3p_\tr^{-1} p_{\fa}^{-1}/2) \cdot (p_{\balcol}^{1/2}p_\tr p_\fa n)>2p_{\balcol}^{3/4}n$ edges in $\Psi''$ around $x$, $y$ or $z$ with any symbol, a contradiction. Therefore, we can find $\Psi''$ as claimed.

Note that $\Psi''$ has a decomposition into directed 2-cycles, $\mathcal{C}''$ say. Let $s=|\mathcal{C}''|$ and $\mathcal{C}''=\{L'_1,\ldots,L'_s\}$. For each $i\in [s]$, label vertices and symbols so that $V(L_i')=\{x_i,y_i\}$ and the edges of $L_i'$ are an edge from $x_i$ to $y_i$ with symbol $\phi_i$ and an edge from $y_i$ to $x_i$ with symbol $\phi'_i$.
For each $\phi\in \mathcal{F}$ and $v\in S_\phi$, we have
\begin{align}
|\{i\in [s]:&\phi_i=\phi,x_i=v\}|-|\{i\in [s]:\phi'_i=\phi,y_i=v\}|\nonumber\\
&=|\{\vec{uv}\in E(\Psi''):u\in S_\phi,\mathrm{symb}(\vec{uv})=\phi\}|-|\{\vec{vu}\in E(\Psi''):u\in S_\phi,\mathrm{symb}(\vec{vu})=\phi\}|\nonumber\\
&=|\{\vec{uv}\in E(\Psi'):u\in S_\phi,\mathrm{symb}(\vec{uv})=\phi\}|-|\{\vec{vu}\in E(\Psi'):u\in S_\phi,\mathrm{symb}(\vec{vu})=\phi\}|\nonumber\\
&=|\{\vec{uv}\in E(\Psi):u\in S_\phi,\mathrm{symb}(\vec{uv})=\phi\}|-|\{\vec{vu}\in E(\Psi):u\in S_\phi,\mathrm{symb}(\vec{vu})=\phi\}|\nonumber\\
&\overset{\eqref{eqn:balanceroundvertex1}}{=}\lambda_{\phi,v}.\label{eqn:balanceroundvertex4}
\end{align}

\medskip

\noindent\textbf{II. Making the necessary changes using $(u,v,L)$-links.}
Now, take a maximal set $I\subset [s]$ for which there are edge-disjoint paths $P_i$, $i\in I$, in $G$ satisfying the following conditions.
\stepcounter{propcounter}
\begin{enumerate}[label = {{\textbf{\Alph{propcounter}\arabic{enumi}}}}]
\item For each $i\in I$, $P_i$ is an $x_i,y_i$-path of length 62 with odd edges in $E_{3,\phi_i}$ and even edges in $E_{3,\phi_i'}$.\label{prop:oddeven}
\item For each $i\in I$, the odd edges of $P_i$ have the same colour (with multiplicity) as the even edges of $P_i$.\label{prop:samecols}
\item For each $\phi\in \mathcal{F}$ and $v\in S_\phi$, there are at most $\frac{1}{8}p_{\balcol}^{1/4}p_\tr p_\fa n$ values of
 $i\in I$ with $\phi\in \{\phi_i,\phi'_i\}$ for which $v$ is an internal vertex of $P_i$.\label{prop:notoftenmiddle}
\end{enumerate}

Suppose, for a contradiction, that $I\neq [s]$. Take some $i\in [s]\setminus I$. Let $P_i$, $i\in I$, be a set of edge-disjoint paths in $G$ satisfying \ref{prop:oddeven}--\ref{prop:notoftenmiddle}. We will show that there is an $(x_i,y_i,L)$-link (with $L$ as defined in Theorem~\ref{thm:Llinks}) whose edges alternate between $E_{3,\phi_i}$ and $E_{3,\phi_i'}$, which are edge-disjoint from the paths $P_j$, $j\in I$, (noting that we only need to actively avoid those paths with edges in $E_{3,\phi}\cup E_{3,\phi_i'}$), and whose interior vertices are not in many interior vertices of paths in $P_j$, $j\in I$.

Let then $E^\mathrm{forb}$ be the set of edges in $E_{3,\phi_i}\cup E_{3,\phi_i'}$ which are in some path $P_j$, $j\in I$, and which do not contain $x_i$ or $y_i$.
Let $V^\mathrm{forb}$ be the set of vertices which are neighbours of $x_i$ or $y_i$ in some edge in $E_{3,\phi_i}\cup E_{3,\phi_i'}$ which appears in some path $P_j$, $j\in I$. Let $W^{\mathrm{forb}}$ be the set of vertices $v$ for which there are at least $\frac{1}{16}p_{\balcol}^{1/4}p_\tr p_\fa n$ values of $j\in I$ with $\{\phi_i,\phi_i'\}\cap \{\phi_j,\phi'_j\}\neq \emptyset$ for which $v$ is an internal vertex of $P_j$.

From \ref{prop:C:forvxbalancing}\ref{prop:C:linksb} and \ref{prop:C:linksc}, we can see that we should take particular care in counting the number of edges we have used which could be the 2nd or 61st edge of the link, as the corresponding bound we have is larger by a factor of $(p_\tr p_\fa)^{-1}$. However, any such edge contains a neighbour of $\{x_i,y_i\}$ using an edge of  $E_{3,\phi_i}\cup E_{3,\phi_i'}$, so the number of edges here, as we will see, will be limited by \ref{prop:notoftenmiddle}.
 For this, let $F^{\mathrm{forb}}$ be the set of edges in $E_{3,\phi_i}\cup E_{3,\phi_i'}$ which are in some path $P_j$, $j\in I$, and which contain a neighbour of $x_i$ or $y_i$ in $E_{3,\phi_i}\cup E_{3,\phi_i'}$.

We will now find an $(x_i,y_i,L)$-link whose edges alternate between $E_{3,\phi_i}$ and $E_{3,\phi_i'}$ which do not use any vertices in $W^\mathrm{forb}$ as interior vertices, or a vertex in $V^\mathrm{forb}$ as a neighbour of $x_i$ or $y_i$, or any edge in $E^\mathrm{forb} \cup F^{\mathrm{forb}}$. First, note that
\begin{align}
|E^\mathrm{forb}|&\leq 62\cdot |\{\vec{uv}\in E(\Psi''):\mathrm{symb}(\vec{uv})\in \{\phi_i,\phi_i'\}\}|\overset{\ref{prop:partC:Sphibound},\ref{prop:upperboundinPhiprimeprime}}{\leq} 62 \cdot 3p_Sn\cdot \sqrt{p_{\balcol}}p_\tr p_\fa n\nonumber\\
&\leq p_{\balcol}^{1/5}p_\tr p_\fa n^2.\label{eqn:Eforbbound}
\end{align}
Now, by \ref{prop:forC3:maxdeg}, there are at most $8p_{\bal}p_{\tr}p_\fa n$ vertices in $S_\phi\cap S_{\phi'}$ which are a neighbour of $x_i$ or $y_i$ in $E_{3,\phi}$ or $E_{3,\phi'}$. Thus, using \ref{prop:notoftenmiddle} and\ref{prop:upperboundinPhiprimeprime},
\begin{equation}\label{eqnFforbbound}
|F^{\mathrm{forb}}|\leq 8p_{\tr}p_\fa n\cdot \left(4\cdot \frac{1}{8}p_{\balcol}^{1/4}p_\tr p_\fa n+2\cdot \sqrt{p_{\balcol}}p_\tr p_\fa n\right)\leq p_{\balcol}^{1/5}(p_\tr p_\fa n)^2.
\end{equation}
Then, note that
\begin{equation}\label{eqn:Vforbbound}
|V^{\mathrm{forb}}|\overset{\ref{prop:notoftenmiddle},\ref{prop:upperboundinPhiprimeprime}}{\leq}
4\cdot \frac{1}{8}p_{\balcol}^{1/4}p_\tr p_\fa n+2\cdot \sqrt{p_{\balcol}}p_\tr p_\fa n\leq p_{\balcol}^{1/5}p_\tr p_\fa n.
\end{equation}
Finally, note that
\begin{align}
|W^{\mathrm{forb}}|&\leq \frac{61\cdot |\{\vec{uv}\in E(\Psi''):\mathrm{symb}(\vec{uv})\in \{\phi_i,\phi_i'\}\}|}{\frac{1}{16}p_{\balcol}^{1/4}p_\tr p_\fa n}\nonumber\\
&\overset{\ref{prop:partC:Sphibound},\ref{prop:upperboundinPhiprimeprime}}{\leq} \frac{61 \cdot 3p_Sn\cdot \sqrt{p_{\balcol}}p_\tr p_\fa n}
{\frac{1}{16}p_{\balcol}^{1/4}p_\tr p_\fa n}\leq p_{\balcol}^{1/5}n.
\label{eqn:Wforbbound}
\end{align}

Let $\mathcal{L}$ be the set of $(x_i,y_i,L)$-links
with odd edges in $E_{3,\phi_i}$
and even edges in $E_{3,\phi'_i}$ and whose vertices are in $S_\phi\cap S_{\phi'}$
Then, using \ref{prop:C:forvxbalancing}\ref{prop:C:linksb}--\ref{prop:C:linkse}, the number of links in $\mathcal{L}$  which either
\begin{enumerate}[label = \roman{enumi})]
\item for some $3\leq k\leq 60$, contain an edge of $E^{\mathrm{forb}}$ as their $k$th edge, or
\item contain a vertex of $W^{\mathrm{forb}}$ as an internal vertex, or
\item use an edge from $F^{\mathrm{forb}}$ as its 2nd or 61st edge, or
\item use a vertex of $V^{\mathrm{forb}}$ as its 2nd or 62nd vertex,
\end{enumerate}
is at most
\begin{align*}
100(p_{\tr}p_{\fa})^{60}&n^{28}\cdot ((p_{\tr}p_{\fa})\cdot |E^{\mathrm{forb}}|+|F^{\mathrm{forb}}|+(p_{\tr}p_{\fa})^2n\cdot |W^{\mathrm{forb}}|+(p_{\tr}p_{\fa})n\cdot |V^{\mathrm{forb}}|)\\
&\overset{\eqref{eqn:Eforbbound}\text{--}\eqref{eqn:Wforbbound}}\le
400p_{\balcol}^{1/5}(p_{\tr}p_{\fa})^{62}n^{30}.
\end{align*}
However, by \ref{prop:C:forvxbalancing}\ref{prop:C:linksa}, $|\mathcal{L}|\geq p_{\balvx}^{63}(p_{\tr}p_{\fa})^{62}n^{30}$.
Thus, there is some $(x_i,y_i,L)$-link with odd edges in $F_{3,\phi_i}$ and even edges in $F_{3,\phi'_i}$ which uses no edge in $E^{\mathrm{forb}}$ or vertex in $W^{\mathrm{forb}}$ (noting that if it has an edge of $E^{\mathrm{forb}}$ as its 2nd or 61st edge, then this edge is in $F^{\mathrm{forb}}$, and if it has an edge of $E^{\mathrm{forb}}$ as its first or last edge then it contains a vertex of $V^{\mathrm{forb}}$ as its 2nd or 62nd vertex).

Let $P_i$ be the path of such a link. Then, we have that \ref{prop:oddeven} and \ref{prop:samecols} hold for $P_i$ from the properties of an $(x_i,y_i,L)$-link. Furthermore, as $P_i$ has no edge in $E^{\mathrm{forb}}$ or vertex in $V^{\mathrm{forb}}$ it has no edge in common with any $E(P_j)$
with $j\in I$ and $\{\phi_i,\phi_i'\}\cap \{\phi_j,\phi_j'\}\neq \emptyset$, and therefore, by \ref{prop:oddeven}, no edge in common with any $E(P_j)$, $j\in I$.  Finally, as $V(P_i)$ contains no vertex in $W^{\mathrm{forb}}$, we have that \ref{prop:notoftenmiddle} holds with $I$ replaced by $I\cup \{i\}$. Therefore, the set $I\cup \{i\}$ contradicts the choice of $I$, as shown by $P_j$, $j\in I\cup\{i\}$.

Thus, we have that $I=[s]$. Take  edge-disjoint paths $P_i$, $i\in [s]$, in $G$, then, satisfying \ref{prop:oddeven}--\ref{prop:notoftenmiddle}. For each $\phi\in \mathcal{F}$, let
\[
\hat{E}^*_\phi=\left(\hat{E}_\phi\setminus \left(\bigcup_{i\in [s]:\phi\in \{\phi_i,\phi_i'\}}E(P_i)\right)\right)\cup
\left(\bigcup_{i\in [s]:\phi\in \{\phi_i,\phi_i'\}}E(P_i)\setminus E_{3,\phi}\right).
\]
We will show that $\hat{E}^*_\phi$, $\phi\in \mathcal{F}$, satisfy \ref{prop:frompartC3:edge}--\ref{prop:frompartC3:roundvertex}. Firstly, as for any $\phi\in \mathcal{F}$ and $i\in [s]$
with $\phi\in \{\phi_i,\phi_i'\}$ we have $E(P_i)\cap \hat{E}_\phi' \subset E_{3,\phi}$, we have that \ref{prop:frompartC3:edge} holds due to \ref{prop:fromcolbalance:edge}.
Next, for each $\phi\in \mathcal{F}$, for each $i\in [s]$ with $\phi\in \{\phi_i,\phi_i'\}$, we have that $E(P_i)\setminus E_{3,\phi_i}$ has the same colours (with multiplicity) as $E(P_i)\cap E_{3,\phi_i}$, and therefore each colour appears the same number of times in $\hat{E}^*_\phi$ as in $\hat{E}_\phi$. Thus, \ref{prop:frompartC3:colour} follows from \ref{prop:fromcolbalance:colour}. Furthermore, we have that \ref{prop:fromcolbalance:stillstillinS} follows from \ref{prop:fromcolbalance:stillinS} as, for each $\phi\in \mathcal{F}$, for every $i\in [s]$, if $\phi\in \{\phi_i,\phi_i'\}$ then $P_i$ has all its vertices in $S_\phi$.

For each $\phi\in \mathcal{F}$ and $v\in V(G)$, we have
\begin{align*}
|\{e\in \hat{E}_\phi^*&:v\in V(e)\}|=|\{e\in \hat{E}_\phi:v\in V(e)\}|-|\{i\in [s]:\phi_i=\phi,x_i=v\}|+|\{i\in [s]:\phi'_i=\phi,y_i=v\}|\\
&\overset{\eqref{eqn:balanceroundvertex4}}{=}|\{e\in \hat{E}_\phi:v\in V(e)\}|-\lambda_{\phi,v}
\overset{\eqref{eqn:lambdaphiv}}{=}|\{i\in I_\phi:v\in \hat{V}_i\setminus V(\tilde{M}_{i,1})\}|-|\{i\in I_\phi:v\in T_i\}|
\end{align*}
and thus \ref{prop:frompartC3:vertex} holds.
Finally, for each $\phi\in \mathcal{F}$, $i\in I_\phi$ and $v\in S_i$, we have
\begin{align*}
|\{e\in \hat{E}_\phi^*:v\in V(e), c(e)\in C_i\}|&\leq |\{e\in \hat{E}_\phi:v\in V(e), c(e)\in C_i\}|+|\{i\in [s]:\phi\in \{\phi_i,\phi_i'\},v\in V(P_i)\}|\\
&\overset{\ref{prop:frompartC2:roundvertex},\ref{prop:notoftenmiddle}}{\leq} \frac{1}{2}p_{\mathrm{pt}}^{3/2}p_\tr p_\fa n+\frac{1}{8}p_{\balcol}^{1/4}p_\tr p_\fa n\leq p_{\mathrm{pt}}^{3/2}p_\tr p_\fa n,
\end{align*}
so therefore \ref{prop:frompartC3:roundvertex} holds.
\end{proof}
\fi


\subsection{Part \ref{partC4}: Partitioning the final edges}\label{sec:C4}
Consider the partition of $\hat{E}\setminus (\bigcup_{i\in [n]}\tilde{M}_{i,1})$ into $\hat{E}^*_\phi$, $\phi\in \mathcal{F}$, satisfying \ref{prop:frompartC3:edge}--\ref{prop:frompartC3:roundvertex}.
It is now sufficient to, for each $\phi\in \mathcal{F}$, partition $\hat{E}_\phi$ into matchings $\tilde{M}_{i,2}$, $i\in I_\phi$, such that, for each $i\in I_\phi$, $V(\tilde{M}_{i,2})\subset R_{i,2}$ and $C(\tilde{M}_{i,2})=\hat{C}_i\setminus C(\tilde{M}_{i,1})$.
Indeed, then, setting $\hat{M}_i=\tilde{M}_{i,1}\cup \tilde{M}_{i,2}$ for each $i\in [n]$, we have that \ref{key3:outcome1}--\ref{key3:outcome3} hold, where \ref{key3:outcome3} follows from \ref{prop:frompartC3:vertex}.
Thus, for the rest of the proof, to simplify the notation slightly, we will fix a family $\phi\in \mathcal{F}$ and omit $\phi$ from the subscripts.

Fix, then, $\phi\in \mathcal{F}$. For each $c\in C$, we wish to assign one each of the edges of colour $c$ in $\hat{E}_\phi^*$ to the $\tilde{M}_{i,2}$ for which $c\in \hat{C}_i\setminus C(\tilde{M}_{i,1})$, doing this in such a way that the assigned edge for $\tilde{M}_{i,2}$ has its vertices in $R_{i,2}$ and that $\tilde{M}_{i,2}$ is a matching.
If possible, this will use exactly all of the colour-$c$ edges in $\hat{E}_\phi^*$, due to \ref{prop:frompartC3:colour}.

For each $c\in C$, let $E_c$ be the edges in $\hat{E}_\phi^*$ with colour $c$ and let $I_{c}=\{i\in I_\phi:c\in \hat{C}_i\setminus C(\tilde{M}_{i,1})\}$. By \ref{prop:frompartC3:colour}, we have $|E_c|=|I_c|$.
 For each $c\in C$, let $L_{c}$ be the bipartite graph with vertex classes $E_{c}$ and $I_{c}$ and edge set $\{ei:e\in E_{c},i\in I_{c},V(e)\subset R_{i,2}\}$.
If we can find a perfect matching $M_c$ in $L_c$ for each $c\in C$, such that no two edges matched to the same $i\in I_\phi$ in any $L_c$ share any vertices, then we will be able to use $M_c$, $c\in C$, to assign the edges in $\hat{E}_\phi^*$. We will now first sparsify the edges of $L_c$ to get $L_c'$, for each $c\in C$, by deleting edges independently at random. (For each $i\in I_\phi$ this will reduce the overlap between the edges we are still considering for $\tilde{M}_{i,2}$.) Then, we will further sparsify $L_c'$ to get $L_c''$ in such a way that if any such perfect matchings $M_c$, $c\in C$, can be found in $L_c'$ then they will automatically have the additional property we wish to have.

Recalling from \eqref{eqn:nzeroetc} that  $n_0=1.01p_{\mathrm{pt}}p_{\tr}p_{\fa}n$, $D_0=p_R^2p_{\mathrm{pt}}p_\tr p_\fa n/8p_S^2$ and $q_0=p_R^2/8p_S^2$, we will use \ref{propforLc:2}--\ref{propforLc:1}. Note that $q_0n_0=1.01D_0$. 
By  \ref{prop:C:colsinfamilies}, \ref{prop:nocolourmissingtoomuch} and \ref{prop:partc1:1b}, we have
 $|I_c|=(1\pm 2p_{\balcol})p_{\mathrm{pt}}p_{\tr}p_{\fa}n$, so that $0.98n_0\leq |I_c|=|E_c|\leq n_0$.
Let $\Delta_0=p_{\mathrm{pt}}^{3/2}p_\tr p_\fa n$ and $\mu=10/(\sqrt{p_T}\Delta_0)$,
so that
\begin{equation}\label{eqn:mudelta}
\mu \Delta_0= \frac{10}{\sqrt{p_T}}=\omega(\log^3 n),\;\;\; \frac{D_0}{\Delta_0}\geq p_{\mathrm{pt}}^{-1/3}=\omega(p_T^{-10}\log^3 n),
\end{equation}
and
\begin{equation}\label{eqn:mudzero}
\mu D_0= \frac{10 D_0}{\Delta_0\sqrt{p_T}}= \frac{10p_R^2}{8p_S^2\sqrt{p_{\mathrm{pt}}p_T}}\geq \frac{1}{p_{\mathrm{pt}}^{1/3}}=\omega(p_T^{-2}\log^3 n),
\end{equation}
 For each $c\in C$, form $L_c'$ by taking $L_c$ and keeping each edge independently at random with probability $\mu$. The following then holds.


\begin{claim}\label{claim:sparsifyLprime}
With high probability, we have the following properties.
\stepcounter{propcounter}
\begin{enumerate}[label = {{\textbf{\Alph{propcounter}\arabic{enumi}}}}]
\item For each $c\in C$ and $I\subset I_c$ with $1\leq |I|\leq n_0(\log^2n)/\mu D_0$, $|N_{L_c'}(I)|\geq \mu D_0|I|/30\log^2n$.\label{prop:LCprime:expands}
\item For each $c\in C$ and $E\subset E_c$ with $1\leq |E|\leq n_0(\log^2n)/\mu D_0$, $|N_{L_c'}(E)|\geq \mu D_0|E|/30\log^2n$.\label{prop:LCprime:expands:2}
\item For each $c\in C$, and any sets $I\subset I_c$ and $E\subset E_c$ with $|I_c|,|E_c|\geq n_0\log^2n/4\mu D_0$, $e_{L'_c}(I,E)\geq \mu q_0|I||E|/2$.\label{prop:LCprime:joined}
\item For each $c\in C$, $i\in I_c$ and $w\in R_{i,2}$,\label{prop:LCprime:dep}
\[
|\{z\in R_{i,2}:wz\in \hat{E}^*_\phi,c(wz)\in {C}_{i},i(wz)\in E(L_{c(wz)}')\}|\leq 2\mu \Delta_0.
\]
\end{enumerate}
\end{claim}
\begin{proof}[Proof of Claim~\ref{claim:sparsifyLprime}]
\ref{prop:LCprime:expands}:
Let $c\in C$ and $I\subset I_c$ with $1\leq |I|\leq n_0/2D_0$. Then, we have $|N_{L_c}(I,E_{4,\phi})|\geq D_0|I|$ by \ref{propforLc:2}. Letting $X_{c,I}=|N_{L_c'}(I,E_{4,\phi})|$, we have $\E X_{c,I}\geq \mu D_0|I|$. Thus, by Lemma~\ref{chernoff}, we have that $\P(X_{c,I}< \mu D_0|I|/2)=\exp(-\omega(|I|\log n))$.
Therefore, by a union bound, with high probability, for every $c\in C$ and $I\subset I_c$ with $1\leq |I|\leq n_0/2D_0$, we have $|N_{L_c'}(I)|\geq \mu D_0|I|/2\geq \mu D_0|I|/30\log^2n$.

Suppose then $c\in C$ and $I\subset I_c$ with $n_0\log^2n/2D_0\leq |I|\leq n_0/\mu D_0$. Let $\ell=q_0|I|/2\leq q_0n_0/\mu D_0$, so that $\ell\leq 1.01\log^2n/\mu$. Let $E$ be the set of $e\in E_c$ with at least $\ell$ neighbours in $L_c$ in $I$. Then,
\[
e_{L_c}(I,E_c\setminus E)\leq \ell\cdot |E_c\setminus E|< q_0|I|\cdot |E_c\setminus E|
\]
so that, by \ref{propforLc:1}, we must have $|E_c\setminus E|<n_0/2D_0$, and, hence, $|E|\geq 0.98n_0-n_0/2D_0\geq 0.97n_0$.
For each $e\in E$, setting $\ell'=\ell/\log^{2}n$ and using that $\ell'\leq 1.01/\mu$,
\[
\P(e\in N_{L_c'}(I))\geq \mu\ell'(1-\mu)^{\ell'}\geq \frac{\mu \ell'}{2e}\geq \frac{\mu q_0|I|}{4e\log^2n}\geq \frac{\mu D_0|I|}{20n_0\log^2n}.
\]
so that, letting $X_{c,I}=|N_{L_c'}(I,E_{4,\phi})|$, $\E X_{c,I}\geq \mu D_0|I||E|/20n_0\geq \mu D_0|I|/25\log^2n$. Therefore, by Lemma~\ref{chernoff}, we have that $\P(|N_{L_c''}(I)|< \mu D_0|I|/30\log^2n)=\exp(-\omega(|I|\log n))$. Thus, using a union bound, this completes the proof that \ref{prop:LCprime:expands:2} holds with high probability.

\ref{prop:LCprime:expands:2} holds with high probability similarly to \ref{prop:LCprime:expands}, using \ref{propforLc:1}, and \ref{propforLc:3} in place of \ref{propforLc:2}, and $\{i\in I_\phi:c\in C_{i,2}\}$ in place of $E_{4,\phi}$.

\ref{prop:LCprime:joined}: Let $c\in C$, $I\subset I_c$ and $E\subset E_c$ satisfy $|I|,|E|\geq n_0(\log^2n)/4\mu D_0$. Then, by \ref{propforLc:1}, there are at least $q_0|I||E|$ edges between $I_c$ and $E_c$ in $L_c$. Letting $X_{I,E}=e_{L_c'}(I_c,E_c)$, we have that $\E X_{I,E}\geq \mu q_0|I||E|$, so that, by Lemma~\ref{chernoff}, there are fewer than $\mu q_0|I||E|/2$ edges between $I$ and $E$ in $L'_c$ with probability at most
\[
\exp\left(-\frac{\mu q_0|I||E|}{12}\right)\leq \exp\left(-\frac{\mu q_0\cdot (n_0(\log^2n)/4\mu D_0)\cdot \max\{|I|,|E|\}}{12}\right)=\exp(-\omega(\max\{|I|,|E|\}\log n)).
\]
Thus, by a union bound, with high probability \ref{prop:LCprime:joined} holds.

 \ref{prop:LCprime:dep}: Let $c\in C$, $i\in I_c$ and $w\in R_{i,2}$. By \ref{prop:frompartC3:roundvertex}, we have
\[
|\{z\in R_{i,2}:wz\in \hat{E}^*_\phi,c(wz)\in {C}_{i},i(wz)\in E(L_{c(wz)})\}|\leq p_{\mathrm{pt}}^{3/2}p_\tr p_\fa n=\Delta_0.
\]
Therefore, $\E|\{z\in R_{i,2}:wz\in \hat{E}^*_\phi,c(wz)\in {C}_{i},i(wz)\in E(L_c')\}|\leq \mu \Delta_0$. Thus,  \ref{prop:LCprime:dep} holds with high probability by Lemma~\ref{chernoff} and a union bound.
\claimproofend


Thus, by Claim~\ref{claim:sparsifyLprime}, we can assume that \ref{prop:LCprime:expands}--\ref{prop:LCprime:dep} hold. Now, let
\begin{equation}\label{eqn:etaD1}
\eta=1/4\mu \Delta_0\overset{\eqref{eqn:mudelta}}={\sqrt{p_T}}/{40}\;\;\text{ and }\;\;
D_1=\frac{\eta^3\mu D_0}{10^3\log^2n}=\frac{\eta^2D_0}{4\cdot 10^3\Delta_0\log^2n},
\end{equation}
so that, by \eqref{eqn:mudelta}, $D_1=\omega(\eta^{-2}p_T^{-1}\log^3 n)$.
For each $c\in C$, $i\in I_\phi$ and $e\in E_c$ with $ie\in E(L_c)$, let $x_{ie}$ be a Bernoulli random variable with probability $\eta$. For each $c\in C$, let $L''_c$ be the subgraph of $L'_c$ of edges $ie$ for which $x_{ie}=1$ and, for each $c'\in C\setminus \{c\}$ and $f\in E_{c'}$ with
$V(f)\cap V(e)\neq \emptyset$ and $if\in E(L_c')$, we have $x_{if}=0$.


\begin{claim}\label{claim:hallinLprime}
With high probability, we have the following properties.
\stepcounter{propcounter}
\begin{enumerate}[label = {{\textbf{\Alph{propcounter}\arabic{enumi}}}}]
\item For each $c\in C$ and $I\subset I_c$ with $1\leq |I|\leq n_0/2D_1$, $|N_{L_c''}(I)|\geq D_1|I|$.\label{prop:hall:expands}
\item For each $c\in C$ and $E\subset E_c$ with $1\leq |E|\leq n_0/2D_1$, $|N_{L_c''}(E)|\geq D_1|E|$.\label{prop:hall:expands:2}
\item For each $c\in C$, and any sets $I\subset I_c$ and $E\subset E_c$ with $|I|,|E|\geq n_0/4D_1$, there is an edge between $I$ and $E$ in $L''_c$.\label{prop:hall:joined}
\end{enumerate}
\end{claim}
\begin{proof}[Proof of Claim~\ref{claim:hallinLprime}] Let $c\in C$, $i\in I_c$ and $e\in E_c$ with $ie\in E(L_c)$. Let $F_{i,e}$ be the set of edges $f\neq e$ for which $V(f)\cap V(e)\neq \emptyset$ and $if\in E(L'_{c(f)})$. Then, we have
\begin{align}
|F_{i,e}|&\leq \sum_{w\in V(e)}|\{z\in R_{i,2}:wz\in \hat{E}^*_\phi,c(wz)\in {C}_{i},if\in E(L'_{c(f)})\}|\overset{\ref{prop:LCprime:dep}}{\leq} 4\mu \Delta_0,\label{eqn:Febound}
\end{align}
and, hence,
\begin{equation}\label{eqn:edgeinLprime}
\P(e\in E(L''_c))\geq \eta(1-\eta)^{|F_{i,e}|}\geq \eta (1-\eta)^{4\mu \Delta_0}\geq \eta/2e.
\end{equation}


\ref{prop:hall:expands} if $|I|\leq n_0(\log^2n)/\mu D_0$:
Let $c\in C$ and $I\subset I_c$ with $1\leq |I|\leq n_0(\log^2n)/\mu D_0$. Then, we have $|N_{L_c'}(I)|\geq \mu D_0|I|/30\log^2n$ by \ref{prop:LCprime:expands}. Letting $X_{c,I}=|N_{L_c''}(I)|$, we have that $X_{c,I}$ is 2-Lipschitz and affected by at most $|N_{L_c'}(I)|\cdot 4\mu \Delta_0$ random variables $x_{ie}$ by \ref{prop:LCprime:dep}, and
$\E X_{c,I}\geq |N_{L_c'}(I)|\cdot \eta/2e\geq \eta\mu D_0|I|/60e\log^2n\geq 2D_1|I|$. Thus, by Lemma~\ref{lem:mcdiarmidchangingc}, we have that
\begin{align}
\P(X_{c,I}< D_1|I|)&\leq 2\exp\left(-\frac{2(|N_{L_c'}(I)|\cdot \eta/4e)^2}{4\cdot |N_{L_c'}(I)|\cdot 4\mu \Delta_0}\right)=2\exp\left(-\Omega\left(\frac{\eta^2 |N_{L_c'}(I)|}{\mu\Delta_0}\right)\right)\nonumber\\
&=2\exp\left(-\Omega\left(\frac{\eta^2 \mu D_0|I|}{\mu \Delta_0\log^2n}\right)\right)
\overset{\eqref{eqn:etaD1}}{=}2\exp\left(-\Omega\left(\frac{p_T D_0|I|}{\Delta_0\log^2n}\right)\right)\nonumber \\
&\overset{\eqref{eqn:mudelta}}{=}\exp(-\omega(|I|\log n)).\label{eqn:applyingMcD}
\end{align}
Therefore, by a union bound, with high probability, for every $c\in C$ and $I\subset I_c$ with $1\leq |I|\leq n_0(\log^2n)/\mu D_0$, we have $|N_{L_c'}(I)|\geq D_1|I|$.


\ref{prop:hall:expands} if $|I|> n_0(\log^2n)/\mu D_0$: Let $c\in C$ and $I\subset I_c$ with $n_0(\log^2n)/\mu D_0<|I|\leq n_0/2D_1$. Let $\ell= \lfloor \mu q_0|I|/4\rfloor$ and note that
\begin{equation}\label{eqn:forlastbit}
\eta \ell\overset{\eqref{eqn:etaD1}}{\geq}\frac{\mu q_0|I|}{32\mu \Delta_0}=\Omega\left(\frac{|I|D_0}{n_0\Delta_0}\right)\overset{\eqref{eqn:etaD1}}=\omega\left(\frac{D_1|I|\log n}{n_0}\right).
\end{equation}
Let $E$ be a maximal set of $e\in E_c$ for which $|N_{L_c'}(e)\cap I|\geq \ell$, and let $L_{c,I}'\subset L_{c}'$ contain exactly $\ell$ edges from each vertex in $E$ into $I$. Suppose, for contradiction, that $|E_c\setminus E|>n_0/10$. Then,
$e_{L_c'}(I,E_c\setminus E)< \ell\cdot |E_c\setminus E|\leq \mu q_0|I||E_c\setminus E|/4$,
which contradicts \ref{prop:LCprime:joined}. Thus, $|E|\geq |E_c|-n_0/10\geq 0.8n_0$.

For each $e\in E$, using \eqref{eqn:forlastbit}, we have $\P(e\in N_{L_c''\cap L_{c,I}'}(I,E))\geq 1-(1-\eta/2e)^\ell\geq \min\{0.99,3D_1|I|/n_0\}$. Thus, letting $X_{I,E}=|N_{L_c''\cap L_{c,I}'}(I,E)|$, we then have that $\E X_{I,E}\geq \min\{0.75n_0,2D_1|I|\}=\Omega(D_1|I|)$.
As $X_{I,E}$ is 2-Lipschitz and affected by at most $|E|\cdot \ell \cdot 4\mu \Delta_0=O(n_0\cdot \mu q_0|I|\cdot \sqrt{p_T})$ random variables $x_{ie}$ by \ref{prop:LCprime:dep}, by
Lemma~\ref{lem:mcdiarmidchangingc},
we have that
\begin{align*}
\P(X_{I,E}<D_1|I|)&\leq \exp\left(-\Omega\left(\frac{(D_1|I|)^2}{n_0\cdot \mu q_0|I|\cdot \sqrt{p_T}}\right)\right)
=\exp\left(-\Omega\left(\frac{D_1^2|I|}{\mu D_0\cdot \sqrt{p_T}}\right)\right)\\
&=\exp\left(-\Omega\left(\frac{\eta^2D_1|I|}{\sqrt{p_T}\log^2n}\right)\right)=\exp(-\omega(|I|\log n)).
\end{align*}
Therefore, by a union bound, with high probability, \ref{prop:hall:expands} holds if $|I|> n_0(\log^2n)/\mu D_0$.

\ref{prop:hall:expands:2}: That \ref{prop:hall:expands:2} hold with high probability follows similarly, again splitting into cases depending on whether $|E|\leq n_0(\log^2n)/\mu D_0$ or not.

\ref{prop:hall:joined}: Let $c\in C$, $I\subset I_c$ and $E\subset E_c$ satisfy $|I|,|E|\geq n_0/4D_1$. Let $N=e_{L_c'}(I,E)$. As $\log^2n/4\mu D_0\leq n_0/4D_1$, by \ref{prop:LCprime:joined} we have $N\geq \mu q_0|I||E|/2$. By \eqref{eqn:edgeinLprime}, we have $\E X_{I,E}\geq \eta N/2e$.
As $X_{I,E}$ is 2-Lipschitz and affected by at most $N\cdot \cdot 4\mu \Delta_0=O(N/\eta)$ random variables $x_{ie}$ by \ref{prop:LCprime:dep} and \eqref{eqn:etaD1}, by Lemma~\ref{lem:mcdiarmidchangingc}, we have that
\begin{align*}
\P(X_{I,E}>0)&\leq \exp\left(-\Omega\left(\frac{(\eta N)^2}{N\cdot \eta^{-1}}\right)\right)
=\exp\left(-\Omega\left(\eta^3 N\right)\right)
=\exp\left(-\Omega\left(\eta^3  \mu q_0|I||E|\right)\right)\\
& =\exp\left(-\Omega\left(\eta^3 \mu q_0 \max\{|I|,|E|\}\cdot n_0/D_1\right)\right)
\overset{\eqref{eqn:etaD1}}{=}\exp\left(-\Omega\left(q_0\log^2n \max\{|I|,|E|\}\cdot n_0/D_0\right)\right)\\
&=\exp(-\omega(\max\{|I|,|E|\}\log n)).
\end{align*}
Therefore, by a union bound, with high probability, \ref{prop:hall:joined} holds.
\claimproofend

Thus, by Claim~\ref{claim:hallinLprime}, we can assume that \ref{prop:hall:expands}--\ref{prop:hall:joined} hold, using which we show the following claim.


\begin{claim}\label{claim:matchinginLprime}
For each $c\in C$, $L_c'$ has a perfect matching.
\end{claim}
\begin{proof}[Proof of Claim~\ref{claim:matchinginLprime}] Fix $c\in C$. If $I\subset I_c$ with $|I|\leq  n_0/2D_1$, then, by \ref{prop:hall:expands}, $|N_{L_c''}(I)|\geq D_1|I|\geq |I|$.
Suppose, then, that $I\subset I_c$ with $n_0/4D_1\leq |I|\leq |I_c|/2$. Then, by \ref{prop:hall:joined}, we must have $|E_c\setminus N_{L_c''}(I)|<n_0/4D_1$, and, hence
\[
|N_{L_c''}(I)|\geq |E_c|-n_0/4D_1=|I_c|-n_0/4D_1\geq |I_c|/2\geq |I|.
\]

Thus, for any $I\subset I_c$ with $|I|\leq |I_c|/2$, we have $|N_{L_c''}(I)|\geq |I|$. Similarly, by \ref{prop:hall:expands:2} and \ref{prop:hall:joined}, we have that $|N_{L_c''}(E)|\geq |E|$ for any $E\subset E_c$ with $|E|\leq |E_c|/2=|I_c|/2$. Then, for any
$I\subset I_c$ with $|I|>|I_c|/2$, if $|N_{L_c''}(I)|<|I|$, then
we have that $|E_c\setminus N_{L_c''}(I)|\leq |I_c|/2$, and thus $|I_c\setminus I|\geq |N_{L_c''}(E_c\setminus N_{L_c''}(I))|\geq |E_c\setminus N_{L_c''}(I)|$, so that
\[
|N_{L_c''}(I)|= |E_c|-|E_c\setminus N_{L_c''}(I)|\geq |E_c|-|I_c\setminus I|=|I|.
\]
Therefore, Hall's matching condition holds, and thus there is a perfect matching in $L_c''$.
\claimproofend

By Claim~\ref{claim:matchinginLprime}, we can choose bijective functions $\psi_{c}:I_c\to E_c$, $c\in C$, such that, for each $c\in C$ and $i\in I_c$, $i\psi_c(i)\in E(L_c'')$.
For each $i\in I_\phi$, let $\tilde{M}_{i,2}=\{\psi_c(i):c\in \hat{C}_i\setminus C(\tilde{M}_{i,1})\}$. Observe that each of these subgraphs is a matching. Indeed, if $i\in I_\phi$ and $e,f\in E_c$ with $ie,if\in \tilde{M}_{i,2}\subset E(L_c'')$ and $e\neq f$, then we have $x_{ie}=x_{if}=1$.
For any edge $f'\in E_c$ intersecting $e$ with $if\in E(L_c')$, we have $x_{if'}=0$, and thus $e$ and $f$ do not intersect, so that $\tilde{M}_{i,2}$ is a matching. Moreover, by construction, $\tilde{M}_{i,2}$ is rainbow with colour set $(\hat{C}_i\setminus C(\tilde{M}_{i,1}))$.

For each $i\in [n]$, let $\tilde{M}_{i}=\tilde{M}_{i,1}\cup \tilde{M}_{i,2}$. We will show that \ref{key3:outcome1}--\ref{key3:outcome3} hold, thus completing the proof of Lemma~\ref{keylemma:completion}.
First, for each $i\in [n]$, as $\tilde{M}_{i,1}$ and $\tilde{M}_{i,2}$ are rainbow matchings with vertex sets in $(\hat{V}_i\setminus R_i)\cup R_{i,1}$ and $R_{i,2}$
 and colour sets $C(\tilde{M}_{i,1})$ and $\hat{C}_i\setminus C(\tilde{M}_{i,1})$ respectively, we have that $\tilde{M}_i$ is a rainbow matching with colour set $\hat{C}_i$ (and thus \ref{key3:outcome1} holds) and $V(\tilde{M}_i)\subset \hat{V}_i$. Thus, by the property from \ref{prop:partc1:1}, we have that \ref{key3:outcome2} holds.

 Finally, for \ref{key3:outcome3}, letting $\phi\in \mathcal{F}$ and $v\in S_\phi$, we wish to show that
 \begin{equation}\label{lastbalance}
|\{i\in I_\phi:v\in R_i\setminus V(\tilde{M}_i)\}|=|\{i\in I_\phi:v\in T_i\}|.
 \end{equation}
Let, then, $\phi\in \mathcal{F}$ and $v\in S_\phi$. First, note that, by \ref{prop:partc1:1},
\begin{equation}\label{lasteq}
|\{i\in I_\phi:v\in R_i\setminus V(\tilde{M}_i)\}|=|\{i\in I_\phi:v\in \hat{V}_i\}|-|\{i\in I_\phi:v\in V(\tilde{M}_i)\}|.
\end{equation}
Now, for each $\phi\in \mathcal{F}$ and $v\in S_\phi$,
\begin{align*}
|\{i\in I_\phi:v\in V(\tilde{M}_i)\}|&=\Big|\Big\{e\in \bigcup_{i\in I_\phi}\tilde{M}_i:v\in V(e)\Big\}\Big|=\Big|\Big\{e\in \Big(\hat{E}_\phi^*\cup \Big(\bigcup_{i\in I_\phi}\tilde{M}_{i,1}\Big)\Big):v\in V(e)\Big\}\Big|\\
&\overset{\ref{prop:frompartC3:vertex}}{=}|\{i\in I_\phi:v\in \hat{V}_i\}|-|\{i\in I_\phi:v\in T_i\}|.
\end{align*}
In combination with \eqref{lasteq}, this implies that \eqref{lastbalance} holds, as required.
This completes the proof of Lemma~\ref{keylemma:completion}.

\fi